\documentclass[10pt,twoside]{amsart}

\usepackage{amssymb}
\usepackage{amscd}
\usepackage{amsmath}
\usepackage{amsrefs}
\usepackage{stmaryrd}
\usepackage{xypic}
\usepackage[mathscr]{euscript}
\usepackage{bbold}
\usepackage{cancel}
\usepackage{amsfonts}
\usepackage{hyperref}

\newcommand{\ba}{\mathbb A}
\newcommand{\bq}{\mathbb Q}

\newcommand{\bz}{\mathbb Z}
\newcommand{\br}{\mathbb R}
\newcommand{\bc}{\mathbb C}
\newcommand{\bg}{\mathbb G}

\newcommand{\bn}{\mathbb N}
\newcommand{\bF}{\mathbb F}
\newcommand{\be}{\mathbb E}
\newcommand{\bs}{\mathbb S}

\newcommand{\cd}{\mathcal D}
\newcommand{\co}{\mathcal O}
\newcommand{\ca}{\mathcal A}
\newcommand{\cf}{\mathcal F}
\newcommand{\cm}{\mathcal M}

\newcommand{\fm}{\mathfrak m}

\newcommand{\X}{\mathcal X}

\newcommand{\mydet}{\mathrm{det}}

\DeclareMathOperator{\id}{id}
\DeclareMathOperator{\Hom}{Hom} 
\DeclareMathOperator{\End}{End}

\DeclareMathOperator{\fil}{Fil}
\DeclareMathOperator{\Cone}{Cone} 
\DeclareMathOperator{\fibre}{Fibre}
\DeclareMathOperator{\Tot}{Tot}
\DeclareMathOperator{\Spec}{Spec}
\DeclareMathOperator{\Ind}{Ind}
\DeclareMathOperator{\Gal}{Gal}

\newtheorem{theorem}{Theorem}[subsection]
\newtheorem{prop}{Proposition}[subsection]
\newtheorem{lemma}{Lemma}[subsection]
\newtheorem{corollary}{Corollary}[subsection]
\newtheorem{conjecture}{Conjecture}[subsection]
\newtheorem{definition}{Definition}[subsection]
\newtheorem{example}{Example}[subsection]
\newtheorem{notation}{Notation}[subsection]
\newtheorem{remark}{Remark}[subsection]

\usepackage{relsize} 
\usepackage[bbgreekl]{mathbbol} 
\DeclareSymbolFontAlphabet{\mathbb}{AMSb} 
\DeclareSymbolFontAlphabet{\mathbbl}{bbold} 
\newcommand{\Prism}{{\mathlarger{\mathbbl{\Delta}}}}
\newcommand{\prism}{{\mathlarger{\mathbbl{\Delta}}}}

\DeclareMathOperator{\Mod}{Mod}
\DeclareMathOperator{\CAlg}{CAlg}
\DeclareMathOperator{\Map}{Map}
\DeclareMathOperator{\Fun}{Fun}
\DeclareMathOperator{\Sym}{Sym}
\DeclareMathOperator{\QCoh}{QCoh}
\DeclareMathOperator{\Spf}{Spf}

\DeclareMathOperator{\can}{can}
\DeclareMathOperator{\rel}{rel}
\DeclareMathOperator{\ani}{ani}
\DeclareMathOperator{\syn}{syn}
\DeclareMathOperator{\add}{add}
\DeclareMathOperator{\et}{\acute et}

\newcommand{\csheaf}{\mathscr{C}^{\mathrm{syn}}}
\newcommand{\gsheaf}{\mathscr{G}^{\mathrm{syn}}}
\newcommand{\synlog}{\mathrm{slog}}
\newcommand{\cdf}{\mathcal{DF}}

\begin{document}
\title{The de Rham and the syntomic logarithm}

\author{Matthias Flach}
\author{Achim Krause}
\author{Baptiste Morin}

\begin{abstract} We define and study an integral refinement of the inverse of the Bloch-Kato exponential map which we call the de Rham logarithm. Our main tool to analyze the de Rham logarithm is the syntomic logarithm, a certain limit construction based on the theory of filtered prismatic cohomology initiated by Antieau, Krause and Nikolaus. We use the syntomic logarithm to prove a version of the Beilinson fibre square for all quasicompact, quasiseparated derived formal schemes. We also use our techniques to prove Conjecture $C_{EP}(\bq_p(n))$ of Fontaine and Perrin-Riou for all local fields $K/\bq_p$ and to compute the correction factor $C(X,n)$ introduced by Flach and Morin in their reformulation of the Bloch-Kato Tamagawa number conjecture for the Zeta function of a smooth projective scheme $X$ over a number ring. \end{abstract}
\date{\today}
\maketitle
\tableofcontents

\section{Introduction} In this article we generalize the logarithm map (from the multiplicative to the additive group of a $p$-complete ring) to a map on higher weight syntomic cohomology which we call the de Rham logarithm. The de Rham logarithm is an integral refinement of the inverse of the Bloch-Kato exponential map and will allow us to prove new properties of the latter. Our main tool to analyze the de Rham logarithm is another, more involved construction which we call the syntomic logarithm.  In particular we use the syntomic logarithm to prove a version of the Beilinson fibre square for all quasicompact, quasiseparated derived formal schemes. First, a brief introduction to these topics in historical order.

\subsection{The Bloch-Kato exponential map} Let $p$ be a prime number, $K/\bq_p$ be a finite extension and $V$ a $\bq_p$-representation of the absolute Galois group $G_K$ of $K$. The Bloch-Kato exponential map 
$$ \exp_V: D_{dR}(V)/D^0_{dR}(V)\to H^1(K,V)$$
was invented in \cite{bk88} in order to formulate the seminal "Tamagawa number conjecture" on special values of motivic L-functions. Its definition employs ideas of $p$-adic Hodge theory which were brand new at the time, most notably the "fundamental exact sequence" \cite{bk88}[Prop. 1.17 or Cor. \ref{funseq}]. With the classical logarithm appearing in the analytic class number formula, and the $p$-adic logarithm in its $p$-adic analogue, it seems reasonable that some kind of logarithm might also be needed for motivic cohomology groups $H^1(K,V)$ over local fields in general in order to formulate a special value conjecture (for the exact role that $p$-adic numbers play in describing a real number - the L-value - we must refer to \cite{bk88} and \cite{fpr91}). A reassuring result is that $\exp_V$ does reproduce the classical exponential map if $V$ is the Tate module of a commutative formal group over $\co_K$, e.g. if $V=\bq_p(1)$ \cite{bk88}[Example 3.10.1]. However, validating that $\exp_V$ is the correct map in other cases requires some control over its relation to integral lattices in the source and target. Even the basic case of Tate motives $V=\bq_p(n)$, $n\geq 2$  (where $D_{dR}(V)/D^0_{dR}(V)=K$ and $\exp_V$ is an isomorphism) has remained out of computational reach for general $K$. If $K/\bq_p$ is unramified then Bloch and Kato accomplished this comparison of lattices \cite{bk88}{Thm. 4.2] and thereby verified the Tamagawa number conjecture for the Riemann Zeta function at $s=n\geq 2$ (modulo powers of $2$).

The recent advent of prismatic cohomology \cite{bhatt-scholze, bhatt-lurie-22, akn23} has already revolutionized integral $p$-adic Hodge theory in a very short period of time. In this article we show that prismatic cohomology also gives a new construction of the inverse of the Bloch-Kato exponential map (for crystalline representations of $G_K$ arising from smooth, proper formal schemes $\mathscr{X}/\co_K$) which provides precise integral information. As a consequence we obtain, for example, the following generalization of \cite{bk88}{Thm. 4.2] to all local fields $K/\bq_p$. This was previously only known if $K/\bq_p$ is contained in a cyclotomic extension \cite{pr94}, or if $e(K/\bq_p)<p/4$ and $n=2$ \cite{flach-daigle-16}. In the statement of Thm. \ref{bktheorem} and throughout this article we use the formalism of determinants of perfect complexes \cite{Knudsen-Mumford-76}, \cite{bhatt-scholze-17}[App.].

\begin{theorem} Let $K/\bq_p$ be a finite extension with discriminant $D_K$ and residue field of cardinality $q$. Then for $n\geq 2$
$$\mydet_{\bz_p}\left(\exp_{\bq_p(n)}(\co_K)\right)\cdot (1-q^{-n})^{-1}\cdot (n-1)!^{[K:\bq_p]}\cdot D_K^{n-1}=\mydet_{\bz_p}^{-1}R\Gamma(K,\bz_p(n)) $$
inside $\mydet_{\bq_p}^{-1}R\Gamma(K,\bq_p(n))\simeq\mydet_{\bq_p}H^1(K,\bq_p(n))$.
\label{bktheorem} \end{theorem}

\begin{proof} See Thm. \ref{bktheorem2}.
\end{proof}

\begin{corollary} Conjecture $C_{EP}(\bq_p(n))$ of \cite{pr95}[App. C.2.9]  holds true for all local fields $K/\bq_p$ and all integers $n$. In particular, the Tamagawa number conjecture for the Dedekind Zeta function $\zeta_F(s)$ of any number field $F$ holds true at $s=n\in\bz$ if and only if it holds true at $s=1-n$.
\label{fneqcomp}\end{corollary} 

\begin{proof} Conjecture $C_{EP}(\bq_p(n))$ was shown to be equivalent to the formula in Thm. \ref{bktheorem} in equation (89) in the proof of \cite{Flach-Morin-16}[Prop. 5.34] (see also Remark \ref{errata} 3)). The fact that Conjecture $C_{EP}$ amounts to compatibility of the Tamagawa number conjecture with the functional equation is \cite{pr95}[App. C.3.3].
\end{proof}
\begin{corollary} The Tamagawa number conjecture for the Dedekind Zeta function $\zeta_F(s)$ holds true in the following cases.
\begin{itemize} \item[a)] $F$ is arbitrary and $n=0,1$.
\item[b)] $F/\bq$ is abelian and $n\in\bz$.
\item[c)] $F$ is totally real and $n<0$ is odd or $n>1$ is even.
\end{itemize}
\end{corollary}

\begin{proof} Cases a) and b) are well known \cite{flach03,flach06-3}. Case c) for $n<0$ is \cite{johnston-nickel}[Thm. 12.4] for the $p$-primary part for odd $p$ and similarly follows for $p=2$ from \cite{dasgupta}. Case c) for $n>1$ then follows from Cor. \ref{fneqcomp}.
\end{proof}

Our analysis of the Bloch-Kato exponential map proceeds in two steps. On the one hand we relate the Bloch-Kato exponential map to the fundamental fibre square, a version of the Beilinson fibre square \cite{beil14, antieau-mmn21}, and on the other hand we give a new, prismatic proof of this fibre square which yields more precise integral information than the proofs of the Beilinson fibre square in \cite{beil14} and \cite{antieau-mmn21}. 

\subsection{The fundamental fibre square} Recall that for a derived $p$-adic formal scheme $\mathscr{X}$ and $n\in\bz$ the {\em syntomic cohomology} is defined by the fibre sequence
\begin{equation} R\Gamma_{\syn}(\mathscr{X},\bz_p(n))\to\mathcal N^{\geq n}\prism_\mathscr{X}\{n\}\xrightarrow{\varphi\{n\}-\can}\prism_\mathscr{X}\{n\}\label{syndef}\end{equation}
where $\mathcal N^{\geq n}=\fil^{\geq n}_{\mathcal N}$ denotes the Nygaard filtration on Breuil-Kisin twisted absolute prismatic cohomology $\prism_\mathscr{X}\{n\}$, $\varphi\{n\}$ denotes the Frobenius and $\can$ the canonical map \cite{bhatt-lurie-22}[Construction 7.4.1]. 
We denote $p$-complete derived de Rham cohomology \cite{bhatt12} of $\mathscr{X}/\bz_p$ by $\widehat{\mathrm{dR}}_{\mathscr{X}}$ and we denote by
$\widehat{\mathrm{dR}}_{\mathscr{X}}^{<n}:=\widehat{\mathrm{dR}}_{\mathscr{X}}/ \fil^{\geq n}_{Hod}\widehat{\mathrm{dR}}_{\mathscr{X}}$
its Hodge truncation. For any contravariant functor $F$ from derived $p$-adic formal schemes to some stable $\infty$-category we set 
$$ F(\mathscr{X})^{\rel}:=\fibre\left(F(\mathscr{X})\to F(\mathscr{X}/p)\right) $$
where $\mathscr{X}/p:=\mathscr{X}\otimes^L\bF_p$.  The following proposition is immediate from the results of \cite{bhatt-lurie-22}.

\begin{prop} (The fundamental square and the de Rham logarithm) For a derived $p$-adic formal scheme $\mathscr{X}$ there is a commutative diagram
\begin{equation}\xymatrix{ 
R\Gamma_{\syn}(\mathscr{X},\bz_p(n)) \ar[rr]\ar[d] & & R\Gamma_{\syn}(\mathscr{X}/p,\bz_p(n))\ar[d] & \\ 
\fil^{\geq n}_{\mathcal N}\prism_\mathscr{X}\{n\}\ar[dd]_{\fil^{\geq n}\gamma^{dR}_{\prism}\{n\}} \ar@{-->}[dr]^\can \ar@{-->}[rr]& &\fil^{\geq n}_{\mathcal N}\prism_{\mathscr{X}/p}\{n\}\ar[dd]^(.3){\beta\circ\can} \ar@{-->}[dr]^\can &\\
 & \prism_\mathscr{X}\{n\}\ar@{-->}[rr] \ar@{-->}[dr]^{\gamma^{dR}_{\prism}\{n\}}& &\prism_{\mathscr{X}/p}\{n\}\ar@{-->}[dl]_\beta\\
\fil^{\geq n}_{Hod}\widehat{\mathrm{dR}}_{\mathscr{X}}\ar[rr]  && \widehat{\mathrm{dR}}_{\mathscr{X}}& 
}\label{beilsquare}\end{equation}
where $\fil^{\geq \bullet}\gamma^{dR}_{\prism}\{n\}$ was defined in \cite{bhatt-lurie-22}[Construction 5.5.3] and the isomorphism $\beta$ in \cite{bhatt-lurie-22}[Thm. 5.4.2]. We call the square of solid arrows the fundamental square. Taking horizontal fibres in (\ref{beilsquare}) we obtain a map
\begin{equation} \log_\mathscr{X}:R\Gamma_{\syn}(\mathscr{X},\bz_p(n))^{\rel} \xrightarrow{} \widehat{\mathrm{dR}}_{\mathscr{X}}^{<n}[-1],\label{comp1}\end{equation}
functorial in $\mathscr{X}$, which we call the de Rham logarithm.
\label{beilsquareprop}\end{prop}

\begin{proof}  See Prop. \ref{relativebeilsquare} for a slightly more general version of the fundamental square.
\end{proof}

The following theorem is the main result of this article. 

\begin{theorem}\label{main} Let $\mathscr{X}$ be a quasi-compact, quasi-separated derived formal scheme.
\begin{enumerate}
\item[a)] (The fundamental fibre square). The map
$$\log_\mathscr{X}:R\Gamma_{\syn}(\mathscr{X},\bz_p(n))^{\rel} \to \widehat{\mathrm{dR}}_{\mathscr{X}}^{<n}[-1]$$
is a rational isomorphism, i.e. the fundamental square is rationally Cartesian.
\item[b)] (Small Tate twist) For $0<n<p-1$ the map $\log_\mathscr{X}$ factors through an isomorphism
$$\log_{\mathscr{X}}:R\Gamma_{\syn}(\mathscr{X},\bz_p(n))^{\rel}\simeq\widehat{\mathrm{dR}}_{\mathscr{X}}^{<n,\rel}[-1]\xrightarrow{\iota} \widehat{\mathrm{dR}}_{\mathscr{X}}^{<n}[-1].$$
\item[c)] (Local volume computation) If $\mathscr{X}/\bz_p$ is proper of finite tor-amplitude and $(L_{\mathscr{X}/\bz_p})^\wedge_p\in\cd(\mathscr{X})$ is perfect then the source and target of $\log_\mathscr{X}$ are perfect complexes of $\bz_p$-modules and $\widehat{\mathrm{dR}}_{\mathscr{X}/p}^{<n}$ is $\bz_p$-perfect with finite cohomology. Moreover, there is an identity of $\bz_p$-lines in $\mydet^{-1}_{\bq_p}\widehat{\mathrm{dR}}_{\mathscr{X},\bq}^{<n}$
\begin{align*} &\mydet_{\bq_p}(\log_{\mathscr{X},\bq})\left(\mydet_{\bz_p}R\Gamma_{\syn}(\mathscr{X},\bz_p(n))^{\rel}\right)\\
=&\mydet_{\bq_p}(\iota_\bq)\left(\mydet^{-1}_{\bz_p}\widehat{\mathrm{dR}}_\mathscr{X}^{<n,\rel}\right)\cdot C_\infty(\mathscr{X},n)^{-1}
\notag\end{align*}
where
\begin{equation}\notag
\label{cinftydef}
C_\infty(\mathscr{X},n):=\prod_{ i\leq n-1;\, j}(n-1-i)!^{(-1)^{i+j}\mathrm{dim}_{\bq_p}H^j(\mathscr{X}_{\bq_p},L\Omega^i)}
\end{equation}
was defined in \cite{Flach-Morin-20}[Eq. (2)].
\item[d)] (The Bloch-Kato exponential map) Let $K/\bq_p$ be a finite extension and $\mathscr{X}/\co_K$ a smooth proper formal scheme. Choose a $G_K$-equivariant decomposition 
$$ R\Gamma_{\et}(\mathscr{X}_{\bc_p},\bq_p)\simeq \bigoplus_{i\geq 0}V^i[-i]$$
where $V^i$ is the $G_K$-representation $V^i:=H^i_{\et}(\mathscr{X}_{\bc_p},\bq_p)$.
Then there is a commutative diagram
$$\xymatrix{R\Gamma_{\syn}(\mathscr{X},\bz_p(n))^{\rel}\ar[r] \ar[d]^{\log_{\mathscr{X}}}& R\Gamma_{\syn}(\mathscr{X},\bz_p(n))\ar[r]^{c_{\et}}& \bigoplus_{i}R\Gamma(K,V^i(n))[-i]\\
\widehat{\mathrm{dR}}_{\mathscr{X}}^{<n}[-1]\ar[rr]^{c_{dR}} & &\bigoplus_i D_{dR}(V^i(n))/F^{\geq 0}[-i-1]
\ar[u]^{\bigoplus_i\exp_{V^i(n)}[-i]}}$$
where $c_{\et}$ is the composite of the \'etale comparison map \cite{bhatt-lurie-22}[Thm. 8.3.1]
$$\gamma^{\et}_{\syn}\{n\}:R\Gamma_{\syn}(\mathscr{X},\bz_p(n))\to R\Gamma_{\et}(\mathscr{X}_{\bq_p},\bz_p(n))$$
with the decomposition
$$ R\Gamma_{\et}(\mathscr{X}_{\bq_p},\bq_p(n))\simeq R\Gamma(K, R\Gamma_{\et}(\mathscr{X}_{\bc_p},\bq_p(n)))\simeq R\Gamma(K,\bigoplus_{i\geq 0}V^i(n)[-i])$$
and $c_{dR}$ is induced by the de Rham comparison isomorphism
$$\widehat{\mathrm{dR}}_{\mathscr{X},\bq_p}\simeq\bigoplus_{i\geq 0}H^i_{dR}(\mathscr{X}_{\bq_p}/K)[-i]\simeq \bigoplus_{i\geq 0} D_{dR}(V^i)[-i].$$
\end{enumerate}
\label{maintheorem}\end{theorem}

\begin{proof} Part a) is Cor. \ref{beilinsonfibresquare}, part b) is Thm. \ref{smalln}, part c) is Thm.  \ref{thm:synlogproper} and part d) is Thm. \ref{exptheorem}. \end{proof}

\begin{remark} (Previous work) For quasisyntomic $\mathscr{X}$ Theorem \ref{maintheorem} a) reproves \cite{antieau-mmn21}[Thm. 6.17 (55)] and part b) reproves \cite{antieau-mmn21}[Thm. 6.17 (56)] provided one can verify that $\log_{\mathscr{X},\bq}$ coincides with the map constructed in \cite{antieau-mmn21}. Note that the latter has also not been verified to coincide with Beilinson's original K-theoretic map \cite{antieau-mmn21}[Footnote 1]. We therefore refrain from calling the fundamental (fibre) square the Beilinson (fibre) square for the time being. If $\mathscr{X}/\bz_p$ is smooth and $n < p$ the fundamental square was implicit in the definition of Fontaine-Messing syntomic cohomology from the beginning \cite{kato87}[Rem. 3.5] but was first systematically discussed by Bloch, Esnault and Kerz \cite{Bloch-Esnault-Kerz-14}. In particular Theorem \ref{maintheorem} b) for Fontaine-Messing syntomic cohomology is due to them \cite{Bloch-Esnault-Kerz-14}[Thm. 5.4] (see Remark \ref{bke} for further details). 

For $\mathscr{X}/\co_K$ smooth and proper and $n\geq\dim(\mathscr{X})$ we expect that Theorem \ref{maintheorem} c) is implied by conjecture $C_{EP}(V^i(n))$ for all $i$ although the details remain to be written out. Conjecture $C_{EP}$ was proven by Benois and Berger \cite{bb05} for crystalline representations over unramified $K'$ base changed to cyclotomic $K/K'$ after previous work by Perrin-Riou for ordinary representations \cite{pr94}. 

In any case, our methods of proof are quite different from any of these references and will be explained in the next section.
\end{remark}

We briefly sketch the proof of Theorem \ref{maintheorem} d) for $\mathscr{X}=\Spf(\co_K)$ and the deduction of Theorem \ref{bktheorem}. The key observation, already made in \cite{antieau-mmn21}[Thm. 7.7], is that the fundamental fibre square for $\mathscr{X}=\Spf(\co_{\bc_p})$ recovers the fundamental exact sequence of $p$-adic Hodge theory. More precisely, for $n\geq 0$ there are $G_K$-isomorphisms 
\begin{align*} \gamma^{\et}_{\syn}\{n\}: R\Gamma_{\syn}(\Spf(\co_{\bc_p}),\bz_p(n))\simeq & R\Gamma(\bc_p,\bz_p(n))=\bz_p(n)\\
R\Gamma_{\syn}(\Spf(\co_{\bc_p}/p),\bz_p(n))_{\bq}\simeq & (B^+_{cris})^{p^{-n}\varphi=1}\\
\log_{\co_{\bc_p}}[1]: R\Gamma_{\syn}(\Spf(\co_{\bc_p}),\bz_p(n))_\bq^{\rel}[1]\simeq &\bigl(\widehat{\mathrm{dR}}_{\co_{\bc_p}}^{< n}\bigr)_\bq\simeq  \bigl(A_{cris}/\ker(\theta)^n\bigr)_\bq\simeq B^+_{dR}/F^n.
\end{align*} 
The top row in (\ref{beilsquare}) then yields the following corollary.
\begin{corollary} (The fundamental exact sequence) For $n\geq 0$ there is a short exact sequence of $G_K$-modules
\[0\to\bq_p(n)\to (B^+_{cris})^{p^{-n}\varphi=1}\to B^+_{dR}/F^n\to 0.\]
\label{funseq} \end{corollary}
Functoriality of syntomic cohomology gives a map
\begin{equation} R\Gamma_{\syn}(\Spf(\co_K),\bz_p(n))\to R\Gamma_{\syn}(\Spf(\co_{\bc_p}),\bz_p(n))^{hG_K}\simeq R\Gamma(K,\bz_p(n))\label{etcomp}\end{equation}
which (by functoriality of $c_{\et}$) coincides with the map $c_{\et}$ for $\Spf(\co_K)$. Similar considerations apply to the other terms in the fundamental exact sequence.
Hence a map of exact triangles
$$\minCDarrowwidth1em\begin{CD}
R\Gamma(K,\bq_p(n)) @>>>  R\Gamma(K,(B^+_{cris})^{p^{-n}\varphi=1}) @>>> R\Gamma(K,B^+_{dR}/F^n)\\
@AAc_{\et} A @AAA @AA\log_{\co_K}[1]A\\
R\Gamma_{\syn}(\Spf(\co_K),\bz_p(n))_\bq @>>> 0 @>>> R\Gamma_{\syn}(\Spf(\co_K),\bz_p(n))^{\rel}_\bq[1].
\end{CD}$$ 
Taking cohomology we obtain a commutative diagram (of isomorphisms if $n\geq 2$)
\begin{equation}\begin{CD}  K@>\exp_{\bq_p(n)}>> H^1(K,\bq_p(n))\\ @A\log_{\co_K}A\sim A @Ac_{\et} AA\\ H^1_{\syn}(\Spf(\co_K),\bz_p(n))^{\rel}_\bq@>\sim>> H^1_{\syn}(\Spf(\co_K),\bz_p(n))_\bq\end{CD}\label{logandexp}\end{equation}
which is Theorem \ref{maintheorem} d). On the other hand, by \cite{bhatt-mathew}[Thm. 1.8] the map (\ref{etcomp}) is an isomorphism if $n\geq 2$. Theorem \ref{bktheorem} readily follows from this fact together with Theorem \ref{maintheorem} c) for $\mathscr{X}=\Spf(\co_K)$.

\subsection{The syntomic logarithm} We define {\em additive syntomic cohomology} by the fibre sequence
\begin{equation} R\Gamma_{\add}(\mathscr{X},\bz_p(n))\xrightarrow{}\mathcal N^{\geq n}\prism_\mathscr{X}\{n\}\xrightarrow{\can}\prism_\mathscr{X}\{n\}\notag\end{equation}
analogous to (\ref{syndef}). Replacing the top row in (\ref{beilsquare}) by additive syntomic cohomology we obtain a map 
\begin{equation} \gamma_\mathscr{X}: R\Gamma_{\add}(\mathscr{X},\bz_p(n))^{\rel} \xrightarrow{} \widehat{\mathrm{dR}}_{\mathscr{X}}^{<n}[-1] \notag\end{equation}
analogous to $\log_\mathscr{X}$ and functorial in the formal scheme $\mathscr{X}$.

\begin{remark} (Concerning terminology) If $\mathscr{X}$ is quasisyntomic the main result of \cite{bms19} gives equivalences
\begin{align*} R\Gamma_{\syn}(\mathscr{X},\bz_p(n))\simeq &\mathrm{gr}^n_{\mathrm{Mot}}TC(\mathscr{X})[-2n]\simeq \mathrm{gr}^n_{\mathrm{Mot}}K_{\et}(\mathscr{X})^\wedge_p[-2n]\\
R\Gamma_{\add}(\mathscr{X},\bz_p(n))\simeq &\mathrm{gr}^n_{\mathrm{Mot}}TC^+(\mathscr{X})[1-2n]
\end{align*}
where $TC^+(\mathscr{X}):=THH(\mathscr{X})_{S^1}$ is the analogue over the sphere spectrum of cyclic homology $HC(\mathscr{X}):=HH(\mathscr{X})_{S^1}$ over $\bz$. Since cyclic homology is sometimes called additive K-theory we call $R\Gamma_{\add}(\mathscr{X},\bz_p(n))$ additive syntomic cohomology. 
There is also a clear analogy between $R\Gamma_{\add}(\mathscr{X},\bz_p(n)):= \prism_\mathscr{X}\{n\}/ \fil^n_{\mathcal N}  \prism_\mathscr{X}\{n\}[-1]$ and Hodge truncated (derived) de Rham cohomology over $\bz$
$$  \mathrm{dR}_{\mathscr{X}}^{<n}[-1]= \mathrm{dR}_{\mathscr{X}}/ \fil^n_{Hod}\mathrm{dR}_{\mathscr{X}}[-1]\simeq \mathrm{gr}^n_{\mathrm{Mot}}HC(\mathscr{X})[1-2n].$$
However,  this analogy is somewhat imperfect since $\prism_\mathscr{X}\{n\}$ depends on $n$ whereas  $\mathrm{dR}_{\mathscr{X}}$ does not. We therefore abandon the notation $ \mathrm{dR}_{\mathscr{X}/\bs}^{<n}=R\Gamma(\mathscr{X},L\Omega_{\mathscr{X}/\mathbb{S}}/F^n)$ used in \cite{Flach-Morin-20} for $R\Gamma_{\add}(\mathscr{X},\bz_p(n))[1]$, at least in this article.
\label{tc}\end{remark}

\begin{theorem}\label{addmain} Let $\mathscr{X}$ be a quasi-compact, quasi-separated derived formal scheme.
There is a commutative diagram
\begin{equation} \xymatrix{
G^{\geq k} R\Gamma_{\syn}(\mathscr{X},\bz_p(n))^{\rel}\ar[d]^{\iota_{\syn}} \ar[rr]^{G^{\geq k}\synlog_\mathscr{X}^{\rel}}_{\sim}&&G^{\geq k} R\Gamma_{\add}(\mathscr{X},\bz_p(n))^{\rel} \ar[d]^{\iota_{\add}}\\
R\Gamma_{\syn}(\mathscr{X},\bz_p(n))^{\rel}\ar[r]^(.6){\log_\mathscr{X}}&    \widehat{\mathrm{dR}}_{\mathscr{X}}^{< n}[-1]   &R\Gamma_{\add}(\mathscr{X},\bz_p(n))^{\rel}\ar[l]_(.6){\gamma_\mathscr{X}}\\
}
\label{synlogdiagram}\end{equation}
such that the following hold.
\begin{enumerate}
\item[a)] The map $G^{\geq k}\synlog_\mathscr{X}^{\rel}$, which we call the {\em syntomic logarithm} is an isomorphism.
\item[b)] The maps $\iota_{\syn}, \iota_{\add}$  are rational isomorphisms.
\item[c)] If $\mathscr{X}/\bz_p$ is proper of finite tor-amplitude and $(L_{\mathscr{X}/\bz_p})^\wedge_p$ is perfect then all complexes in (\ref{synlogdiagram}) are perfect complexes of $\bz_p$-modules. Moreover
$$\mydet_{\bq_p}(\log'_{\bq_p})\left(\mydet_{\bz_p}R\Gamma_{\syn}(\mathscr{X},\bz_p(n))^{\rel}\right)=\mydet_{\bz_p}R\Gamma_{\add}(\mathscr{X},\bz_p(n))^{\rel}$$
where $\log'_{\bq_p}=\iota_{\add,\bq_p}\circ G^{\geq k}\synlog_{\mathscr{X},\bq_p}^{\rel} \circ\iota^{-1}_{\syn,\bq_p}$.
\item[d)] The map $\gamma_\mathscr{X}$ is a rational isomorphism. If $\mathscr{X}/\bz_p$ is proper of finite tor-amplitude and $(L_{\mathscr{X}/\bz_p})^\wedge_p$ is perfect then
\begin{equation} \mydet_{\bq_p}(\gamma_{\mathscr{X},\bq})\left(\mydet_{\bz_p}R\Gamma_{\add}(\mathscr{X},\bz_p(n))^{\rel}\right)=\mydet_{\bq_p}(\iota_\bq)\bigl(\mydet^{-1}_{\bz_p}\widehat{\mathrm{dR}}_\mathscr{X}^{<n,\rel}\bigr) \cdot C_\infty(\mathscr{X},n)^{-1}.
\notag\end{equation}
\end{enumerate}
\end{theorem}

\begin{proof} Take $\mathscr{X}$ with the $p$-adic filtration in Thm. \ref{thm:synlog} to obtain a) and b) and use Thm. \ref{thm:synlogproper} for c). Part d) is immediate from \cite{bhatt-lurie-22}[Prop. 5.5.12] and \cite{bhatt-lurie-22}[4.7.2, 4.7.14, 5.5.8].
\end{proof}

As indicated by our notation, the source and target of the syntomic logarithm are sufficiently deep pieces of a certain descending filtration on syntomic and additive syntomic cohomology. The $G$-filtration, which also exists on de Rham cohomology, is the key innovation of this paper and will be explained below. A fairly straightforward analysis of graded pieces (of the closely related $F$-filtration) will give Theorem \ref{addmain} b) and c). 

Theorem \ref{addmain} a), b) and d) imply that $\log_{\mathscr{X}}$ is a rational isomorphisms and hence Theorem \ref{maintheorem} a). Theorem \ref{maintheorem} c) follows from Theorem \ref{addmain} c) and d). Finally, our proof of Theorem \ref{maintheorem} b) is independent of Theorem \ref{addmain} or the $F$- or $G$-filtration and will instead use Fontaine-Messing syntomic cohomology as defined in \cite{antieau-mmn21}.

In order to explain the $F$- and $G$-filtration and the syntomic logarithm we go back to the classical case $n=1$ and $\mathscr{X}=\Spf(\co_K)$ where
\begin{align*} R\Gamma_{\syn}(\mathscr{X},\bz_p(1))[1]\simeq &(\co_K^\times)^\wedge_p\simeq U^1:=1+\fm\\
R\Gamma_{\add}(\mathscr{X},\bz_p(1))[1]\simeq &\co_K.
\end{align*}
The exponential and logarithm are given by the usual power series which are formally inverse to each other and converge on sufficiently small $p$-adic disks. In particular, they yield an isomorphism
$$ \exp: R\Gamma_{\add}(\mathscr{X},\bz_p(1))^{\rel}[1]\simeq p\cdot\co_K \simeq 1+p\cdot \co_K\simeq R\Gamma_{\syn}(\mathscr{X},\bz_p(1))^{\rel}[1]:\log$$ 
if $p$ is odd (if $p=2$ replace $p$ by $p^2$). However, there is an alternative construction of this isomorphism which uses the natural filtrations 
\begin{equation} \co_K\supset \fm\supseteq \cdots\supseteq \fm^i\supseteq \cdots \label{addfil}\end{equation}
$$ \co_K^\times\supseteq U^1:=1+\fm\supseteq \cdots\supseteq U^i:=1+\fm^i\supseteq \cdots$$
on the additive and the multiplicative group. It is immediate that the map $\alpha\mapsto \alpha-1$ gives a group isomorphism $U^i/U^{i+1}\simeq\fm^i/\fm^{i+1}$ for $i\geq 1$ but in fact it gives slightly more, namely a group isomorphism
\begin{equation} U^i/U^{i+c}\simeq\fm^i/\fm^{i+c};\quad \alpha\mapsto \alpha-1 \label{point1}\end{equation}
for $i\geq c\geq 1$. On the other hand it is easy to analyze the effect of multiplication by $p$ on both filtrations. By \cite{serre61}[1.7, Cor. 1] one has 
$$ p:U^i \simeq U^{i+e}$$
if $i>e/(p-1)$ and it is also clear that 
$$p:\fm^i\simeq \fm^{i+e}$$
if $i\geq 0$. So for $i>e/(p-1)$ and $k\geq 1+(c-i)/e\geq 1$ one obtains a commutative diagram of isomorphisms
\begin{equation}\begin{CD} U^{i}/U^{i+c} @>p^k>> U^{i+ke}/U^{i+ke+c}@>p>> U^{i+(k+1)e}/U^{i+(k+1)e+c}\\
@V\log_c VV @VV\alpha\mapsto \alpha-1 V  @VV\alpha\mapsto \alpha-1 V \\
\fm^{i}/\fm^{i+c} @>p^k>> \fm^{i+ke}/\fm^{i+ke+c}@>p>> \fm^{i+(k+1)e}/\fm^{i+(k+1)e+c}
\end{CD}\label{multiplicationbyp}\end{equation}
which defines the isomorphism $\log_c$ and at the same time shows that $\log_c$ is independent of the choice of $k$. This in turn implies commutativity of
$$\begin{CD} U^{i}/U^{i+c+1} @>>> U^{i}/U^{i+c} \\
@V\log_{c+1} VV @V\log_c VV  \\
\fm^{i}/\fm^{i+c+1}@>>> \fm^{i}/\fm^{i+c} 
\end{CD}$$ 
and completeness of the filtrations then yields an isomorphism 
$$\log=\varprojlim_c\log_c:U^{i}\simeq \fm^{i}$$
for $i>e/(p-1)$. 
That one has indeed recovered the classical logarithm follows from diagram (\ref{multiplicationbyp}) and the formula
$$\log(\alpha)=\left.\frac{d}{ds}\alpha^s\right\vert_{s=0}=\lim_{k\to\infty} \frac{\alpha^{p^k}-1}{p^k}.$$
We generalize this to higher Tate twists as follows. 
The filtration (\ref{addfil}) is multiplicative and therefore turns $\co_K$ into a filtered ring $F^{\geq \star}\co_K$. By one of the main results of \cite{akn23} such a  filtration on a ring induces a multiplicative bifiltration $F^{\geq \star}{\mathcal N}^{\geq \star}\prism_{\co_K}\{n\}$ on (Breuil-Kisin twisted) absolute prismatic cohomology so that
$\varphi\{n\}$ in filtration degree $i$ factors  \cite{akn23}[Footnote 21]
$$ \varphi\{n\}: F^{\geq i} {\mathcal N^{\geq n}}\prism_{\co_K}\{n\} \to F^{\geq pi} \prism_{\co_K}\{n\}\to F^{\geq i}\prism_{\co_K}\{n\}.$$
Since $p\geq 2$ this implies  
$$ F^{[i,i+c[}(\varphi\{n\})=0$$
for $i\geq c\geq 1$. The identity map on $F^{\geq \star}{\mathcal N}^{\geq \star}\prism_{\co_K}\{n\}$ therefore induces an isomorphism 
$$ F^{[i,i+c[}R\Gamma_{\syn}(\co_K,\bz_p(n))\simeq F^{[i,i+c[}R\Gamma_{\add}(\co_K,\bz_p(n)),\quad i\geq c\geq 1,$$
a generalization of (\ref{point1}). Unfortunately, if $n\geq 2$ it may not be true that multiplication by $p$ factors through a map
$\tilde{p}: F^{\geq i}C\to F^{\geq i+e}C$
for $C$ either of the two complexes since $F^{[i,i+e[}C$ may not be annihilated by $p$. We remedy this defect by replacing the multiplicative bifiltration $F^{\geq \star}{\mathcal N}^{\geq \star}\prism_{\co_K}\{n\}$ with the multiplicative bifiltration $G^{\geq \star}{\mathcal N}^{\geq \star}\prism_{\co_K}\{n\}$ where
$$ G^{\geq i}{\mathcal N}^{\geq j}\prism_{\co_K}\{n\}:=\mathrm{colim}_{a+be\geq i, b\geq j} F^{\geq a}{\mathcal N}^{\geq b}\prism_{\co_K}\{n\}.$$
Now $G^{[i,i+e[}{\mathcal N}^{\geq j}\prism_{\co_K}\{n\}$ is a module over 
$$G^{[0,e[}{\mathcal N}^{\geq 0}\prism_{\co_K}\{0\}\simeq F^{[0,e[}\mathrm{gr}^0_{\mathcal N}\prism_{\co_K}\{0\}\simeq F^{[0,e[}\co_K\simeq\co_K/p$$ and hence annihilated by $p$. We also show that there is an induced filtration $G^{\geq i} R\Gamma_{\syn}(\co_K,\bz_p(n))$ and an isomorphism
\begin{equation} G^{[i,i+c[}R\Gamma_{\syn}(\co_K,\bz_p(n))\simeq G^{[i,i+c[}R\Gamma_{\add}(\co_K,\bz_p(n)),\quad i\geq c\geq 1,\label{gtriv}\end{equation}
for $i$ large enough, and that moreover multiplication by $p$ factors through an {\em isomorphism} 
$$\tilde{p}: G^{\geq i}C\xrightarrow{\sim} G^{\geq i+e}C$$
for both $C=R\Gamma_{\syn}(\co_K,\bz_p(n))$ and $C=R\Gamma_{\add}(\co_K,\bz_p(n))$ and $i$ large enough. We moreover show that both filtrations are complete. This means that all ingredients are in place for the above construction of a logarithm isomorphism to go through, albeit in the derived (infinity) category rather than the category of abelian groups. We finally show that the so constructed syntomic logarithm fits into the commutative diagram (\ref{synlogdiagram}) by extending both the $F$- and $G$- bifiltration to Hodge filtered derived de Rham cohomology.

\begin{remark} (Concerning $\log_{\mathscr{X}}$)  The de Rham logarithm and the syntomic logarithm agree rationally but are of quite different nature. The map $\log_{\mathscr{X}}$ is the map of main interest because it is an integral version of the Bloch-Kato logarithm. Like the Bloch-Kato exponential map the de Rham logarithm is easy to construct but hard to analyze. The maps $G^{\geq k}\synlog_\mathscr{X}^{\rel}$ on the other hand are hard to construct and of auxilliary nature but very useful by virtue of being isomorphisms.
\end{remark}

\subsection{Organisation of this article} In section \ref{prismaticsection} we give the definition of prismatic cohomology of graded and filtered animated rings (and derived formal schemes with filtered structure sheaf), expanding on the discussion in \cite{akn23}[Sec. 10]. In section \ref{beilsection} we establish a filtered version of the fundamental square and the de Rham logarithm. Section \ref{synlogsection} is the heart of this article. Here we construct the syntomic logarithm following the outline given above and prove Thm.  \ref{addmain}, or rather a slightly more general version relative to a $\delta$-ring. In section \ref{smalltwist} we define Fontaine-Messing syntomic cohomology relative to a $\delta$-ring and prove a generalization of Thm. \ref{main} b) relative to a $\delta$-ring. In section \ref{bksection} we prove Thm. \ref{main} d). Finally, in section \ref{zetasection} we discuss an application of Thm. \ref{main} c) to special value conjectures for Zeta functions \cite{Flach-Morin-16} and prove Thm. \ref{bktheorem}.

\subsection{Acknowledgements} We would like to thank Lars Hesselholt for bringing the authors together at a master class in Kopenhagen in January 2023 to report about their respective work on Zeta-values \cite{Flach-Morin-16,Flach-Morin-20} and the K-theory of $\bz/p^n$ \cite{akn23,akn24}. It was on this occasion where the idea arose that the theory of filtered prismatic cohomology initiated in \cite{akn23} might be useful for the local volume computation in Thm. \ref{main} c). We would also like to thank Ben Antieau, H\'el\`ene Esnault, Akhil Mathew, Matthew Morrow and Thomas Nikolaus for helpful comments regarding this project.

\subsection{Conventions} In this article all rings are implicitly assumed commutative and we fix a prime $p$ throughout.
 A {\em (derived) formal scheme} is a (derived) formal scheme with respect to the ideal $(p)$. A  {\em (derived) formal stack} is a sheaf of anima for the fpqc topology on the opposite of the category of $p$-nilpotent (animated) rings. Formal stacks form a full subcategory of derived formal stacks \cite{bhatt-lurie-22b}[Notation 8.1]. For an animated ring $R$ we denote by $\Spf(R)$ the functor on $p$-nilpotent animated rings corepresented by $R$. 
 
 We refer to \cite{akn23}[App. A] for a discussion of the category $\cd(\mathfrak{X})$ of quasi-coherent sheaves on a (derived) formal stack $\mathfrak{X}$. If $R$ is an animated ring $\cd(R):=\cd(\Spf(R))$ is the category $(\Mod_R)^\wedge_p$ of $p$-complete objects in $\Mod_R$. We denote by 
 $$\cdf(\mathfrak{X}):=\Fun(\bz^{op},\cd(\mathfrak{X}))$$ the category of filtered objects in $\cd(\mathfrak{X})$, and for $F^{\geq\star}\mathcal{R}\in\CAlg(\cdf(\mathfrak{X}))$ we put $\cdf(F^{\geq\star}\mathcal{R}):=\Mod_{F^{\geq\star}\mathcal{R}}(\cdf(\mathfrak{X}))$.

We introduce the following definition of "$p$-quasisyntomic" in order to state Prop. \ref{deRham} and Thm. \ref{packageproperties} in optimal generality. Otherwise quasisyntomicity will only play a minor role in this article (in the definition of Fontaine-Messing syntomic cohomology in section \ref{FM}). If $A=\bz_p$ and $R$ is discrete with bounded $p$-power torsion then our definition coincides with \cite{bhatt-lurie-22}[C.6].

\begin{definition} Let $A\to R$ be a morphism of animated rings. We call $R/A$ \underline{$p$-quasisyntomic} if the following hold.
\begin{enumerate}
\item[a)] $A^\wedge_p$ is $m$-truncated for some $m\geq 0$.
\item[b)] $R^\wedge_p\in\cd(A)$ has finite $p$-complete tor-amplitude
\item[c)] $(L_{R/A})^\wedge_p\in\cd(R)$ has $p$-complete tor-amplitude in $[-1,0]$.
\end{enumerate}
\label{lcidef}\end{definition}
The prime $p$ being fixed we usually shorten this to "quasisyntomic" and we will omit "$p$-complete" in "$p$-complete tor-amplitude" and "$p$-completely flat". We will often omit the notation $(-)^\wedge_p$ since $p$-completion is automatic when we consider modules as objects of $p$-complete derived categories.

\section{Prismatic cohomology of graded and filtered rings}\label{prismaticsection}

In this section we recall the construction of the bifiltration $F^{\geq \star}{\mathcal N}^{\geq \star}\prism_{\mathscr{X}}\{n\}$ from \cite{akn23}[Sec. 10] and define a bifiltration $F^{\geq \star}{\fil}^{\geq \star}_{Hod}\widehat{\mathrm{dR}}_{\mathscr{X}}$ in an entirely parallel way. In order to illustrate the basic idea we give a direct construction of $F^{\geq \star}{\fil}^{\geq \star}_{Hod}\widehat{\mathrm{dR}}_{R/A}$ in the affine case in \ref{sec:gradeddeRham}. In \ref{filtrations} we briefly recall the relation between filtrations and stacks over $\widehat{\ba}^1/\widehat{\bg}_m$ and define $F^{\geq \star}{\fil}^{\geq \star}_{Hod}\widehat{\mathrm{dR}}_{\mathscr{X}/A}$ for a derived formal scheme $\mathscr{X}$ with filtered structure sheaf. Constructing $F^{\geq \star}{\mathcal N}^{\geq \star}\prism_{\mathscr{X}/A}\{n\}$ requires a definition of prismatic cohomology relative to a $\delta$-stack, based in turn on a theory of prismatic cohomology relative to a $\delta$-ring, as first laid out in \cite{akn23}.  In order to accommodate $F^{\geq \star}{\fil}^{\geq \star}_{Hod}\widehat{\mathrm{dR}}_{\mathscr{X}/A}$ we need to slightly expand the foundational framework of \cite{akn23} by working directly with prismatic $F$-gauges rather than their global sections. We do this in \ref{deltapairs}. In \ref{deltastacks} we construct prismatic cohomology of graded and filtered rings and formal schemes following the pattern laid out in \ref{sec:gradeddeRham} and \ref{filtrations}.

\subsection{Derived de Rham cohomology of graded and filtered rings} \label{sec:gradeddeRham}For the definition of $p$-complete Hodge filtered derived de Rham cohomology $\fil^{\geq \star}_{Hod}\widehat{\mathrm{dR}}_{R/A}$ over a fixed base ring $A$ we refer to \cite{bhatt-lurie-22}[App. E]. Here we need to consider $\fil^{\geq \star}_{Hod}\widehat{\mathrm{dR}}_{R/A}$ more fully as a functor of pairs $R/A$. The following definition is justified by the fact that $\fil^{\geq \star}_{Hod}\widehat{\mathrm{dR}}_{R/A}$ commutes with sifted colimits of pairs (see the proof of \cite{akn23}[Lemma 3.16]). For the general process of animation of $1$-categories with compact projective generators we refer to \cite{scholze-cesnavicius}[5.1.4].

\begin{definition} ($p$-complete Hodge filtered derived de Rham cohomology) Let $\mathrm{Pairs}$ be the category of ring homomorphisms $A\to R$ and $\mathrm{Poly}_{\mathrm{Pairs}}$ the full subcategory of objects
$$\bz[X_1,\dots,X_n]\to \bz[X_1,\dots,X_n, Y_1,\dots,Y_m]$$ (the free objects on finitely many variables). 
Let $\mathrm{Pairs}^{\ani}$ be the free completion of $\mathrm{Poly}_{\mathrm{Pairs}}$ under sifted colimits (equivalent to the category of morphisms of animated rings). Define $\fil^{\geq \star}_{Hod}\widehat{\mathrm{dR}}_{R/A}$ as the left Kan extension of the functor
$$\mathrm{Poly}_{\mathrm{Pairs}} \to \CAlg(\Fun(\bn^{op},\cd(?)));\quad\quad R/A\mapsto \widehat{\Omega}^{\geq \star}_{R/A}$$
along $\mathrm{Poly}_{\mathrm{Pairs}}\to\mathrm{Pairs}^{\ani}$. Here $\widehat{\Omega}^\bullet_{R/A}$ is the algebraic de Rham complex with $p$-completed terms and the target is the category of pairs $(A,F^{\geq\star}M)$ with $A$ an animated ring and $F^{\geq\star}M\in\CAlg(\Fun(\bn^{op},\cd(A)))$.
\label{padicderiveddeRham}\end{definition}

For a morphism of animated rings $R\to R^0$ we denote by $R^\bullet$ its Cech conerve.
\begin{prop} Let $A\to R$ be a morphism of animated rings.
\begin{itemize}
\item[a)] (Base change) For any morphism of animated rings $A\to A'$ the natural map 
$$ \fil^{\geq \star}_{Hod}\widehat{\mathrm{dR}}_{R/A}\otimes_AA'\to \fil^{\geq \star}_{Hod}\widehat{\mathrm{dR}}_{R\otimes_AA'/A'}$$
is an equivalence. Here $\otimes$ denotes the $p$-complete tensor product.
\item[b)] (\'Etale descent in $R$) For an \'etale faithfully flat map $R\to R^0$ the natural map
$$\fil^{\geq \star}_{Hod}\widehat{\mathrm{dR}}_{R/A}\to \Tot \fil^{\geq \star}_{Hod}\widehat{\mathrm{dR}}_{R^\bullet/A}$$
is an equivalence.
\item[c)] (Descent for quasisyntomic $R/A$) Assume $A\to R$ is quasisyntomic.
\begin{itemize}
\item[c1)] (Quasisyntomic descent in $R$) If $R\to R^0$ is quasisyntomic and faithfully flat then the natural map
$$\fil^{\geq \star}_{Hod}\widehat{\mathrm{dR}}_{R/A}\to \Tot \fil^{\geq \star}_{Hod}\widehat{\mathrm{dR}}_{R^\bullet/A}$$
is an equivalence.
\item[c2)] (Flat descent) If $A\to A^0$ is faithfully flat and $R^0:=R\otimes_AA^0$ then the natural map
$$\fil^{\geq \star}_{Hod}\widehat{\mathrm{dR}}_{R/A}\to \Tot \fil^{\geq \star}_{Hod}\widehat{\mathrm{dR}}_{R^\bullet/A^\bullet}$$
is an equivalence.
\end{itemize}
\item[d)] (Invariance under quasi-\'etale extensions) For $A'\to A$ with vanishing $L_{A/A'}\in\cd(A)$ the natural map
$$  \fil^{\geq \star}_{Hod}\widehat{\mathrm{dR}}_{R/A'}\to \fil^{\geq \star}_{Hod}\widehat{\mathrm{dR}}_{R/A}$$
is an equivalence.
\end{itemize}
\label{deRham}\end{prop}

\begin{proof} Base change can be directly checked for $\widehat{\Omega}^{\geq \star}_{R/A}$ and follows by compatibility of both sides with sifted colimits (see also \cite{bhatt12}[Prop. 2.7] for the discrete case). Since $\mathrm{gr}^i_{Hod}\widehat{\mathrm{dR}}_{R/A}\simeq L\widehat{\Omega}^i_{R/A}[-i]$ satisfies descent for $(A\to R)\to (A^0\to R^0)$ with $A\to A^0$ flat and $R\to R^0$ faithfully flat \cite{akn23}[Lemma 3.14] (flat descent for short),  in order to prove b) and c) it suffices to prove the respective descent property for $\fil^{\geq 0}_{Hod}\widehat{\mathrm{dR}}_{R/A}=\widehat{\mathrm{dR}}_{R/A}$. By $p$-completeness and base change, in order to prove descent for $\widehat{\mathrm{dR}}_{R/A}$ it suffices to do so for $R$ and $A$ of characteristic $p$. By \cite{bhatt12}[Cor. 3.14] (stated for discrete $R$ and $R^0$ but true in general) one has an isomorphism
$$\widehat{\mathrm{dR}}_{R^0/A}\simeq \widehat{\mathrm{dR}}_{R/A}\otimes_{R^{(1)}}R^{0,(1)}$$
for \'etale $R\to R^0$ where $R^{(1)}:=R\otimes_{A,\mathrm{Frob}_A}A$. \'Etale descent then follows from faithfully flat descent for modules \cite{lurieSAG}[Cor. D.6.3.4]. Part c) can be proven along the lines of \cite{bhatt-lurie-22}[E.17]. The comparison map 
\begin{equation}\widehat{\mathrm{dR}}_{R/A}\to \Tot \widehat{\mathrm{dR}}_{R^\bullet/A^\bullet}\label{d11}\end{equation}
is a map of filtered complexes using the conjugate filtration \cite{bhatt-lurie-22}[E.7]. The associated graded $\mathrm{gr}^{conj}_i\widehat{\mathrm{dR}}_{R/A}\simeq L\widehat{\Omega}^i_{R^{(1)}/A}[-i]$ satisfy flat descent as already remarked above. Note here that $R^{(1)}\to R^{0,(1)}$ is again faithfully flat, by base change along $\mathrm{Frob}_A$ in c1) and since $A\to A^0$ is faithfully flat and hence so is
$$R^{(1)}\to R^{(1)}\otimes_AA^0\simeq R^0\otimes_{A^0,\mathrm{Frob}_{A^0}}A^0=R^{0,(1)}$$ 
in c2). It follows in both cases that all $R^{k,(1)}$ are $m$-truncated if $R^{(1)}$ is $m$-truncated for some $m\geq 0$. Such $m$ exists since $R^{(1)}/A$ is again quasisyntomic. In fact all $R^{k,(1)}/A^k$ are quasisyntomic so that $L\widehat{\Omega}^i_{R^{k,(1)}/A^k}[-i]$ has $R^{k,(1)}$-tor-amplitude in $[0,i]$, hence is concentrated in cohomological degrees $\geq -m$. It follows that all $\fil^{conj}_i\widehat{\mathrm{dR}}_{R^k/A^k}$ are concentrated in cohomological degrees $\geq -m$ and that therefore totalization commutes with the colimit of the conjugate filtration. This concludes the proof that (\ref{d11}) is an equivalence. Part d) is \cite{bhatt12}[Lemma 8.3 (5)].
\end{proof}

Let $\Mod_\bz$ be the presentable, stable $\infty$-category of $\bz$-modules. We denote by $\bz^\delta$ the discrete category on the set $\bz$ with symmetric monoidal structure given by addition. The category $\Fun(\bz^{\delta},\Mod_\bz)$ of $\bz$-graded $\bz$-modules with Day convolution symmetric monoidal structure may be viewed in several equivalent ways. $\Fun(\bz^{\delta},\Mod_\bz)$ is also the category of algebraic $\bg_m$-representations over $\bz$ or the category of $\co(\bg_m)$-comodules in $\Mod_\bz$. For $(i\mapsto M^i)\in \Fun(\bz^{\delta},\Mod_\bz)$ we denote by
$$M:=\bigoplus_{i\in\bz}M^i;\quad M\to \co(\bg_m)\otimes M$$
the underlying $\co(\bg_m)$-comodule. We often denote a graded module just by $M$ and leave the grading implicit.

A graded ring is an object of $\CAlg(\Fun(\bz^\delta,\Mod_\bz^\heartsuit))$ and an animated graded ring is an object of 
$\CAlg(\Fun(\bz^\delta,\Mod_\bz^\heartsuit))^{\ani}$. An animated graded ring $A$ gives rise to an object $A^{gr}$ of $\CAlg(\Fun(\bz^\delta,\Mod_\bz))$ and we abbreviate $$\Mod_{A^{gr}}:=\Mod_{A^{gr}}(\Fun(\bz^\delta,\Mod_\bz)).$$

\begin{definition} (Graded derived de Rham cohomology) Let $A\to R$ be a morphism of animated graded rings. Denote by 
$$ \fil^{\geq \star}_{Hod}\widehat{\mathrm{dR}}_{ R/A}\in\CAlg(\Fun(\bn^{op},\Mod_{A^{gr}})^\wedge_p)$$
the multiplicative graded filtration whose $\co(\bg_m)$-comodule structure is obtained by functoriality along 
$$\begin{CD} A @>>> \co(\bg_m)\otimes A\\ @VVV@VVV\\  R @>>> \co(\bg_m)\otimes R\end{CD}$$ 
combined with base change (Prop. \ref{deRham} a)). The commutative diagram
$$\begin{CD} A\otimes \widehat{\mathrm{dR}}_{R/A} @>>> \widehat{\mathrm{dR}}_{R/A}\\
@VVV@VVV\\
(\co(\bg_m)\otimes A)\otimes \widehat{\mathrm{dR}}_{\co(\bg_m)\otimes R/\co(\bg_m)\otimes A} @>>> \widehat{\mathrm{dR}}_{\co(\bg_m)\otimes R/\co(\bg_m)\otimes A}\end{CD}$$
shows that the $A$-module structure of $\widehat{\mathrm{dR}}_{R/A}$ is compatible with the $\co(\bg_m)$-comodule structure of both objects so that we do get an object of $\Mod_{A^{gr}}$. 
\label{gradeddeRham}\end{definition}

Viewing $\bz$ as an ordered set, i.e. category with morphisms $i\to j$ for $j\geq i$, the category $\Fun(\bz^{op},\Mod_\bz)$ is the category of decreasing filtrations $F^{\geq\star} M$, again with Day convolution symmetric monoidal structure. There is a $t$-exact symmetric monoidal equivalence
\begin{equation}  \mathrm{Rees}: \Fun(\bz^{op},\Mod_\bz) \simeq \Mod_{\bz[t]^{gr}}(\Fun(\bz^\delta,\Mod_\bz))\label{Reeseq}\end{equation}
sending a filtered $\bz$-module $F^{\geq \star} M$ to the graded $\bz[t]$-modules with underlying $\co(\bg_m)$-comodule
$$ \mathrm{Rees}(F^{\geq \star} M):=\bigoplus_{i\in\bz}F^{\geq i}M\cdot t^{-i} $$
where $t$ is an element of grading weight $-1$.  

A filtered ring is an object of $\CAlg(\Fun(\bz^{op},\Mod_\bz^\heartsuit))$ and an animated filtered ring is an object of 
$\CAlg(\Fun(\bz^{op},\Mod_\bz^\heartsuit))^{\ani}$.  An animated filtered ring $F^{\geq\star} A$ gives rise to an object of $\CAlg(\Fun(\bz^{op},\Mod_\bz))$ and we abbreviate 
$$\Mod_{F^{\geq\star} A}:=\Mod_{F^{\geq\star} A}(\Fun(\bz^{op},\Mod_\bz))\simeq \Mod_{ \mathrm{Rees}(F^{\geq \star} A)^{gr}}(\Fun(\bz^\delta,\Mod_\bz)).$$
\begin{example} The initial filtered ring, given by the filtration
$$ F^{\geq i}\bz=\begin{cases} \bz & i\leq 0 \\ 0 & i\geq 1\end{cases} $$
is the unit for the symmetric monoidal structure of $\Fun(\bz^{op},\Mod_\bz)$ and $$\mathrm{Rees}(F^{\geq\star}\bz)=\bz[t].$$
\label{initial}\end{example}

\begin{definition} (Filtered derived de Rham cohomology) Let $F^{\geq\star} A\to F^{\geq\star} R$ be a morphism of animated filtered rings.
Denote by 
$$ F^{\geq\star}\fil^{\geq \star}_{Hod}\widehat{\mathrm{dR}}_{F^{\geq\star}R/F^{\geq\star} A}\in\CAlg(\Fun(\bn^{op},\Mod_{F^{\geq\star}A})^\wedge_p)$$
the multiplicative bifiltration obtained under the Rees equivalence (\ref{Reeseq}) from
$$ \fil^{\geq \star}_{Hod}\widehat{\mathrm{dR}}_{\mathrm{Rees}(F^{\geq\star}R)/\mathrm{Rees}(F^{\geq\star} A)}\in\CAlg(\Fun(\bn^{op},\Mod_{\mathrm{Rees}(F^{\geq\star}A)^{gr}})^\wedge_p)$$
as defined in Def. \ref{gradeddeRham}.
\label{filtereddeRham}\end{definition}

\begin{definition} For a filtration $F^{\geq\star} M\in \Fun(\bz^{op},\Mod_\bz)$ we define the \\ \underline{underlying module} of the filtration as the colimit
$$ F^\infty M:=\mathrm{colim}_{i\to -\infty}F^{\geq i} M$$
and the \underline{associated graded} as
$$ \mathrm{gr}^i_FM:=F^{\geq i} M/F^{\geq i+1} M.$$
We call a filtration $F^{\geq\star}M$ 
\begin{enumerate}
\item[a)] \underline{$\bn^{op}$-indexed} if $\mathrm{gr}^i_FM=0$ for $i<0$.
\item[b)] \underline{complete} if $\varprojlim_jF^{\geq j}M=0$.
\item[c)] \underline{bounded} (in some interval $[r,s]$) if $F^{\geq\star}M$ is complete and $\mathrm{gr}^i_FM=0$ for all but finitely many $i$ (all but $i\in [r,s]$).
\item[d)] \underline{strict} if all $F^{\geq j}M$ are discrete and all maps $F^{\geq j+1}M\to F^{\geq j}M$ are injective.
 \end{enumerate}
\end{definition} 

We have isomorphisms of $\co(\bg_m)$-comodules 
\begin{align}\mathrm{Rees}(F^{\geq \star} M)\otimes_{\bz[t]}\bz[t,t^{-1}]\simeq\  &F^\infty M\otimes_\bz\bz[t,t^{-1}]\label{reesunder}\\
\mathrm{Rees}(F^{\geq \star} M)\otimes_{\bz[t]}\bz\simeq\ & \bigoplus_{i\in\bz}\mathrm{gr}^i_FM\label{reesgraded}\end{align}
where in the last tensor product $t\mapsto 0$. Note that $F^\infty M$ can be (functorially and multiplicatively) recovered from (\ref{reesunder}) as the $\bg_m$-invariants, i.e. the weight $0$ graded part. In fact one has a symmetric monoidal $t$-exact equivalence
\begin{equation}  \Mod_\bz \simeq \Mod_{\bz[t,t^{-1}]^{gr}}(\Fun(\bz^\delta,\Mod_\bz));\quad\quad M\mapsto M\otimes_\bz\bz[t,t^{-1}].\label{modequivalence}\end{equation}
Also note that if $F^{\geq\star} M$ is $\bn^{op}$-indexed then $F^\infty M=F^{\geq 0}M$.

\begin{prop} Let $F^{\geq\star} A\to F^{\geq\star} R$ be a morphism of animated filtered rings. 

a) We have natural isomorphisms of $F^{\infty} A$-modules
$$ F^{\infty}\fil^{\geq \star}_{Hod}\widehat{\mathrm{dR}}_{F^{\geq\star}R/F^{\geq\star} A}\simeq \fil^{\geq \star}_{Hod}\widehat{\mathrm{dR}}_{F^{\infty}R/F^{\infty} A}$$
and of graded $\bigoplus_i\mathrm{gr}^i_FA$-modules
$$ \bigoplus_{i\in\bz}\mathrm{gr}^i_F\fil^{\geq \star}_{Hod}\widehat{\mathrm{dR}}_{F^{\geq\star}R/F^{\geq\star} A}\simeq \fil^{\geq \star}_{Hod}\widehat{\mathrm{dR}}_{\bigoplus_i\mathrm{gr}^i_FR/\bigoplus_i\mathrm{gr}^i_FA}$$

b) If $F^{\geq\star} A$ and $F^{\geq\star} R$ are $\bn^{op}$-indexed, so is $F^{\geq\star}\fil^{\geq \star}_{Hod}\widehat{\mathrm{dR}}_{F^{\geq\star}R/F^{\geq\star} A}$. Hence in this case
$$ F^{\geq 0}\fil^{\geq \star}_{Hod}\widehat{\mathrm{dR}}_{F^{\geq\star}R/F^{\geq\star} A}\simeq \fil^{\geq \star}_{Hod}\widehat{\mathrm{dR}}_{F^{\geq 0}R/F^{\geq 0} A}.$$
Moreover in this case
$$ \mathrm{gr}^0_F\fil^{\geq \star}_{Hod}\widehat{\mathrm{dR}}_{F^{\geq\star}R/F^{\geq\star} A}\simeq \fil^{\geq \star}_{Hod}\widehat{\mathrm{dR}}_{\mathrm{gr}^0_FR/\mathrm{gr}^0_FA}$$
where the map is induced by a) and the map of animated graded pairs 
$$(\mathrm{gr}^0_FR/\mathrm{gr}^0_FA)\to (\bigoplus_i\mathrm{gr}^i_FR/\bigoplus_i\mathrm{gr}^i_FA).$$
\label{deRhambc}\end{prop} 

\begin{proof} Part a) is an immediate consequence of (\ref{reesunder}) and (\ref{reesgraded}) combined with base change (Prop. \ref{deRham} a)). For example
\begin{align*}
&\fil^{\geq \star}_{Hod}\widehat{\mathrm{dR}}_{F^{\infty}R/F^{\infty} A}\hat{\otimes}_\bz\bz[t,t^{-1}]\\ \simeq &\fil^{\geq \star}_{Hod}\widehat{\mathrm{dR}}_{F^{\infty}R\otimes_\bz\bz[t,t^{-1}]/F^{\infty} A\otimes_\bz\bz[t,t^{-1}]}\\
\simeq &\fil^{\geq \star}_{Hod}\widehat{\mathrm{dR}}_{\mathrm{Rees}(F^{\geq\star}R)\otimes_{\bz[t]}\bz[t,t^{-1}]/\mathrm{Rees}(F^{\geq\star}A)\otimes_{\bz[t]}\bz[t,t^{-1}]}\\
\simeq &\fil^{\geq \star}_{Hod}\widehat{\mathrm{dR}}_{\mathrm{Rees}(F^{\geq\star}R)/\mathrm{Rees}(F^{\geq\star}A)}\hat{\otimes}_{\bz[t]}\bz[t,t^{-1}]\\
\simeq & \mathrm{Rees}\left(F^{\geq\star}\fil^{\geq \star}_{Hod}\widehat{\mathrm{dR}}_{F^{\geq\star}R/F^{\geq\star} A}\right)\hat{\otimes}_{\bz[t]}\bz[t,t^{-1}]\\
\simeq & F^{\infty}\fil^{\geq \star}_{Hod}\widehat{\mathrm{dR}}_{F^{\geq\star}R/F^{\geq\star} A}\hat{\otimes}_{\bz}\bz[t,t^{-1}]
\end{align*}
which implies the first isomorphism in a) by the $p$-complete version of (\ref{modequivalence}).

Part b) follows from the second isomorphism in a) and the following Lemma. 
\end{proof}

\begin{lemma} For a morphism of non-negatively graded animated rings $A\to R$ $\fil^{\geq \star}_{Hod}\widehat{\mathrm{dR}}_{ R/A}$ is non-negatively graded and $(R^0/A^0)\to(R/A)$ induces an isomorphism
$$\left(\fil^{\geq \star}_{Hod}\widehat{\mathrm{dR}}_{ R/A}\right)^0\simeq \fil^{\geq \star}_{Hod}\widehat{\mathrm{dR}}_{R^0/A^0}.$$
\label{gr0lemma}\end{lemma}

\begin{proof}
Writing $A\to R$ as a sifted colimit of free non-negatively graded pairs we can assume
$$A:=\bz[X_1,\dots,X_n]\to \bz[X_1,\dots,X_n, Y_1,\dots,Y_m]=: R$$ 
where $X_i,Y_j$ are variables of degrees $\deg(X_i),\deg(Y_j)\geq 0$. We need to verify that
$\Omega^k_{ R/A}$ is non-negatively graded and that
$$\left(\Omega^k_{ R/A} \right)^0=\Omega^k_{R^0/A^0}.$$
But $\Omega^k_{ R/A}$ is a free $ R$-module on generators $\{dY_{i_1}\wedge\cdots\wedge dY_{i_k} \vert \ 1\leq i_1<\cdots < i_k\leq m\}$ of degrees $\deg(Y_{i_1})+\cdots +\deg(Y_{i_k})\geq 0$, hence is non-negatively graded. We have 
$$A^0=\bz[X_i\vert\ \deg(X_i)=0]\to \bz[X_i,Y_j\vert \deg(X_i)=\deg(Y_j)=0]=R^0$$ 
and
$$\left(\Omega^k_{ R/A} \right)^0=\bigoplus_{\deg(Y_{i_j})=0} R^0\cdot dY_{i_1}\wedge\cdots\wedge dY_{i_k} =\Omega^k_{R^0/A^0}.$$
This proves Lemma \ref{gr0lemma}.
\end{proof}

\subsection{Filtrations and derived formal stacks}\label{filtrations} There is yet another way to view the category of graded modules over a graded ring $A$, namely as $\bg_m$-equivariant quasi-coherent sheaves on $\Spec(A)$, or equivalently quasi-coherent sheaves on the quotient stack $\Spec(A)/\bg_m$. We will use this point of view primarily as a convenient notational device to globalize the constructions of the previous section. Since
$$ \fil^{\geq \star}_{Hod}\widehat{\mathrm{dR}}_{R/A}\simeq\fil^{\geq \star}_{Hod}\widehat{\mathrm{dR}}_{R^\wedge_p/A^\wedge_p}$$
(by $p$-completeness and base change to the mod p reduction) we do this in the context of derived $p$-adic formal geometry.  
For any animated graded ring $A$ we have
$$ (\Mod_{A^{gr}})^\wedge_p\simeq\cd(\Spf(A)/\widehat{\bg}_m)\simeq\Tot \cd(\Spf(\ca^\bullet))$$
where the simplicial derived formal scheme $\Spf(\ca^\bullet)$ is the action groupoid of the $\widehat{\bg}_m$-action on $\Spf(A)$, i.e.
$$ \ca^i\simeq A\hat{\otimes}\co(\widehat{\bg}_m)^{\hat{\otimes} i}\simeq A[u_1^{\pm 1},\dots,u_{i}^{\pm 1}]_p^\wedge$$
for $i\geq 0$. From this perspective, the construction of the $\co(\widehat{\bg}_m)$-coaction on $\widehat{\mathrm{dR}}_{R/A}$ in Def. \ref{gradeddeRham} becomes a definition
$$\widehat{\mathrm{dR}}_{R\otimes_A\ca^\bullet/\ca^\bullet}\in\Tot \cd(\Spf(\ca^\bullet))$$
via descent (this also makes precise higher coherences which we have ignored in Def. \ref{gradeddeRham}). The association 
$$R\quad \mapsto\quad \left( \Spf(R)/\widehat{\bg}_m \to \mathfrak{A}:=\Spf(A)/\widehat{\bg}_m\right)$$ sets up a contravariant equivalence between the category of $p$-complete animated graded $A$-algebras and the category of affine morphisms of derived formal stacks 
$$g:\mathfrak{X}\to \mathfrak{A}$$
with inverse
$$g\mapsto\co(\mathfrak{X}\times_\mathfrak{A}\Spf(A))\simeq \co(\Spf(R))=R.$$

\begin{definition} ($\widehat{\bg}_m$-equivariant  derived de Rham cohomology of derived formal schemes) Let $A$ be a animated graded ring and
$$\mathfrak{X}\to \mathfrak{A}:=\Spf(A)/\widehat{\bg}_m$$ 
a morphism of derived formal stacks such that $\mathscr{X}_0:=\mathfrak{X}\times_{\mathfrak{A}}\Spf(A)$ is a derived formal scheme.
Define
$$\fil^{\geq \star}_{Hod}\widehat{\mathrm{dR}}_{\mathfrak{X}/\mathfrak{A}}:=\left(\fil^{\geq \star}_{Hod}\widehat{\mathrm{dR}}_{\mathscr{X}_\bullet/\ca^\bullet}\right) \in\Tot\CAlg(\Fun(\bn^{op},\cd(\Spf(\ca^\bullet))))$$
where  
$$\ca^i:= A\hat{\otimes}\co(\widehat{\bg}_m)^{\hat{\otimes} i},\quad
\mathscr{X}_i:=\mathfrak{X}\times_\mathfrak{A}\Spf(\ca^i)\simeq \mathscr{X}_0\times\widehat{\bg}_m^i.$$ 
\label{schematicdeRham} \end{definition}

\begin{remark} Note that for a morphism of derived formal schemes $\mathscr{X}\to\Spf(A)$ $\fil^{\geq \star}_{Hod}\widehat{\mathrm{dR}}_{\mathscr{X}/A}$ is well defined by Zariski descent (Prop. \ref{deRham} b)) and satisfies base change by Prop. \ref{deRham} a).
\end{remark}

For any animated filtered ring $F^{\geq\star} A$ define the derived formal stack
$$ F\Spf(A):=\Spf(\mathrm{Rees}(F^{\geq \star} A))/\widehat{\bg}_m.$$
We have the following stack theoretic incarnation of the Rees construction (see also \cite{bhatt22}[Prop. 2.2.6]).

\begin{prop} Let $F^{\geq\star} A$ be a animated filtered ring. The Rees construction (\ref{Reeseq}) gives a $t$-exact symmetric monoidal equivalence
\begin{equation} (\Mod_{F^{\geq\star} A})^\wedge_p =\Mod_{ \mathrm{Rees}(F^{\geq \star} A)^{gr}}(\Fun(\bz^\delta,\Mod_\bz))^\wedge_p\simeq \cd(F\Spf(A))
\notag\end{equation}
with inverse sending $\cf\in \cd(F\Spf(A))$ to the filtered complex
$$ \cdots\to R\Gamma( F\Spf(A)  ,\cf(-n-1)) \to R\Gamma(F\Spf(A) ,\cf(-n)) \to\cdots $$
where the transition maps are induced by the tautological cosection \cite{bhatt22}[Construction 2.2.5] 
$$\co_{F\Spf(A)}(-1)\to \co_{F\Spf(A)}.$$
Base change to 
$$ \Spf(\mathrm{Rees}(F^{\geq \star} A)[t^{-1}])/\widehat{\bg}_m \simeq \Spf(F^{\infty} A\otimes_\bz\bz[t,t^{-1}])/\widehat{\bg}_m\simeq \Spf(F^{\infty} A)$$
recovers the underlying complex of the filtered complex (i.e. its $p$-complete colimit), and base change to
$$\Spf\left(\bigoplus_i\mathrm{gr}^i_FA\right)/\widehat{\bg}_m$$
recovers the associated graded.
\label{filtered}\end{prop} 

The association 
$$F^{\geq\star} R\quad \mapsto\quad \left( F\Spf(R) \to F\Spf(A)\right)$$ sets up a contravariant equivalence between the category of $p$-complete animated filtered  $F^{\geq\star} A$-algebras and the category of affine morphisms of derived formal stacks 
$$\mathfrak{X}\to \mathfrak{A}:=F\Spf(A).$$
The Rees construction localizes on $\Spf(F^{\geq 0} R)$: For a multiplicative system $S$ in $\pi_0F^{\geq 0}R$ we obtain a multiplicative system $S$ in $\pi_0\mathrm{Rees}(F^{\geq \star} R)$ and 
$$ \left(S^{-1}\mathrm{Rees}(F^{\geq \star} R)\right)_p^\wedge \simeq \left(\bigoplus_{i\in\bz}S^{-1}\fil^{\geq i}R\cdot t^{-i}\right)_p^\wedge \simeq  \mathrm{Rees}\left((S^{-1}F^{\geq \star} R)_p^\wedge\right).$$
Note that this is not in general true for a multiplicative system of the underlying ring $F^\infty R=\varinjlim_iF^{\geq i}R$ of the filtration unless it happens to agree with $F^{\geq 0}R$. Hence when gluing filtered rings to schemes it is natural to restrict to $\bn^{op}$-indexed filtered rings.

\begin{definition} Let $F^{\geq\star} A$ be an animated filtered ring and set $\mathfrak{A}:=F\Spf(A)$. We call a morphism of derived formal stacks 
$$\mathfrak{X}\to \mathfrak{A}$$ 
 \underline{schematic} if there exists a quasi-compact and quasi-separated derived formal scheme $\mathscr{X}$, a $\bn^{op}$-indexed, filtered, quasi-coherent $F^{\geq \star} A$-algebra $F^{\geq\star}\co_{\mathscr{X}}$ with $F^{\geq 0}\co_{\mathscr{X}}=\co_{\mathscr{X}}$ and an isomorphism
 $$\mathfrak{X} \simeq \mathscr{S}pf(\mathrm{Rees}(F^{\geq\star}\co_{\mathscr{X}}))/\widehat{\bg}_m$$
 over $\mathfrak{A}$. Here $\mathscr{S}pf$ denotes the relative derived formal spectrum of a quasi-coherent sheaf of algebras on $\mathscr{X}$.
\label{schematic} \end{definition}

\begin{lemma} If $\mathfrak{X}\xrightarrow{} \mathfrak{A} $ is schematic we have equivalences
\begin{align*} \mathfrak{X}\times_{\mathfrak{A}} \Spf(\mathrm{Rees}(F^{\geq \star} A))\simeq & \mathscr{S}pf(\mathrm{Rees}(F^{\geq\star}\co_{\mathscr{X}}))\\ 
\mathfrak{X}\times_{\mathfrak{A}}\Spf(F^\infty A)\simeq &\mathscr{X}\\
\mathfrak{X}\times_{\mathfrak{A}}\Spf(\bigoplus_i\mathrm{gr}^i_FA)/\widehat{\bg}_m \simeq & \mathscr{S}pf(\bigoplus_i \mathrm{gr}^i_F\co_{\mathscr{X}})/\widehat{\bg}_m\\
 \cd(\mathfrak{X})\simeq &\cdf(\mathscr{X},F^{\geq\star}\co_{\mathscr{X}})\simeq \cd^{gr}(\mathscr{X},\mathrm{Rees}(F^{\geq\star}\co_{\mathscr{X}})) .
\end{align*}
In particular, the derived formal stack $\mathscr{X}_0:=\mathfrak{X}\times_{\mathfrak{A}}\Spf(\mathrm{Rees}(F^{\geq\star}A))$ is a derived formal scheme, hence
$$\fil^{\geq \star}_{Hod}\widehat{\mathrm{dR}}_{\mathfrak{X}/\mathfrak{A}}\in\CAlg(\Fun(\bn^{op},\cd(\mathfrak{A})))$$
is defined as in Def. \ref{schematicdeRham}.
\label{schematiclemma}\end{lemma}
\begin{proof} Clear from the definitions. \end{proof}

The following examples describe some formal stacks of interest for this article.

\begin{example} The (derived) formal stack $\mathfrak{A}:=\widehat{\ba}^1/\widehat{\bg}_m$ corresponds to the initial filtered ring of Example \ref{initial}.
\label{initialstack}\end{example}

\begin{example} Let $K/\bq_p$ be a finite extension with ring of integers $\co_K$ and maximal ideal $\fm$. Then $\co_K$ can be viewed as a filtered ring with 
$$ F^{\geq i}_\fm\co_K=\begin{cases} \co_K & i\leq 0\\ \fm^i & i\geq 1.\end{cases}$$
For every (animated) $\co_K$-algebra $R$ we obtain an induced filtration
$$ F^{\geq i}_\fm\co_K\otimes_{\co_K}R$$
compatibly with $p$-complete localization. Hence for any quasi-compact and  quasi-separated (derived) formal scheme $\mathscr{X}\to \Spf(\co_K)$
we obtain a schematic morphism of (derived) formal stacks
$$ \mathfrak{X}=F_\fm\mathscr{X}\to \widehat{\ba}^1/\widehat{\bg}_m=\mathfrak{A} $$ with underlying (derived) formal scheme $\mathscr{X}$. 
The choice of a uniformizer $\varpi\in\fm$ gives an alternative description 
$$ \cdots \xrightarrow{\cdot \varpi}R\xrightarrow{\cdot \varpi}R\xrightarrow{\cdot \varpi}R$$
of the filtration $F^{\geq i}_\fm\co_K\otimes_{\co_K}R$.
\label{keyexample}\end{example}

\subsection{Prismatic cohomology relative to a $\delta$-ring} \label{deltapairs}
The original definition of prismatic cohomology of Bhatt and Scholze \cite{bhatt-scholze} associates a $\be_\infty$-$A$-algebra $\prism_{R/(A,I)}$ (together with Breuil-Kisin twists and Hodge-Tate and Nygaard filtrations) to a pair $(R, (A, I))$ where $(A,I)$ is a prism and $R$ is an $A/I$-algebra. This was generalized in \cite{akn23} to arbitrary $\delta$-pairs $(R,A)$ where $A$ is only assumed to be a $\delta$-ring and $R$ is any (animated) $A$-algebra, also recovering absolute prismatic cohomology of Bhatt and Lurie \cite{bhatt-lurie-22} for $A=\bz_p$.  More precisely, for any $\delta$-pair $(R,A)$ one obtains a structure (the "prismatic package" \cite{akn23}[Sec. 8]) 
\begin{equation} \underline{\prism}_{R/A}=\left(\prism_{R/A}^{[\star]}\{\star\}, \mathcal N^{\geq\star}\widehat{\prism}^{(1)}_{R/A}\{\star\},c,\varphi\right)\in\mathscr{C}_A\label{package}\end{equation}
where $\mathscr{C}_A$ is the $\infty$-category consisting of quadruples $(H,N,c,\varphi)$ of the following objects
\begin{itemize} 
\item[a)] $H$ is a complete $\bz$-filtered, $\bz$-graded $\be_\infty$-$A$-algebra 
\item[b)] $N$ is a complete $\bn$-filtered, $\bz$-graded $\be_\infty$-$A$-algebra 
\item[c)] $c$ is a $\varphi_A$-semilinear $\bn$-filtered, $\bz$-graded map of $A$-algebras $H\to \mathrm{scp}(N)$ where
$$ \mathrm{scp}(N)^{\geq \star}\{n\}:=N^{\geq p\star}\{n\}\otimes p^\star\bz_p.$$
\item[d)] $\varphi$ is a map of $\bn$-filtered, $\bz$-graded $A$-algebras $N\to\mathrm{sh}(H)$ where $$\mathrm{sh}(H)^{\geq \star}\{n\}:=H^{\geq \star-n}\{n\}.$$
\end{itemize} 
The notation $\prism_{R/A}^{[\star]}$ refers to the Hodge-Tate filtration, $\widehat{\prism}$ indicates completion with respect to the Nygaard filtration and the notation $\prism^{(1)}_{R/A}$ denotes scalar extension with respect to the Frobenius of $A$ on the underlying prismatic crystal \cite{akn23}[Def. 6.1]. 

\subsubsection{Prismatic F-gauges} For our purposes we need to include $\fil^{\geq \star}_{Hod}\widehat{\mathrm{dR}}_{R/A}$ into the prismatic package, as well as additional filtrations  which allow to make reductions to de Rham or Hodge cohomology (e.g. the conjugate filtration). Enlarging the prismatic package to include these objects would be too unwieldy but working with sheaves on the Cartier-Witt stack leads to a somewhat more manageable structure. The following definition is essentially a description of $A$-modules in (a slight weakening of) absolute prismatic $F$-gauges, directly in terms of the Cartier-Witt stack whose theory is well documented in \cite{bhatt-lurie-22}.

First, recall the commutative diagram
\begin{equation}\begin{CD} \mathrm{WCart}^{\mathrm{HT}} @>\iota >> \mathrm{WCart}\\
@V\pi VV @V F VV\\
\Spf(\bz_p) @>\rho_{dR}>> \mathrm{WCart}
\end{CD}\label{wcart}\end{equation}
from \cite{bhatt-lurie-22}[Thm. 3.6.7].
Second, recall that for the $\bz$-filtered quasicoherent algebra $\mathscr{I}^\star\co_{\mathrm{WCart}}$ in $\cd(\mathrm{WCart})$ there is an equivalence $$ \cdf(\mathscr{I}^\star\co_{\mathrm{WCart}})\simeq\cd(\mathrm{WCart}),\quad \mathscr{I}^\star\mathscr{E}\mapsto \mathscr{I}^0\mathscr{E}=\mathscr{E}$$
analogous to (\ref{modequivalence}). Lastly, note that the filtered relative Frobenius map \cite{bhatt-lurie-22}[Constr. 5.7.1]
$$\Phi:\fil^{\geq\star}_{\mathcal N}F^*\co_{\mathrm{WCart}}\to \mathscr{I}^\star\co_{\mathrm{WCart}}$$
is an isomorphism with the $\bn^{op}$-indexed truncation of the target by  \cite{bhatt-lurie-22}[Example 5.1.4].

\begin{definition} (Lax prismatic F-gauges with coefficients) For a $\delta$-ring $A$ denote by $\gsheaf_A$ the $\infty$-category of sextuples 
$$\underline{\mathscr{E}}=(\fil^{\geq\star}_{\mathcal N}\mathscr{E}',\fil^{\geq\star}_HD,\alpha, c_{dR}, c_F, c_{\square})$$ 
where  
\begin{itemize}
\item[a)] $\fil^{\geq\star}_{\mathcal N}\mathscr{E}'$ is an $\bn^{op}$-indexed $A$-module in $\cdf(\fil^{\geq\star}_{\mathcal N}F^*\co_{\mathrm{WCart}})$. Put
$$ \mathscr{I}^\star\mathscr{E}:= \fil^{\geq\star}_{\mathcal N}\mathscr{E}'\otimes_{ \fil^{\geq\star}_{\mathcal N}F^*\co_{\mathrm{WCart}},\Phi}\mathscr{I}^\star\co_{\mathrm{WCart}} , \quad \mathscr{E}:=\mathscr{I}^0\mathscr{E},\quad \mathscr{E}^{(1)}:=\mathscr{E}\otimes_{A,\varphi}A$$
and denote by 
$$\Phi:   \fil^{\geq\star}_{\mathcal N}\mathscr{E}' \to \mathscr{I}^\star\mathscr{E}$$
the tautological filtered $A$-linear map.
\item[b)] $\fil^{\geq\star}_HD$ is an $\bn^{op}$-indexed object of $\mathcal{DF}(A)$. 
\item[c)] $\alpha$ is an $A$-linear isomorphism 
$$\alpha: \fil^{\geq\star}_{\mathcal N}\mathscr{E}'/\mathscr{I}\fil^{\geq\star-1}_{\mathcal N}\mathscr{E}'\simeq \fil^{\geq\star}_HD\otimes\co_{\mathrm{WCart}^{\mathrm{HT}}}\simeq \pi^*\fil^{\geq\star}_HD$$
in $\cdf(\mathrm{WCart}^{\mathrm{HT}})$.
\item[d)] $c_{dR}$ is an $A$-linear morphism 
$$ c_{dR}:\rho_{dR}^*\mathscr{E}^{(1)}\to \fil^{\geq 0}_HD.$$ 
\item[e)] $c_F$ is an $A$-linear morphism in $\cd(\mathrm{WCart})$
$$c_F: F^*\mathscr{E}^{(1)}\to \fil^{\geq 0}_{\mathcal N}\mathscr{E}'.$$
\item[f)] $c_{\square}$ is an $A$-linear commutative diagram in $\cd(\mathrm{WCart}^{\mathrm{HT}})$
$$\begin{CD} \iota^*F^*\mathscr{E}^{(1)} @>\sim>> \pi^*\rho_{dR}^*\mathscr{E}^{(1)}\\
@V\iota^*c_F VV @VV \pi^*c_{dR}V\\
\iota^*\fil^{\geq 0}_{\mathcal N}\mathscr{E}' @>\alpha^0>\sim> \pi^*\fil^{\geq 0}_HD . \end{CD}$$ 
\end{itemize}  
\label{fgauges}\end{definition} 

For the proper definition of $\gsheaf_A$ as a symmetric monoidal presentable, stable $\infty$-category we refer to Def. \ref{rigorousfgauges} in App. A. We set
$$ \csheaf_A=\CAlg(\gsheaf_A).$$
For a morphism of $\delta$-rings $A\to A'$ there is a symmetric monoidal base change functor
\begin{equation} \gsheaf_A\xrightarrow{-\otimes_AA'} \gsheaf_{A'}\label{basechange}\end{equation}
defined by 
\begin{multline*}\underline{\mathscr{E}}=(\fil^{\geq\star}_{\mathcal N}\mathscr{E}',\fil^{\geq\star}_HD,\alpha, c_{dR}, c_F, c_{\square})   \mapsto \\ 
\underline{\mathscr{E}}\otimes_AA':=(\fil^{\geq\star}_{\mathcal N}\mathscr{E}'\otimes_AA',\fil^{\geq\star}_HD\otimes_AA',\alpha, c_{dR}\otimes_AA',c_F\otimes_AA', c_{\square}\otimes_AA') 
\end{multline*}
where $-\otimes_AA'$ denotes the base change functor of Lemma \ref{DAgeneralities} a) in App. A.  The right adjoint $\mathrm{Res}^{A'}_{A}$ of (\ref{basechange}) is given by restricting the module structures of $\fil^{\geq\star}_{\mathcal N}\mathscr{E}'$ and $\fil^{\geq\star}_HD$ along $A\to A'$. This simple description of the right adjoint is the reason we do not work with prismatic $F$-gauges, i.e. we do not require $c_{dR}$ and $c_F$ to be isomorphisms.

\begin{remark} While every object of $\gsheaf_A$ gives rise to an object of $\gsheaf_{\bz}$ by restriction, this restriction does not in general have an $A$-module structure. Hence the equivalence $\mathscr{G}_A\simeq \Mod_A(\mathscr{G}_\bz)$ of \cite{akn23}[Rem. 8.2 (ii)] does not hold for $\gsheaf_A$.
\end{remark}

In order to accommodate graded and filtered objects we need to further extend the definition of $\gsheaf_A$ from $\delta$-rings $A$ to $\delta$-stacks $\mathfrak{A}$. We are only interested in simple quotient stacks such as $\mathfrak{A}=F\Spf(A)$ for a filtered $\delta$-ring $F^{\geq\star} A$ in the sense of Def. \ref{fildefs}. Therefore we restrict our use of the terminology "formal $\delta$-stack" to the following situation.

\begin{definition} A \underline{formal $\delta$-stack} is a formal stack $\mathfrak{A}$ (with a Frobenius lift $\varphi_{\mathfrak{A}}$) that can be presented as a colimit  
\begin{equation} \mathfrak{A} \simeq \mathrm{colim}_{\Delta^{op}}\Spf(\ca^\bullet)\label{deltastack}\end{equation}
where $\ca^\bullet$ is a cosimplicial $\delta$-ring with flat differentials and $\Spf(\ca^\bullet)$ is a groupoid object in the $\infty$-topos of formal stacks \cite{lurieHTT}[Def. 6.1.2.7].
\label{deltastackdef}\end{definition}

For a formal $\delta$-stack (\ref{deltastack}) the definition of $\gsheaf_{\mathfrak{A}}$ is also given in Def. \ref{rigorousfgauges} in App. A. Lemma \ref{omnibus} a) in App. A gives equivalences 
$$\gsheaf_{\mathfrak{A}}\simeq \Tot \gsheaf_{\ca^\bullet},\quad\quad \csheaf_{\mathfrak{A}}:=\CAlg(\gsheaf_{\mathfrak{A}})\simeq \Tot \csheaf_{\ca^\bullet}$$
which will be our key tool to define objects of $\gsheaf_{\mathfrak{A}}$ by descent. We note however that the definition of  $\gsheaf_{\mathfrak{A}}$ is independent of the presentation (\ref{deltastack}).

\begin{definition} (Nygaard global sections) For a formal $\delta$-stack $\mathfrak{A}$ and
$$ \underline{\mathscr{E}}=(\fil^{\geq\star}_{\mathcal N}\mathscr{E}',\fil^{\geq\star}_HD,\alpha, c_{dR}, c_F, c_{\square})\in\gsheaf_\mathfrak{A}$$ 
let $R\Gamma_{\mathcal N}^{\geq\bullet}(\underline{\mathscr{E}}^{(1)}\{\star\})$ be the filtered graded object of $\cd(\mathfrak{A})$ (of  $\CAlg(\cd(\mathfrak{A}))$ if $\underline{\mathscr{E}}\in\csheaf_\mathfrak{A}$)  defined by the fibre product
\begin{equation}\begin{CD} R\Gamma_{\mathcal N}^{\geq\bullet}(\underline{\mathscr{E}}^{(1)}\{\star\})  @>>> \fil^{\geq\bullet}_HD\\
@VVV@VVV\\
R\Gamma(\mathrm{WCart}, \fil^{\geq\bullet}_{\mathcal N}\mathscr{E}'\otimes F^*\co\{\star\}) @>\alpha_*>>R\Gamma(\mathrm{WCart}^{\mathrm{HT}},\fil^{\geq\bullet}_HD\otimes\co_{\mathrm{WCart}^{\mathrm{HT}}}). \end{CD}\label{nygaardef}\end{equation}
The diagram $c_\square$ together with \cite{bhatt-lurie-22}[Thm. 3.6.7] gives a map
$$c_{\square,*} :\,R\Gamma(\mathrm{WCart},\mathscr{E}^{(1)}\{\star\})\to R\Gamma_{\mathcal N}^{\geq 0}(\underline{\mathscr{E}}^{(1)}\{\star\}).$$
\label{globalnygaard}\end{definition}

\begin{definition} (Syntomic global sections) For a formal $\delta$-stack $\mathfrak{A}$,
$$ \underline{\mathscr{E}}=(\fil^{\geq\star}_{\mathcal N}\mathscr{E}',\fil^{\geq\star}_HD,\alpha, c_{dR}, c_F, c_{\square})\in\gsheaf_\mathfrak{A}$$ 
and $n\in\bz$ define \underline{syntomic cohomology} by the fiber sequence
$$ R\Gamma_{\syn}(\underline{\mathscr{E}}^{(1)}\{n\})\to R\Gamma_{\mathcal N}^{\geq n}(\underline{\mathscr{E}}^{(1)}\{n\})\xrightarrow{\can-c^0\varphi\{n\}}R\Gamma_{\mathcal N}^{\geq 0}(\underline{\mathscr{E}}^{(1)}\{n\})$$
where $\varphi\{n\}$ is the  composite
$$ R\Gamma_{\mathcal N}^{\geq n}(\underline{\mathscr{E}}^{(1)}\{n\})\to R\Gamma(\mathrm{WCart}, \fil^{\geq n}_{\mathcal N}\mathscr{E}'\otimes F^*\co\{n\}) \xrightarrow{\Phi_*}
R\Gamma(\mathrm{WCart},\mathscr{E}\{n\})$$
using $F^*\co\{n\}\simeq\mathscr{I}^{-n}\{n\}$, and $c^0$ is the map
$$c^0:R\Gamma(\mathrm{WCart},\mathscr{E}\{n\})\to R\Gamma(\mathrm{WCart},\mathscr{E}^{(1)}\{n\})\to R\Gamma_{\mathcal N}^{\geq 0}(\underline{\mathscr{E}}^{(1)}\{n\})$$
induced by the natural $\varphi_{\mathfrak{A}}$-semilinear map $\mathscr{E}\to \mathscr{E}^{(1)}$ together with $c_{\square,*}$.

Define \underline{additive syntomic cohomology} by the fiber sequence
$$ R\Gamma_{\add}(\underline{\mathscr{E}}^{(1)}\{n\})\to R\Gamma_{\mathcal N}^{\geq n}(\underline{\mathscr{E}}^{(1)}\{n\})\xrightarrow{\can}R\Gamma_{\mathcal N}^{\geq 0}(\underline{\mathscr{E}}^{(1)}\{n\}).$$
\label{globalsyntomic}\end{definition}

The following Lemma will allow to reduce many properties of an object $\underline{\mathscr{E}}\in\gsheaf_\mathfrak{A}$ to properties of $\fil^{\geq\star}_HD$.

\begin{lemma} Let $\mathfrak{A}$ be a formal $\delta$-stack and 
$$ \underline{\mathscr{E}}=(\fil^{\geq\star}_{\mathcal N}\mathscr{E}',\fil^{\geq\star}_HD,\alpha, c_{dR}, c_F, c_{\square})\in\gsheaf_\mathfrak{A}.$$

a) The object $\fil^{\geq\star}_{\mathcal N}\mathscr{E}'$ has a complete $\bn^{op}$-indexed filtration
$\mathscr{I}^{\bullet}\fil^{\geq\star-\bullet}_{\mathcal N}\mathscr{E}'$ with graded pieces
$$\mathscr{I}^{i}\fil^{\geq\star-i}_{\mathcal N}\mathscr{E}'/\mathscr{I}^{i+1}\fil^{\geq\star-i-1}_{\mathcal N}\mathscr{E}'\simeq (\pi^*\fil^{\geq\star-i}_HD)\{i\}.$$

b) For $m\geq 0$ the object $\mathrm{gr}^m_{\mathcal N}\mathscr{E}'$ has a bounded filtration with graded pieces
$$ \mathscr{I}^{i}\mathrm{gr}^{m-i}_{\mathcal N}\mathscr{E}'/\mathscr{I}^{i+1}\mathrm{gr}^{m-i-1}_{\mathcal N}\mathscr{E}'\simeq (\pi^*\mathrm{gr}_H^{m-i}D)\{i\}$$
for $i=0,1,\dots,m$.

c) For any $n\in\bz$ the object $R\Gamma_{\mathcal N}^{\geq\star}(\underline{\mathscr{E}}^{(1)}\{n\})$ has a complete $\bn^{op}$-indexed filtration with graded pieces
$$\mathrm{gr}^i=\begin{cases} \fil^{\geq\star}_HD & i=0\\ R\Gamma(\mathrm{WCart}^{\mathrm{HT}},(\pi^*\fil^{\geq{\star-i}}_HD)\{i\}) &i\geq 1.\end{cases} $$

d) For $m\geq 0$ and any $n\in\bz$ the object $\mathrm{gr}_{\mathcal N}^m R\Gamma_{\mathcal N}^{\geq\star}(\underline{\mathscr{E}}^{(1)}\{n\})$ has a bounded filtration with graded pieces
$$\mathrm{gr}^i=\begin{cases} \mathrm{gr}^m_HD & i=0\\ R\Gamma(\mathrm{WCart}^{\mathrm{HT}},(\pi^*\mathrm{gr}^{m-i}_HD)\{i\}) &i = 1,\dots,m.\end{cases} $$

\label{filtration}\end{lemma}

\begin{proof} The filtration in a) is immediate from tensoring the isomorphism $\alpha$ in Def. \ref{fgauges} c) with $\mathscr{I}$. Since $\fil^{\geq\star}_{\mathcal N}\mathscr{E}'$ is $\bn^{op}$-indexed the filtration agrees with the $\mathscr{I}$-adic filtration on $\fil^{\geq 0}_{\mathcal N}\mathscr{E}'$ for $\bullet\geq \star$ and is therefore complete.

Part c) follows by taking global sections in a) together with the fact that (\ref{nygaardef}) is a fibre square. Part b), resp. d), is immediate from part a), resp. c), by passing to associated graded objects.
\end{proof}

\subsubsection{The syntomic package} For any $\delta$-pair $(R,A)$ recall the prismatic crystal $\mathscr{H}_\prism(R/A)\in\cd(\mathrm{WCart})$ from \cite{akn23}[Def. 3.3] and define 
$\mathscr{H}^{(1)}_\prism(R/A):=\mathscr{H}_\prism(R/A)\otimes_{A,\varphi}A$ \cite{akn23}[Def. 6.1]. 

\begin{lemma} a) For any $\delta$-pair $(R,A)$ there is a cofibre sequence of filtered $A$-modules in $\cd(\mathrm{WCart})$
\begin{equation}\mathscr{I}\,\mathcal N^{\geq {\bullet-1}}F^*\mathscr{H}^{(1)}_\prism(R/A)\xrightarrow{\tau^\bullet} \mathcal N^{\geq\bullet}F^*\mathscr{H}^{(1)}_\prism(R/A)\xrightarrow{\alpha^\bullet} \iota_*\pi^*\fil^{\geq \bullet}_{Hod}\widehat{\mathrm{dR}}_{R/A}\end{equation}
and a commutative diagram
$$\begin{CD} \iota^*F^*\mathscr{H}^{(1)}_\prism(R/A) @>>> \pi^*\rho_{dR}^*\mathscr{H}^{(1)}_\prism(R/A)\\
\Vert@. @VV \pi^*c_{dR}V\\
\iota^*\mathcal N^{\geq 0}F^*\mathscr{H}^{(1)}_\prism(R/A) @>\alpha^0>> \pi^*\fil^{\geq 0}_{Hod}\widehat{\mathrm{dR}}_{R/A}\end{CD}$$ 
where $c_{dR}:\rho_{dR}^*\mathscr{H}^{(1)}_\prism(R/A)\simeq \widehat{\mathrm{dR}}_{R/A}$ is the isomorphism of \cite{akn23}[Lemma 6.4].

b) The colimit of the diagram
$$ F^*\mathscr{H}^{(1)}_\prism(R/A)\xrightarrow{\mathscr{I}^{-1}\tau^1} \mathscr{I}^{-1}\mathcal{N}^{\geq 1}F^*\mathscr{H}^{(1)}_\prism(R/A)\xrightarrow{\mathscr{I}^{-2}\tau^2} \mathscr{I}^{-2}\mathcal{N}^{\geq 2}F^*\mathscr{H}^{(1)}_\prism(R/A)\to\cdots$$
identifies with $\mathscr{H}_\prism(R/A)$ under the relative Frobenius maps $\Phi$ of \cite{bhatt-lurie-22}[Prop. 5.1.1].
\label{alphadef}\end{lemma}

\begin{proof}As a functor on transversal prisms $\mathscr{H}_\prism(R/A)$ is given by the assignment
$$ (B,I)\mapsto \prism_{R\hat{\otimes}\bar{B}/(A\hat{\otimes} B,A\hat{\otimes}I)}.$$
Defining $\mathscr{H}^{(1)}_\prism(R/A):=\mathscr{H}_\prism(R/A)\otimes_{A,\varphi}A$ the crystal $F^*\mathscr{H}^{(1)}_\prism(R/A)$ is given by the assignment
$$ (B,I)\mapsto \prism_{R\hat{\otimes}\bar{B}/(A\hat{\otimes} B,A\hat{\otimes}I)}\otimes_{A\hat{\otimes} B,\varphi}A\hat{\otimes} B=: \prism^{(1)}_{R\hat{\otimes}\bar{B}/A\hat{\otimes} B}$$
and the relative de Rham comparison \cite{bhatt-lurie-22}[Prop. 5.2.5] gives an equivalence
$$ \prism_{R\hat{\otimes}\bar{B}/(A\hat{\otimes} B,A\hat{\otimes}I)}\otimes_{A\hat{\otimes} B,\varphi}\overline{A\hat{\otimes} B}\simeq \widehat{\mathrm{dR}}_{R\hat{\otimes}\bar{B}/\overline{A\hat{\otimes} B}} \simeq \widehat{\mathrm{dR}}_{R/A}\otimes \overline{B}$$
and hence an equivalence
$$ \pi^*\rho_{dR}^*\mathscr{H}^{(1)}_\prism(R/A)\simeq\iota^*F^*\mathscr{H}^{(1)}_\prism(R/A)\simeq \pi^*\widehat{\mathrm{dR}}_{R/A}.$$
It was shown in the proof of \cite{akn23}[Lemma 6.4] that this isomorphism coincides with $\pi^*c_{dR}$.
On the other hand by \cite{bhatt-lurie-22}[Cor. 5.2.8] for any transversal prism $(B,I)$ there is a fibre sequence of $A\hat{\otimes}B$-modules 
\begin{equation} A\hat{\otimes}I\otimes_{A\hat{\otimes}B}^L\mathcal N^{\geq {\bullet-1}}\prism^{(1)}_{R\hat{\otimes}\bar{B}/A\hat{\otimes} B} \to \mathcal N^{\geq\bullet}\prism^{(1)}_{R\hat{\otimes}\bar{B}/A\hat{\otimes} B}\xrightarrow{\alpha^\bullet}\fil^{\geq \bullet}_{Hod}\widehat{\mathrm{dR}}_{R\hat{\otimes}\bar{B}/\overline{A\hat{\otimes} B}} \notag\end{equation}
where $\alpha^0$ induces the comparison isomorphism of \cite{bhatt-lurie-22}[Prop. 5.2.5]. This gives part a). Part b) is immediate from \cite{bhatt-lurie-22}[Cor. 5.2.16] after translating it into a statement about functors on transversal prisms.
\end{proof}

Note that the colimit in Lemma \ref{alphadef} b) computes $\mathscr{I}^0\mathscr{E}$ where 
$$ \mathscr{I}^\bullet\mathscr{E}:= \mathcal N^{\geq\bullet}F^*\mathscr{H}^{(1)}_\prism(R/A)\otimes_{ \mathcal N^{\geq\bullet}F^*\co_{\mathrm{WCart}},\Phi}\mathscr{I}^\bullet\co_{\mathrm{WCart}}.$$

\begin{definition} For any $\delta$-pair $(R,A)$ define the \underline{syntomic package}
$$ \underline{\mathscr{E}}({R/A})=\left(\mathcal N^{\geq\bullet}F^*\mathscr{H}^{(1)}_\prism(R/A),\fil^{\geq \bullet}_{Hod}\widehat{\mathrm{dR}}_{R/A},\alpha, c_{dR},c_F, c_\square\right)\in\csheaf_A$$
where $\alpha, c_{dR}, c_\square$ were defined in Lemma \ref{alphadef} a) and $c_F$ in Lemma \ref{alphadef} b). 
\label{sheafpackage}\end{definition}

\begin{remark} The maps $c_{dR}$ and $c_F$ in  $\underline{\mathscr{E}}({R/A})$ are isomorphisms by \cite{akn23}[Lemma 6.4] and \cite{bhatt-lurie-22}[Prop. 5.1.1].
\end{remark}

For a map of rings $R\to R^0$ denote by $R^\bullet$ its Cech conerve.

\begin{theorem} The objects $\underline{\mathscr{E}}({R/A})\in\csheaf_A$ depend functorially on the $\delta$-pair $(R,A)$ and satisfy the following properties
\begin{itemize}
\item[a)] (Base Change) For any map $A\to A'$ of $\delta$-rings the natural map in $\gsheaf_{A'}$
$$\underline{\mathscr{E}}({R/A})\otimes_AA'\to \underline{\mathscr{E}}({R\otimes_AA'/A'}) $$ 
is an equivalence.
\item[b)] (\'Etale descent in $R$) For any $\delta$-pair $(R,A)$ and for an \'etale faithfully flat map $R\to R^0$ the natural map in $\csheaf_A$
$$ \underline{\mathscr{E}}({R/A}) \to \Tot \underline{\mathscr{E}}({R^\bullet/A}) $$ 
is an equivalence.
\item[c)] (Descent for quasisyntomic $R/A$) Assume $(R,A)$ is a $\delta$-pair with $A\to R$ quasisyntomic.
\begin{itemize}
\item[c1)] (Quasisyntomic descent in $R$) If $R\to R^0$ is quasisyntomic and faithfully flat then the natural map in $\csheaf_A$
$$ \underline{\mathscr{E}}({R/A}) \to \Tot \underline{\mathscr{E}}({R^\bullet/A}) $$ 
is an equivalence.
\item[c2)] (Flat descent) If $A\to A^0$ is faithfully flat and $R^0:=R\otimes_AA^0$ then the natural map in $\csheaf_A$
$$ \underline{\mathscr{E}}({R/A}) \to \Tot \underline{\mathscr{E}}({R^\bullet/A^\bullet}) $$ 
is an equivalence.
\end{itemize}
\item[d)] (Invariance under quasi-\'etale extensions) For $A'\to A$ with vanishing $L_{A/A'}\in\cd(A)$ the natural map in $\csheaf_{A'}$
$$  \underline{\mathscr{E}}({R/A'}) \to \mathrm{Res}^A_{A'} \underline{\mathscr{E}}({R/A})$$
is an equivalence.
\item[e)] The functor $(R,A)\mapsto \underline{\mathscr{E}}({R/A})$ commutes with sifted colimits. 
\item[f)] The natural map $\underline{\mathscr{E}}({R/A})\to\underline{\mathscr{E}}({R^\wedge_p/A^\wedge_p})$ is an equivalence.
\end{itemize} 
\label{packageproperties}\end{theorem}

\begin{proof} All properties follow from the corresponding properties of $\fil^{\geq \star}_{Hod}\widehat{\mathrm{dR}}_{R/A}$ and the complete filtration of Lemma \ref{filtration} a).  It suffices to show the maps on associated graded 
$$ \mathrm{gr}^i\mathcal N^{\geq\bullet}F^*\mathscr{H}^{(1)}_\prism(R/A)\simeq \pi^*\fil^{\geq \bullet}_{Hod}\widehat{\mathrm{dR}}_{R/A}\{i\},\quad i\geq 0$$
are isomorphisms. Since $\pi^*\{i\}=-\otimes_{\bz_p}\co_{\mathrm{WCart}^{\mathrm{HT}}}\{i\}$ has finite tor-dimension it commutes with totalization and hence Prop. \ref{deRham} proves the isomorphisms on $\mathrm{gr}^i$.  For a) note that $\pi^*\{i\}$ also commutes with $-\otimes_AA'$, as does the colimit in Lemma \ref{alphadef} b) because of base change for relative prismatic cohomology. 
For e) see also \cite{bhatt-lurie-22}[Rem. 4.4.2].
\end{proof}

\begin{definition} For a derived formal scheme $\mathscr{X}$ over a $\delta$-ring $A$ we define $\underline{\mathscr{E}}(\mathscr{X}/A)\in\csheaf_A$ by Zariski descent, using Theorem \ref{packageproperties} b). Define 
\begin{align*}
\mathcal N^{\geq\bullet}\prism^{(1)}_{\mathscr{X}/A}\{\star\}:= &R\Gamma_{\mathcal N}^{\geq\bullet}(\underline{\mathscr{E}}(\mathscr{X}/A)^{(1)}\{\star\})\\
R\Gamma_{\syn}(\mathscr{X}/A,\bz_p(n)):= &R\Gamma_{\syn}(\underline{\mathscr{E}}(\mathscr{X}/A)^{(1)}\{n\})\\
R\Gamma_{\add}(\mathscr{X}/A,\bz_p(n)):= &R\Gamma_{\add}(\underline{\mathscr{E}}(\mathscr{X}/A)^{(1)}\{n\})
\end{align*}
\end{definition}

\subsubsection{Compatibility with the definitions of  \cite{akn23}} The uncompleted Nygaard filtration was defined in \cite{akn23}[Def. 6.6] as a fibre product
\begin{equation}\xymatrix{ \mathcal N^{\geq\bullet}\prism^{(1)}_{\mathscr{X}/A}\{\star\}\ar[d]
\ar[r]^{\fil^{\geq \bullet}\gamma^{dR}_{\prism,\mathscr{X}/A}\{\star\}} & \fil^{\geq \bullet}_{Hod}\widehat{\mathrm{dR}}_{\mathscr{X}/A}\ar[d]\\
R\Gamma(\mathrm{WCart}, \mathcal N^{\geq\bullet}F^*\mathscr{H}^{(1)}_\prism(\mathscr{X}/A)\{\star\})\ar[r]& \fil^{\geq \bullet}_{Hod}\widehat{\mathrm{dR}}_{\mathscr{X}/A}\otimes R\Gamma(\mathrm{WCart}^{\mathrm{HT}},\co_{\mathrm{HT}}),
}\label{gammadef}\end{equation}
which is our Definition \ref{globalnygaard}. Syntomic cohomology was defined in \cite{akn23}[Def. 7.8] by the fibre sequence
$$R\Gamma_{\syn}(\mathscr{X}/A,\bz_p(n)) \to\mathcal N^{\geq n}\widehat{\prism}^{(1)}_{\mathscr{X}/A}\{n\}\xrightarrow{\can-c^0\varphi\{n\}}\widehat{\prism}^{(1)}_{\mathscr{X}/A}\{n\}$$
which is equivalent to our Definition \ref{globalsyntomic} by \cite{akn23}[Prop. 7.12] (see also Lemma \ref{completesyn} below). Finally, the prismatic package (\ref{package}) of \cite{akn23}[Thm. 8.8] can be recovered from the syntomic package by taking global sections
$$\underline{\prism}_{R/A}=R\Gamma( \underline{\mathscr{E}}(R/A))$$
in the sense of the following definition.

\begin{definition} (Prismatic global sections) Define 
$$R\Gamma:\gsheaf_A\to \mathscr{G}_A$$ 
$$\underline{\mathscr{E}}= (\fil^{\geq\star}_{\mathcal N}\mathscr{E}',\fil^{\geq\star}_HD,\alpha, c_{dR}, c_F, c_{\square})   \mapsto (H,N,c^0,\varphi)$$ 
by
\begin{itemize}
\item[a)] $H:=R\Gamma(\mathrm{WCart},\mathscr{I}^\bullet\mathscr{E}\{\star\})$
\item[b)] $N$ is the filtration completion of $R\Gamma_{\mathcal N}^{\geq\bullet}(\underline{\mathscr{E}}^{(1)}\{\star\})$
\item[c)] $c^0$ is the map from Def. \ref{globalsyntomic} composed with the map to filtration completion.
\item[d)] As in Def. \ref{globalsyntomic} the map $\varphi^\bullet\{n\}$ is the filtration completion of the composite
$$R\Gamma_{\mathcal N}^{\geq\bullet}(\underline{\mathscr{E}}^{(1)}\{\star\})\to R\Gamma(\mathrm{WCart}, \fil^{\geq\bullet}_{\mathcal N}\mathscr{E}'\otimes F^*\co\{n\}) \xrightarrow{\Phi_*}
R\Gamma(\mathrm{WCart},\mathscr{I}^{\bullet-n}\mathscr{E}\{n\})$$
\end{itemize}
\label{globalpackage}\end{definition}

We will not try to define $c$ in arbitrary filtration degrees. This was done in \cite{akn23}[Lemma 7.5] if $\underline{\mathscr{E}}= \underline{\mathscr{E}}(R/A)$ is a syntomic package as in Def. \ref{sheafpackage}. We remark that the functor $R\Gamma$ is lax symmetric monoidal and therefore induces a functor
$$ R\Gamma: \csheaf_A\to\mathscr{C}_A.$$

\subsection{Prismatic cohomology of graded and filtered $\delta$-pairs}\label{deltastacks} 

In this section we recast the definitions of \cite{akn23}[Sec. 10] in terms of $\csheaf_A$ and generalize to derived formal schemes. 
  
\begin{definition} A \underline{graded $\delta$-ring} is a graded ring $A=\bigoplus_{i\in\bz}A^i$  with a $\delta$-ring structure on $A$ such that $\delta(A^i)\subseteq A^{pi}$. 

A \underline{graded $\delta$-pair} $( R,A)$ is a map of animated graded rings $A\to  R$ where $A$ is a graded $\delta$-ring.

\cite{akn23}[Def. 10.3] A \underline{filtered $\delta$-ring} is a strict, $\bn^{op}$-indexed filtered ring $F^{\geq \star} A$ with a $\delta$-ring structure on $A=F^{\geq 0} A$ such that $\delta(F^{\geq i}A)\subseteq F^{\geq pi}A$. 

 \cite{akn23}[Def. 10.10] A \underline{filtered $\delta$-pair} $(F^{\geq \star} R,F^{\geq \star} A)$ is a map of $\bn^{op}$-indexed animated filtered rings $F^{\geq \star} A\to F^{\geq \star} R$ where $F^{\geq \star} A$ is a filtered $\delta$-ring.
 
\label{fildefs}\end{definition}

\begin{example} The Hopf algebra $\co(\bg_m)=\bz[u,u^{-1}]$ is a graded $\delta$-ring where $\varphi(u)=u^p$. For any graded $\delta$-ring $A$ the coaction
$$ A\to \co(\bg_m)\otimes A $$
is a map of $\delta$-rings \cite{akn23}[Lemma 10.20].
If $F^{\geq \star} A$ is a filtered $\delta$-ring then $\mathrm{Rees}(F^{\geq \star} A)$ and $\bigoplus_{i\in\bz}\mathrm{gr}_F^iA\simeq \mathrm{Rees}(F^{\geq \star} A)\otimes_{\bz[t]}\bz$ are graded $\delta$-rings (with $\varphi(t)=t^p$).
\label{deltacoaction}\end{example}

\begin{example} The initial filtered ring of Example \ref{initial} is a filtered $\delta$-ring when $\bz$ is equipped with its unique $\delta$-ring structure ($\varphi=\id$). Hence $\mathrm{Rees}(F^{\geq\star}\bz)=\bz[t]$ is a graded $\delta$-ring.
\label{finital}\end{example}

\begin{definition} For a graded $\delta$-ring $A$ denote by $\gsheaf_{A^{gr}}$ the $\infty$-category of sextuples 
$$\underline{\mathscr{E}}=(\fil^{\geq\star}_{\mathcal N}\mathscr{E}',\fil^{\geq\star}_HD,\alpha, c_{dR}, c_F, c_{\square})$$ 
where  
\begin{itemize}
\item[a)] $\fil^{\geq\star}_{\mathcal N}\mathscr{E}'$ is an object of $\Mod_{A^{gr}}\otimes\cdf(\fil^{\geq\star}_{\mathcal N}F^*\co_{\mathrm{WCart}})$.
\item[b)] $\fil^{\geq\star}_HD$ is an object of $\Mod_{A^{gr}}\otimes\cdf(\bz_p)$. 
\item[c)] $\alpha$ is an $A^{gr}$-linear isomorphism, $c_{dR}$, $c_F$ are $A^{gr}$-linear morphisms and $c_\square$ is an $A^{gr}$-linear commutative diagram as in Def. \ref{fgauges}. 
\end{itemize}  
\label{gradedfgauges}\end{definition} 

The precise definition of $\gsheaf_{A^{gr}}$ is analogous to Def. \ref{rigorousfgauges} in App. A. The following definition is slightly informal and will be made more rigorous in Def. \ref{package}.

\begin{definition} (The graded syntomic package) Let $( R,A)$ be a graded $\delta$-pair.  Denote by 
$$  \underline{\mathscr{E}}( R/A)^{gr}\in\csheaf_{A^{gr}}$$
the object $\underline{\mathscr{E}}(R/A)\in\csheaf_{A}$ with $\co(\bg_m)$-comodule structure obtained by functoriality along 
$$\begin{CD} A @>>> \co(\bg_m)\otimes A\\ @VVV@VVV\\  R @>>> \co(\bg_m)\otimes R\end{CD}$$ 
combined with base change (Thm. \ref{packageproperties} a)) and the isomorphism
$$\underline{\mathscr{E}}( R/A)\otimes_A(A\otimes\co(\bg_m))\simeq \underline{\mathscr{E}}( R/A)\otimes\co(\bg_m).$$
\label{gradedprismatic}\end{definition}

\begin{remark} We will use the notation $\underline{\mathscr{E}}(R/A)^{gr}=\bigoplus_{i\in\bz} \underline{\mathscr{E}}(R/A)^{gr,i}$ but caution the reader that the individual graded pieces $\underline{\mathscr{E}}(R/A)^{gr,i}$ do not lie in $\gsheaf_{A^0}$ but only in $\mathscr{G}^{\mathcal N}_{A^0}$ (in the notation of Def. \ref{rigorousfgauges} in App. A). However, $\underline{\mathscr{E}}(R/A)^{gr,0}$ is an object of $\csheaf_{A^{0}}$.
\end{remark}

\begin{lemma} (Graded base change) For a map of graded $\delta$-rings $A\to A'$ the natural map
$$  \underline{\mathscr{E}}(R/A)^{gr}\otimes_{A^{gr}}A'^{gr}\to  \underline{\mathscr{E}}(R\otimes_AA'/A')^{gr}$$
is an equivalence.
\label{gradedbc}\end{lemma}
\begin{proof} The forgetful functor $\gsheaf_{A^{gr}}\to\gsheaf_{A}$ is conservative by Lemma \ref{omnibus} b) and commutes with base change. Hence the result follows from Thm. \ref{packageproperties} a).
\end{proof}

\begin{definition} (The filtered syntomic package) Let $( F^{\geq\star}R ,F^{\geq\star} A)$ be a filtered $\delta$-pair. Define
$$  \underline{\mathscr{E}}(F^{\geq\star}R/F^{\geq\star} A):= \underline{\mathscr{E}}(\mathrm{Rees}(F^{\geq\star}R)/\mathrm{Rees}(F^{\geq\star} A))^{gr} \in\csheaf_{\mathrm{Rees}(F^{\geq\star} A)^{gr}}$$
with \underline{underlying object}
$$ F^\infty \underline{\mathscr{E}}(F^{\geq\star}R/F^{\geq\star} A) \in\csheaf_{\mathrm{Rees}(F^{\geq\star} A)[t^{-1}]^{gr}}$$
and \underline{associated graded}
$$ \bigoplus_{i\in\bz}\mathrm{gr}_F^i \underline{\mathscr{E}}(F^{\geq\star}R/F^{\geq\star} A) \in\csheaf_{(\bigoplus_i\mathrm{gr}_F^i A)^{gr}}$$
given by the images of $\underline{\mathscr{E}}(\mathrm{Rees}(F^{\geq\star}R)/\mathrm{Rees}(F^{\geq\star} A))^{gr}$ under the respective base change functors.
\label{filteredprismatic}\end{definition}

For a filtered $\delta$-ring $F^{\geq\star}A$ the underlying ring $F^\infty A\simeq F^{\geq 0}A$ is again a $\delta$-ring. We have an isomorphism of graded $\delta$-rings $$\mathrm{Rees}(F^{\geq\star} A)[t^{-1}]\simeq F^{\geq 0}A\otimes_\bz\bz[t,t^{-1}]$$
and a symmetric monoidal equivalence
\begin{equation} \gsheaf_{F^{\geq 0}A}\simeq \gsheaf_{F^{\geq 0}A\otimes_\bz\bz[t,t^{-1}]^{gr}}\simeq \gsheaf_{\mathrm{Rees}(F^{\geq\star} A)[t^{-1}]^{gr}}\label{gmodequivalence}\end{equation}
given by scalar extension with inverse given by passing to weight $0$ parts.

\begin{prop} Let $(F^{\geq \star} R,F^{\geq \star} A)$ be a filtered $\delta$-pair. 

\begin{itemize}

\item[a)] There is a natural isomorphism in $\csheaf_{F^{\geq 0}A}$ 
$$ F^\infty \underline{\mathscr{E}}(F^{\geq\star}R/F^{\geq\star} A)\simeq \underline{\mathscr{E}}(F^{\geq 0}R/F^{\geq 0} A)$$
using (\ref{gmodequivalence}) to map the left hand side into $\csheaf_{F^{\geq 0}A}$.

\item[b)] There is a natural isomorphism in $\csheaf_{(\bigoplus_i\mathrm{gr}_F^i A)^{gr}}$
$$ \bigoplus_{i\in\bz}\mathrm{gr}_F^i \underline{\mathscr{E}}(F^{\geq\star}R/F^{\geq\star} A)\simeq  \underline{\mathscr{E}}\left( \bigoplus_i\mathrm{gr}_F^i R/\bigoplus_i\mathrm{gr}_F^i A\right)^{gr}.$$

\item[c)] There is a natural isomorphisms in $\csheaf_{\mathrm{gr}_F^0 A}$
$$ \mathrm{gr}_F^0\underline{\mathscr{E}}(F^{\geq\star}R/F^{\geq\star} A)\simeq \underline{\mathscr{E}}\left(\mathrm{gr}_F^0 R/\mathrm{gr}_F^0 A\right) $$
and $\mathrm{gr}_F^i \underline{\mathscr{E}}(F^{\geq\star}R/F^{\geq\star} A)=0$ for $i<0$.
\end{itemize}

\label{affinebc}\end{prop} 

\begin{proof} Note that by definition  $F^{\geq\star} R$ is also $\bn^{op}$-indexed so that $F^\infty R=F^{\geq 0} R$. Part a) and b) then follow by graded base change (Lemma \ref{gradedbc}). Part c) is a consequence of b) and the following Lemma. 
\end{proof}

\begin{lemma} For a non-negatively graded $\delta$-pair $(R,A)$ the object $\underline{\mathscr{E}}(R/A)^{gr}$ is non-negatively graded and the map of $\delta$-pairs $(R^0,A^0)\to(R,A)$ induces an isomorphism
$$ \underline{\mathscr{E}}(R^0/A^0)\xrightarrow{\sim} \underline{\mathscr{E}}(R/A)^{gr,0}$$
in $\csheaf_{A^{0}}$.
\end{lemma}

\begin{proof}
Using the filtration of Lemma \ref{filtration} a) this follows from Lemma \ref{gr0lemma}.
 \end{proof}

For 
$$\underline{\mathscr{E}}= (\mathrm{Rees}(F^{\geq\bullet}\fil^{\geq\star}_{\mathcal N}\mathscr{E}'),\mathrm{Rees}(F^{\geq\bullet}\fil^{\geq\star}_HD),\alpha, c_{dR}, c_F, c_{\square})\in \csheaf_{\mathrm{Rees}(F^{\geq\star} A)^{gr}}$$ 
the object $F^{\geq\bullet}\fil^{\geq\star}_{\mathcal N}\mathscr{E}'$ is a bifiltered quasicoherent sheaf on $\mathrm{WCart}$ with a module structure over the bifiltered quasicoherent sheaf of algebras $F^{\geq \bullet} A\otimes\fil^{\geq\star}_{\mathcal N}F^*\co_{\mathrm{WCart}}$ and $F^{\geq\bullet}\fil^{\geq\star}_HD$ is a bifiltered $F^{\geq \bullet} A$-module. The objects $R\Gamma_{\mathcal N}^{\geq\star}(\underline{\mathscr{E}}^{(1)}\{n\})$ are likewise bifiltered $F^{\geq \bullet} A$-modules, whereas $R\Gamma_{\add}(\underline{\mathscr{E}}^{(1)}\{n\})$ is a filtered $F^{\geq \bullet} A$-module and $R\Gamma_{\syn}(\underline{\mathscr{E}}^{(1)}\{n\})$ is a filtered $(F^{\geq 0}A)^{\varphi=1}$-module.

\begin{prop} Let $F^{\geq \bullet} A$ be a filtered $\delta$-ring and  $\underline{\mathscr{E}}\in \gsheaf_{\mathrm{Rees}(F^{\geq\star} A)^{gr}}$. Then for any $i\geq 1$ there is an isomorphism of bounded filtrations $\mathrm{triv}$ fitting into a commutative diagram
$$\xymatrix{F^{[i,pi[}R\Gamma_{\syn}(\underline{\mathscr{E}}^{(1)}\{n\})\ar[rr]_{\sim}^{\mathrm{triv}}\ar[dr] && F^{[i,pi[}R\Gamma_{\add}(\underline{\mathscr{E}}^{(1)}\{n\})\ar[dl]\\ &F^{[i,pi[}R\Gamma_{\mathcal N}^{\geq n}(\underline{\mathscr{E}}^{(1)}\{n\})&
}$$ 
functorial in $\underline{\mathscr{E}}$.
In particular, for $i\geq c\geq 1$ there is a commutative diagram
$$\xymatrix{F^{[i,i+c[}R\Gamma_{\syn}(\underline{\mathscr{E}}^{(1)}\{n\})\ar[rr]_{\sim}^{\mathrm{triv}}\ar[dr] && F^{[i,i+c[}R\Gamma_{\add}(\underline{\mathscr{E}}^{(1)}\{n\})\ar[dl]\\ &F^{[i,i+c[}R\Gamma_{\mathcal N}^{\geq n}(\underline{\mathscr{E}}^{(1)}\{n\})&
}$$
functorial in $\underline{\mathscr{E}}$. 
\label{griso} \end{prop}

\begin{proof} The map $c^0$ from Def. \ref{globalsyntomic} multiplies filtration degrees by $p$
$$c^0:F^{\geq i}R\Gamma(\mathrm{WCart},\mathscr{E}\{n\})\to F^{\geq pi}R\Gamma(\mathrm{WCart},\mathscr{E}^{(1)}\{n\})\to F^{\geq pi}R\Gamma_{\mathcal N}^{\geq 0}(\underline{\mathscr{E}}^{(1)}\{n\})$$
since $\varphi: \mathrm{Rees}(F^{\geq\star} A)\to \mathrm{Rees}(F^{\geq\star} A)$ multiplies grading weights by $p$. Hence $F^{[i,pi[}c^0=0$ and there are fibre sequences
$$ F^{[i,pi[}R\Gamma_{?}(\underline{\mathscr{E}}^{(1)}\{n\})\to F^{[i,pi[}R\Gamma_{\mathcal N}^{\geq n}(\underline{\mathscr{E}}^{(1)}\{n\})\xrightarrow{F^{[i,pi[}\can}F^{[i,pi[}R\Gamma_{\mathcal N}^{\geq 0}(\underline{\mathscr{E}}^{(1)}\{n\})$$
for both $?=\syn,\add$.
For $i\geq  c\geq 1$ we have $pi\geq i+c$, giving the second statement.  

\end{proof}

\subsubsection{Globalisation} We reformulate the definitions of the last section in terms of derived formal stacks and generalize to derived formal schemes. For a graded $\delta$-ring $A$ the formal stack 
$$ \mathfrak{A}:=\Spf(A)/\widehat{\bg}_m$$
is a formal $\delta$-stack in the sense of Def. \ref{deltastackdef}. Indeed, in view of Example \ref{deltacoaction}
we may choose $\Spf(\ca^\bullet)$ to be the action groupoid of the $\widehat{\bg}_m$-action on $\Spf(A)$. Since $(\Mod_{A^{gr}})^\wedge_p\simeq\cd(\mathfrak{A})$  we have a symmetric monoidal equivalence
\begin{equation}\gsheaf_{A^{gr}}\simeq \gsheaf_{\mathfrak{A}}\label{stackyreformulation}\end{equation}
by Lemma \ref{DAgeneralities} c) in App. A.

\begin{definition}  Let $A$ be a graded $\delta$-ring and 
$$ \mathfrak{X}\to\mathfrak{A}:=\Spf(A)/\widehat{\bg}_m$$
a morphism of derived formal stacks such that $\mathscr{X}_0:=\mathfrak{X}\times_{\mathfrak{A}}\Spf(A)$ is a derived formal scheme. Write 
$$\ca^i:= A\hat{\otimes}\co(\widehat{\bg}_m)^{\hat{\otimes} i},\quad
\mathscr{X}_i:=\mathfrak{X}\times_\mathfrak{A}\Spf(\ca^i)\simeq \mathscr{X}_0\times\widehat{\bg}_m^i.$$ 
Define
$$\underline{\mathscr{E}}(\mathfrak{X}/\mathfrak{A}):=\left(\underline{\mathscr{E}}(\mathscr{X}_\bullet/\mathcal{A}^\bullet)\right)\in\Tot \csheaf_{\ca^\bullet}\simeq\csheaf_\mathfrak{A}$$
using base change (Theorem \ref{packageproperties} a)). Define
\begin{align*}
\underline{\mathscr{E}}(\mathfrak{X}/\mathfrak{A})=: &(\fil^{\geq\star}_{\mathcal N}\mathscr{E}'(\mathfrak{X}/\mathfrak{A}),\fil^{\geq \star}_{Hod}\widehat{\mathrm{dR}}_{\mathfrak{X}/\mathfrak{A}} ,\alpha, c_{dR}, c_F, c_{\square})\\
\mathcal N^{\geq\bullet}\prism^{(1)}_{\mathfrak{X}/\mathfrak{A}}\{\star\}:= &R\Gamma_{\mathcal N}^{\geq\bullet}(\underline{\mathscr{E}}(\mathfrak{X}/\mathfrak{A})^{(1)}\{\star\})\\
R\Gamma_{\syn}(\mathfrak{X}/\mathfrak{A},\bz_p(n)):= &R\Gamma_{\syn}(\underline{\mathscr{E}}(\mathfrak{X}/\mathfrak{A})^{(1)}\{n\})\\
R\Gamma_{\add}(\mathfrak{X}/\mathfrak{A},\bz_p(n)):= &R\Gamma_{\add}(\underline{\mathscr{E}}(\mathfrak{X}/\mathfrak{A})^{(1)}\{n\})
\end{align*}
and let $$\fil^{\geq \bullet}\gamma^{dR}_{\prism,\mathfrak{X}/\mathfrak{A}}\{\star\}:\mathcal N^{\geq\bullet}\prism^{(1)}_{\mathfrak{X}/\mathfrak{A}}\{\star\}\to \fil^{\geq \bullet}_{Hod}\widehat{\mathrm{dR}}_{\mathfrak{X}/\mathfrak{A}}$$ be the induced map of filtered graded objects of $\cd(\mathfrak{A})$.
\label{package}\end{definition} 

For a graded $\delta$-pair $(R,A)$ the object $\underline{\mathscr{E}}(R/A)^{gr}\in\csheaf_{A^{gr}}$ of Def. \ref{gradedprismatic} is mapped to 
$$\underline{\mathscr{E}}\left((\Spf(R)/\widehat{\bg}_m)/(\Spf(A)/\widehat{\bg}_m)\right) \in\csheaf_\mathfrak{A}$$ 
under the isomorphism (\ref{stackyreformulation}).

\begin{lemma} In the situation of Def. \ref{package} let
$\mathcal N^{\geq\bullet}\widehat{\prism}^{(1)}_{\mathfrak{X}/\mathfrak{A}}\{\star\}$ be the completion of $\mathcal N^{\geq\bullet}\prism^{(1)}_{\mathfrak{X}/\mathfrak{A}}\{\star\}$ with respect to the Nygaard filtration and define
$$\widehat{R\Gamma}_{\syn}(\mathfrak{X}/\mathfrak{A},\bz_p(n)):=\fibre\left(\mathcal N^{\geq n}\widehat{\prism}^{(1)}_{\mathfrak{X}/\mathfrak{A}}\{n\}\xrightarrow{\can-c^0\varphi\{n\}}\widehat{\prism}^{(1)}_{\mathfrak{X}/\mathfrak{A}}\{n\}\right).$$
Then the natural map 
$$ R\Gamma_{\syn}(\mathfrak{X}/\mathfrak{A},\bz_p(n))\to \widehat{R\Gamma}_{\syn}(\mathfrak{X}/\mathfrak{A},\bz_p(n))$$
is an equivalence.
\label{completesyn}\end{lemma}

\begin{proof} The proof of \cite{bhatt-lurie-22}[Prop. 7.4.6] applies, see also \cite{akn23}[Prop. 7.12].  Note that the map $c^0\varphi\{n\}$ is well defined since the target of the map in Def. \ref{globalpackage} d) is already filtration complete.
\end{proof}

\begin{prop} Let $F^{\geq\star} A$ be a filtered $\delta$-ring and 
$$\mathfrak{X}\to \mathfrak{A}:=\Spf(\mathrm{Rees}(F^{\geq\star} A))/\widehat{\bg}_m$$
a schematic morphism of derived formal stacks in the sense of Def. \ref{schematic}. Put 
$$A=:F^{\geq 0}A=F^\infty A$$ and let $\mathscr{X}\simeq \mathfrak{X}\times_{\mathfrak{A}}\Spf(A)$ be the underlying derived formal scheme of $\mathfrak{X}$.

\begin{itemize}

\item[a)] There is a natural isomorphism in $\csheaf_{A}$ 
$$ F^\infty \underline{\mathscr{E}}(\mathfrak{X}/\mathfrak{A})\simeq \underline{\mathscr{E}}(\mathscr{X}/A).$$

\item[b)] There is a natural isomorphism in $\csheaf_{(\bigoplus_i\mathrm{gr}_F^i A)^{gr}}$
$$ \bigoplus_{i\in\bz}\mathrm{gr}_F^i \underline{\mathscr{E}}(\mathfrak{X}/\mathfrak{A})\simeq  \underline{\mathscr{E}}\left(\mathscr{S}pf\bigl(\bigoplus_i\mathrm{gr}_F^i \co_{\mathscr{X}}\bigr)/\widehat{\bg}_m/\Spf\bigl(\bigoplus_i\mathrm{gr}_F^i A\bigr)/\widehat{\bg}_m\right).$$

\item[c)] There is a natural isomorphisms in $\csheaf_{\mathrm{gr}_F^0 A}$
$$ \mathrm{gr}_F^0\underline{\mathscr{E}}(\mathfrak{X}/\mathfrak{A})\simeq \underline{\mathscr{E}}\left(\mathscr{S}pf(\mathrm{gr}_F^0 \co_{\mathscr{X}})/\mathrm{gr}_F^0 A\right) $$
and $\mathrm{gr}_F^i \underline{\mathscr{E}}(\mathfrak{X}/\mathfrak{A})=0$ for $i<0$.
\end{itemize}

\label{schematicbc}\end{prop} 

\begin{proof} This follows from Prop. \ref{affinebc} by choosing an affine cover of $\mathscr{X}$. 
\end{proof}

\begin{corollary} (Filtrations on global sections) Let $F^{\geq\star} A$ be a filtered $\delta$-ring, $\mathfrak{A}=F\Spf(A)$ and $\mathfrak{X}\to \mathfrak{A}$ a schematic morphism of derived formal stacks. Denote by $\mathscr{X}$ the underlying derived formal scheme of $\mathfrak{X}$, set  $\mathrm{gr}^0_F\mathscr{X}:=\mathscr{S}pf(\mathrm{gr}_F^0 \co_{\mathscr{X}})$ and $A:=F^{\geq 0}A$.

a) The $F^{\geq\star} A$-module $\mathcal N^{\geq\bullet}\prism^{(1)}_{\mathfrak{X}/\mathfrak{A}}\{\star\}$ is $\bn^{op}$-indexed and 
the underlying module (i.e. $F^{\geq 0}$) coincides with $\mathcal N^{\geq\bullet}\prism^{(1)}_{\mathscr{X}/A}\{\star\}$.

b) For $?=\syn,\add$ the filtration $R\Gamma_{?}(\mathfrak{X}/\mathfrak{A},\bz_p(n))$ is $\bn^{op}$-indexed and the underlying module (i.e. $F^{\geq 0}$) coincides with $R\Gamma_{?}(\mathscr{X}/A,\bz_p(n))$. 

c) There are natural isomorphisms 
\begin{align*}\mathrm{gr}^0_F\mathcal N^{\geq\bullet}\prism^{(1)}_{\mathscr{X}/A}\{\star\}\simeq\, &\mathcal N^{\geq\bullet}\prism^{(1)}_{\mathrm{gr}^0_F\mathscr{X}/\mathrm{gr}^0_FA}\{\star\}\\
\mathrm{gr}^0_FR\Gamma_{\syn}(\mathscr{X}/A,\bz_p(n))\simeq\, & R\Gamma_{\syn}(\mathrm{gr}^0_F\mathscr{X}/\mathrm{gr}^0_FA,\bz_p(n))\\
\mathrm{gr}^0_FR\Gamma_{\add}(\mathscr{X}/A,\bz_p(n))\simeq\, & R\Gamma_{\add}(\mathrm{gr}^0_F\mathscr{X}/\mathrm{gr}^0_FA,\bz_p(n))
\end{align*}

\label{globalfiltrations} \end{corollary}

\begin{proof}  The filtrations $\fil^{\geq\star}_{\mathcal N}\mathscr{E}'(\mathfrak{X}/\mathfrak{A})$ and $\fil^{\geq \bullet}_{Hod}\widehat{\mathrm{dR}}_{\mathfrak{X}/\mathfrak{A}}$ are $\bn^{op}$-indexed by Prop. \ref{schematicbc} c). Their underlying modules (i.e. $F^{\geq 0}$) coincide with $ \mathcal N^{\geq\bullet}F^*\mathscr{H}^{(1)}_\prism(\mathscr{X}/A)$ and $\fil^{\geq \bullet}_{Hod}\widehat{\mathrm{dR}}_{\mathscr{X}/A}$, respectively, by Prop. \ref{schematicbc} a).  By Prop. \ref{schematicbc} c) there are natural isomorphisms 
\begin{align*}\mathrm{gr}^0_F\mathcal N^{\geq\bullet}F^*\mathscr{H}^{(1)}_\prism(\mathscr{X}/A)\simeq &\mathcal N^{\geq\bullet}F^*\mathscr{H}^{(1)}_\prism( \mathrm{gr}^0_F\mathscr{X}/ \mathrm{gr}^0_FA)\\
\mathrm{gr}^0_F\fil^{\geq \bullet}_{Hod}\widehat{\mathrm{dR}}_{\mathscr{X}/A}\simeq &\fil^{\geq \bullet}_{Hod}\widehat{\mathrm{dR}}_{\mathrm{gr}^0_F\mathscr{X}/\mathrm{gr}^0_FA}.\end{align*}
The corollary then follows from the fact that $F^{\geq i}$ and $\mathrm{gr}^i_F$ commute with the fibre products, resp. sequences, in Def. \ref{globalnygaard}, resp. Def. \ref{globalsyntomic}, as well as with taking global sections on $\mathrm{WCart}$.
\end{proof}

\begin{notation} In the situation of Cor. \ref{globalfiltrations} we will also denote ${\mathcal N}^{\geq \star}\prism^{(1)}_{ \mathfrak{X}/\mathfrak{A}}\{n\}$
by
$$F^{\geq\star}{\mathcal N}^{\geq \star}\prism^{(1)}_{ \mathscr{X}/A}\{n\}\in \cd(F^{\geq\star}A)$$ and similarly for $\fil^{\geq\star}_{\mathcal N}\mathscr{E}'(\mathfrak{X}/\mathfrak{A})$, $\fil^{\geq \star}_{Hod}\widehat{\mathrm{dR}}_{\mathfrak{X}/\mathfrak{A}}$ and $R\Gamma_{?}(\mathfrak{X}/\mathfrak{A},\bz_p(n))$ for $?=\syn,\add$.  We also use the notation 
$$F^{\geq\star}L_{\mathscr{X}/A}\in\cdf(\mathscr{X},F^{\geq\star}\co_{\mathscr{X}})\simeq \cd(\mathfrak{X})$$ 
for the cotangent complex $L_{\mathfrak{X}/\mathfrak{A}}:= (L_{\mathscr{X}_\bullet/\mathcal{A}^\bullet})\in\Tot\cd(\mathscr{X}_\bullet)\simeq \cd(\mathfrak{X})$.
\end{notation}

Note that $F^{\geq\star} R\Gamma_{\add}(\mathscr{X}/A ,\bz_p(n))$ is a $F^{\geq\star} A$-module but $F^{\geq\star} R\Gamma_{\syn}(\mathscr{X}/A ,\bz_p(n))$ is only a filtered $A^{\varphi =1}$-module.

\subsection{Filtered amplitude} In this section we briefly discuss completeness and boundedness of the filtrations $F^{\geq\star} R\Gamma_{\add}(\mathscr{X}/A ,\bz_p(n))$ and  $F^{\geq\star} R\Gamma_{\syn}(\mathscr{X}/A ,\bz_p(n))$ assuming that the filtration $F^{\geq\star}\co_{\mathscr{X}}$ is complete, resp. bounded. The additional assumption which allows us to connect one to the other is that $F^{\geq\star}L_{\mathscr{X}/A}$ has finite filtered amplitude in the sense of Def. \ref{amplitude}. The main result of this section, Prop. \ref{completeandbounded}, is an elaboration of \cite{akn24}[Prop. 2.11, Rem. 2.15] but is not needed in the remainder of this article (except for a brief appearance  in the proof of Lemma \ref{mulemma}).

\begin{definition} Let $F^{\geq\star}R$ be an animated filtered ring. A filtered $F^{\geq\star}R$-module $F^{\geq\star}M$ has \underline{finite filtered amplitude}, resp. \underline{filtered amplitude in $[a,b]$}, if $F^{\geq\star}M$ lies in the thick subcategory of $\cd(F^{\geq\star}R)$ generated by the filtered shifts $F^{\geq\star-m}R$ with $m\in\bz$, resp. with $m\in [a,b]$. 

Let $g:\mathfrak{X}\to \widehat{\ba}^1/\widehat{\bg}_m$ be a schematic morphism of derived formal stacks.
An object $F^{\geq\star}\cm\in\cdf(\mathscr{X},F^{\geq\star}\co_{\mathscr{X}})\simeq\cd(\mathfrak{X})$ has \underline{finite filtered amplitude}, resp. \underline{filtered ampli-} \underline{tude in $[a,b]$}, if its restriction to any affine open does.  
\label{amplitude}\end{definition}

Any complex of finite filtered amplitude has filtered amplitude in some interval $[a,b]$. This is clear in the affine case and follows in general since $\mathscr{X}$ is quasi-compact.

\begin{remark} Since the $F^{\geq\star}R$-modules $F^{\geq\star-m}R$ correspond to line bundles on the derived formal stack
$$ F\Spf(R):=\Spf(\mathrm{Rees}(F^{\geq \star} R))/\widehat{\bg}_m$$
any module $F^{\geq\star}M$ of finite filtered amplitude is a perfect complex on $F\Spf(R)$. However, we do not know if conversely any perfect complex on $F\Spf(R)$ has finite filtered amplitude.
\end{remark}

\begin{lemma} If $F^{\geq\star}\cm$ has filtered amplitude in $[a,b]$ then the derived exterior power $L\bigwedge_{F^{\geq\star}\co_\mathscr{X}}^iF^{\geq\star}\cm$ has filtered amplitude in $[ia,ib]$.
\label{exterior}\end{lemma}

\begin{proof} We may assume $\mathfrak{X}=F\Spf(R)$. Denote by 
$$\mathscr{C}^{[a,b]}\subseteq \mathscr{D}^{[a,b]}\subseteq\cd(F^{\geq\star}R)$$ 
the additive and thick (i.e. idempotent complete stable) subcategory generated by the set of objects $\{\mathcal{L}(m):=F^{\geq\star-m}R\,\vert\,m\in [a,b]\}$. Denote by 
$\mathrm{Stab}(\mathscr{C}^{[a,b]})$ the stable envelope of the additive category $\mathscr{C}^{[a,b]}$ \cite{bgmn21}[Constr. 2.16] and by $\overline{\mathrm{Stab}(\mathscr{C}^{[a,b]})}$ its idempotent completion. Let $\mathscr{E}$ be a stable $\infty$-category. By \cite{bgmn21}[Thm. 2.19] restriction along the natural additive functor induces an equivalence
$$ \Fun_{\leq i}(\mathscr{C}^{[a,b]},\mathscr{E})\simeq \Fun_{\leq i}(\mathrm{Stab}(\mathscr{C}^{[a,b]}),\mathscr{E})$$
between polynomial functors of degree $\leq i$ in the sense of additive categories \cite{bgmn21}[Def. 2.4] and polynomial functors of degree $\leq i$ in the sense of stable categories \cite{bgmn21}[Def. 2.11]. By the universal property of idempotent completion combined with \cite{lurieHA}[Prop. 6.1.5.4] (using that the idempotent completion is a full subcategory of the Ind-category) restriction along the natural exact functor gives an equivalence
$$\Fun_{\leq i}(\mathrm{Stab}(\mathscr{C}^{[a,b]}),\mathscr{E})\simeq \Fun_{\leq i}(\overline{\mathrm{Stab}(\mathscr{C}^{[a,b]})},\mathscr{E})$$
if $\mathscr{E}$ is idempotent complete. Assume that the functor $\bigwedge^i:=L\bigwedge_{F^{\geq\star}R}^i$ on $\mathscr{C}^{[a,b]}$ takes values in a thick subcategory $\mathscr{E}\subseteq \cd(F^{\geq\star}R)$. Then the unique extension $\bigwedge^i:\overline{\mathrm{Stab}(\mathscr{C}^{[a,b]})}\to\mathscr{E}$ composed with the inclusion $\mathscr{E}\subseteq \cd(F^{\geq\star}R)$ coincides with the composite
$$ \overline{\mathrm{Stab}(\mathscr{C}^{[a,b]})}\xrightarrow{\pi}\mathscr{D}^{[a,b]}\xrightarrow{\bigwedge^i}\cd(F^{\geq\star}R)$$
since both restrict to the same functor on $\mathscr{C}^{[a,b]}$. Since $\mathscr{D}^{[a,b]}$ is generated as a thick subcategory by $\mathscr{C}^{[a,b]}$ the exact functor $\pi$ is essentially surjective. It follows that $\bigwedge ^i$ on $\mathscr{D}^{[a,b]}$ takes values in $\mathscr{E}$. Taking $\mathscr{E}= \mathscr{D}^{[ia,ib]}$ it therefore suffices to show 
\begin{equation} \bigwedge^i(\mathcal{L}(m_1)\oplus\cdots\oplus\mathcal{L}(m_k))\in \mathscr{D}^{[ia,ib]}\label{ithpower}\end{equation}
since any object of $\mathscr{C}^{[a,b]}$ is such a finite direct sum of generators $\mathcal L(m)$. The module in (\ref{ithpower}) is trivial for $k<i$ and has a bounded filtration with graded pieces
$$ \bigwedge^j(\mathcal{L}(m_1)\oplus\cdots\oplus\mathcal{L}(m_i))\otimes \bigwedge^{i-j}(\mathcal{L}(m_1)\oplus\cdots\oplus\mathcal{L}(m_{k-i}))$$
for $k\geq i$. By induction on $i$ (the case $i=0$ being trivial) and then induction over $k\geq i$, the case $k=i$ being
$$\bigwedge^i(\mathcal{L}(m_1)\oplus\cdots\oplus\mathcal{L}(m_i))\simeq\mathcal{L}(m_1)\otimes\cdots\otimes\mathcal{L}(m_i)$$
we find that indeed (\ref{ithpower}) has filtered amplitude in $[ia,ib]$. Here we use that if $\mathcal{N}$, resp. $\mathcal{M}$, has filtered amplitude in $[a,b]$, resp. $[c,d]$, then $\mathcal{N}\otimes\mathcal{M}$ has filtered amplitude in $[a+c,b+d]$.
\end{proof}

We call a filtration $F^{\geq\star}\cm$ {\em $\square$-complete} if $\varprojlim_k F^{\geq k}\cm=0$ in $\cd_{\square}(\mathscr{X})$, the solid enlargement of $\cd(\mathscr{X})$ \cite{scholze-sixfunctor}[Lecture IX]. This is a stronger condition than completeness in $\cd(\mathscr{X})$ (even in the affine case) but has the advantage of passing to open subspaces $j:U\to\mathscr{X}$ as $j^*$ has a left adjoint and hence commutes with limits in the solid formalism. This notion will play no further role in this paper. If $\mathscr{X}$ is affine then $\square$-completeness can be replaced by completeness in Lemma \ref{lemmacomplete} and Prop. \ref{completeandbounded}.

\begin{lemma}\label{lemmacomplete} Assume $F^{\geq\star}\cm\in\cd(\mathfrak{X})$ has filtered amplitude in $[a,b]$.
\begin{enumerate}
\item[a)] If the filtration $F^{\geq\star}\co_{\mathscr{X}}$ is $\square$-complete then the filtration $F^{\geq\star}\cm$ is $\square$-complete.
\item[b)] If the filtration $F^{\geq\star}\co_{\mathscr{X}}$ is bounded in $[0,r]$ then the filtration $F^{\geq\star}\cm$ is bounded in $[a,b+r]$.
\end{enumerate}
\end{lemma}
\begin{proof} Choosing an affine Zariski hypercovering $j_i:U_i\to\mathscr{X}$ we have 
$$\cd_{\square}(\mathscr{X})\simeq\varprojlim_\Delta \cd_{\square}(U_i),\quad \cm\simeq \varprojlim_\Delta j_{i,*}j^*_i\cm$$ 
for each $\cm\in\cd_{\square}(\mathscr{X})$. So it suffices to assume that $\mathscr{X}$ is affine. In the affine case, $\square$-completeness of $F^{\geq \star}M$ is preserved by fibers, shifts and retracts, and is satisfied by $F^{\geq\star-m}R$ for $m\in\bz$ by assumption. This gives a). 

Similarly, boundedness of $F^{\geq \star}M$ in $[a,b+r]$ is preserved by fibers, shifts and retracts, and is satisfied by $F^{\geq\star-m}R$ for $m\in [a,b]$. This gives b). 
\end{proof}

\begin{prop} \label{completeandbounded} Let $F^{\geq\star} A$ be a filtered $\delta$-ring, $\mathfrak{A}=F\Spf(A)$ and $\mathfrak{X}\to \mathfrak{A}$ a schematic morphism of derived formal stacks. Denote by $\mathscr{X}$ the underlying derived formal scheme of $\mathfrak{X}$ and set $A=F^{\geq 0}A$. Assume $F^{\geq \star}L_{\mathscr{X}/A}$ has filtered amplitude in $[a,b]$ with $b\geq 0$.
\begin{enumerate}
\item[a)] If the filtration $F^{\geq\star}\co_{\mathscr{X}}$ is $\square$-complete then for any $i\geq 0$ the filtrations 
$$F^{\geq \star}L\widehat{\Omega}^{i}_{\mathscr{X}/A},\quad F^{\geq \star}\mathrm{gr}^i_{\mathcal N}\prism^{(1)}_{\mathscr{X}/A}\{n\}$$
as well as the filtrations
$$F^{\geq \star}R\Gamma_{\add}(\mathscr{X}/A,\bz_p(n)),\  F^{\geq \star}R\Gamma_{\syn}(\mathscr{X}/A,\bz_p(n)),\  F^{\geq \star} \widehat{\mathrm{dR}}_{\mathscr{X}/A}^{<n}[-1]$$ 
are complete.
\item[b)] If the filtration $F^{\geq\star}\co_{\mathscr{X}}$ is bounded in $[0,r]$ then the filtrations
$$F^{\geq \star}R\Gamma_{\add}(\mathscr{X}/A,\bz_p(n)),\  F^{\geq \star}R\Gamma_{\syn}(\mathscr{X}/A,\bz_p(n)),\  F^{\geq \star} \widehat{\mathrm{dR}}_{\mathscr{X}/A}^{<n}[-1]$$ 
are bounded in $[0,(n-1)b+r]$.
\end{enumerate}
\end{prop}

\begin{proof} a) By Lemma \ref{exterior} $L\bigwedge^i F^{\geq \star}L_{\mathscr{X}/A}$ has filtered amplitude in $[ia,ib]$. By Lemma \ref{lemmacomplete} a) the filtration $L\bigwedge^i F^{\geq \star}L_{\mathscr{X}/A}$ is $\square$-complete hence complete, hence so is 
$$F^{\geq \star}L\widehat{\Omega}^{i}_{\mathscr{X}/A}=R\Gamma(\mathscr{X},L\bigwedge{}^i F^{\geq \star}L_{\mathscr{X}/A}).$$
By Lemma \ref{filtration} d) $F^{\geq \star}\mathrm{gr}^i_{\mathcal N}\prism^{(1)}_{\mathscr{X}/A}\{n\}$ has a bounded filtration with graded pieces
\begin{equation}\mathrm{gr}^j=\begin{cases} F^{\geq \star}L\widehat{\Omega}^{i}_{\mathscr{X}/A}[-i] & j=0\\ 
\fibre\left(F^{\geq \star}L\widehat{\Omega}^{i-j}_{\mathscr{X}/A}\xrightarrow{j}F^{\geq \star}L\widehat{\Omega}^{i-j}_{\mathscr{X}/A}\right)[j-i] &j = 1,\dots,i\end{cases} \label{nygaardfilt}\end{equation}
where we use the fact that the Sen operator \cite{bhatt-lurie-22}[Prop. 3.5.11] operates trivially on $\pi^*M$ for any $M$ and by multiplication with $j$ on $\co_{\mathrm{WCart}^\mathrm{HT}}\{j\}$. We conclude that $F^{\geq \star}\mathrm{gr}^i_{\mathcal N}\prism^{(1)}_{\mathscr{X}/A}\{n\}$ is complete.

The filtration $F^{\geq \star}R\Gamma_{\add}(\mathscr{X}/A,\bz_p(n))$, resp. $F^{\geq \star} \widehat{\mathrm{dR}}_{\mathscr{X}/A}^{<n}$, has a bounded filtration with graded pieces 
\begin{equation}F^{\geq \star}\mathrm{gr}^i_{\mathcal N}\prism^{(1)}_{\mathscr{X}/A}\{n\}[-1],\quad\text{resp. } F^{\geq \star}L\widehat{\Omega}^{i}_{\mathscr{X}/A}[-i], \quad i=0,\dots,n-1,\label{addfilt}\end{equation}
hence is complete. Since a limit of complete filtrations is complete,  the filtration 
$F^{\geq \star}\mathcal N^{\geq m}\widehat{\prism}^{(1)}_{\mathscr{X}/A}\{n\}$ is complete for any $m$. It follows from Lemma \ref{completesyn} that the filtration $F^{\geq \star}R\Gamma_{\syn}(\mathscr{X}/A,\bz_p(n))$ is complete.

b) This is immediate from Lemma \ref{lemmacomplete} b), the filtrations (\ref{nygaardfilt}) and (\ref{addfilt}) and Cor. \ref{globalfiltrations} b).

\end{proof}

We note the following consequence of the proof of Prop. \ref{completeandbounded}, generalizing \cite{bhatt-lurie-22}[Prop. 5.5.12] (the case $A=\bz_p$ and $\star=0$).

\begin{lemma} Let $F^{\geq\star} A$ be a filtered $\delta$-ring, $\mathfrak{A}=F\Spf(A)$ and $\mathfrak{X}\to \mathfrak{A}$ a schematic morphism of derived formal stacks. Denote by $\mathscr{X}$ the underlying derived formal scheme of $\mathfrak{X}$ and set $A=F^{\geq 0}A$. The map
$$F^{\geq\star} \gamma^{dR,<n}_{\prism, \mathscr{X}/A}\{n\}[-1]: F^{\geq\star} R\Gamma_{\add}(\mathscr{X}/A,\bz_p(n))\to F^{\geq\star} \widehat{\mathrm{dR}}_{\mathscr{X}/A}^{<n}[-1]
$$
is a rational isomorphism for any $n\geq 1$ and an isomorphism for $0<n<p+1$.
\label{gammaiso}\end{lemma}

\begin{proof} This follows from the filtrations (\ref{addfilt}) and (\ref{nygaardfilt}) as multiplication by $j$ in (\ref{nygaardfilt}) is an equivalence for $j\leq n-1\leq p-1$ and a rational equivalence for any $j$. 
\end{proof}

\section{The fundamental square and the de Rham logarithm}\label{beilsection}

The following result generalizes \cite{bhatt-lurie-22}[Thm. 5.4.2] which corresponds to the special case $A=\bz_p$.

\begin{prop} (de Rham comparison relative to a $\delta$-ring) For a derived formal scheme $\mathscr{X}$ over a $\delta$-ring $A$ there is a commutative diagram, contravariantly functorial in the pair $\mathscr{X}/\Spf(A)$,
\begin{equation}\xymatrix{ 
  \prism^{(1)}_{\mathscr{X}/A}\{n\}\ar[rr] \ar[d]_{\gamma^{dR}_{\prism,{\mathscr{X}/A}}\{n\}}& &\prism^{(1)}_{(\mathscr{X}/p)/A}\{n\}\ar[dll]_{\beta_{\mathscr{X}/A}}^{\sim}\\
 \widehat{\mathrm{dR}}_{\mathscr{X}/A}&&
}\notag\end{equation}
where $\beta_{\mathscr{X}/A}$ is an isomorphism.
\label{betaprop}\end{prop}

\begin{proof} There is a commutative diagram
\begin{equation}\xymatrix{
\prism^{(1)}_{\mathscr{X}/A}\{n\}\ar[d]^{\sim}_{c_{\square,*}^{-1}}\ar[r]^\psi & \prism^{(1)}_{(\mathscr{X}/p)/A}\{n\}\ar[d]^{\sim} \\
R\Gamma(\mathrm{WCart},\mathscr{H}^{(1)}_\prism(\mathscr{X}/A)\{n\})\ar[r]\ar[d] & R\Gamma(\mathrm{WCart},\mathscr{H}^{(1)}_\prism((\mathscr{X}/p)/A)\{n\})\ar[d]\\
\rho_{dR}^*\mathscr{H}^{(1)}_\prism(\mathscr{X}/A)\{n\} \ar[d]^{\sim}\ar[r]& \rho_{dR}^*\mathscr{H}^{(1)}_\prism((\mathscr{X}/p)/A)\{n\}\ar[d]^{\sim}\\
 \prism^{(1),\rel}_{(\mathscr{X}/p)/(A,(p))}\ar[r] &  \prism^{(1),\rel}_{(\mathscr{X}/p)\otimes^L\bF_p/(A,(p))}\ar@/_/[l]_\mu
}
\label{relativecomparison}\end{equation}
where the bottom vertical isomorphisms are evaluation of the prismatic crystal at the de Rham point \cite{akn23}[Rem. 6.2] and $\mu$ is induced by the map of $A$-algebras 
$$(\co_\mathscr{X}/p)\otimes^L\bF_p\to \co_\mathscr{X}/p:=\co_\mathscr{X}\otimes^L\bF_p$$
 given by multiplication. The composite map
$$ \alpha:\prism^{(1)}_{(\mathscr{X}/p)/A}\{n\}\to  \prism^{(1),\rel}_{(\mathscr{X}/p)/(A,(p))}$$
in (\ref{relativecomparison}) is the comparison map (between prismatic cohomology relative to the $\delta$-ring $A$ and prismatic cohomology relative to the (animated) prism $(A,(p))$) which was constructed in the proof of \cite{akn23}[Prop. 5.17]. We define $\beta_{\mathscr{X}/A}$ as the composite
\begin{equation} \beta_{\mathscr{X}/A}: \prism^{(1)}_{(\mathscr{X}/p)/A}\xrightarrow{\alpha} \prism^{(1),\rel}_{(\mathscr{X}/p)/(A,(p))}\simeq \rho_{dR}^*\mathscr{H}^{(1)}_\prism(\mathscr{X}/A)\{n\}\simeq \widehat{\mathrm{dR}}_{\mathscr{X}/A}\label{betadef}\end{equation}
where the last isomorphism is $c_{dR}$ \cite{akn23}[Lemma 6.4].  By \cite{akn23}[Prop. 5.17] $\alpha$ is an isomorphism, hence so is $\beta_{\mathscr{X}/A}$.

By definition of $\gamma^{dR}_{\prism,\mathscr{X}/A}\{n\}$ in (\ref{gammadef}) there is a commutative diagram
\begin{equation}\xymatrix{
\prism^{(1)}_{\mathscr{X}/A}\{n\}\ar[d]_{\gamma^{dR}_{\prism,\mathscr{X}/A}\{n\}} & R\Gamma(\mathrm{WCart},\mathscr{H}^{(1)}_\prism(\mathscr{X}/A)\{n\})\ar[l]^(.65){\sim}_(.65){c_{\square,*}} \ar[d] \\
 \widehat{\mathrm{dR}}_{\mathscr{X}/A}& \rho_{dR}^*\mathscr{H}^{(1)}_\prism(\mathscr{X}/A)\{n\} \ar[l]_(.6){c_{dR}}^(.6){\sim}
}
\notag\end{equation}
which together with (\ref{betadef}) and (\ref{relativecomparison}) shows that $\gamma^{dR}_{\prism,\mathscr{X}/A}\{n\}=\beta_{\mathscr{X}/A}\circ\psi$. 
\end{proof}

\begin{prop} (The fundamental square relative to a $\delta$-ring) For a derived formal scheme $\mathscr{X}$ over a $\delta$-ring $A$ and $?=\syn,\add$ one has a commutative diagram, contravariantly functorial in the pair $\mathscr{X}/\Spf(A)$,
\begin{equation}\xymatrix@C-10pt{ 
R\Gamma_{?}(\mathscr{X}/A,\bz_p(n)) \ar[rr]\ar[d] & & R\Gamma_{?}((\mathscr{X}/p)/A,\bz_p(n))\ar[d] & \\ 
\fil^{\geq n}_{\mathcal N}\prism^{(1)}_{\mathscr{X}/A}\{n\}\ar[dd]_{\fil^{\geq n}\gamma^{dR}_{\prism,\mathscr{X}/A}\{n\}} \ar@{-->}[dr]^\can \ar@{-->}[rr]& &\fil^{\geq n}_{\mathcal N}\prism^{(1)}_{(\mathscr{X}/p)/A}\{n\}\ar[dd]_(.3){\beta_{\mathscr{X}/A}\circ\can} \ar@{-->}[dr]^\can &\\
 & \prism^{(1)}_{\mathscr{X}/A}\{n\}\ar@{-->}[rr] \ar@{-->}[dr]^{\gamma^{dR}_{\prism,\mathscr{X}/A}\{n\}}& &\prism^{(1)}_{(\mathscr{X}/p)/A}\{n\}\ar@{-->}[dl]_{\beta_{\mathscr{X}/A}}\\
\fil^{\geq n}_{Hod}\widehat{\mathrm{dR}}_{\mathscr{X}/A}\ar[rr]  && \widehat{\mathrm{dR}}_{\mathscr{X}/A}& 
}\label{relative-beilinson}\end{equation}
where $\beta_{\mathscr{X}/A}$ was defined in Prop. \ref{betaprop}. Taking horizontal fibres one obtains two maps
\begin{equation} R\Gamma_{\syn}(\mathscr{X}/A,\bz_p(n))^{\rel} \xrightarrow{\log_{\mathscr{X}/A} } \widehat{\mathrm{dR}}_{\mathscr{X}/A}^{<n}[-1]\xleftarrow{\gamma_{\mathscr{X}/A}} R\Gamma_{\add}(\mathscr{X}/A,\bz_p(n))^{\rel}\label{comp3}\end{equation}
contravariantly functorial in the pair $\mathscr{X}/\Spf(A)$.
\label{relativebeilsquare}\end{prop}

\begin{proof} The commutativity of (\ref{relative-beilinson}) is clear from functoriality of the syntomic package for the map $\mathscr{X}/p\to\mathscr{X}$ over $A$ together with Prop. \ref{betaprop}.  \end{proof}

\begin{remark} Note that $\gamma_{\mathscr{X}/A}$ can be obtained from (\ref{relative-beilinson}) {\em without} the maps involving the isomorphism $\beta_{\mathscr{X}/A}$. One simply has a commutative diagram 
$$\xymatrix@C-8pt{R\Gamma_{\add}(\mathscr{X}/A,\bz_p(n))^{\rel}\ar[r]\ar[dr]^{\gamma_{\mathscr{X}/A}}\ar[d]_{\gamma^{dR,<n,\rel}_{\prism,\mathscr{X}/A}\{n\}[-1]} & R\Gamma_{\add}(\mathscr{X}/A,\bz_p(n))\ar[r]\ar[d]^{\gamma^{dR,<n}_{\prism, \mathscr{X}/A}\{n\}[-1]} &R\Gamma_{\add}((\mathscr{X}/p)/A,\bz_p(n))
\ar[d]^{\gamma^{dR,<n}_{\prism, (\mathscr{X}/p)/A}\{n\}[-1]}\\
\widehat{\mathrm{dR}}_{\mathscr{X}/A}^{<n,\rel}[-1]\ar[r]&\widehat{\mathrm{dR}}_{\mathscr{X}/A}^{<n}[-1]\ar[r]&\widehat{\mathrm{dR}}_{(\mathscr{X}/p)/A}^{<n}[-1]
}$$
where the vertical maps are induced by ${\fil^{\geq \bullet}\gamma^{dR}_{\prism, -\,/A}\{n\}}$ in (\ref{gammadef}).
\label{gammaremark}\end{remark}

\begin{theorem} (The filtered fundamental square) Let $F^{\geq\star} A$ be a filtered $\delta$-ring, $\mathfrak{A}=F\Spf(A)$ and $\mathfrak{X}\to \mathfrak{A}$ a schematic morphism of derived formal stacks. Denote by $\mathscr{X}$ the underlying derived formal scheme of $\mathfrak{X}$ and set $A=F^{\geq 0}A$.
Then for $?=\syn,\add$ there is a commutative diagram of $\bn^{op}$-indexed filtrations
\begin{equation}\xymatrix@C-10pt{ 
R\Gamma_{?}(\mathfrak{X}/\mathfrak{A},\bz_p(n)) \ar[rr]\ar[d] & & R\Gamma_{?}((\mathfrak{X}/p)/\mathfrak{A},\bz_p(n))\ar[d] & \\ 
\fil^{\geq n}_{\mathcal N}\prism^{(1)}_{\mathfrak{X}/\mathfrak{A}}\{n\}\ar[dd]_{\fil^{\geq n}\gamma^{dR}_{\prism,\mathfrak{X}/\mathfrak{A}}\{n\}} \ar@{-->}[dr]^\can \ar@{-->}[rr]& &\fil^{\geq n}_{\mathcal N}\prism^{(1)}_{(\mathfrak{X}/p)/\mathfrak{A}}\{n\}\ar[dd]_(.3){\beta_{\mathfrak{X}/\mathfrak{A}}\circ\can} \ar@{-->}[dr]^\can &\\
 & \prism^{(1)}_{\mathfrak{X}/\mathfrak{A}}\{n\}\ar@{-->}[rr] \ar@{-->}[dr]^{\gamma^{dR}_{\prism,\mathfrak{X}/\mathfrak{A}}\{n\}}& &\prism^{(1)}_{(\mathfrak{X}/p)/\mathfrak{A}}\{n\}\ar@{-->}[dl]_{\beta_{\mathfrak{X}/\mathfrak{A}}}\\
\fil^{\geq n}_{Hod}\widehat{\mathrm{dR}}_{\mathfrak{X}/\mathfrak{A}}\ar[rr]  && \widehat{\mathrm{dR}}_{\mathfrak{X}/\mathfrak{A}}& 
}\label{filtered-relative-beilinson}\end{equation}
functorial in the pair $\mathfrak{X}/\mathfrak{A}$, whose underlying diagram (i.e. $F^{\geq 0}$) coincides with  (\ref{relative-beilinson}). Taking horizontal fibres one obtains two maps of $\bn^{op}$-indexed filtrations
\begin{equation} R\Gamma_{\syn}(\mathfrak{X}/\mathfrak{A},\bz_p(n))^{\rel} \xrightarrow{\log_{\mathfrak{X}/\mathfrak{A}}}  \widehat{\mathrm{dR}}_{\mathfrak{X}/\mathfrak{A} }^{<n}[-1]\xleftarrow{\gamma_{\mathfrak{X}/\mathfrak{A}}} R\Gamma_{\add}(\mathfrak{X}/\mathfrak{A},\bz_p(n))^{\rel}\label{comp67}\end{equation}
functorial in the pair $\mathfrak{X}/\mathfrak{A}$, and with underlying maps (i.e. $F^{\geq 0}$) those in (\ref{comp3}). 
\label{filteredbeilsquare}\end{theorem}

\begin{proof} The base change isomorphism of Thm. \ref{packageproperties} a) induces an isomorphism
$$\prism^{(1)}_{(\mathscr{X}/p)/A}\{n\}\otimes_AA'\xrightarrow{\sim}\prism^{(1)}_{(\mathscr{X}/p)\times_{\Spf(A)}\Spf(A')/A'}\{n\}$$
on Nygaard global sections since $\mathscr{X}/p$ has characteristic $p$ and therefore the Hodge-Tate filtration coincides with the $p$-adic filtration.
Functoriality of $\beta_{\mathscr{X}/A}$ gives equivalences
$$\beta_{\mathscr{X}/A}\otimes_AA'\simeq \beta_{\mathscr{X}\times_{\Spf(A)}\Spf(A')/A'}.$$
Using the notation of Def. \ref{package} this defines an isomorphism $\beta_{\mathfrak{X}/\mathfrak{A}}:=(\beta_{\mathscr{X}_\bullet/\ca^\bullet})$ in  $\Tot \cd(\ca^\bullet)\simeq \cd(\mathfrak{A})$ and a commutative diagram  
\begin{equation}\xymatrix{ 
  \prism^{(1)}_{\mathfrak{X}/\mathfrak{A}}\{n\}\ar[rr] \ar[d]_{\gamma^{dR}_{\prism,{\mathfrak{X}/\mathfrak{A}}}\{n\}}& &\prism^{(1)}_{(\mathfrak{X}/p)/\mathfrak{A}}\{n\}
  \ar[dll]_{\beta_{\mathfrak{X}/\mathfrak{A}}}^{\sim}\\
 \widehat{\mathrm{dR}}_{\mathfrak{X}/\mathfrak{A}}.&&
}\notag\end{equation}
Together with functoriality of the syntomic package for $(\mathfrak{X}/p)\to \mathfrak{X}$ this implies Theorem \ref{filteredbeilsquare}, as in the proof of  Prop. \ref{relativebeilsquare}. The statements about the filtrations follow from Cor. \ref{globalfiltrations}.
\end{proof}

\begin{definition} We refer to the maps $\log_{\mathscr{X}/A}$ in (\ref{comp3}) and $\log_{\mathfrak{X}/\mathfrak{A}}$ in (\ref{comp67}) as the \underline{de Rham logarithm}.
\end{definition}

\section{The syntomic logarithm}\label{synlogsection}

In this section we use the filtered fundamental square of Thm. \ref{filteredbeilsquare} to investigate the de Rham logarithm $\log_{\mathscr{X}/A}$ for a given quasi-compact and quasi-separated derived formal scheme $\mathscr{X}$ over a $\delta$-ring $A$. We endow $\co_{\mathscr{X}}$ with the $p$-adic filtration, or rather a slight generalization which also covers Example \ref{keyexample} and which we call "$e$-filtration". We endow $A$ with any filtration which turns it into a filtered $\delta$-ring in the sense of Def. \ref{fildefs} and so that there is a filtered map $F^{\geq\star}A\to F^{\geq\star}\co_{\mathscr{X}}$ (typically the trivial filtration). The only extra condition we impose on the resulting schematic morphism of derived formal stacks $\mathfrak{X}\to\mathfrak{A}$ is that $L_{\mathfrak{X}/\mathfrak{A}}$ has finite weak filtered amplitude in the sense of Def. \ref{weakamplitude} below. This will allow us to prove completeness of the filtrations $F^{\geq \star}R\Gamma_{?}(\mathscr{X}/A,\bz_p(n))$ for $?=\syn,\add$ and it will allow the construction of the syntomic logarithm to go through. The condition of finite weak filtered amplitude for $L_{\mathfrak{X}/\mathfrak{A}}$ does not impose any restrictions on $L_{\mathscr{X}/A}$ such as perfectness.  It rather amounts to a simple condition on the filtration $F^{\geq\star}A$ that is satisfied in all examples of interest. 

We discuss weak filtered amplitude in \ref{filperf} and $e$-filtered rings in \ref{efiltered}. In \ref{gfiltration} we introduce the $G$-filtration in general and in \ref{gefiltration} we study its properties if $\mathfrak{X}\to\mathfrak{A}$ is $e$-filtered. In \ref{synlog} we construct the syntomic logarithm and establish its main properties in Thm. \ref{thm:synlog} and  Thm. \ref{thm:synlogproper}. We obtain Thm. \ref{addmain} of the introduction as an immediate consequence.

\subsection{Weak filtered amplitude}\label{filperf} The following is a weakening of the notion of finite filtered amplitude defined in Def. \ref{amplitude}. 

\begin{definition} Let $F^{\geq\star}R$ be an animated filtered ring. A filtered $F^{\geq\star}R$-module $F^{\geq\star}M$ has \underline{weak filtered amplitude in $[a,b]$} if $F^{\geq\star}M$ is contained in the subcategory of $\cd(F^{\geq\star}R)$ generated under colimits by $(F^{\geq\star-m}R)[k]$ with $m\in [a,b]$ and $k<0$. 

Let $g:\mathfrak{X}\to \widehat{\ba}^1/\widehat{\bg}_m$ be a schematic morphism of derived formal stacks.
An object $F^{\geq\star}\cm\in\cdf(\mathscr{X},F^{\geq\star}\co_{\mathscr{X}})\simeq\cd(\mathfrak{X})$ has \underline{weak filtered amplitude in $[a,b]$} if its restriction to any affine open has weak filtered amplitude in $[a,b]$.  

We say that $F^{\geq\star}\cm$ has \underline{finite weak filtered amplitude} if it has weak filtered amplitude in some interval $[a,b]$.
\label{weakamplitude}\end{definition}

\begin{lemma} If $F^{\geq\star}\cm$ has weak filtered amplitude in $[a,b]$ then the derived exterior power $L\bigwedge_{F^{\geq\star}\co_\mathscr{X}}^iF^{\geq\star}\cm$ has weak filtered amplitude in $[ia,ib]$.
\label{weakexterior}\end{lemma}
\begin{proof} We may assume $\mathfrak{X}=F\Spf(R)$. We continue to use the notation of the proof of Lemma \ref{exterior} and denote by $\mathscr{D}_w^{[a,b]}\subseteq \cd(F^{\geq\star}R)$ the full subcategory of modules of weak filtered amplitude in $[a,b]$. Since $\mathscr{D}_w^{[a,b]}$ has all filtered colimits the inclusion $\mathscr{D}^{[a,b]}\subseteq\mathscr{D}^{[a,b]}_w$ factors
$$\mathscr{D}^{[a,b]}\to \Ind(\mathscr{D}^{[a,b]})\xrightarrow{\pi}\mathscr{D}_w^{[a,b]}$$
where $\pi$ preserves filtered colimits and is essentially surjective by definition of $\mathscr{D}_w^{[a,b]}$ \cite{lurieHTT}[Prop. 5.3.5.12].
Since $\Lambda^i:=L\bigwedge_{F^{\geq\star}R}^i$ commutes with filtered colimits, the composite 
$$\Ind(\mathscr{D}^{[a,b]})\xrightarrow{\pi}\mathscr{D}_w^{[a,b]}\xrightarrow{\Lambda^i}\cd(F^{\geq\star}R)$$
is equivalent to
$$\Ind(\mathscr{D}^{[a,b]})\xrightarrow{\tilde{\Lambda}^i} \mathscr{D}_w^{[ia,ib]}\subseteq \cd(F^{\geq\star}R)$$
where $\tilde{\Lambda}^i$ is the unique filtered colimit preserving extension of $\Lambda^i:\mathscr{D}^{[a,b]}\to \mathscr{D}_w^{[ia,ib]}$. This is because both functors extend $\Lambda^i:\mathscr{D}^{[a,b]}\to \cd(F^{\geq\star}R)$. Since $\pi$ is essentially surjective we obtain the desired factorization $\Lambda^i:\mathscr{D}^{[a,b]}_w\to \mathscr{D}_w^{[ia,ib]}$.
\end{proof}

The following computation of weak filtered amplitude covers all examples of interest for this article.

\begin{lemma} Let $F^{\geq\star} A$ be a $\bn^{op}$-indexed filtered ring such that $F^{\geq k}A=I^k$ for an ideal $I$ generated by a regular sequence $z_0,\dots,z_s\in A$ (possibly empty). Let $F^{\geq\star} R$ be a $\bn^{op}$-indexed animated filtered ring whose filtration is isomorphic to the $\varpi$-adic filtration
$$ \cdots \xrightarrow{\cdot \varpi}R\xrightarrow{\cdot \varpi}R\xrightarrow{\cdot \varpi}R$$
on $R:=F^{\geq 0}R$ for some $\varpi\in\pi_0R$. For any map $F^{\geq\star} A\to F^{\geq\star} R$ of animated filtered rings
$L_{F\Spf(F^{\geq \star}R)/F\Spf(F^{\geq \star}A)}$
has weak filtered amplitude in $[0,1]$.
\label{lcicor}\end{lemma}

\begin{proof} For brevity we use the notation $\mathrm{Rees}(R):=\mathrm{Rees}(F^{\geq\star}R)$. We must show that $L_{\mathrm{Rees}(R)/\mathrm{Rees}(A)}\in\Mod_{\mathrm{Rees}(R)^{gr}}$ lies in the subcategory generated under colimits by $\mathrm{Rees}(R)[k]$ and $\mathrm{Rees}(F^{\geq\star-1}R)[k]$. 
The transitivity triangles for $A\to R\to \mathrm{Rees}(R)$ and $A\to \mathrm{Rees}(A)\to \mathrm{Rees}(R)$ give a commutative diagram with exact rows and columns
$$\begin{CD}{}@. L_{R/A}\underset{R}{\otimes}\mathrm{Rees}(R)@=L_{R/A}\underset{R}{\otimes}\mathrm{Rees}(R)\\
@.@VVV@VVV\\L_{\mathrm{Rees}(A)/A}\underset{\mathrm{Rees}(A)}{\otimes}\mathrm{Rees}(R) @>>> L_{\mathrm{Rees}(R)/A}@>>> L_{\mathrm{Rees}(R)/\mathrm{Rees}(A)}\\
\Vert@. @VVV@VVV\\
L_{\mathrm{Rees}(A)/A}\underset{\mathrm{Rees}(A)}{\otimes}\mathrm{Rees}(R) @>>> L_{\mathrm{Rees}(R)/R}@>>> K.
\end{CD}$$ 
Noting that $L_{R/A}$ is an iterated colimit of copies of  shifts of $R$ it suffices to show that $K$ has weak filtered amplitude in $[0,1]$. Since $I$ is generated by a regular sequence the natural surjection of non-negatively graded rings
$$S:=\Sym_A^\star I\to A\oplus I\oplus I^2\oplus I^3\oplus\cdots$$
is an isomorphism by \cite{micali}[Thm. 1]. Denoting by $u_i\in S$ the elements of degree $1$ given by the generators $z_i\in I$ we have an isomorphism of graded rings
$$ \mathrm{Rees}(A)\simeq S[t]/(tu_i-z_i)=S'/J$$
where $t$ has degree $-1$ and $S'=S[t]\simeq \Sym_A^\star (I\oplus A)$. One verifies that the $tu_i-z_i\in S'$ form a regular sequence. Indeed, $tu_i-z_i$ is a nonzerodivisor in $S'$ since $z_i$ is a nonzerodivisor in $A$ and hence in $S'$. One then concludes by induction using $S'/(u_s,tu_s-z_s)\simeq\Sym_{(A/z_s)}^\star (I/z_s\oplus A/z_s)$.  In the transitivity triangle
\begin{equation} L_{S'/A}\otimes_{S'}\mathrm{Rees}(A)\to L_{\mathrm{Rees}(A)/A}\to L_{\mathrm{Rees}(A)/S'}\simeq J/J^2[1]\label{trans}\end{equation}
$J/J^2$ is a free $\mathrm{Rees}(A)$-module on generators $tu_i-z_i$ of degree $0$ and $L_{S'/A}\simeq (I\oplus A)\otimes_AS'$ is a module of finite projective dimension on generators $u_i$ of degree $1$ and $t$ of degree $-1$ (we denote the images of these generators in $L_{S'/A}$ by $du_i$ and $dt$). We have $\mathrm{Rees}(R)\simeq R[t,u]/(tu-\varpi)$ and an analogous triangle to (\ref{trans}). The commutative diagram
$$\begin{CD} S' @>>> R[t,u]\\
@VVV@VVV\\
\mathrm{Rees}(A)@>>>\mathrm{Rees}(R)
\end{CD}$$ 
induces a diagram with exact rows and columns
$$\begin{CD}L_{S'/A}\underset{S'}{\otimes}\mathrm{Rees}(R) @>\phi>> L_{R[t,u]/R}\underset{R[t,u]}{\otimes}\mathrm{Rees}(R) @>>> K_1\\
@VVV @VVV@VVV\\
L_{\mathrm{Rees}(A)/A}\underset{\mathrm{Rees}(A)}{\otimes}\mathrm{Rees}(R) @>>> L_{\mathrm{Rees}(R)/R}@>>> K\\
@VVV@VVV@VVV\\
J/J^2[1]\underset{\mathrm{Rees}(A)}{\otimes}\mathrm{Rees}(R) @>>> (tu-\varpi)/(tu-\varpi)^2[1] @>>> K_0
\end{CD}$$ 
where the terms in the bottom row have generators of degree $0$, hence are colimits of copies of $\mathrm{Rees}(R)$. Both the source and the target of $\phi$ have a direct summand $\mathrm{Rees}(R)dt$ where $dt$ has degree $-1$ and $\phi$ induces an isomorphism between these direct summands. The remaining generators $du, du_i$ in the top row have degree $1$. Hence $K_1$ has a presentation by generators of degree $1$, i.e. is a colimit of copies of $\mathrm{Rees}(F^{\geq\star-1}R)[k]$. This concludes the proof that $K$ has weak filtered amplitude in $[0,1]$. \end{proof}

\begin{example} Let $g:\mathscr{X}\to \Spf(\co_K)$ be a morphism with $\mathscr{X}$ quasi-compact and quasi-separated, let
$$ F_\fm\mathscr{X}\to \widehat{\ba}^1/\widehat{\bg}_m\times \Spf(\co_K) $$ 
be the filtered formal stack of Example \ref{keyexample} and $A^s=W(k)[[z_0,\dots,z_s]]$ the filtered $\delta$-ring with $\delta(z_i)=0$ and the $(z_0,\dots,z_s)$-adic filtration \cite{akn23}[Example 10.38] over which $\co_K$ becomes an algebra by sending the $z_i$ to a uniformizer $\varpi$. Then Lemma \ref{lcicor} implies that the cotangent complex of
$$ F_\fm\mathscr{X}\to  F\Spf(F^{\geq \star}A^s)$$
has weak filtered amplitude in $[0,1]$. 
\label{keyexamplectd}\end{example}

\begin{example} \label{trivexample} If  $A$ is trivially filtered (i.e. $I=0$) and $R$ is $\varpi$-adically filtered, then Lemma \ref{lcicor} shows that the cotangent complex of 
$$F\Spf(F^{\geq \star}R)\to \widehat{\ba}^1/\widehat{\bg}_m\times \Spf(A)$$ 
has weak filtered amplitude in $[0,1]$.
\end{example}

\subsection{$e$-filtered rings and morphisms}\label{efiltered} The following notion is a mild generalization of the $p$-adic filtration (which corresponds to $e=1$).

\begin{definition} ($e$-filtered rings) Let $e\geq 1$ be an integer. We call a animated filtered ring $F^{\geq\star}R$ \underline{$e$-filtered} if the filtration is isomorphic to the $\varpi$-adic filtration
$$ \cdots \xrightarrow{\cdot \varpi}R\xrightarrow{\cdot \varpi}R\xrightarrow{\cdot \varpi}R$$
on $R:=F^{\geq 0}R$ for some $\varpi\in\pi_0R$ such that $\varpi^e=up$ with $u\in\pi_0R^\times$.
 We say a derived formal stack $\mathfrak{X}$, equipped with a schematic morphism $g:\mathfrak{X}\to \widehat{\ba}^1/\widehat{\bg}_m$, is \underline{$e$-filtered} if the filtered structure sheaf $F^{\geq \star}\co_{\mathscr{X}}$ of the underlying derived formal scheme $\mathscr{X}$ is $e$-filtered for some $\varpi\in H^0(\mathscr{X},\pi_0F^{\geq 0}\co_{\mathscr{X}})$.
\end{definition} 

\begin{remark} If $F^{\geq\star}R$ is $e$-filtered then the filtration is $\bn^{op}$-indexed and there exists $\tilde{p}\in \pi_0F^{\geq e}R$ and a commutative diagram of filtrations ($\star\geq 0$)
$$\xymatrix{F^{\geq \star+e}R\ar[r]  \ar[d]^{\cdot p} & F^{\geq \star}R\ar[dl]_{\cdot \tilde{p}}^{\sim} \ar[d]^{\cdot p} \\
F^{\geq \star+e}R\ar[r]  & F^{\geq \star}R }$$
where the maps $ \cdot\tilde{p}$ are isomorphisms. If $e=1$ then $F^{\geq \star} R$ is isomorphic to the $p$-adic filtration
$$ \cdots \xrightarrow{\cdot p}R\xrightarrow{\cdot p}R\xrightarrow{\cdot p}R$$
via the maps $\cdot\tilde{p}^k$. Taking $R$ to be a $\bF_p$-algebra this remark also shows that an $e$-filtration on a discrete ring need not be strict, nor need the element $\tilde{p}$ be unique. If $R$ is strictly $e$-filtered then $R$ is flat over $\bz_p$ and $\tilde{p}=p\in F^{\geq e}R$ is unique.
\label{ptilderemark}\end{remark}

\begin{example} Let $g:\mathscr{X}\to \Spf(\co_K)$ be a flat morphism of formal schemes where $\mathscr{X}$ is quasi-compact and quasi-separated. Then the filtered formal stack 
$F_\fm\mathscr{X}$
of Example \ref{keyexample} is strictly $e$-filtered  where $e$ is the ramification index of $K/\bq_p$. 
\end{example}

\begin{prop}\label{propperfect1} Let $\mathfrak{X}$ be $e$-filtered and assume $F^{\geq\star}\cm\in\cd(\mathfrak{X})$ has weak filtered amplitude in $[a,b]$.
\begin{enumerate}
\item[a)] If $k\geq b$ then 
$$ \cdot\tilde{p}: F^{\geq k}\cm \xrightarrow{\sim} F^{\geq k+e}\cm$$
is an equivalence.
\item[b)] The filtration $F^{\geq\star}\cm$ is complete. 
\end{enumerate}
\end{prop}
\begin{proof}
a) It suffices to check that $\cdot\tilde{p}$ is an equivalence on an affine cover. In the affine case, the property that $\cdot\tilde{p}:F^{\geq k }M\rightarrow F^{\geq k+e}M$ is an equivalence for $k\geq b$ is preserved by shifts and arbitrary colimits and is satisfied by $F^{\geq\star-m}R$ for $m\leq b$. This gives a). 

b) We have
$$\varprojlim_k F^{\geq k}\cm \simeq \varprojlim_k F^{\geq b+ke}\cm\simeq \varprojlim\left(\cdots \xrightarrow{p} F^{\geq b}\cm\xrightarrow{p}F^{\geq b}\cm\right)=0$$ since any $F^{\geq b}\cm\in\cd(\mathscr{X})$ is $p$-complete.
\end{proof}

In the following definition we could only require $L_{\mathfrak{X}/\mathfrak{A}}$ to have finite weak filtered amplitude without invalidating the results in this section but such generality seems unhelpful.

\begin{definition} ($e$-filtered morphisms) Let $F^{\geq\star} A$ be a filtered $\delta$-ring. We call a morphism of derived formal stacks
$$g:\mathfrak{X}\to \mathfrak{A}:=F\Spf(A)$$
\underline{$e$-filtered} if $g$ is schematic, $\mathfrak{X}$ is $e$-filtered and $L_{\mathfrak{X}/\mathfrak{A}}$ has weak filtered amplitude in $[0,1]$.
\end{definition}

We note that if $\mathfrak{X}$ is $e$-filtered then the morphism $g:\mathfrak{X}\to \widehat{\ba}^1/\widehat{\bg}_m$ is $e$-filtered in view of Example \ref{trivexample}.

For the following proposition note that $F^{\geq \star} \mathrm{gr}^i_{\mathcal N}\prism^{(1)}_{\mathscr{X}/A}\{n\}$ is a $F^{\geq \star}\mathrm{gr}^0_{\mathcal N}\prism^{(1)}_{\mathscr{X}/A}\simeq R\Gamma(\mathscr{X}, F^{\geq \star}\co_{\mathscr{X}})$-module.

\begin{prop} \label{corforA} Let $F^{\geq\star} A$ be a filtered $\delta$-ring, $\mathfrak{A}=F\Spf(A)$ and $\mathfrak{X}\to \mathfrak{A}$ a $e$-filtered morphism of derived formal stacks. 
\begin{enumerate}
\item[a)] For any $i\geq 0$ the filtrations $F^{\geq \star}L\widehat{\Omega}^{i}_{\mathscr{X}/A}$ and $F^{\geq \star}\mathrm{gr}^i_{\mathcal N}\prism^{(1)}_{\mathscr{X}/A}\{n\}$ are complete.
\item[b)] The filtrations
$$F^{\geq \star}R\Gamma_{\add}(\mathscr{X}/A,\bz_p(n)),\  F^{\geq \star}R\Gamma_{\syn}(\mathscr{X}/A,\bz_p(n)),\  F^{\geq \star} \widehat{\mathrm{dR}}_{\mathscr{X}/A}^{<n}[-1]$$ 
are complete.
\item[c)] Multiplication by $\tilde{p}\in H^0(\mathscr{X},\pi_0F^{\geq e}\co_{\mathscr{X}})$ induces equivalences
$$F^{\geq k}L\widehat{\Omega}^{i}_{\mathscr{X}/A} \stackrel{\sim}{\rightarrow} F^{\geq k+e}L\widehat{\Omega}^{i}_{\mathscr{X}/A} $$
and 
$$ F^{\geq k}\mathrm{gr}^i_{\mathcal N}\prism^{(1)}_{\mathscr{X}/A}\{n\} \stackrel{\sim}{\rightarrow} F^{\geq k+e}\mathrm{gr}^i_{\mathcal N}\prism^{(1)}_{\mathscr{X}/A}\{n\}$$
for $k\geq i$.  
\item[d)] $\mathfrak{X}/p\to \mathfrak{A}$ is $e$-filtered.
\item[e)] The filtrations
$$F^{\geq \star}R\Gamma_{\add}(\mathscr{X}/A,\bz_p(n))^{\rel},\  F^{\geq \star}R\Gamma_{\syn}(\mathscr{X}/A,\bz_p(n))^{\rel},\  F^{\geq \star} \widehat{\mathrm{dR}}_{\mathscr{X}/A}^{<n,\rel}[-1]$$ 
are complete.
\end{enumerate}
\end{prop}

\begin{proof} a) By Lemma \ref{weakexterior} $L\bigwedge^i L_{\mathfrak{X}/\mathfrak{A}}$ has weak filtered amplitude in $[0,i]$. By Prop. \ref{propperfect1} b) the filtration $L\bigwedge^i L_{\mathfrak{X}/\mathfrak{A}}$ is complete hence so is 
$$F^{\geq \star}L\widehat{\Omega}^{i}_{\mathscr{X}/A}=R\Gamma(\mathscr{X},L\bigwedge{}^i L_{\mathfrak{X}/\mathfrak{A}}).$$
The bounded filtration (\ref{nygaardfilt}) then shows that $F^{\geq \star}\mathrm{gr}^i_{\mathcal N}\prism^{(1)}_{\mathscr{X}/A}\{n\}$ is complete.

b) This follows exactly as in the proof of Prop. \ref{completeandbounded} a), using the bounded filtration (\ref{addfilt}) together with part a).

c) This is immediate from Prop. \ref{propperfect1} a) and the bounded filtration (\ref{nygaardfilt}).

d) Since $F^{\geq\star}\co_{\mathscr{X}/p}=(F^{\geq\star}\co_{\mathscr{X}})/p$ the image of $\varpi$ in $\pi_0\co_{\mathscr{X}/p}$ exhibits $\mathfrak{X}/p$ as $e$-filtered. The transitivity triangle for $(\mathfrak{X}/p)\xrightarrow{g}\mathfrak{X}\to\mathfrak{A}$   
$$ g^*L_{\mathfrak{X}/\mathfrak{A}}\to L_{(\mathfrak{X}/p)/\mathfrak{A}}\to L_{(\mathfrak{X}/p)/\mathfrak{X}}$$
and $L_{(\mathfrak{X}/p)/\mathfrak{X}}\simeq L_{\Spf(\bF_p)/\Spf(\bz_p)}\otimes\co_{\mathfrak{X}}\simeq \co_{\mathfrak{X}/p}[1]$ show that 
$L_{(\mathfrak{X}/p)/\mathfrak{A}}$ has weak filtered amplitude in $[0,1]$ if $L_{\mathfrak{X}/\mathfrak{A}}$ does.

e) This follows by combining b) for $\mathfrak{X}$ and $\mathfrak{X}/p$.
\end{proof}

\subsection{The G-filtration} \label{gfiltration} Let $F^{\geq\star} A$ be a filtered $\delta$-ring and $g:\mathfrak{X}\to \mathfrak{A}:=F\Spf(A)$ a schematic morphism of derived formal stacks.  If $\mathfrak{X}$ is $e$-filtered  we have seen (before the statement of Prop. \ref{corforA}) that multiplication by $p$ lifts to a map of filtered graded $F^{\geq\star} A$-modules
$$ \cdot\tilde{p}:F^{\geq \star}\mathrm{gr}^\bullet_{\mathcal N}\prism^{(1)}_{\mathscr{X}/A}\{n\}\to F^{\geq \star+e}\mathrm{gr}^\bullet_{\mathcal N}\prism^{(1)}_{\mathscr{X}/A}\{n\}$$
which by Prop. \ref{corforA} c) is an isomorphism if $g$ is $e$-filtered and $\star$ is large enough.
However it is not in general true that multiplication by $p$ lifts to a map of bifiltered $F^{\geq\star} A$-modules
$$ F^{\geq \star}\mathcal {N}^{\geq \star}\prism^{(1)}_{\mathscr{X}/A}\{n\}\to F^{\geq \star+e}\mathcal {N}^{\geq \star} \prism^{(1)}_{\mathscr{X}/A}\{n\}$$
and therefore also to a map of filtered modules
$$ F^{\geq \star}R\Gamma_{?}(\mathscr{X}/A,\bz_p(n))\to F^{\geq \star+e}R\Gamma_{?}(\mathscr{X}/A,\bz_p(n))$$
for $?=\syn,\add$. In order to achieve this property we introduce the following modified bifiltration.

\begin{definition} \label{gdef}
Let $F^{\geq \star}\mathcal {N}^{\geq \star}C$ be a bifiltered $F^{\geq\star}A$-module and $e\geq 1$. Define the trifiltered $G^{\geq\star}F^{\geq\star}A$-module
$G^{\geq \star}F^{\geq \star}\mathcal {N}^{\geq \star}C$, resp. bifiltered $G^{\geq\star}A$-module
$G^{\geq \star}\mathcal {N}^{\geq \star}C$  by left Kan extension along the map
$$ \rho:\bz\times\bz\to\bz\times\bz\times\bz ,\quad (a,b)\mapsto (a+be,a,b),$$
resp. the map $\mathrm{pr_{1,3}}\circ\rho: \bz\times\bz \to \bz\times\bz ,\ (a,b)\mapsto (a+be,b)$. That is we have
$$G^{\geq k}F^{\geq j}\mathcal {N}^{\geq i}C:=\underset{a+be\geq k, a\geq j, b\geq i}{\mathrm{colim}} F^{\geq a}\mathcal {N}^{\geq b}C$$
and
$$ G^{\geq k}{\mathcal N}^{\geq i}C:=  \underset{a+be\geq k, b\geq i}{\mathrm{colim}} F^{\geq a}{\mathcal N}^{\geq b}C.$$
\end{definition}

In order to clarify multiplicative structures in Def. \ref{gdef} we recall the following well known Lemma on functoriality of the Day convolution symmetric monoidal structure.

\begin{lemma} \label{day}
Let $\tau:M\to N$ be an additive, order preserving map between ordered commutative monoids and 
$$ \xymatrix{\Fun(M^{op},\Mod_\bz)\ar@<1ex>[r]^{\tau_!} & \Fun(N^{op},\Mod_\bz)\ar[l]^{\tau^*}}$$
the induced adjunction. Then $\tau_!$ is symmetric monoidal and its right adjoint $\tau^*$ is lax symmetric monoidal.
\end{lemma} 

\begin{proof} See for example \cite{nikolaus16}[Cor. 3.8]. \end{proof}

Given $F^{\geq\star}A\in\CAlg(\Fun(\bz^{op},\Mod_\bz))$, in Def. \ref{gdef} we also use the notation $F^{\geq\star}A$ for 
$$F^{\geq\star}A\otimes F^{\geq\star}_{\mathrm{triv}}\bz\in\CAlg(\Fun(\bz^{op}\times\bz^{op},\Mod_\bz)).$$
Likewise in the following we shall omit the factor $\otimes F^{\geq\star}_{\mathrm{triv}}\bz$ for the last index.
By definition we have $G^{\geq\star}F^{\geq\star}A:=\rho_!F^{\geq\star}A$, $G^{\geq\star}A:=(\mathrm{pr_{1,3}}\circ\rho)_!F^{\geq\star}A$
and an easy computation shows
$$ G^{\geq k}F^{\geq j}A\simeq F^{\geq\max(k,j)}A,\quad \quad G^{\geq k}A\simeq F^{\geq k}A.$$

\begin{lemma} Let $F^{\geq \star}\mathcal {N}^{\geq \star}C$ be a bifiltered $F^{\geq\star}A$-module and $G^{\geq \star}F^{\geq \star}\mathcal {N}^{\geq \star}C$ 
and $G^{\geq \star}\mathcal {N}^{\geq \star}C$ as in Def. \ref{gdef}.
\begin{enumerate}
\item[a)] We have \begin{equation}\notag
\mathrm{gr}^k_G\mathcal{N}^{\geq i}C\simeq\bigoplus_{j\geq i}\mathrm{gr}^{k-je}_F\mathrm{gr}^{j}_{\mathcal{N}} C.
\end{equation}
In particular if $F^{\geq \star}\mathcal {N}^{\geq \star}C$ is $\bn^{op}\times\bn^{op}$-indexed, so is $G^{\geq \star}\mathcal {N}^{\geq \star}C$.
\item[b)] If $F^{\geq \star}\mathcal {N}^{\geq \star}C$ is multiplicative (i.e. a bifiltered $\be_\infty$-$F^{\geq\star}A$-algebra) then $G^{\geq \star}F^{\geq \star}\mathcal {N}^{\geq \star}C$, resp. $G^{\geq \star}{\mathcal N}^{\geq \star}C$ is multiplicative (i.e. a trifiltered $\be_\infty$-$G^{\geq \star}F^{\geq\star}A$-algebra, resp. bifiltered $\be_\infty$-$G^{\geq \star}A$-algebra).
\item[c)] 
On $\mathcal N^{[0,n[}C$  the $F$-filtration and the $G$-filtration are cofinal in each other. More precisely, there is a factorization 
$$ F^{\geq k}{\mathcal N}^{[0,n[}C\to G^{\geq k}{\mathcal N}^{[0,n[}C\to F^{\geq k-e(n-1)}{\mathcal N}^{[0,n[}C \to G^{\geq k-e(n-1)}{\mathcal N}^{[0,n[}C$$
of the transition maps from index $k$ to index $k-e(n-1)$ in both filtrations.
\item[d)] The filtration $G^{\geq \star}{\mathcal N}^{[0,n[}C$ is complete if and only if the filtration $F^{\geq \star}{\mathcal N}^{[0,n[}C$ is complete.
\end{enumerate}
\label{Ggeneralities}\end{lemma}

\begin{proof} a) It suffices to treat the case $i=-\infty$, the general case follows by looking at the $\bz^{op}\times\bn^{op}_{\geq i}$-indexed truncation of $F^{\geq \star}\mathcal {N}^{\geq \star}C$. Let $\sigma:\bz\times\bz\to\bz$ be the map $(a,b)\mapsto a+be$ and denote by $\bz^{0,0}\in\CAlg(\Fun(\bz^{op}\times\bz^{op},\Mod_\bz))$ the bifiltered algebra consisting of $\bz$ in bidegree $(0,0)$ and $0$ elsewhere so that $-\otimes\bz^{0,0}$ is the functor of passing to the associated bigraded. There is a commutative diagram
 \begin{equation}  \begin{CD} \Fun(\bz^{op}\times\bz^{op},\Mod_\bz) @>\sigma_!>> \Fun(\bz^{op},\Mod_\bz)\\
 @VV-\otimes\bz^{0,0} V @VV-\otimes\sigma_!\bz^{0,0}V\\
\Mod_{\bz^{0,0}} \Fun(\bz^{op}\times\bz^{op},\Mod_\bz) @>\sigma_!>> \Mod_{\sigma_!\bz^{0,0}} \Fun(\bz^{op},\Mod_\bz).
  \end{CD}\notag\end{equation}
For $M^{\star,\star}\in\Mod_{\bz^{0,0}} \Fun(\bz^{op}\times\bz^{op},\Mod_\bz)$ (a bifiltered module with transition maps all zero) one checks that
$$(\sigma_!M^{\star,\star})^k=\bigoplus_{i+je=k}M^{i+j}\oplus \bigoplus_{i+je=k+1}M^{i+j}[1],$$
in particular $\sigma_!\bz^{0,0}=\bz^0\oplus\bz^0\langle-1\rangle[1]$. So the associated graded functor $-\otimes\bz^0=(-\otimes {\sigma_!\bz^{0,0}})\otimes_{\sigma_!\bz^{0,0}} \bz^0$ sends $G^{\geq \star}C=\sigma_!F^{\geq \star}\mathcal {N}^{\geq \star}C$ to 
$$\bigoplus_{i+je=k}\mathrm{gr}^{i}_F\mathrm{gr}^{j}_{\mathcal{N}} C.$$

b) This follows from Lemma \ref{day} since both $\rho$ and $\mathrm{pr_{1,3}}\circ\rho$ are additive and order preserving. 

c) The map $F^{\geq k}{\mathcal N}^{[0,n[}C\to G^{\geq k}{\mathcal N}^{[0,n[}C$ is induced by a map of  $\bz^{op}\times\bn^{op}$-indexed bifiltrations. Indeed there is a natural map
$$ F^{\geq k}{\mathcal N}^{\geq i}C\to  \underset{a+be\geq k, b\geq i}{\mathrm{colim}} F^{\geq a}{\mathcal N}^{\geq b}C \simeq G^{\geq k}{\mathcal N}^{\geq i}C $$
as long as $i\geq 0$. There is an equivalence
\begin{equation} \begin{CD} G^{\geq k}{\mathcal N}^{[0,n[}C @<\kappa<\sim< G^{\geq k}F^{\geq j}{\mathcal N}^{[0,n[}C \\
\Vert @. \Vert @.\\
\underset{\substack{a+be\geq k\\ n-1\geq b\geq 0}}{\mathrm{colim}} F^{\geq a}{\mathcal N}^{[b,n[}C@=\underset{\substack{a+be\geq k\\ n-1\geq b\geq 0, a\geq j}}{\mathrm{colim}} F^{\geq a}{\mathcal N}^{[b,n[}C\end{CD}\label{kappadef}\end{equation}
as long as $0\leq j\leq k-e(n-1)$ and both composite maps 
$$ G^{\geq k}{\mathcal N}^{[0,n[}C\xrightarrow{\kappa^{-1}}G^{\geq k}F^{\geq k-\epsilon}{\mathcal N}^{[0,n[}C\to F^{\geq k-\epsilon}{\mathcal N}^{[0,n[}C\to G^{\geq k-\epsilon}{\mathcal N}^{[0,n[}C$$
and 
$$ F^{\geq k}{\mathcal N}^{[0,n[}C\to G^{\geq k}{\mathcal N}^{[0,n[}C\xrightarrow{\kappa^{-1}}G^{\geq k}F^{\geq k-\epsilon}{\mathcal N}^{[0,n[}C\to F^{\geq k-\epsilon}{\mathcal N}^{[0,n[}C$$
are the natural transition maps (abbreviating $\epsilon=e(n-1)$). 

We prove d). Assume $F^{\geq \star}{\mathcal N}^{[0,n[}C$ is complete, i.e. $\mathrm{lim}_{k}\,F^{\geq k}{\mathcal N}^{[0,n[}C=0$. The canonical equivalence 
$$\mathrm{lim}_{k}\,G^{\geq k+e(n-1)} {\mathcal N}^{[0,n[}C\stackrel{\sim}{\rightarrow} \mathrm{lim}_{k}\,G^{\geq k} {\mathcal N}^{[0,n[}C $$
factors through
$$\mathrm{lim}_{k}\,G^{\geq k+e(n-1)} {\mathcal N}^{[0,n[}C \to \mathrm{lim}_{k}\,F^{\geq k}{\mathcal N}^{[0,n[}C=0\to \mathrm{lim}_{k}\,G^{\geq k}{\mathcal N}^{[0,n[}C$$
by part c). Hence $\mathrm{lim}_{k}\,G^{\geq k} {\mathcal N}^{[0,n[}C=0$. The other implication follows similarly.
\end{proof}

We now discuss the $G$-filtration in the situation of Theorem \ref{filteredbeilsquare}. Let $F^{\geq\star} A$ be a filtered $\delta$-ring, $\mathfrak{A}:=F\Spf(A)$ and $\mathfrak{X}\to \mathfrak{A}$ a schematic morphism of derived formal stacks. Denote by $\mathscr{X}$ the underlying derived formal scheme of $\mathfrak{X}$ and set $A=F^{\geq 0}A$. The map of $\bn^{op}\times\bn^{op}$-indexed bifiltrations
$$F^{\geq \star}\gamma^{dR}_{\prism,\mathscr{X}/A}\{n\}:  F^{\geq \star}\mathcal N^{\geq\star}\prism^{(1)}_{\mathscr{X}/A}\{n\} \to F^{\geq \star}\fil^{\geq \star}_{Hod}\widehat{\mathrm{dR}}_{\mathscr{X}/A}$$
from Def. \ref{package} (multiplicative for $n=0$, a map of $F^{\geq \star}\mathcal N^{\geq\star}\prism^{(1)}_{\mathscr{X}/A}$-modules for any $n$) induces a map of $\bn^{op}\times\bn^{op}$-indexed bifiltrations
$$G^{\geq \star}\gamma^{dR}_{\prism,\mathscr{X}/A}\{n\}:  G^{\geq \star}\mathcal N^{\geq\star}\prism^{(1)}_{\mathscr{X}/A}\{n\} \to G^{\geq \star}\fil^{\geq \star}_{Hod}\widehat{\mathrm{dR}}_{\mathscr{X}/A}$$
(multiplicative for $n=0$, a map of $G^{\geq \star}\mathcal N^{\geq\star}\prism^{(1)}_{\mathscr{X}/A}$-modules for any $n$), functorial in $\mathfrak{X}$, and hence also a map of $\bn^{op}$-indexed filtrations (see Remark \ref{gammaremark})
\begin{equation} G^{\geq \star} \widehat{\mathrm{dR}}_{\mathscr{X}/A}^{<n}[-1]  \leftarrow G^{\geq \star}R\Gamma_{\add}(\mathscr{X}/A,\bz_p(n))^{\rel}:G^{\geq \star}\gamma_{\mathscr{X}/A}.\notag\end{equation}
In contrast, the map $\log_{\mathfrak{X}/\mathfrak{A}}=F^{\geq \star}\log_{\mathscr{X}/A}$ in (\ref{comp67}) does not immediately induce a map of $G$-filtrations for two reasons. First, $c^0\varphi\{n\}$ does not preserve the Nygaard filtration and hence not the $F^{\geq \star}\mathcal{N}^{\geq \star}$-bifiltration so that there is no immediate $G$-filtration on $R\Gamma_{\syn}(\mathscr{X}/A,\bz_p(n))$. Second, the isomorphism $\beta_{\mathscr{X}/A}$ in Prop. \ref{betaprop} does not extend to a map from the Nygaard to the Hodge filtration, hence also not to a map of bifiltrations and $G$-filtrations.  Therefore the diagram (\ref{filtered-relative-beilinson}) of $F$-filtrations does not induce a corresponding diagram of $G$-filtrations. We now explain how to resolve each of these issues. 

\subsubsection{The $G$-filtration on $R\Gamma_{\syn}$} It turns out that $G^{\geq k}F^{\geq j}R\Gamma_{\syn}(\mathscr{X}/A,\bz_p(n))$ does exist for large enough $k$. In order to see this, first note that for $0\leq j\leq k-e(n-1)$ we have the equivalence (\ref{kappadef}) 
\begin{align*}&G^{\geq k}F^{\geq j}R\Gamma_{\add}(\mathscr{X}/A,\bz_p(n))\\ =\,& G^{\geq k}F^{\geq j}{\mathcal N}^{[0,n[}\prism^{(1)}_{\mathscr{X}/A}\{n\} \xrightarrow{\sim} G^{\geq k}F^{\geq 0}\mathcal N^{[0,n[}\prism^{(1)}_{\mathscr{X}/A}\{n\}\\
=\,&G^{\geq k}R\Gamma_{\add}(\mathscr{X}/A,\bz_p(n))
\end{align*}
whereas for $j\geq k-e(n-1)$ we have
\begin{align}
&G^{\geq k}F^{\geq j}R\Gamma_{\add}(\mathscr{X}/A,\bz_p(n))\label{rewriting}\\
\simeq &\fibre\left( G^{\geq k}F^{\geq j}{\mathcal N}^{\geq n}\prism^{(1)}_{\mathscr{X}/A}\{n\}\xrightarrow{\can} G^{\geq k}F^{\geq j}\prism^{(1)}_{\mathscr{X}/A}\{n\}\right)\notag\\
\simeq &\fibre\left( F^{\geq j}{\mathcal N}^{\geq n}\prism^{(1)}_{\mathscr{X}/A}\{n\}\xrightarrow{\can}  G^{\geq k}F^{\geq j}\prism^{(1)}_{\mathscr{X}/A}\{n\}\right)\notag
\end{align} 
since $a\geq j\geq k-e(n-1)$ and $b\geq n$ imply $a+be\geq k$. Finally, for $j\geq k$ we have 
\begin{align}
&G^{\geq k}F^{\geq j}R\Gamma_{\add}(\mathscr{X}/A,\bz_p(n))
\simeq F^{\geq j}R\Gamma_{\add}(\mathscr{X}/A,\bz_p(n))\notag
\end{align} 
since $a\geq j\geq k$ implies $a+be\geq k$.

If $pj\geq p(k-e(n-1))\geq k$ then $c^0\varphi\{n\}$ maps  $F^{\geq j}{\mathcal N}^{\geq n}\prism^{(1)}_{\mathscr{X}/A}\{n\}$    to $F^{\geq k}\prism^{(1)}_{\mathscr{X}/A}\{n\}$ which in turn maps to $G^{\geq k}F^{\geq j}\prism^{(1)}_{\mathscr{X}/A}\{n\}$ as long as $j\leq k$. This motivates the following

\begin{definition}\label{Achimdef3} 
For $k\geq k_0:= pe(n-1)/(p-1)$ we set
\begin{multline*} G^{\geq k}F^{\geq j}R\Gamma_{\syn}(\mathscr{X}/A,\bz_p(n)):=\\
\begin{cases} F^{\geq j}R\Gamma_{\syn}(\mathscr{X}/A,\bz_p(n)) & j\geq k\\
\fibre\left( F^{\geq j}{\mathcal N}^{\geq n}\prism^{(1)}_{\mathscr{X}/A}\{n\}\xrightarrow{\can-c^0\varphi\{n\}}  G^{\geq k}F^{\geq j}\prism^{(1)}_{\mathscr{X}/A}\{n\}\right) & k-e(n-1)\leq j\leq k\\
G^{\geq k}F^{\geq k-e(n-1)}R\Gamma_{\syn}(\mathscr{X}/A,\bz_p(n)) & 0\leq j\leq k-e(n-1)\end{cases} 
\end{multline*}
We consider $G^{\geq \star}F^{\geq\star}R\Gamma_{\syn}(\mathscr{X}/A,\bz_p(n))$ as an $\mathbb{N}_{\geq k_{0}}^{op}\times\bn^{op}$-indexed bifiltration, and deduce an $\mathbb{N}_{\geq k_{0}}^{op}$-indexed filtration
\begin{multline} G^{\geq k}R\Gamma_{\syn}(\mathscr{X}/A,\bz_p(n)):= G^{\geq k}F^{\geq 0}R\Gamma_{\syn}(\mathscr{X}/A,\bz_p(n))\\
=\fibre\left( F^{\geq k-e(n-1)}{\mathcal N}^{\geq n}\prism^{(1)}_{\mathscr{X}/A}\{n\}\xrightarrow{\can-c^0\varphi\{n\}}  G^{\geq k}F^{\geq k-e(n-1)}\prism^{(1)}_{\mathscr{X}/A}\{n\}\right)
\notag\end{multline}

\end{definition}

This filtration is functorial in $\mathfrak{X}$, hence we also have an $\mathbb{N}_{\geq k_{0}}^{op}$-indexed filtration $G^{\geq k}R\Gamma_{\syn}(\mathscr{X}/A,\bz_p(n))^{\rel}$. There is the following analogue of Lemma \ref{Ggeneralities} c) and d).

\begin{lemma} 
On $C:=R\Gamma_{\syn}(\mathscr{X}/A,\bz_p(n))$  the $F$-filtration and the $G$-filtration are cofinal in each other. More precisely, for $k\geq k_0+e(n-1)$  there is a factorization 
$$ F^{\geq k}C\to G^{\geq k}C\to F^{\geq k-e(n-1)}C \to G^{\geq k-e(n-1)}C$$
of the transition maps from index $k$ to index $k-e(n-1)$ in both filtrations. Consequently, $G^{\geq \star}R\Gamma_{\syn}(\mathscr{X}/A,\bz_p(n))$ is complete if and only if $F^{\geq \star}R\Gamma_{\syn}(\mathscr{X}/A,\bz_p(n))$ is complete. 
\label{syncofinal}\end{lemma}

\begin{proof} This is just the map on fibres induced by the commutative diagram
$$\begin{CD}
F^{\geq k}{\mathcal N}^{\geq n}\prism^{(1)}_{\mathscr{X}/A}\{n\}@>\can-c^0\varphi\{n\} >>  G^{\geq k+e(n-1)}F^{\geq k}\prism^{(1)}_{\mathscr{X}/A}\{n\}\\
@VVV@VVV\\
F^{\geq k}{\mathcal N}^{\geq n}\prism^{(1)}_{\mathscr{X}/A}\{n\}@>\can-c^0\varphi\{n\} >>  F^{\geq k}\prism^{(1)}_{\mathscr{X}/A}\{n\}\\
@VVV@VVV\\
F^{\geq k-e(n-1)}{\mathcal N}^{\geq n}\prism^{(1)}_{\mathscr{X}/A}\{n\}@>\can-c^0\varphi\{n\} >>  G^{\geq k}F^{\geq k-e(n-1)}\prism^{(1)}_{\mathscr{X}/A}\{n\}.
\end{CD}$$
The argument in the proof of Lemma \ref{Ggeneralities} d) gives the second statement.
\end{proof} 

\subsubsection{The map $G_\epsilon^{\geq \star}\log_{\mathscr{X}/A}$} It turns out that $G^{\geq \star}\log_{\mathscr{X}/A}$ exists with a filtration shift. 

\begin{lemma}\label{add-G-epsilon} Set $\epsilon:=e(n-1)$. The composite map
\begin{equation} G^{\geq \star}_\epsilon\gamma_{\mathscr{X}/A}: G^{\geq \star}R\Gamma_{\add}(\mathscr{X}/A,\bz_p(n))^{\rel} \xrightarrow{G^{\geq \star}\gamma_{\mathscr{X}/A}} G^{\geq \star} \widehat{\mathrm{dR}}_{\mathscr{X}/A}^{<n}[-1]\to G^{\geq \star-\epsilon} \widehat{\mathrm{dR}}_{\mathscr{X}/A}^{<n}[-1] \notag\end{equation}
factors as follows:
\[ \xymatrix{
G^{\geq k} R\Gamma_{\add}(\mathscr{X}/A,\bz_p(n))^{\rel}\ar[r]&F^{\geq k-\epsilon} R\Gamma_{\add}(\mathscr{X}/A,\bz_p(n))^{\rel}\ar[d]^{F^{\geq k-\epsilon}\gamma_{\mathscr{X}/A}}\ar[ld]& \\
F^{\geq k-\epsilon}{\mathcal N}^{\geq n}\Prism^{(1),\rel}_{\mathscr{X}/A}\{n\}\ar[r]&F^{\geq k-\epsilon} \widehat{\mathrm{dR}}_{\mathscr{X}/A}^{< n}[-1]\ar[r]&G^{\geq k-\epsilon} \widehat{\mathrm{dR}}_{\mathscr{X}/A}^{< n}[-1]
}
\] 
\end{lemma}

\begin{proof}This is clear from Lemma \ref{Ggeneralities} c).\end{proof}

\begin{definition}\label{syn-G-epsilon} Define the map 
\begin{equation} G^{\geq \star}_\epsilon\log_{\mathscr{X}/A}: G^{\geq \star}R\Gamma_{\syn}(\mathscr{X}/A,\bz_p(n))^{\rel} \to G^{\geq \star-\epsilon} \widehat{\mathrm{dR}}_{\mathscr{X}/A}^{<n}[-1] \notag\end{equation}
as the composite
\[ \xymatrix{
G^{\geq k} R\Gamma_{\syn}(\mathscr{X}/A,\bz_p(n))^{\rel}\ar[r]&F^{\geq k-\epsilon} R\Gamma_{\syn}(\mathscr{X}/A,\bz_p(n))^{\rel}\ar[d]^{F^{\geq k-\epsilon}\log_{\mathscr{X}/A}}\ar[ld]& \\
F^{\geq k-\epsilon}{\mathcal N}^{\geq n}\Prism^{(1),\rel}_{\mathscr{X}/A}\{n\}\ar[r]&F^{\geq k-\epsilon} \widehat{\mathrm{dR}}_{\mathscr{X}/A}^{< n}[-1]\ar[r]&G^{\geq k-\epsilon} \widehat{\mathrm{dR}}_{\mathscr{X}/A}^{< n}[-1]
}
\] for $k\geq k_0$ using Lemma \ref{syncofinal}.
\end{definition}

\subsubsection{The definition of $\mathrm{triv}$} The following Lemma is the analogue of Prop. \ref{griso} for the $G$-filtration and proves the isomorphism (\ref{gtriv}) from the introduction. 

\begin{lemma}\label{lem-triv}
Set $\epsilon:=e(n-1)$. For $c\geq 1$ and $$k\geq  \frac{pe(n-1)+c}{p-1}=\frac{p\epsilon+c}{p-1}$$ there is an isomorphism of bounded filtrations $\mathrm{triv}$ fitting into a commutative diagram
\[ \xymatrix@C-40pt{
G^{[k,k+c[} R\Gamma_{\syn}(\mathscr{X}/A,\bz_p(n)) \ar[dr]\ar[rr]^{\mathrm{triv}}_{\sim}&&
G^{[k,k+c[} R\Gamma_{\add}(\mathscr{X}/A,\bz_p(n))\ar[dl] \\
&F^{[k-\epsilon, k-\epsilon+c[}{\mathcal N}^{\geq n}\Prism^{(1)}_{\mathscr{X}/A}\{n\} &
}
\] 
\end{lemma}

\begin{proof} This follows from the commutative diagram with exact rows
\[ \xymatrix{
G^{[k,k+c[} R\Gamma_{\syn}(\mathscr{X}/A,\bz_p(n))\ar[d]^{\simeq}\ar[r]^{}&F^{[k-\epsilon,k-\epsilon+c[} \mathcal{N}^{\geq n}\Prism\{n\}\ar[d]^{\mathrm{Id}}\ar[r]^{\hspace{0.4cm}\mathrm{can}-c^0\varphi\{n\}}& 
\frac{G^{\geq k}F^{\geq k-\epsilon}\Prism\{n\}}{G^{\geq k+c}F^{\geq k-\epsilon+c}\Prism\{n\}}\ar[d]^{\mathrm{Id}}\\
\mathrm{Fib}\ar[r]\ar[d]^{\simeq}&F^{[k-\epsilon,k-\epsilon+c[} \mathcal{N}^{\geq n}\Prism\{n\}\ar[r]^{\hspace{0.2cm}\mathrm{can}}\ar[d]&\frac{G^{\geq k}F^{\geq k-\epsilon}\Prism\{n\}}{G^{\geq k+c}F^{\geq k-\epsilon+c}\Prism\{n\}}\ar[d]\\
G^{[k,k+c[} R\Gamma_{\add}(\mathscr{X}/A,\bz_p(n))\ar[r]&G^{[k,k+c[}\mathcal{N}^{\geq n}\Prism\{n\}\ar[r]&G^{[k,k+c[}\Prism\{n\}
}
\] 
where we abbreviate  $\prism=\prism^{(1)}_{\mathscr{X}/A}$. The upper right square commutes since the map $c^0\varphi\{n\}$ sends $F^{\geq k-\epsilon} \mathcal{N}^{\geq n}\prism$ to $G^{\geq k+c}F^{\geq k-\epsilon+c}\prism$ by our assumption $p(k-\epsilon)\geq k+c$. The isomorphism of the fibres in the lower two rows is a consequence of the equivalences (\ref{kappadef}) and (\ref{rewriting}) for $j=k-e(n-1)$.
\end{proof}

We note that $\mathrm{triv}$ is functorial in $\mathfrak{X}$, hence we also have a commutative diagram
\begin{equation}\xymatrix@C-55pt{
G^{[k,k+c[} R\Gamma_{\syn}(\mathscr{X}/A,\bz_p(n))^{\rel} \ar[dr]\ar[rr]^{\mathrm{triv}}_{\sim}&&
G^{[k,k+c[} R\Gamma_{\add}(\mathscr{X}/A,\bz_p(n))^{\rel}\ar[dl] \\
&F^{[k-\epsilon, k-\epsilon+c[}{\mathcal N}^{\geq n}\Prism^{(1),\rel}_{\mathscr{X}/A}\{n\} &
}
\label{trivrel}\end{equation} 

\begin{lemma}\label{fundlemma}
For $c\geq 1$ and $k\geq (pe(n-1)+c)/(p-1)$ there is a commutative diagram
\[ \xymatrix@C-40pt{
G^{[k,k+c[} R\Gamma_{\syn}(\mathscr{X}/A,\bz_p(n))^{\rel} \ar[dr]_(.4){G^{[k,k+c[}_\epsilon\log_{\mathscr{X}/A}}\ar[rr]^{\mathrm{triv}}_{\sim}& &
G^{[k,k+c[} R\Gamma_{\add}(\mathscr{X}/A,\bz_p(n))^{\rel}\ar[dl]^(.4){G^{[k,k+c[}_\epsilon\gamma_{\mathscr{X}/A}} \\
&G^{[k-\epsilon, k-\epsilon+c[} \widehat{\mathrm{dR}}_{\mathscr{X}/A}^{< n}[-1]&
}
\] 
\end{lemma}
\begin{proof}
This follows from (\ref{trivrel}) combined with Lemma \ref{add-G-epsilon} and Definition \ref{syn-G-epsilon}. 
\end{proof}

\subsection{The $G$-filtration for $e$-filtered morphisms} \label{gefiltration} In the last section we have recorded properties of the $G$-filtration for an arbitrary schematic morphism $\mathfrak{X}\to \mathfrak{A}$. In this section we see what can be said if $\mathfrak{X}\to \mathfrak{A}$ is $e$-filtered (for the same integer $e$ used in the $G$-filtration). 

\begin{lemma}\label{gcomplete}Let $F^{\geq\star} A$ be a filtered $\delta$-ring, $\mathfrak{A}=F\Spf(A)$ and $\mathfrak{X}\to \mathfrak{A}$ a $e$-filtered morphism of derived formal stacks.
\begin{itemize}
\item[a)] The filtrations 
$$G^{\geq \star}R\Gamma_{\add}(\mathscr{X}/A,\bz_p(n)),\  G^{\geq \star}R\Gamma_{\syn}(\mathscr{X}/A,\bz_p(n)),\  G^{\geq \star} \widehat{\mathrm{dR}}_{\mathscr{X}/A}^{<n}[-1]$$ 
are complete.
\item[b)] The filtrations 
$$G^{\geq \star}R\Gamma_{\add}(\mathscr{X}/A,\bz_p(n))^{\rel},\  G^{\geq \star}R\Gamma_{\syn}(\mathscr{X}/A,\bz_p(n))^{\rel},\  G^{\geq \star} \widehat{\mathrm{dR}}_{\mathscr{X}/A}^{<n,\rel}[-1]$$ 
are complete.
\end{itemize}
\end{lemma}

\begin{proof} Part a) is an immediate consequence of Prop. \ref{corforA} b) combined with Lemma \ref{Ggeneralities} d) and Lemma \ref{syncofinal}.   Applying a) to $\mathfrak{X}$ and $\mathfrak{X}/p$ (see Prop. \ref{corforA} d)) gives b).
\end{proof}

\subsubsection{Multiplication by $p$} We come to the key property of the $G$-filtration.

\begin{lemma} Given a commutative square $bf=f'a$ in a stable $\infty$-category 
$$\xymatrix@R-18pt{A \ar[dd]^a\ar[r]^f& B\ar[dd]^b\ar[r]\ar[ddl]_(.45){\tilde{b}} & B/A \ar[dd]^c\ar[dr] &\\
 & & &0\ar[dl] \\
A'\ar[r]^{f'} & B'\ar[r] & B'/A'&}$$
the space of lifts $\tilde{b}$ is homotopy equivalent to the space of homotopies between $0$ and the induced map $c$ on cofibres. 
\label{plift}\end{lemma}

\begin{proof} We have an equivalence
$$\Map(A,B'/A'[-1])\simeq \Map(A,A')\times_{\Map(A,B')}\{bf\},\quad\lambda\mapsto a+\lambda.$$
The fibre product  $\Map(B,A')\times_{\Map(B,B')}\{b\}$ is either empty or a principal homogeneous space $\tilde{b}+\Map(B,B'/A'[-1])$ under $\Map(B,B'/A'[-1])$. It follows that the space of lifts is either empty or a principal homogeneous space
$$ (\tilde{b}+\Map(B,B'/A'[-1]))\times_{(a+\Map(A,B'/A'[-1])}\{a\}$$
 under $\Map(B/A,B'/A'[-1])$, i.e. the space of zero homotopies on the cofibre.
\end{proof}

 If $\mathfrak{X}$ is $e$-filtered, Lemma \ref{plift} implies that a choice of $\tilde{p}\in H^0(\mathscr{X},\pi_0F^{\geq e}\co_{\mathscr{X}})$ as in Rem. \ref{ptilderemark} gives a homotopy between $p$ and $0$ on  
$$ R\Gamma(\mathscr{X}, F^{[0,e[}\co_{\mathscr{X}})\simeq F^{[0,e[}\mathrm{gr}^0_{\mathcal N}\prism^{(1)}_{\mathscr{X}/A}\{0\}\simeq G^{[0,e[}\prism^{(1)}_{\mathscr{X}/A}$$
and hence a commutative diagram
$$\xymatrix{G^{\geq e}\prism^{(1)}_{\mathscr{X}/A}\ar[r]  \ar[d]^{\cdot p} & G^{\geq 0}\prism^{(1)}_{\mathscr{X}/A}\ar[dl]_{\tilde{p}} \ar[d]^{\cdot p} \\
G^{\geq e}\prism^{(1)}_{\mathscr{X}/A}\ar[r]  & G^{\geq 0}\prism^{(1)}_{\mathscr{X}/A}.}$$
Composing with the unit map of the $\be_\infty$-$A$-algebra $G^{\geq 0}\prism^{(1)}_{\mathscr{X}/A}$ we may also view $\tilde{p}$ as a section 
\begin{equation}\tilde{p}\in\pi_0G^{\geq e}\prism^{(1)}_{\mathscr{X}/A}\simeq \Map_{h\Mod_A}(HA,G^{\geq e}\prism^{(1)}_{\mathscr{X}/A})\label{choice}\end{equation}
mapping to the section
$ p\in\pi_0G^{\geq 0}\prism^{(1)}_{\mathscr{X}/A}$, as well as to the image of the given section in 
$$\pi_0R\Gamma(\mathscr{X}, F^{\geq e}\co_{\mathscr{X}})\simeq \pi_0F^{\geq e}\mathrm{gr}^0_{\mathcal N}\prism^{(1)}_{\mathscr{X}/A}\simeq   \pi_0G^{\geq e}\mathrm{gr}^0_{\mathcal N}\prism^{(1)}_{\mathscr{X}/A}.$$
The image of $\tilde{p}$ in $\pi_0G^{\geq e}\prism^{(1)}_{(\mathscr{X}/p)/A}$ likewise maps to the image of the given section in $\pi_0R\Gamma(\mathscr{X}, F^{\geq e}\co_{(\mathscr{X}/p)})\simeq\pi_0G^{\geq e}\mathrm{gr}^0_{\mathcal N}\prism^{(1)}_{(\mathscr{X}/p)/A}$.

\begin{lemma}\label{mult-by-p}
Let $F^{\geq\star} A$ be a filtered $\delta$-ring, $\mathfrak{A}=F\Spf(A)$ and $\mathfrak{X}\to \mathfrak{A}$ a $e$-filtered morphism of derived formal stacks.
\begin{itemize} 
\item[a)] Multiplication with the section (\ref{choice}) induces a commutative diagram of maps of bifiltrations
\begin{equation}\xymatrix@C-40pt{ 
&G^{\geq \star+e}\mathcal N^{\geq\star}\prism^{(1)}_{\mathscr{X}/A}\{n\}\ar[rr] \ar[dd]&& G^{\geq \star+e}\mathcal N^{\geq\star}\prism^{(1)}_{(\mathscr{X}/p)/A}\{n\} \ar[dd]\\
G^{\geq \star}\mathcal N^{\geq\star}\prism^{(1)}_{\mathscr{X}/A}\{n\}\ar[rr] \ar[dd]\ar[ur]^{\cdot\tilde{p}}&& G^{\geq \star}\mathcal N^{\geq\star}\prism^{(1)}_{(\mathscr{X}/p)/A}\{n\} \ar[dd] \ar[ur]^{\cdot\tilde{p}}& \\ 
&\ar[rr] G^{\geq \star+e}\fil^{\geq \star}_{Hod}\widehat{\mathrm{dR}}_{\mathscr{X}/A}  & &G^{\geq \star+e}\fil^{\geq \star}_{Hod}\widehat{\mathrm{dR}}_{(\mathscr{X}/p)/A}\\
\ar[rr] G^{\geq \star}\fil^{\geq \star}_{Hod}\widehat{\mathrm{dR}}_{\mathscr{X}/A} \ar[ur]^{\cdot\tilde{p}} & &G^{\geq \star}\fil^{\geq \star}_{Hod}\widehat{\mathrm{dR}}_{(\mathscr{X}/p)/A}\ar[ur]^{\cdot\tilde{p}}&
}\label{mult-by-e}\end{equation}
where the maps $\cdot\tilde{p}$ lift multiplication by $p$ on the source.
\item[b)] 
If $k\geq (e+1)(n-1)$ the induced maps
$ \cdot\tilde{p}: G^{\geq k}C\to G^{\geq k+e}C $
are equivalences for $C$ any of the complexes 
$$R\Gamma_{\add}(\mathscr{X}/A,\bz_p(n)),\  R\Gamma_{\add}((\mathscr{X}/p)/A,\bz_p(n)),\  R\Gamma_{\add}(\mathscr{X}/A,\bz_p(n))^{\rel} $$
or
$$\widehat{\mathrm{dR}}_{\mathscr{X}/A}^{<n},\  \widehat{\mathrm{dR}}_{(\mathscr{X}/p)/A}^{<n},\  \widehat{\mathrm{dR}}_{\mathscr{X}/A}^{<n,\rel}.$$ 
\end{itemize}
\end{lemma}

\begin{proof}
Part a) is immediate from the fact that $G^{\geq \star}\mathcal N^{\geq\star}\prism^{(1)}_{\mathscr{X}/A}\{n\}$ is a $G^{\geq \star}\prism^{(1)}_{\mathscr{X}/A}$-module, that  the maps 
$$\begin{CD} G^{\geq \star}\prism^{(1)}_{\mathscr{X}/A}@>>> G^{\geq \star}\prism^{(1)}_{(\mathscr{X}/p)/A}\\
@VVV@VVV\\
G^{\geq \star}\widehat{\mathrm{dR}}_{\mathscr{X}/A}@>>> G^{\geq \star}\widehat{\mathrm{dR}}_{(\mathscr{X}/p)/A}
\end{CD}$$ are multiplicative and that
$$G^{\geq \star}\gamma^{dR}_{\prism,A}\{n\}:  G^{\geq \star}\mathcal N^{\geq\star}\prism^{(1)}_{\mathscr{X}/A}\{n\} \to G^{\geq \star}\fil^{\geq \star}_{Hod}\widehat{\mathrm{dR}}_{\mathscr{X}/A}$$
is a map of $G^{\geq \star}\prism^{(1)}_{\mathscr{X}/A}$-modules. 

b)  The filtration $G^{\geq \star}R\Gamma_{\add}(\mathscr{X}/A,\bz_p(n))$, resp. $G^{\geq \star} \widehat{\mathrm{dR}}_{\mathscr{X}/A}^{<n}$, has a bounded filtration with graded pieces 
$$G^{\geq \star}\mathrm{gr}^i_{\mathcal N}\prism^{(1)}_{\mathscr{X}/A}\{n\}[-1]\simeq F^{\geq \star-ei}\mathrm{gr}^i_{\mathcal N}\prism^{(1)}_{\mathscr{X}/A}\{n\}[-1],$$ resp. $G^{\geq \star}L\widehat{\Omega}^{i}_{\mathscr{X}/A}[-i]\simeq F^{\geq \star-ei}L\widehat{\Omega}^{i}_{\mathscr{X}/A}[-i]$,  for $i=0,\dots,n-1$.  By part a) this filtration is preserved by multiplication with $\tilde{p}$.  By Prop. \ref{corforA} c) the maps $\cdot\tilde{p}$ are equivalences on graded pieces for 
$$ k\geq \max_{0\leq i\leq n-1}\{i+ei\}=(1+e)(n-1),$$
hence they are equivalences for $k$ in the same range. 
By Prop. \ref{corforA} d) the same is true for $\mathfrak{X}/p$ and hence for the relative groups.
\end{proof}

\begin{lemma}\label{psyn} Let $F^{\geq\star} A$ be a filtered $\delta$-ring, $\mathfrak{A}=F\Spf(A)$ and $\mathfrak{X}\to \mathfrak{A}$ a $e$-filtered morphism of derived formal stacks. 

a) For $k_1\geq \frac{pe(n-1)+e}{p-1}$ there is a map of $\bn^{op}_{\geq k_1}$-indexed filtrations
$$\tilde{p}_{\syn}:G^{\geq \star}R\Gamma_{\syn}(\mathscr{X}/A,\bz_p(n))\rightarrow G^{\geq \star+e}R\Gamma_{\syn}(\mathscr{X}/A,\bz_p(n))$$
lifting multiplication by $p$ on the source. 

b) For $c\geq 1$ and $k\geq \frac{pe(n-1)+c+e}{p-1}$ there is a commutative diagram 
\begin{equation} \xymatrix{
G^{[k,k+c[} R\Gamma_{\syn}(\mathscr{X}/A,\bz_p(n))\ar[d]_{\tilde{p}_{\syn}}\ar[r]^{\mathrm{triv}}&G^{[k,k+c[} R\Gamma_{\add}(\mathscr{X}/A,\bz_p(n))\ar[d]^{\cdot\tilde{p}} \\
G^{[k+e,k+c+e[} R\Gamma_{\syn}(\mathscr{X}/A,\bz_p(n))\ar[r]^{\mathrm{triv}}&G^{[k+e,k+c+e[} R\Gamma_{\add}(\mathscr{X}/A,\bz_p(n))
}
\label{triv-and-p}\end{equation}

c) If $k_1\geq \frac{pe(n-1)+e+1}{p-1}$ is large enough so that 
$$\cdot \tilde{p}:G^{\geq \star}R\Gamma_{\add}(\mathscr{X}/A,\bz_p(n))\rightarrow G^{\geq \star+e}R\Gamma_{\add}(\mathscr{X}/A,\bz_p(n))$$
is an equivalence of $\mathbb{N}^{op}_{\geq k_1}$-indexed filtrations, then so is $\tilde{p}_{\syn}$.
\end{lemma}
\begin{proof} 


a) By Lemma \ref{plift}, in order to construct $\tilde{p}_{\syn}$ it suffices to exhibit a homotopy between $p$ and $0$ on
$$ G^{[\star,\star+e[} R\Gamma_{\syn}(\mathscr{X}/A,\bz_p(n))\underset{\sim}{\xrightarrow{\mathrm{triv}}}G^{[\star,\star+e[} R\Gamma_{\add}(\mathscr{X}/A,\bz_p(n))$$
where the isomorphism $\mathrm{triv}$ exists by the assumption on $k_1$.  Again using Lemma \ref{plift}, such a homotopy is induced by multiplication with $\tilde{p}$ on $G^{\geq \star}R\Gamma_{\add}(\mathscr{X}/A,\bz_p(n))$ (see Lemma \ref{mult-by-p}).


b) The arrows $\tilde{p}_{\syn}$ and $\cdot\tilde{p}$ in (\ref{triv-and-p}) are the pushout of the corresponding arrows $\tilde{p}$ on $G^{\geq k+c}$ and $G^{\geq k}$:
$$\xymatrix{G^{\geq k+c+e}\ar[r]\ar[d] & G^{\geq k+c}\ar[r]\ar[d]\ar@/_1pc/[l]_{\tilde{p}} & G^{[k+c,k+c+e[}\ar[d]\\
G^{\geq k+e} \ar[r]\ar[d] & G^{\geq k} \ar[r]\ar[d]\ar@/_1pc/[l]_{\tilde{p}}&G^{[k,k+e[}\\
G^{[k+e,k+c+e[} \ar[r] & G^{[k,k+c[}\ar@/_1pc/[l]_{\tilde{p}}  &
}$$
They are also the pushout of the arrows $\tilde{p}$ on $G^{[k+c,k+c+e[}$ and $G^{[k,k+c+e[}$
\begin{equation}\xymatrix{ 0 \ar[r]\ar[d] &  G^{[k+c,k+c+e[} \ar[r]^{\sim}\ar[d]\ar@/_1pc/[l]_{\tilde{p}} & G^{[k+c,k+c+e[}\ar[d]\\
G^{[k+e,k+c+e[} \ar[r]\ar[d] & G^{[k,k+c+e[} \ar[r]\ar[d]\ar@/_1pc/[l]_{\tilde{p}}&G^{[k,k+e[}\\
G^{[k+e,k+c+e[} \ar[r] & G^{[k,k+c[}\ar@/_1pc/[l]_{\tilde{p}}  &
}\label{truncation}\end{equation}
since the latter are induced via Lemma \ref{plift} by the same homotopies between $p$ and $0$ on $G^{[k+c,k+c+e[}$ and $G^{[k,k+e[}$ respectively. By the assumption on $k$ the isomorphism $\mathrm{triv}$ exists on all terms in (\ref{truncation}) and moreover commutes with the maps $\tilde{p}$ in the top and middle row by construction. Hence $\mathrm{triv}$ commutes with $\tilde{p}$ in the bottom row, i.e. (\ref{triv-and-p}) commutes.

c)  For $k\geq k_1$ the map 
$$\tilde{p}_{\syn}:\mathrm{gr}_G^{k}R\Gamma_{\syn}(\mathscr{X}/A,\bz_p(n))\rightarrow \mathrm{gr}_G^{k+e}R\Gamma_{\syn}(\mathscr{X}/A,\bz_p(n))$$
induced by $\tilde{p}_{\syn}$ on associated graded is an isomorphism by (\ref{triv-and-p}) for $c=1$. Since the filtration $G^{\geq \star}R\Gamma_{\syn}(\mathscr{X}/A,\bz_p(n))$ is complete by Lemma \ref{gcomplete} a) it follows that $\tilde{p}_{\syn}$ is an equivalence of $\mathbb{N}_{\geq k_1}$-indexed filtrations.
\end{proof}

\subsection{The map $G^{\geq k}\synlog_{\mathscr{X}/A}$} \label{synlog} We combine the results of the last section to define the syntomic logarithm following the outline in the introduction.

\begin{prop}\label{logdef}
Let $F^{\geq\star} A$ be a filtered $\delta$-ring, $\mathfrak{A}=F\Spf(A)$ and $\mathfrak{X}\to \mathfrak{A}$ a $e$-filtered morphism of derived formal stacks.
Let $k_1\geq \frac{pe(n-1)+e+1}{p-1}$ be such that 
\begin{equation}  \cdot\tilde{p}: G^{\geq \star}R\Gamma_{\add}(\mathscr{X}/A,\bz_p(n))\xrightarrow{\sim} G^{\geq \star+e}R\Gamma_{\add}(\mathscr{X}/A,\bz_p(n))  \notag\end{equation}
is an equivalence of $\mathbb{N}^{op}_{\geq k_1}$-indexed filtrations. Then there is a unique isomorphism of $\bn^{op}_{\geq k_1}$-indexed filtrations
  \[
   G^{\geq \star}\synlog_{\mathscr{X}/A}: G^{\geq \star}R\Gamma_{\syn}(\mathscr{X}/A,\bz_p(n))\xrightarrow{\sim} G^{\geq \star}R\Gamma_{\add}(\mathscr{X}/A,\bz_p(n))
  \]
  characterized by the following two properties:
  \begin{enumerate}
    \item[a)] The diagrams
     \[
      \begin{CD}
        G^{\geq \star}R\Gamma_{\syn}(\mathscr{X}/A,\bz_p(n)) @>G^{\geq \star}\synlog_{\mathscr{X}/A}>> G^{\geq \star}R\Gamma_{\add}(\mathscr{X}/A,\bz_p(n))\\
        @V\tilde{p}_{\syn} VV  @V\cdot\tilde{p}VV\\
        G^{\geq \star+e}R\Gamma_{\syn}(\mathscr{X}/A,\bz_p(n)) @>G^{\geq \star+e}\synlog_{\mathscr{X}/A}>> G^{\geq \star+e}R\Gamma_{\add}(\mathscr{X}/A,\bz_p(n))
      \end{CD}
    \]
    commute.
  \item[b)] The diagrams
     \[
      \begin{CD}
      G^{\geq k}R\Gamma_{\syn}(\mathscr{X}/A,\bz_p(n)) @>G^{\geq k}\synlog_{\mathscr{X}/A}>> G^{\geq k}R\Gamma_{\add}(\mathscr{X}/A,\bz_p(n))\\
        @VVV @VVV\\
        G^{[k,k+c[}R\Gamma_{\syn}(\mathscr{X}/A,\bz_p(n)) @>\mathrm{triv}>> G^{[k,k+c[}R\Gamma_{\add}(\mathscr{X}/A,\bz_p(n))
      \end{CD}
    \]
      commute for $k\geq (pe(n-1)+c)/(p-1)$, where $\mathrm{triv}$ was defined in Lemma \ref{lem-triv}.
       \end{enumerate}
\end{prop}

\begin{proof} Given $k\geq k_1$ and $c\geq 1$, pick $m$ with $k+me\geq (pe(n-1)+c+e)/(p-1) $ so that the equivalence $\mathrm{triv}$ is defined on $G^{[k+me,k+me+c[}$. Then define the equivalence $G^{[k,k+c[}\synlog_{\mathscr{X}/A}$ by the commutative diagram
  \[
      \xymatrix{
      G^{[k,k+c[}R\Gamma_{\syn}(\mathscr{X}/A,\bz_p(n)) \ar[r]^{G^{[k,k+c[}\synlog_{\mathscr{X}/A}} \ar[d]^{\tilde{p}_{\syn}^m}& G^{[k,k+c[}R\Gamma_{\add}(\mathscr{X}/A,\bz_p(n))\ar[d]^{\cdot\tilde{p}^m}\\
       G^{[k+me,k+me+c[}R\Gamma_{\syn}(\mathscr{X}/A,\bz_p(n)) \ar[r]^{\mathrm{triv}} & G^{[k+me,k+me+c[}R\Gamma_{\add}(\mathscr{X}/A,\bz_p(n))
      }
  \]
  where the vertical maps are equivalences by Lemma \ref{psyn} c). This definition is independent of $m$ by (\ref{triv-and-p}). Hence we also have commutative diagrams
  \[
      \xymatrix@C+40pt{
      G^{[k,k+c+1[}R\Gamma_{\syn}(\mathscr{X}/A,\bz_p(n)) \ar[r]^{G^{[k,k+c+1[}\synlog_{\mathscr{X}/A}} \ar[d]& G^{[k,k+c+1[}R\Gamma_{\add}(\mathscr{X}/A,\bz_p(n))\ar[d]\\
      G^{[k,k+c[}R\Gamma_{\syn}(\mathscr{X}/A,\bz_p(n)) \ar[r]^{G^{[k,k+c[}\synlog_{\mathscr{X}/A}}& G^{[k,k+c[}R\Gamma_{\add}(\mathscr{X}/A,\bz_p(n)).} 
  \]
  Passing to the inverse limit over $c$ and using completeness of the $G$-filtration proven in Lemma \ref{gcomplete} a) we obtain the desired isomorphism
 \[
    G^{\geq k}\synlog_{\mathscr{X}/A}: G^{\geq k}R\Gamma_{\syn}(\mathscr{X}/A,\bz_p(n))\xrightarrow{\sim} G^{\geq k}R\Gamma_{\add}(\mathscr{X}/A,\bz_p(n)).
  \]
  
\end{proof}

\begin{corollary}\label{relativelogdef} Let $F^{\geq\star} A$ be a filtered $\delta$-ring, $\mathfrak{A}=F\Spf(A)$ and $\mathfrak{X}\to \mathfrak{A}$ a $e$-filtered morphism of derived formal stacks.  Then for
$$k_1:=\max\left\{\frac{pe(n-1)+e+1}{p-1},(e+1)(n-1)\right\}$$
there are isomorphisms of $\bn^{op}_{\geq k_1}$-indexed filtrations
\[
G^{\geq \star}\synlog_{\mathscr{X}/A}: G^{\geq \star}R\Gamma_{\syn}(\mathscr{X}/A,\bz_p(n))\xrightarrow{\sim} G^{\geq \star}R\Gamma_{\add}(\mathscr{X}/A,\bz_p(n))
\]
and 
  \[
   G^{\geq \star}\synlog^{\rel}_{\mathscr{X}/A}: G^{\geq \star}R\Gamma_{\syn}(\mathscr{X}/A,\bz_p(n))^{\rel}\xrightarrow{\sim} G^{\geq \star}R\Gamma_{\add}(\mathscr{X}/A,\bz_p(n))^{\rel}
  \]
  satisfying a) and b) in Prop. \ref{logdef}.
\end{corollary}

\begin{proof} This follows from Lemma \ref{mult-by-p} b) combined with Prop. \ref{logdef} for $\mathfrak{X}$ and $\mathfrak{X}/p$.
\end{proof}

\begin{remark} The construction of $G^{\geq \star}\synlog_{\mathscr{X}/A}$ depends on a choice of $\tilde{p}$ and the maps $G^{\geq \star}\synlog_{\mathscr{X}/A}$ are only functorial for maps of $e$-filtered pairs $(\mathfrak{X}/\mathfrak{A},\tilde{p})\to (\mathfrak{X}'/\mathfrak{A}',\tilde{p}')$ compatible with $\tilde{p}$. 
\end{remark} 

\begin{theorem} \label{thm:synlog} Let $F^{\geq\star} A$ be a filtered $\delta$-ring, $\mathfrak{A}=F\Spf(A)$ and $\mathfrak{X}\to \mathfrak{A}$ a $e$-filtered morphism of derived formal stacks.

a) For large enough $k$ there is a commutative diagram
\begin{equation} \xymatrix{
G^{\geq k} R\Gamma_{\syn}(\mathscr{X}/A,\bz_p(n))^{\rel}\ar[d]^{\iota_{\syn}}\ar[dr]^{G^{\geq k}_\epsilon\log_{\mathscr{X}/A}}\ar[rr]^{G^{\geq k}\synlog^{\rel}_{\mathscr{X}/A}}_{\sim}&&G^{\geq k} R\Gamma_{\add}(\mathscr{X}/A,\bz_p(n))^{\rel}\ar[dl]_{G^{\geq k}_\epsilon\gamma_{\mathscr{X}/A}} \ar[d]^{\iota_{\add}}\\
R\Gamma_{\syn}(\mathscr{X}/A,\bz_p(n))^{\rel}\ar[dr]_{\log_{\mathscr{X}/A}}&G^{\geq k-\epsilon}\widehat{\mathrm{dR}}_{\mathscr{X}/A}^{< n}[-1]\ar[d]^{\iota_{dR}}&R\Gamma_{\add}(\mathscr{X}/A,\bz_p(n))^{\rel}\ar[dl]^{\gamma_{\mathscr{X}/A}}\\
&\widehat{\mathrm{dR}}_{\mathscr{X}/A}^{< n}[-1]&
}
\label{synlogdiagram2}\end{equation}
functorial in the pair $(\mathfrak{X}/\mathfrak{A},\tilde{p})$.

b) For $?=\syn,\add$ the filtration 
\begin{multline*}G^{[0,k[}F^{\geq \star}R\Gamma_{?}(\mathscr{X}/A,\bz_p(n))^{\rel}\\ :=\Cone(G^{\geq k}F^{\geq\star}R\Gamma_{?}(\mathscr{X}/A,\bz_p(n))^{\rel}\to F^{\geq\star}R\Gamma_{?}(\mathscr{X}/A,\bz_p(n))^{\rel})\end{multline*}
is $\bn^{op}$-indexed and bounded with underlying complex $\Cone(\iota_?)$. There is an isomorphism of associated graded
\begin{equation}G^{[0,k[}\mathrm{gr}^m_FR\Gamma_{\syn}(\mathscr{X}/A,\bz_p(n))^{\rel}\simeq G^{[0,k[}\mathrm{gr}^m_F R\Gamma_{\add}(\mathscr{X}/A,\bz_p(n))^{\rel}\label{grfiso}\end{equation}
for $m\geq 1$. For $m=0$ the complex 
$$G^{[0,k[}\mathrm{gr}^0_FR\Gamma_{?}(\mathscr{X}/A,\bz_p(n))^{\rel}\simeq \mathrm{gr}^0_FR\Gamma_{?}(\mathscr{X}/A,\bz_p(n))^{\rel}$$ has a bounded filtration $\Phi^{\geq\star}\mathrm{gr}^0_FR\Gamma_{?}(\mathscr{X}/A,\bz_p(n))^{\rel}$ for both $?=\syn, \add$ such that there is an isomorphism of associated graded
\begin{equation}\mathrm{gr}^i_\Phi\mathrm{gr}^0_FR\Gamma_{\syn}(\mathscr{X}/A,\bz_p(n))^{\rel}\simeq \mathrm{gr}^i_\Phi\mathrm{gr}^0_F R\Gamma_{\add}(\mathscr{X}/A,\bz_p(n))^{\rel}.\label{grf0iso}\end{equation}
The complexes in (\ref{grfiso}) and (\ref{grf0iso}) are $p^N$-torsion for some large enough $N$, and therefore $\iota_{\syn}$ and $\iota_{\add}$ are rational isomorphisms.

c) The map $\gamma_{\mathscr{X}/A}$ and therefore the map $\log_{\mathscr{X}/A}$ is a rational isomorphism.  
\end{theorem}

\begin{proof} a) The maps $\iota_{\add}$ and $\iota_{dR}$ are simply the canonical maps to $G^0$, so the right hand square in (\ref{synlogdiagram2}) commutes. The map $\iota_{\syn}$ is the composite $G^{\geq k}\to F^{\geq k-\epsilon}\to F^{\geq 0}$ (see Def. \ref{syn-G-epsilon}), so the left hand square commutes as well. 

 By completeness of all three filtrations (see Lemma \ref{gcomplete}) commutativity of the top triangle in (\ref{synlogdiagram2}) reduces to commutativity of the top triangle in the following diagram for all $c\geq 1$. Here we choose $k_1$ as in Cor. \ref{relativelogdef}, i.e. large enough so that the maps $\cdot\tilde{p}$ and $\tilde{p}_{\syn}$ are equivalences for $k\geq k_1$ (see Lemma \ref{mult-by-p} b) and Lemma \ref{psyn} c)).
 \[
      \xymatrix@C-80pt{
      G^{[k,k+c[}R\Gamma_{\syn}(\mathscr{X}/A,\bz_p(n))^{\rel} \ar[rr]^{G^{[k,k+c[}\synlog_{\mathscr{X}/A}} \ar[dd]^{\tilde{p}_{\syn}^m} \ar[dr]&& G^{[k,k+c[}R\Gamma_{\add}(\mathscr{X}/A,\bz_p(n))^{\rel}\ar[dd]^{\cdot\tilde{p}^m}\ar[dl]\\
      &G^{[k-\epsilon,k-\epsilon+c[}\widehat{\mathrm{dR}}_{\mathscr{X}/A}^{< n,\rel}[-1]\ar[dd]^(.25){\cdot\tilde{p}^m}&\\
       G^{[k+me,k+me+c[}R\Gamma_{\syn}(\mathscr{X}/A,\bz_p(n))^{\rel} \ar[rr]^(.45){\mathrm{triv}} \ar[dr]& &G^{[k+me,k+me+c[}R\Gamma_{\add}(\mathscr{X}/A,\bz_p(n))^{\rel}\ar[dl]\\
      &G^{[k+me-\epsilon,k+me-\epsilon+c[}\widehat{\mathrm{dR}}_{\mathscr{X}/A}^{< n,\rel}[-1]&
      }
  \]
Since we know commutativity of the bottom triangle by Lemma \ref{fundlemma} it suffices to verify commutativity of the three faces. The right hand face commutes by Lemma \ref{mult-by-p} a) and the back face by the definition of $G^{[k,k+c[}\synlog_{\mathscr{X}/A}$. By Lemma \ref{plift}, commutativity of the left hand face is equivalent to compatibility of the homotopies between $p$ and $0$ with $G^{[k,k+e[}_\epsilon\log_{\mathscr{X}/A}$. By Lemma \ref{fundlemma} for $c=e$, and the fact that the nullhomotopies are compatible with $\mathrm{triv}$ by definition, it remains to remark that they are also compatible with $G^{[k,k+e[}_\epsilon\gamma_{\mathscr{X}/A}$ by Lemma \ref{mult-by-p} a) (and Lemma \ref{plift}).

b) By Def. \ref{Achimdef3} and the discussion before it we have 
$$G^{[0,k[}F^{\geq j}R\Gamma_{?}(\mathscr{X}/A,\bz_p(n))^{\rel}=\begin{cases} 0 & j\geq k\\ \Cone(\iota_?) & j=0  \end{cases}$$
for both $?=\syn,\add$, i.e. the filtrations are bounded with underlying complex $\Cone(\iota_?)$. For $m\geq 1$ the isomorphism (\ref{grfiso}) follows as in the proof of Prop. \ref{griso}: in Def. \ref{Achimdef3} the map $c^0\varphi\{n\}$ factors through $F^{\geq pm}$, hence through $F^{\geq m+1}$. For $m=0$ on the other hand we have
\begin{equation}G^{[0,k[}\mathrm{gr}^0_FR\Gamma_{?}(\mathscr{X}/A,\bz_p(n))^{\rel}=\mathrm{gr}^0_FR\Gamma_{?}(\mathscr{X}/A,\bz_p(n))^{\rel}\simeq R\Gamma_{?}(\mathrm{gr}^0_F\mathscr{X}/\mathrm{gr}^0_FA,\bz_p(n))^{\rel}\notag\end{equation}
for $?=\syn,\add$ by Def. \ref{Achimdef3} and Cor. \ref{globalfiltrations}. Note that  the derived formal scheme $\mathrm{gr}^0_F\mathscr{X}$ has characteristic $p$ since $\mathfrak{X}$ is $e$-filtered.

\begin{prop} Let $\mathscr{X}$ be a derived formal scheme of characteristic $p$ over a $\delta$-ring $A$. Consider the grading
$$ \co_{\mathscr{X}/p}=\co_{\mathscr{X}}\oplus\co_{\mathscr{X}}[1]=:\co_{\mathscr{X}/p}^0\oplus\co_{\mathscr{X}/p}^1$$
and the corresponding split $\bn^{op}$-indexed filtration 
$$\Phi^{\geq \star}\co_{\mathscr{X}/p}:=\bigoplus_{i\geq\star}\co_{\mathscr{X}/p}^i.$$ 

a) There is a functorial decomposition in $\gsheaf_A$
$$\underline{\mathscr{E}}((\mathscr{X}/p)/A)\simeq \underline{\mathscr{E}}(\mathscr{X}/A)\oplus\Phi^{\geq 1}\underline{\mathscr{E}}((\mathscr{X}/p)/A)$$ 
and hence an isomorphism
$$ \underline{\mathscr{E}}(\mathscr{X}/A)^{\rel}[1]\simeq\Phi^{\geq 1}\underline{\mathscr{E}}((\mathscr{X}/p)/A)$$
functorial in $\mathscr{X}$.

b) The filtrations $\Phi^{\geq \star}L\widehat{\Omega}^{i}_{(\mathscr{X}/p)/A}$ and $\Phi^{\geq \star}\mathrm{gr}^i_{\mathcal N}\prism^{(1)}_{(\mathscr{X}/p)/A}\{n\}$ are bounded in $[0,i+1]$.

c) The filtrations 
$$\Phi^{\geq \star} \widehat{\mathrm{dR}}_{(\mathscr{X}/p)/A}^{<n},\ \Phi^{\geq \star}R\Gamma_{\add}((\mathscr{X}/p)/A,\bz_p(n)),\  \Phi^{\geq \star}R\Gamma_{\syn}((\mathscr{X}/p)/A,\bz_p(n))$$
are bounded in $[0,n]$.
\label{charp}\end{prop}

\begin{proof} Note that $\co_{\mathscr{X}}\to\co_{\mathscr{X}/p}=\co_{\mathscr{X}}\otimes^L_\bz\bF_p$ is functorially split by the multiplication map. Part a) is an immediate consequence of Prop. \ref{schematicbc} c) for the filtration $\Phi^{\geq \star}\co_{\mathscr{X}/p}$.

The animated $\co_{\mathscr{X}}$-algebra
$$\co_{\mathscr{X}/p}=\co_{\mathscr{X}}\oplus\co_{\mathscr{X}}[1]\simeq \bigoplus_{i\geq 0}(\bigwedge^i_{\co_{\mathscr{X}}}\co_{\mathscr{X}})[i]\simeq \mathrm{LSym}^\star_{\co_{\mathscr{X}}}\co_{\mathscr{X}}[1]$$
is free on the $\co_{\mathscr{X}}$-module $\co_{\mathscr{X}}[1]\simeq\co_{\mathscr{X}/p}^1$, hence 
$$L_{(\mathscr{X}/p)/\mathscr{X}}\simeq \co_{\mathscr{X}}[1]\otimes_{\co_{\mathscr{X}}}\co_{\mathscr{X}/p}\simeq \co_{\mathscr{X}/p}\langle1\rangle[1]$$
where we denote by $M\langle k\rangle^i:=M^{i-k}$ a shift in grading weight. The transitivity triangle for the cotangent complex
$$ L_{\mathscr{X}/A}\otimes_{\co_{\mathscr{X}}}\co_{\mathscr{X}/p}\to L_{(\mathscr{X}/p)/A}\to L_{(\mathscr{X}/p)/\mathscr{X}}$$
shows that $\bigwedge^iL_{(\mathscr{X}/p)/A}$ has a bounded filtration with graded pieces
$$\bigwedge^{i-j}(L_{\mathscr{X}/A}\otimes_{\co_{\mathscr{X}}}\co_{\mathscr{X}/p})\otimes_{\co_{\mathscr{X}/p}}\bigwedge^jL_{(\mathscr{X}/p)/\mathscr{X}}\simeq 
\bigwedge^{i-j}L_{\mathscr{X}/A}\otimes_{\co_{\mathscr{X}}}\co_{\mathscr{X}/p}\langle j\rangle[j]$$
for $j=0,\dots,i$. Hence $L\widehat{\Omega}^{i}_{(\mathscr{X}/p)/A} = R\Gamma(\mathscr{X},\bigwedge^iL_{(\mathscr{X}/p)/A})$ has a bounded filtration with graded pieces $R\Gamma(\mathscr{X}, \bigwedge^{i-j}L_{\mathscr{X}/A}\otimes_{\co_{\mathscr{X}}}\co_{\mathscr{X}/p}\langle j\rangle[j])$ in grading weights $j,j+1$ for $j=0,\dots,i$. We conclude that $L\widehat{\Omega}^{i}_{(\mathscr{X}/p)/A}$ is bounded in grading weights $[0,i+1]$. Part b) then follows from the filtration (\ref{nygaardfilt}). Part c) for the first two filtrations follows from the filtration (\ref{addfilt}). The filtration $\Phi^{\geq \star}R\Gamma_{\syn}((\mathscr{X}/p)/A,\bz_p(n))$ is complete by the argument in the proof of Prop. \ref{completeandbounded} a) and bounded in $[0,n]$ because of the isomorphism on associated graded of Prop. \ref{griso}.\end{proof}

Applying Prop. \ref{charp} to $\mathrm{gr}^0_F\mathscr{X}/\mathrm{gr}^0_FA$ we find for $?=\syn,\add$
$$R\Gamma_{?}(\mathrm{gr}^0_F\mathscr{X}/\mathrm{gr}^0_FA,\bz_p(n))^{\rel}[1]\simeq\Phi^{\geq 1}R\Gamma_{?}((\mathrm{gr}^0_F\mathscr{X}/p)/\mathrm{gr}^0_FA,\bz_p(n))$$
where the filtration $\Phi^{\geq\star}$ is bounded in $[0,n]$. The isomorphism (\ref{grf0iso}) on associated graded then follows again from Prop. \ref{griso}.

To show that $G^{[0,k[}\mathrm{gr}^m_F R\Gamma_{\add}(\mathscr{X}/A,\bz_p(n))^{\rel}$ is $p^N$-torsion for suitable $N$, note that the complex $G^{[0,k[}\mathrm{gr}^m_F R\Gamma_{\add}(\mathscr{X}/A,\bz_p(n))$ has a bounded filtration with graded pieces
\begin{equation} \mathrm{gr}^m_F\mathrm{gr}^i_{\mathcal N}\prism^{(1)}_{\mathscr{X}/A}\{n\}[-1],\quad\quad i=0,\dots,\min(n-1,\bigl\lfloor \frac{k-m}{e} \bigr\rfloor)\label{grffilt}\end{equation} 
and $\mathrm{gr}^m_F\mathrm{gr}^i_{\mathcal N}\prism^{(1)}_{\mathscr{X}/A}\{n\}$ is a $\mathrm{gr}^0_F\mathrm{gr}^0_{\mathcal N}\prism^{(1)}_{\mathscr{X}/A}\simeq R\Gamma(\mathscr{X}, \mathrm{gr}^0_F\co_{\mathscr{X}})$-module. This is an $\bF_p$-algebra since $\mathfrak{X}$ is $e$-filtered. The same reasoning applies to $\mathfrak{X}/p$, concluding the proof. For $m=0$ we note in addition that $ \mathrm{gr}^j_\Phi R\Gamma_{\add}((\mathrm{gr}^0_F\mathscr{X}/p)/\mathrm{gr}^0_FA,\bz_p(n))$ has a bounded filtration (\ref{addfilt}) with graded pieces
$$ \mathrm{gr}^j_\Phi\mathrm{gr}^i_{\mathcal N}\prism^{(1)}_{(\mathrm{gr}^0_F\mathscr{X}/p)/\mathrm{gr}^0_FA}\{n\}[-1], \quad \quad i=0,\dots,n-1$$
which are modules over the $\bF_p$-algebra
$$\mathrm{gr}^0_\Phi\mathrm{gr}^0_{\mathcal N}\prism^{(1)}_{(\mathrm{gr}^0_F\mathscr{X}/p)/\mathrm{gr}^0_FA}\simeq
R\Gamma(\mathrm{gr}^0_F\mathscr{X}, \mathrm{gr}^0_\Phi\co_{\mathrm{gr}^0_F\mathscr{X}/p})\simeq
R\Gamma(\mathrm{gr}^0_F\mathscr{X}, \co_{\mathrm{gr}^0_F\mathscr{X}}).$$
This shows that the modules in (\ref{grf0iso}) are $p^N$-torsion for suitable $N$.

c) By Lemma \ref{gammaiso} for $\mathscr{X}$ and $\mathscr{X}/p$ the map $\gamma^{dR,<n,\rel}_{\prism, \mathscr{X}/A}\{n\}[-1]$ is a rational isomorphism.  
So is the map
$$\iota:\widehat{\mathrm{dR}}_{\mathscr{X}/A}^{< n,\rel}[-1]\to\widehat{\mathrm{dR}}_{\mathscr{X}/A}^{< n}[-1]$$
since $\widehat{\mathrm{dR}}_{(\mathscr{X}/p)/A}^{< n}$ has a bounded filtration with graded pieces
$L\widehat{\Omega}^{i}_{(\mathscr{X}/p)/A}[-i]$, $i=0,\dots,n-1$ which are $\bF_p$-modules.
It follows that $\gamma_{\mathscr{X}/A}$ is a rational isomorphism. By part b) and diagram (\ref{synlogdiagram2}) $\log_{\mathscr{X}/A}$ is a rational isomorphism.

\end{proof}

\begin{corollary} (The fundamental fibre square relative to a $\delta$-ring) Let $\mathscr{X}$ be a quasi-compact, quasi-separated derived formal scheme over a $\delta$-ring $A$. Then the square
\begin{equation}
\xymatrix{ 
R\Gamma_{\syn}(\mathscr{X}/A,\bz_p(n)) \ar[r]\ar[d] &  R\Gamma_{\syn}((\mathscr{X}/p)/A,\bz_p(n))\ar[d] \\ 
\fil^{\geq n}_{\mathcal N}\prism^{(1)}_{\mathscr{X}/A}\{n\}\ar[d]_{\fil^{\geq n}\gamma^{dR}_{\prism,\mathscr{X}/A}\{n\}}  & \fil^{\geq n}_{\mathcal N}\prism^{(1)}_{(\mathscr{X}/p)/A}\{n\}\ar[d]^{\beta_{\mathscr{X}/A}\circ\can} \\
\fil^{\geq n}_{Hod}\widehat{\mathrm{dR}}_{\mathscr{X}/A}\ar[r] & \widehat{\mathrm{dR}}_{\mathscr{X}/A}
}
\notag\end{equation}
of Prop. \ref{relativebeilsquare} is rationally Cartesian. In particular Thm. \ref{main} a) holds true.
\label{beilinsonfibresquare}\end{corollary}

\begin{proof} Equipping $A$ with the trivial filtration and $\co_{\mathscr{X}}$ with the $p$-adic filtration  the resulting morphism $\mathfrak{X}\to\mathfrak{A}$ is $1$-filtered by Example \ref{trivexample}. By Thm. \ref{thm:synlog} the map $\log_{\mathscr{X}/A}$ is a rational isomorphism.  Corollary \ref{beilinsonfibresquare} follows and Thm. \ref{main} a) is the special case $A=\bz_p$.
\end{proof}

\begin{example} In the situation of Example \ref{keyexamplectd} the morphism
$$F_\fm\mathscr{X}\to F\Spf(F^{\geq \star}A^s)$$
is $e$-filtered for any $s\geq -1$. Assume $\mathscr{X}=\Spf(R)$ is affine for simplicity. The $A^s=W(k)[[z_0,\dots,z_s]]$ for $s\geq 0$ form a cosimplicial filtered $\delta$-ring which is in fact the Cech conerve of the map of filtered $\delta$-rings $W(k)\to W(k)[[z_0]]$. Hence the cosimplicial filtered $\delta$-pair $(F_\fm^{\geq \star}R, F^{\geq \star}A^\bullet)$ is the Cech conerve of the map of filtered $\delta$-pairs 
$$(F_\fm^{\geq \star}R, F^{\geq \star}_{triv}W(k))\to (F_\fm^{\geq \star}R, F^{\geq \star}A^0).$$
By the proof of \cite{akn24}[Thm. 2.22] if $R$ is quasisyntomic then all terms of (\ref{synlogdiagram2}) satisfy descent for this map of filtered $\delta$-pairs. Functoriality of (\ref{synlogdiagram2}) then gives an isomorphism 
$$ D(F_\fm^{\geq \star}R, F^{\geq \star}_{triv}W(k))\simeq \Tot D(F_\fm^{\geq \star}R, F^{\geq \star}A^\bullet)$$
where $D(\mathfrak{X},\mathfrak{A})$ denotes the diagram $(\ref{synlogdiagram2})$. Such descent computations were the initial motivation for introducing prismatic cohomology relative to a (filtered) $\delta$-ring \cite{akn24}.
\label{keyexamplectdctd}\end{example}

\begin{theorem} \label{thm:synlogproper} Let $\mathfrak{X}\to \widehat{\ba}^1/\widehat{\bg}_m$ be a $e$-filtered morphism of derived formal stacks with underlying derived formal scheme $\mathscr{X}/\bz_p$. Assume that $\mathscr{X}/\bz_p$ is proper of finite tor-amplitude and that $L_{\mathscr{X}/\bz_p}\in\cd(\mathscr{X})$ is perfect.

a) All complexes in (\ref{synlogdiagram2}), (\ref{grfiso}) and (\ref{grf0iso}) are perfect complexes of $\bz_p$-modules.

b) Writing $\mathscr{X}$ for $\mathscr{X}/\bz_p$ we have 
$$\mydet_{\bq_p}(\log'_{\bq_p})\left(\mydet_{\bz_p}R\Gamma_{\syn}(\mathscr{X},\bz_p(n))^{\rel}\right)=\mydet_{\bz_p}R\Gamma_{\add}(\mathscr{X},\bz_p(n))^{\rel}$$
where $\log'_{\bq_p}=\iota_{\add,\bq_p}\circ G^{\geq k}\synlog_{\mathscr{X},\bq_p}^{\rel} \circ\iota^{-1}_{syn,\bq_p}$.

c)  For any $n\geq 1$ we have
\begin{equation} \mydet_{\bq_p}(\gamma_{\mathscr{X},\bq})\left(\mydet_{\bz_p}R\Gamma_{\add}(\mathscr{X},\bz_p(n))^{\rel}\right)=\mydet_{\bq_p}(\iota_\bq)\bigl(\mydet^{-1}_{\bz_p}\widehat{\mathrm{dR}}_\mathscr{X}^{<n,\rel}\bigr) \cdot C_\infty(\mathscr{X},n)^{-1}
\notag\end{equation}
where
\begin{equation}\notag
\label{cinftydef}
C_\infty(\mathscr{X},n):=\prod_{ i\leq n-1;\, j}(n-1-i)!^{(-1)^{i+j}\mathrm{dim}_{\bq_p}H^j(\mathscr{X}_{\bq_p},L\Omega^i)}
\end{equation}
was defined in \cite{Flach-Morin-20}[Eq. (2)].
\end{theorem} 

\begin{proof} a) We first show that $R\Gamma_{\add}(\mathscr{X},\bz_p(n))$ is $\bz_p$-perfect. Using the bounded filtrations (\ref{addfilt}) and (\ref{nygaardfilt}) it suffices to show that $L\widehat{\Omega}^{i}_{\mathscr{X}}$ is $\bz_p$-perfect. This follows from perfectness of $L\bigwedge^i L_{\mathscr{X}}$  over $\co_{\mathscr{X}}$ and the fact that $\mathscr{X}$ is proper of finite tor-amplitude over $\bz_p$. The same reasoning applies to $\mathscr{X}/p$ relative to $\bz_p$, showing that $R\Gamma_{\add}(\mathscr{X},\bz_p(n))^{\rel}$ is $\bz_p$-perfect. An analogous but simpler argument shows that $\widehat{\mathrm{dR}}_{\mathscr{X}/A}^{< n,\rel}$ is $\bz_p$-perfect.

We next show that $G^{[0,k[}\mathrm{gr}^m_F R\Gamma_{\add}(\mathscr{X},\bz_p(n))$ is $\bz_p$-perfect. Using the bounded filtrations (\ref{grffilt}) and (\ref{nygaardfilt}) it suffices to show that $\mathrm{gr}^m_F L\widehat{\Omega}^{i}_{\mathscr{X}}$ is $\bz_p$-perfect. Since $\mathscr{X}/\bz_p$ is proper of finite tor-amplitude it suffices to show that $\mathrm{gr}^m_F L\bigwedge^i L_{\mathscr{X}}$ is a bounded complex of coherent $\co_\mathscr{X}$-modules. This can be checked locally on affines $\Spf(R)$. We know that $F^{\geq\star}L\bigwedge^i L_{\mathscr{X}}$ lies in the thick subcategory of $\cd(F^{\geq\star}R)$ generated by $F^{\geq \star-j}R$, $j\in\bz$. Hence, for any $m$, the complexes $F^{\geq m} L\bigwedge^i L_{\mathscr{X}}$ and $\mathrm{gr}^m_F L\bigwedge^i L_{\mathscr{X}}$  lie in the thick subcategory of $\cd(R)$ generated by the $R$-modules $F^{\geq j}R$, $j\in\bz$. But $F^{\geq j}R\simeq R$ since $\mathfrak{X}$ is $e$-filtered.  The same reasoning shows that $\mathrm{gr}^m_F L\widehat{\Omega}^{i}_{\mathscr{X}/p}$ is $\bz_p$-perfect, hence so is
$G^{[0,k[}\mathrm{gr}^m_F R\Gamma_{\add}(\mathscr{X},\bz_p(n))^{\rel}$.

We next show that  
$$\mathrm{gr}^j_\Phi\mathrm{gr}^0_F R\Gamma_{\add}(\mathscr{X},\bz_p(n))^{\rel}\simeq \mathrm{gr}^j_\Phi R\Gamma_{\add}((\mathrm{gr}^0_F\mathscr{X})/p,\bz_p(n))$$ 
is $\bz_p$-perfect. Using the bounded filtrations (\ref{addfilt}) and (\ref{nygaardfilt}) it suffices to prove that $\mathrm{gr}^j_\Phi L\widehat{\Omega}^{i}_{(\mathrm{gr}^0_F\mathscr{X})/p}$ is $\bz_p$-perfect. It was shown in the proof of Prop. \ref{charp} b) that $L\widehat{\Omega}^{i}_{(\mathrm{gr}^0_F\mathscr{X})/p}$ has a bounded filtration with graded pieces 
$$R\Gamma(\mathscr{X}, \bigwedge^{i-j}L_{ \mathrm{gr}^0_F\mathscr{X}         }\otimes_{\co_{\mathrm{gr}^0_F\mathscr{X} }}\co_{(\mathrm{gr}^0_F\mathscr{X})/p}\langle j\rangle[j])$$ 
in $\Phi$-grading weights $j,j+1$ for $j=0,\dots,i$. Since $\mathrm{gr}^{j+\varepsilon}_\Phi\co_{(\mathrm{gr}^0_F\mathscr{X})/p}\langle j\rangle\simeq \co_{\mathrm{gr}^0_F\mathscr{X} }[\varepsilon]$ for $\varepsilon=0,1$ it suffices to show that $R\Gamma(\mathscr{X}, \bigwedge^{i-j}L_{ \mathrm{gr}^0_F\mathscr{X}})$ is $\bz_p$-perfect. Since $\mathscr{X}$ is proper this reduces to perfectness of $L_{ \mathrm{gr}^0_F\mathscr{X}}$ which follows from the fact that  $\mathrm{gr}^0_F\mathscr{X}\to \mathscr{X}$ is a regular closed embedding in the derived sense given by the generalized Cartier divisor $\co_{\mathscr{X}}\xrightarrow{\varpi}\co_{\mathscr{X}}$.

Having shown that (\ref{grfiso}), (\ref{grf0iso}) and $R\Gamma_{\add}(\mathscr{X},\bz_p(n))^{\rel}$ are $\bz_p$-perfect, it follows from Thm. \ref{thm:synlog} b) that $G^{\geq k}R\Gamma_{\add}(\mathscr{X},\bz_p(n))^{\rel}$ is $\bz_p$-perfect, and that $G^{\geq k}R\Gamma_{\syn}(\mathscr{X},\bz_p(n))^{\rel}$ is $\bz_p$-perfect if and only if $R\Gamma_{\syn}(\mathscr{X},\bz_p(n))^{\rel}$ is. Since $G^{\geq k}\synlog^{\rel}_{\mathscr{X}}$ is an isomorphism both complexes are indeed $\bz_p$-perfect.

Perfectness of $G^{\geq k-\epsilon}\widehat{\mathrm{dR}}_{\mathscr{X}/A}^{< n,\rel}$ follows from that of $G^{[0,k-\epsilon[}\widehat{\mathrm{dR}}_{\mathscr{X}/A}^{< n,\rel}$ which has a bounded $F$-filtration whose graded are perfect by the facts established above.

b) It follows from Thm. \ref{thm:synlog} b) and part a) that $\Cone(\iota_?)$ is a perfect complex of $\bz_p$-modules with torsion cohomology and that
$$ \mydet_{\bz_p}\Cone(\iota_{\syn})=\mydet_{\bz_p}\Cone(\iota_{\add})$$
inside $\mydet_{\bq_p}\Cone(\iota_{\syn})_{\bq_p}=\bq_p=\mydet_{\bq_p}\Cone(\iota_{\add})_{\bq_p}$. This implies b).

c) Denote by
$$ \chi^\times(C):=\prod_{i\in\bz} |H^i(C)|^{(-1)^i}$$
the multiplicative Euler characteristic of a perfect complex of $\bz_p$-modules $C$  with torsion cohomology. We abbreviate
$\gamma^{dR,<n}_{\prism,\mathscr{X}}:= \gamma^{dR,<n}_{\prism, \mathscr{X}/\bz_p}\{n\}[-1]$ for the map in Lemma \ref{gammaiso} (with $A=\bz_p$ and $\star=0$). For any $n\geq 1$ the filtrations (\ref{addfilt}) and (\ref{nygaardfilt}) give
\begin{align*}  \chi^\times(\fibre( \gamma^{dR,<n}_{\prism,\mathscr{X}} )) = & \prod_{i=0}^{n-1}\chi^\times\fibre\left(\mathrm{gr}^i_{\mathcal N}\prism^{(1)}_{\mathscr{X}}\{n\}[-1]\to \mathrm{gr}^i_{Hod}\widehat{\mathrm{dR}}_{\mathscr{X}}[-1]  \right)\\
=& \prod_{i=1}^{n-1} \prod_{j=1}^{i} \chi^\times\fibre\left(L\widehat{\Omega}^{i-j}_{\mathscr{X}}\xrightarrow{j}L\widehat{\Omega}^{i-j}_{\mathscr{X}}\right)^{-(-1)^{j-i}}\\
=& \prod_{i=1}^{n-1}\prod_{j=1}^{i}j^{(-1)^{j-i}\chi(\mathscr{X}_{\bq_p},i-j)}\\
=&\prod_{i=0}^{n-2}(n-1-i)!^{(-1)^i\chi(\mathscr{X}_{\bq_p},i)}=C_\infty(\mathscr{X},n)
\end{align*} 
where we denote by 
$$ \chi(\mathscr{X}_{\bq_p},i):=\sum_{j\in\bz}(-1)^j\dim_{\bq_p}(\pi_{-j}L\widehat{\Omega}_{\mathscr{X}}^i)_{\bq_p}=\sum_{j\in\bz}(-1)^j\dim_{\bq_p}H^j(\mathscr{X}_{\bq_p},L\Omega_{\mathscr{X}_{\bq_p}}^i)$$
the rational Euler characteristic of derived Hodge cohomology (using properness of $\mathscr{X}/\bz_p$ for the second identity). We also use the fact that for any perfect complex of $\bz_p$-modules $C$ and integer $j$ one has
$$ \chi^\times\fibre (C\xrightarrow{j} C) = j^{-\chi(C_\bq)} $$
where $\chi(C_\bq)=\sum_{a\in\bz}(-1)^a\dim_{\bq_p} H^a(C_\bq)$ is the rational Euler characteristic of $C$. 

Since the groups $\pi_{-j}L\widehat{\Omega}_{\mathscr{X}/p}^i$ are torsion they have trivial rational Euler characteristic. The above computation then shows that $\chi^\times(\fibre(\gamma^{dR,<n}_{\prism,\mathscr{X}/p} ))=1$ and hence
$$ \chi^\times(\fibre(\gamma^{dR,<n,\rel}_{\prism,\mathscr{X}}))=C_\infty(\mathscr{X},n).$$
The identity in Thm. \ref{thm:synlogproper} c) then follows from the fact that
$$ \mydet_{\bz_p}C=\chi^\times(C)^{-1}\cdot\bz_p\subseteq \bq_p=\mydet_{\bq_p}C_{\bq_p} $$
for a perfect complex of $\bz_p$-modules $C$  with torsion cohomology.
\end{proof}

The following result was proven in two different ways in \cite{akn24}[Prop. 2.18, Rem. 4.66]. Here we add a third proof, a variant of the proof of Thm. \ref{thm:synlogproper}. 

\begin{prop} (Angeltveit quotient) Let $K/\bq_p$ be a finite extension with residue field $\kappa$ of cardinality $q$ and put $R=\co_K/\varpi^k$. Then for $n\geq 1$
$$ \chi^\times R\Gamma_{\syn}(R/\bz_p,\bz_p(n))=q^{-n(k-1)}$$
\label{angeltveit}\end{prop}

\begin{proof} The transitivity triangle for the cotangent complex associated to $\bz_p\to\co_K\to R$ shows that $L\widehat{\Omega}^{i}_{R/\bz_p}$ has a finite filtration with graded pieces
\begin{align*} &\bigwedge^jL_{\co_K/\bz_p}\otimes_{\co_K}\bigwedge^{i-j}L_{R/\co_K}\simeq  \bigwedge^jL_{\co_K/\bz_p}\otimes_{\co_K}\bigwedge^{i-j} (R[1])\\
\simeq  &\bigwedge^jL_{\co_K/\bz_p}\otimes_{\co_K}(\Gamma^{i-j} R)[i-j]\simeq  \bigwedge^jL_{\co_K/\bz_p}\otimes_{\co_K}R[i-j] 
\end{align*}
for $j=0,\dots,i$. By \cite{Flach-Morin-16}[Prop. 5.36] for $j>0$ the $\co_K$-complex $\bigwedge^jL_{\co_K/\bz_p}$ is concentrated in degree $j-1$ with finite cohomology.  For a perfect complex of $\co_K$-modules $F$ with finite cohomology we have  
\begin{equation} \chi^\times(F\otimes_{\co_K}^LR)=\chi^\times\Cone(F\xrightarrow{\varpi^k}F)=\chi^\times(F)\chi^\times(F)^{-1}=1.\label{fincoh}\end{equation}
Therefore
$$\chi^\times (L\widehat{\Omega}^{i}_{R/\bz_p}) =  \chi^\times\bigwedge^0L_{\co_K/\bz_p}\otimes_{\co_K}R[i]=|R|^{(-1)^i}.$$
By (\ref{addfilt}) and (\ref{nygaardfilt}) for $\star=0$ we have 
$$\chi^\times R\Gamma_{\add}(R/\bz_p,\bz_p(n))^{-1}=\prod_{i=0}^{n-1}\chi^\times (L\widehat{\Omega}^{i}_{R/\bz_p}[-i])=|R|^n=q^{kn}$$
where we have also used an argument analogous to (\ref{fincoh}) in (\ref{nygaardfilt}). Now consider $R$ with its $\varpi$-adic filtration and apply Prop. \ref{griso} and Cor. \ref{globalfiltrations} c). For $n\geq 1$ we find 
\begin{align*} &\chi^\times R\Gamma_{\syn}(R/\bz_p,\bz_p(n))=\prod_{j\geq 0} \chi^\times \mathrm{gr}_F^j R\Gamma_{\syn}(R/\bz_p,\bz_p(n))\\
=&\chi^\times R\Gamma_{\syn}(\kappa/\bz_p,\bz_p(n))\cdot\prod_{j\geq 1} \chi^\times \mathrm{gr}_F^j R\Gamma_{\add}(R/\bz_p,\bz_p(n))\\
=&\chi^\times R\Gamma_{\add}(\kappa/\bz_p,\bz_p(n))^{-1}\cdot \chi^\times R\Gamma_{\add}(R/\bz_p,\bz_p(n))=q^{n}\cdot q^{-kn}=q^{-n(k-1)}.
\end{align*}
\end{proof}

\section{The de Rham logarithm for small Tate twist}\label{smalltwist} 

In this section we show the following generalization of Theorem \ref{main} b). 

\begin{theorem} Let $\mathscr{X}$ be a quasi-compact, quasi-separated derived formal scheme over a $\delta$-ring $A$. For $0<n<p-1$ the sequence 
$$ R\Gamma_{\syn}(\mathscr{X}/A,\bz_p(n))^{\rel} \xrightarrow{\log_{\mathscr{X}/A}}  \widehat{\mathrm{dR}}_{\mathscr{X}/A}^{<n}[-1]\xrightarrow{} \widehat{\mathrm{dR}}_{(\mathscr{X}/p)/A}^{<n}[-1]$$
is a fibre sequence.
\label{smalln}\end{theorem}

Our proof uses a generalization of Fontaine-Messing syntomic cohomology to the relative situation. In the absolute case ($A=\bz_p$) Fontaine-Messing syntomic cohomology was defined in \cite{antieau-mmn21}[Def. 6.9] for quasisyntomic $R/\bz_p$ and already in \cite{fm87}, \cite{kato87} for smooth $R/\bz_p$. Theorem \ref{smalln} was proven for quasisyntomic $\mathscr{X}/\bz_p$ in \cite{antieau-mmn21}[Thm. 6.17] by homotopy theoretic techniques and then used to prove that syntomic cohomology and Fontaine-Messing syntomic cohomology are isomorphic \cite{antieau-mmn21}[Thm. 6.22]. Here we go the opposite route and first give a direct argument for the latter isomorphism and then deduce Theorem \ref{smalln} as a consequence.

\begin{remark} Theorem \ref{smalln} is equivalent to the existence of an isomorphism
$$ \log'_{\mathscr{X}/A}:R\Gamma_{\syn}(\mathscr{X}/A,\bz_p(n))^{\rel}\simeq \widehat{\mathrm{dR}}^{<n,\rel}_{\mathscr{X}/A}[-1]$$
over $\widehat{\mathrm{dR}}^{<n}_{\mathscr{X}/A}[-1]$ for $0<n<p-1$. By Lemma \ref{gammaiso} there is then also an isomorphism
$$\bigl(\gamma^{dR,<n,\rel}_{\prism,\mathscr{X}/A}\{n\}[-1]\bigr)^{-1}\circ \log'_{\mathscr{X}/A}:R\Gamma_{\syn}(\mathscr{X}/A,\bz_p(n))^{\rel}\simeq R\Gamma_{\add}(\mathscr{X}/A,\bz_p(n))^{\rel}$$
over $\widehat{\mathrm{dR}}^{<n}_{\mathscr{X}/A}[-1]$ for $0<n<p-1$. It is tempting to try to use the syntomic logarithm to construct this isomorphism. We cannot choose $k=0$ in Thm. \ref{thm:synlog} in order for the maps $\iota_{\syn}$ and $\iota_{\add}$ in (\ref{synlogdiagram2}) to be isomorphisms. We may instead use the $p$-adic filtration ($e=1$) in which case $R\Gamma_{?}(\mathscr{X}/A,\bz_p(n))^{\rel}\simeq F^{\geq 1}R\Gamma_{?}(\mathscr{X}/A,\bz_p(n))$ and see if we can apply Cor. \ref{relativelogdef} with a value of $k_1$ so that $G^{\geq k_1}=F^{\geq 1}$. This works precisely for $n=1$ and $p\geq 3$ in which case $k_1=1$, the $F$- and $G$-filtration coincide and we recover the construction of the logarithm described in the introduction.
\end{remark}

\subsection{Fontaine-Messing syntomic cohomology relative to a $\delta$-ring} \label{FM} Given the discussion in \cite{akn23}[Sec. 9] we can define Fontaine-Messing syntomic cohomology relative to a $\delta$-ring for quasisyntomic $R/A$ by quasisyntomic descent, as in the absolute case \cite{antieau-mmn21}[Def. 6.9]. We then generalize to  all $R/A$ by left Kan extension.

\begin{definition} ( \cite{akn23}[Def. 9.2]) A $\delta$-pair $(R,A)$ with $R$ $p$-complete is called \underline{relatively} \underline{quasiregular semiperfectoid} (relatively qrsp) if there is a factorization $A\to A'/J\simeq R$ where $A\to A'$ is a $p$-completely relatively perfect map of $\delta$-rings and such that there exists $I\subseteq J$ turning $(A')_{(p,I)}^\wedge$ into a prism with $L_{R/(A'/I)}$ having $p$-complete tor-amplitude in $[-1,-1]$.
\end{definition}

Note that since $I$ defines a prism structure it is locally generated by a regular element. This implies that $L_{R/A'}$ also has $p$-complete tor-amplitude in $[-1,-1]$, i.e. $J$ is locally generated by a (possible infinite) $p$-completely regular sequence.

\begin{lemma} Assume $(R,A)$ is a $p$-complete $\delta$-pair with $A$ $\bz_p$-flat. Then $L_{R/A}$ has $p$-complete tor-amplitude in $[-1,0]$ if and only if there exists a quasisyntomic faithfully flat $R\to R'=A'/J$ with $(R',A)$ relatively qrsp such that $A'$ is $\bz_p$-flat and an algebra over $A'_0:=A[z]^\wedge_p$ with $\delta(z)=0$ and $z-p\in J$.
\label{qrsplemma}\end{lemma}

\begin{proof} Follow the proof of \cite{akn23}[Lemma 9.13].
\end{proof}

\begin{prop} \label{filteredbeta} For a derived formal scheme $\mathscr{X}$ over a $\delta$-ring $A$ the commutative diagram of Prop. \ref{betaprop}
\begin{equation}\xymatrix{ 
  \prism^{(1)}_{\mathscr{X}/A}\{n\}\ar[rr]^{i^*} \ar[d]_{\gamma^{dR}_{\prism,{\mathscr{X}/A}}\{n\}}& &\prism^{(1)}_{(\mathscr{X}/p)/A}\{n\}\\
 \widehat{\mathrm{dR}}_{\mathscr{X}/A}\ar[urr]^{\beta_{\mathscr{X}/A}^{-1}}_{\sim}&&
}\notag\end{equation}
 extends to a commutative diagram of bounded filtrations indexed by $0\leq k\leq p-1$
\begin{equation}\xymatrix{ 
 \fil^{\geq k}_{\mathcal N} \prism^{(1)}_{\mathscr{X}/A}\{n\}\ar[rr]^{\fil^{\geq k}_{\mathcal N}(i^*)} \ar[d]_{\fil^{\geq k}\gamma^{dR}_{\prism,{\mathscr{X}/A}}\{n\}}& &\fil^{\geq k}_{\mathcal N}\prism^{(1)}_{(\mathscr{X}/p)/A}\{n\}\\
\fil^{\geq k}_{Hod} \widehat{\mathrm{dR}}_{\mathscr{X}/A}\ar[urr]^{\fil^{\geq k}\beta_{\mathscr{X}/A}^{-1}}&&
}\label{extendedbeta}\end{equation}
functorially in $\mathscr{X}/A$.
\end{prop}

\begin{proof} By Zariski descent we can assume $\mathscr{X}=\Spf(R)$ is affine. Then all terms in (\ref{extendedbeta}) commute with sifted colimits of $\delta$-pairs $(R,A)$ where we view $\fil^{\geq k}_{\mathcal N} \prism^{(1)}_{R/A}\{n\}$ and $\fil^{\geq k}_{\mathcal N}\prism^{(1)}_{(R/p)/A}\{n\}$ (endowed with the complete filtration of Lemma \ref{filtration} c)) as taking values in complete filtrations. In particular all terms in (\ref{extendedbeta}) are left Kan extended from free $\delta$-pairs $(R,A)$ where $A$ is a free $\delta$-ring and $R$ is a polynomial $A$-algebra. Hence we can assume that both $R$ and $A$ are $\bz_p$-flat and $R/A$ is quasisyntomic. By  quasisyntomic descent (Thm. \ref{packageproperties} c1)) together with \cite{akn23}[Prop. 9.14] we can further assume that $A\to R=A'/J$ is relatively qrsp (and $R$ and $A'$ $\bz_p$-flat by Lemma \ref{qrsplemma}).  In this case 
$$  \widehat{\mathrm{dR}}_{R/A}\simeq \widehat{\mathrm{dR}}_{R/A'}\simeq D_J(A')$$
is the $p$-complete divided power envelope of $J$ and $\fil^{\geq k}_{Hod} \widehat{\mathrm{dR}}_{R/A}\simeq D^{(k)}_J(A')$ is the $k$-th divided power ideal. In particular
$$D_J(A')\simeq \prism^{(1)}_{(R/p)/A'}\{n\}\simeq \prism^{(1)}_{(R/p)/A}\{n\}$$
is discrete and the Nygaard filtration is given by the pullback of the $p$-adic filtration under the (injective) relative Frobenius (see the proof of \cite{akn24}[Lemma 3.14])
$$\varphi_{(R/p)/A'}:\prism^{(1)}_{(R/p)/A'}\{n\}\simeq A'\left\{\frac{\varphi(J)}{p}\right\}_{p}^\wedge\subseteq A'\left\{\frac{J}{p}\right\}_{p}^\wedge\simeq\prism_{(R/p)/A'}\{n\}.$$
Since both the Hodge and the Nygaard filtrations are strict, the existence (and uniqueness and functoriality on relatively qrsp $R/A$) of $\fil^{\geq k}\beta_{R/A}^{-1}$ amounts to the inclusion $D^{(k)}_J(A')\subseteq \fil^{\geq k}_{\mathcal N}\prism^{(1)}_{(R/p)/A'}\{n\}$. This is equivalent to the inclusion
\begin{equation}D^{(k)}_J(A')=\left\langle \frac{\xi^j}{j!}\,\,\vert\,\, \xi\in J, j\geq k\right\rangle_p^\wedge\subseteq p^kA'\left\{\frac{J}{p}\right\}_{p}^\wedge.\label{divpowers}\end{equation}
For $\xi\in J$ we have
$$ \frac{\xi^j}{j!}=\frac{p^j}{j!}\left(\frac{\xi}{p}\right)^j$$
and $\frac{p^j}{j!}$ has $p$-adic valuation at least $p-1$ for $j\geq p$ and valuation $j$ for $j\leq p-1$. This implies (\ref{divpowers}) for $k\leq p-1$.
\end{proof}

In the following we abbreviate $\varphi\{n\}$ for $c^0\varphi\{n\}$.

\begin{definition} For $0\leq n\leq p-1$ define
$$ \varphi\{n\}_{dR}:=\beta_{\mathscr{X}/A}\circ\varphi\{n\}\circ \fil^{\geq n}\beta_{\mathscr{X}/A}^{-1}:\fil^{\geq n}_{Hod} \widehat{\mathrm{dR}}_{\mathscr{X}/A}\to \widehat{\mathrm{dR}}_{\mathscr{X}/A}$$
and define Fontaine-Messing syntomic cohomology by the fibre sequence
$$ R\Gamma^{\mathrm{FM}}_{\syn}(\mathscr{X}/A,\bz_p(n)) \to \fil^{\geq n}_{Hod}\widehat{\mathrm{dR}}_{\mathscr{X}/A}\xrightarrow{\can- \varphi\{n\}_{dR} } \widehat{\mathrm{dR}}_{\mathscr{X}/A}.$$
\end{definition}

\begin{prop} a)  For $0\leq n\leq p-1$ there is a commutative diagram with exact rows
\begin{equation}\xymatrix@C+15pt{ 
R\Gamma_{\syn}(\mathscr{X}/A,\bz_p(n)) \ar[r]\ar[d]_\kappa &\fil^{\geq n}_{\mathcal N}\prism^{(1)}_{\mathscr{X}/A}\{n\} \ar[d]_{\fil^{\geq n}\gamma^{dR}_{\prism,\mathscr{X}/A}\{n\}}\ar[r]^{\can-\varphi\{n\}} & \prism^{(1)}_{\mathscr{X}/A}\{n\} \ar[d]_{\gamma^{dR}_{\prism,\mathscr{X}/A}\{n\}} \\
R\Gamma^{\mathrm{FM}}_{\syn}(\mathscr{X}/A,\bz_p(n)) \ar[r] \ar[d]_\sigma&\fil^{\geq n}_{Hod}\widehat{\mathrm{dR}}_{\mathscr{X}/A}\ar[r]^{\can- \varphi\{n\}_{dR}} \ar[d]_{\fil^{\geq n}\beta_{\mathscr{X}/A}^{-1}}  & \widehat{\mathrm{dR}}_{\mathscr{X}/A}\ar[d]_{\beta_{\mathscr{X}/A}^{-1}}^{\sim}\\
R\Gamma_{\syn}((\mathscr{X}/p)/A,\bz_p(n)) \ar[r] &\fil^{\geq n}_{\mathcal N}\prism^{(1)}_{(\mathscr{X}/p)/A}\{n\} \ar[r]^(.55){\can-\varphi\{n\}} & \prism^{(1)}_{(\mathscr{X}/p)/A}\{n\} }\notag\end{equation}

b) The map $\kappa$ is an isomorphism for $0\leq n\leq p-2$.
\end{prop}

\begin{proof} a) We claim that the right hand squares commute separately for the $\can$-maps and the $\varphi\{n\}$-maps. For the $\can$-maps this is clear for the top square and the content of Prop. \ref{filteredbeta} for the bottom square. For the $\varphi\{n\}$-maps the bottom square commutes by definition of $\varphi\{n\}_{dR}$. The top square commutes since 
\begin{align*} \varphi\{n\}_{dR}\circ\fil^{\geq n}\gamma^{dR}_{\prism,\mathscr{X}/A}\{n\}=&\beta_{\mathscr{X}/A}\circ\varphi\{n\}\circ \fil^{\geq n}\beta_{\mathscr{X}/A}^{-1}\circ\fil^{\geq n}\gamma^{dR}_{\prism,\mathscr{X}/A}\{n\}\\
=&\beta_{\mathscr{X}/A}\circ\varphi\{n\}\circ \fil^{\geq n}_{\mathcal N}(i^*) \\
=&\beta_{\mathscr{X}/A}\circ i^*\circ\varphi\{n\} = \gamma^{dR}_{\prism,\mathscr{X}/A}\{n\}\circ\varphi\{n\}\end{align*}
by Prop. \ref{filteredbeta}.

b) By Lemma \ref{gammaiso} the fibres of the $\can$-maps in the top right hand square are isomorphic for $n\leq p$. Hence $\can$ induces an isomorphism
\begin{equation}\can:\fibre(\fil^{\geq n}\gamma^{dR}_{\prism,\mathscr{X}/A}\{n\})\simeq \fibre(\gamma^{dR}_{\prism,\mathscr{X}/A}\{n\}).\label{cangamma}\end{equation}
For $0\leq n\leq p-2$ we claim that there exists a morphism
$$\tilde{\varphi}\{n+1\}:\fibre(\fil^{\geq n}\gamma^{dR}_{\prism,\mathscr{X}/A}\{n\})\to\fibre(\gamma^{dR}_{\prism,\mathscr{X}/A}\{n\})$$
such that $\varphi\{n\}=p\tilde{\varphi}\{n+1\}$. By $p$-completeness this implies that
$$\can-\varphi\{n\}: \fibre(\fil^{\geq n}\gamma^{dR}_{\prism,\mathscr{X}/A}\{n\})\to\fibre(\gamma^{dR}_{\prism,\mathscr{X}/A}\{n\})$$
is an isomorphism and therefore that $\kappa$ is an isomorphism.

To verify the claim we follow the proof of Prop. \ref{filteredbeta} and reduce to relatively qrsp $R/A$ with $R=A'/J$ and $R$ and $A'$ $\bz_p$-flat. In this case
$$\prism_{R/A}\{n\}\simeq\prism_{R/A'}\{n\}\simeq A'\{n\}\left\{\frac{J}{I}\right\}_{(p,I)}^\wedge$$
is discrete, isomorphic to a prismatic envelope \cite{bhatt-scholze}[Example 7.9] and moreover $(p,I)$-completely flat over $A'$ \cite{bhatt-scholze}[Prop. 3.13]. Hence
$$\prism^{(1)}_{R/A'}\{n\}=(\prism_{R/A'}\{n\}\otimes_{A',\varphi}A')^\wedge_{(p,I)}\simeq A'\{n\}\left\{\frac{\varphi(J)}{\varphi(I)}\right\}_{(p,I)}^\wedge$$
is also discrete and the Nygaard filtration is given by the pullback of the $I$-adic filtration under the relative Frobenius
$$\varphi_{R/A'}:\prism^{(1)}_{R/A'}\{n\}\simeq A'\{n\}\left\{\frac{\varphi(J)}{\varphi(I)}\right\}_{(p,I)}^\wedge\subseteq  A'\{n\}\left\{\frac{J}{I}\right\}_{(p,I)}^\wedge \simeq\prism_{R/A'}\{n\},$$
in particular is strict. The map
$$ \gamma^{dR}_{\prism,R/A'}\{n\}:\prism^{(1)}_{R/A'}\{n\}\to\widehat{\mathrm{dR}}_{R/A'}\simeq\prism^{(1)}_{(R/p)/A'}\{n\}$$
is injective since it is isomorphic to the inclusion 
$$A'\{n\}\left\{\frac{\varphi(J)}{\varphi(I)}\right\}_{(p,I)}^\wedge\subseteq A'\{n\}\left\{\frac{\varphi(J)}{\varphi(I)},\frac{p}{\varphi(I)}\right\}_{(p,I)}^\wedge\simeq 
A'\{n\}\left\{\frac{\varphi(J)}{p}\right\}_{(p)}^\wedge.$$
Since the Hodge filtration was shown to be strict in the proof of Prop. \ref{filteredbeta} it follows that $\fil^{\geq n}\gamma^{dR}_{\prism,R/A'}\{n\}$ is also injective. The isomorphism (\ref{cangamma}) becomes the $[-1]$-shift of the isomorphism 
\begin{equation} \frac{D^{(n)}_J(A')}{\fil^{\geq n}_{\mathcal N}\prism^{(1)}_{R/A'}\{n\}}\simeq \frac{D_J(A')}{\prism^{(1)}_{R/A'}\{n\}}\simeq \frac{\prism^{(1)}_{(R/p)/A'}\{n\}}{\prism^{(1)}_{R/A'}\{n\}}\underset{\sim}{\xleftarrow{\cdot s^n}}\frac{\prism^{(1)}_{(R/p)/A'}}{\prism^{(1)}_{R/A'}}\label{quotient}\end{equation}
induced by the inclusions. Here the last isomorphism is induced by the choice of a Breuil-Kisin orientation of the orientable prism $(A')^\wedge_{(p,I)}$
 \cite{bhatt-lurie-22}[Rem. 2.5.8], \cite{akn24}[Lemma 4.7]. Moreover by Lemma \ref{qrsplemma}, after choosing $s$ on the orientable prism $((A'_0)^\wedge_{(p,z)},(z-p))$ we obtain a functorial choice on the qrsp-site relative to $A$.  By \cite{akn24}[Lemma 4.6] a Breuil-Kisin orientation $s$ determines an orientation $(d_s)=I=(z-p)$ of $(A')^\wedge_{(p,I)}$.  By \cite{akn24}[Def. 4.22] the map
$\varphi\{n\}$ on $\fil^{\geq n}_{\mathcal N}\prism^{(1)}_{(R/p)/A'}\{n\}$ is given by
$\varphi\{n\}(x\cdot s^n)= \varphi(d_s)^{-n}\varphi(x)s^n$ where $\varphi$ is the Frobenius of $\prism^{(1)}_{(R/p)/A'}$. Hence on the quotient (\ref{quotient}) we have
$\varphi\{n\}=\varphi(d_s)\varphi\{n+1\}s^{-1}$ for $0\leq n\leq p-2$. It remains to remark that $d_s\in J$ and therefore $d_s$ becomes divisible by $p$ in the prismatic envelope $A'\left\{\frac{J}{p}\right\}_{p}^\wedge \simeq\prism_{(R/p)/A'}$. Hence we may define $\tilde{\varphi}\{n+1\}:=\varphi(d_s/p)\varphi\{n+1\}s^{-1}$ functorially on the qrsp-site relative to $A$.
\end{proof}

\subsection{The proof of Theorem \ref{smalln}} We insert Fontaine-Messing syntomic cohomology into the fundamental square and obtain the following commutative diagram
\begin{equation}\xymatrix@C-14pt{ 
R\Gamma_{\syn}(\mathscr{X}/A,\bz_p(n)) \ar[r]^\kappa\ar[d] &R\Gamma^{\mathrm{FM}}_{\syn}(\mathscr{X}/A,\bz_p(n))\ar[r]^(.45){\sigma}\ar[d] & R\Gamma_{\syn}((\mathscr{X}/p)/A,\bz_p(n))\ar[d] & \\ 
\fil^{\geq n}_{\mathcal N}\prism^{(1)}_{\mathscr{X}/A}\{n\}\ar[dd]^{\fil^{\geq n}\gamma^{dR}_{\prism,\mathscr{X}/A}\{n\}}  \ar[r]& 
\fil^{\geq n}_{Hod}\widehat{\mathrm{dR}}_{\mathscr{X}/A}\ar[r]^(.4){\fil^{\geq n}\beta_{\mathscr{X}/A}^{-1}}\ar@{=}[dd]&\fil^{\geq n}_{\mathcal N}\prism^{(1)}_{(\mathscr{X}/p)/A}\{n\}\ar[d]_{\can'} \ar[r]^{\tilde{\gamma}^n}&\fil^{\geq n}_{Hod}\widehat{\mathrm{dR}}_{(\mathscr{X}/p)/A}\ar[d]_{\can_{dR}'}\\
 && \prism^{(1)}_{(\mathscr{X}/p)/A}\{n\} \ar[d]_{\beta_{\mathscr{X}/A}}^{\sim}\ar[r]^{  \tilde{\gamma}^0     }&\widehat{\mathrm{dR}}_{(\mathscr{X}/p)/A}\\
\fil^{\geq n}_{Hod}\widehat{\mathrm{dR}}_{\mathscr{X}/A}\ar@{=}[r]  &\fil^{\geq n}_{Hod}\widehat{\mathrm{dR}}_{\mathscr{X}/A}\ar[r]^{\can_{dR}}& \widehat{\mathrm{dR}}_{\mathscr{X}/A}& 
}
\notag\end{equation}
where we abbreviate $\tilde{\gamma}^n:=\fil^{\geq n}\gamma^{dR}_{\prism,(\mathscr{X}/p)/A}\{n\}$.  This diagram induces the following sequence of isomorphisms over $\fibre(\can_{dR})\simeq \widehat{\mathrm{dR}}^{<n}_{\mathscr{X}/A}[-1]$ for $0\leq n\leq p-2$
\begin{align} 
&R\Gamma_{\syn}(\mathscr{X}/A,\bz_p(n))^{\rel}\label{chain}\\
\simeq& \fibre{\sigma}\simeq\fibre({\fil^{\geq n}\beta_{\mathscr{X}/A}^{-1}})\notag\\
\simeq &\fibre\left(\fibre(\can_{dR})\xrightarrow{\fibre(\beta_{\mathscr{X}/A}^{-1},\fil^{\geq n}\beta_{\mathscr{X}/A}^{-1})}\fibre(\can')\right)\notag\\
\simeq &\fibre\left(\fibre(\can_{dR})\xrightarrow{\fibre(\tilde{\gamma}^0\circ\beta_{\mathscr{X}/A}^{-1},\tilde{\gamma}^n\circ\fil^{\geq n}\beta_{\mathscr{X}/A}^{-1})}\fibre(\can'_{dR})\right)\notag
\end{align}
where we have used Lemma \ref{gammaiso} for this last isomorphism. 

\begin{lemma} Denote by $i:\mathscr{X}/p\to\mathscr{X}$ the natural map. For $0\leq n\leq p$ there exists a commutative diagram
$$\xymatrix{ \fil^{\geq n}_{Hod}\widehat{\mathrm{dR}}_{(\mathscr{X}/p)/A} \ar[d]^{\can'_{dR}} \ar[r]^{\tau^n}_{\sim}&\fil^{\geq n}_{Hod}\widehat{\mathrm{dR}}_{(\mathscr{X}/p)/A}\ar[d]^{\can'_{dR}}\\
\widehat{\mathrm{dR}}_{(\mathscr{X}/p)/A} \ar[r]^{\tau^0}_\sim&\widehat{\mathrm{dR}}_{(\mathscr{X}/p)/A}
}$$
where $\tau^n$ is an automorphism such that $i^*=\tau^n\circ \tilde{\gamma}^n \circ \fil^{\geq n}\beta_{\mathscr{X}/A}^{-1}$. Moreover
$$\tau^0:\widehat{\mathrm{dR}}_{(\mathscr{X}/p)/A}\xrightarrow{\beta_{(\mathscr{X}/p)/A}^{-1}}\prism^{(1)}_{(\mathscr{X}/p)/p/A}\{n\}\xrightarrow{\mathrm{sw}^*}\prism^{(1)}_{(\mathscr{X}/p)/p/A}\{n\}\xrightarrow{\beta_{(\mathscr{X}/p)/A}}\widehat{\mathrm{dR}}_{(\mathscr{X}/p)/A}$$
is induced by the automorphism $\mathrm{sw}$ of $(\mathscr{X}/p)/p\simeq\mathscr{X}\otimes\bF_p\otimes\bF_p$ switching the two factors $\bF_p$.
\label{swlemma}\end{lemma}

\begin{proof} We first prove the statement about $\tau^0$. For clarity of notation we assume $\mathscr{X}=\Spf(R)$ is affine. For $i=1,2$ let $\psi_i:R\otimes\bF_p\to R\otimes\bF_p\otimes\bF_p$ be the morphism of animated rings sending $\bF_p$ to the first and second factor, respectively. Then we have a commutative diagram
$$\xymatrix{
\prism^{(1)}_{R\otimes\bF_p /A}\{n\}\ar[d]^{\beta_{R/A}}\ar[r]^{\psi_{2,*}} & \prism^{(1)}_{R\otimes\bF_p\otimes\bF_p /A}\{n\}\ar[d]^{\beta_{R\otimes\bF_p/A}} \\ 
 \widehat{\mathrm{dR}}_{R/A}\ar[r]^{i^*}  & \widehat{\mathrm{dR}}_{R\otimes\bF_p/A} 
}$$
since $\beta_{R/A}$ is functorial in $R$. On the other hand we have a commutative diagram
$$\xymatrix{
\prism^{(1)}_{R\otimes\bF_p /A}\{n\}\ar[dr]_{\gamma^{dR}_{\prism,R\otimes\bF_p/A}\{n\}}\ar[r]^{\psi_{1,*}} & \prism^{(1)}_{R\otimes\bF_p\otimes\bF_p /A}\{n\}\ar[d]^{\beta_{R\otimes\bF_p/A}} \\ 
{} & \widehat{\mathrm{dR}}_{R\otimes\bF_p/A} 
}$$
by Prop. \ref{betaprop}. An easy computation using that $\mathrm{sw}\circ\psi_1=\psi_2$ then shows that $i^*=\tau^0\circ\gamma^{dR}_{\prism,R\otimes\bF_p/A}\{n\}\circ \beta_{R/A}^{-1}$.

In order to define $\tau^n$ we may assume $R/A$ is relatively qrsp and $R$ and $A$ are $\bz_p$-flat. By the Kuenneth formula for derived de Rham cohomology 
$$ \widehat{\mathrm{dR}}_{R\otimes\bF_p/A} \simeq  \widehat{\mathrm{dR}}_{R/A}\otimes_A \widehat{\mathrm{dR}}_{(A/p)/A}\simeq \widehat{\mathrm{dR}}_{R/A}\otimes_{\bz_p} \widehat{\mathrm{dR}}_{\bF_p/\bz_p}\simeq D_J(A')\otimes_{\bz_p} \widehat{\mathrm{dR}}_{\bF_p/\bz_p}$$
is discrete, as $\widehat{\mathrm{dR}}_{\bF_p/\bz_p}\simeq\bz_p\oplus T$ is discrete and $D_J(A')$ is $\bz_p$-flat. Here 
$T\simeq\bigl(\bigoplus_{i=1}^\infty\bz/i\bz\bigr)^\wedge_p$ \cite{bhatt12}[Cor. 8.6].
Moreover, for $n\leq p$ one has $\fil^{\geq n}_{Hod}\widehat{\mathrm{dR}}_{\bF_p/\bz_p}=p^n\bz_p\oplus T$. It follows that the filtration
\begin{align*}\fil^{\geq n}_{Hod}\widehat{\mathrm{dR}}_{R\otimes\bF_p/A}=&\sum_{j=0}^n\fil^{\geq n-j}_{Hod}\widehat{\mathrm{dR}}_{R/A}\otimes_{\bz_p}\fil^{\geq j}_{Hod}\widehat{\mathrm{dR}}_{\bF_p/\bz_p}\\
=&\sum_{j=0}^nD^{(n-j)}_J(A')\otimes_{\bz_p}(p^j\bz_p\oplus T)\end{align*}
is strict (for $n\leq p$) and the existence of $\tau^n$ reduces to the question whether $\tau^0$ preserves $\fil^{\geq n}_{Hod}\widehat{\mathrm{dR}}_{R\otimes\bF_p/A}$. By functoriality of $\beta_{R/A}$ the automorphism $\tau^0$ of $\widehat{\mathrm{dR}}_{R\otimes\bF_p/A}$ is induced by $\tau^0$ of $\widehat{\mathrm{dR}}_{\bF_p/\bz_p}$. Being an algebra automorphism $\tau^0$ acts trivially on $\bz_p$. It also preserves $T=\ker \left(\prism_{\bF_p\otimes\bF_p}\to\prism_{\bF_p}=\bz_p\right)$. So clearly $\tau^0$ preserves $\fil^{\geq n}_{Hod}\widehat{\mathrm{dR}}_{R\otimes\bF_p/A}$.
\end{proof}

We can now continue the chain of isomorphisms (\ref{chain}) over $\fibre(\can_{dR})\simeq \widehat{\mathrm{dR}}^{<n}_{\mathscr{X}/A}[-1]$
\begin{align} 
&R\Gamma_{\syn}(\mathscr{X}/A,\bz_p(n))^{\rel}\notag\\
\simeq &\fibre\left(\fibre(\can_{dR})\xrightarrow{\fibre(\tilde{\gamma}^0\circ\beta_{\mathscr{X}/A}^{-1},\tilde{\gamma}^n\circ\fil^{\geq n}\beta_{\mathscr{X}/A}^{-1})}\fibre(\can'_{dR})\right)\notag\\
\simeq &\fibre\left(\fibre(\can_{dR})\xrightarrow{\fibre(\tau^0\circ\tilde{\gamma}^0\circ\beta_{\mathscr{X}/A}^{-1},\tau^n\circ\tilde{\gamma}^n\circ\fil^{\geq n}\beta_{\mathscr{X}/A}^{-1})}\fibre(\can'_{dR})\right)\notag\\
\simeq &\fibre\left(\widehat{\mathrm{dR}}^{<n}_{\mathscr{X}/A}[-1]\xrightarrow{i^*} \widehat{\mathrm{dR}}^{<n}_{(\mathscr{X}/p)/A}[-1]\right)=\widehat{\mathrm{dR}}^{<n,\rel}_{\mathscr{X}/A}[-1].\notag
\end{align}
This concludes the proof of Theorem \ref{smalln}.

\begin{remark} \label{bke} For $0\leq n\leq p-1$ the above proof gives an isomorphism 
$$\fibre{\sigma}\simeq \widehat{\mathrm{dR}}^{<n,\rel}_{\mathscr{X}/A}[-1].$$
If $A=\bz_p$, $\mathscr{X}=\Spf(R)$ is affine for simplicity and $R/\bz_p$ is smooth then $\widehat{\mathrm{dR}}_R\simeq \widehat{\Omega}_{R}^{\bullet}$ is the usual  $p$-complete de Rham complex and (see the proof of Lemma \ref{swlemma})
$$\widehat{\mathrm{dR}}_{R\otimes\bF_p}\simeq \widehat{\mathrm{dR}}_R\otimes\widehat{\mathrm{dR}}_{\bF_p}\simeq \widehat{\Omega}_{R}^{\bullet}\oplus\widehat{\Omega}_{R}^{\bullet}\otimes T$$
with $\fil^{\geq n}_{Hod}\widehat{\mathrm{dR}}_{R\otimes\bF_p}\simeq p^{n-\bullet}\widehat{\Omega}_{R}^{\bullet}\oplus\widehat{\Omega}_{R}^{\bullet}\otimes T$ for $n<p$. Hence 
$$\widehat{\mathrm{dR}}_{R\otimes\bF_p}^{<n}\simeq  \bz/p^{n-\bullet}\bz\otimes\widehat{\Omega}_{R}^{\bullet,<n} $$
and
$$\fibre{\sigma}\simeq\widehat{\mathrm{dR}}_{R}^{<n,\rel}[-1]\simeq p^{n-\bullet}\widehat{\Omega}_{R}^{\bullet,<n}[-1]=:p(n) \widehat{\Omega}_{R}^{\bullet,<n}[-1].$$
On the other hand by \cite{bms19}[Sec. 8]
$$R\Gamma_{\syn}((R/p)/\bz_p,\bz_p(n))\simeq W\Omega^n_{R/p,\log}[-n]$$
so that we obtain a fibre sequence
$$ p(n) \widehat{\Omega}_{R}^{\bullet,<n}[-1]\to R\Gamma^{\mathrm{FM}}_{\syn}(R/\bz_p,\bz_p(n))\to W\Omega^n_{R/p,\log}[-n]$$
for $0\leq n\leq p-1$. This is precisely the statement of \cite{Bloch-Esnault-Kerz-14}[Thm. 5.4].

\end{remark}

\newcommand{\XX}{\mathscr{X}}
\section{The Bloch-Kato exponential map}\label{bksection}

We first recall the definition of the Bloch-Kato exponential map. Theorem \ref{main} a) gives a fiber sequence
\begin{equation}\label{fundamental-fib-seq}
R\Gamma_{\syn}(\XX,\bq_p(n))\rightarrow R\Gamma_{\syn}(\XX/p,\bq_p(n))\rightarrow (\widehat{\mathrm{dR}}_{\XX}^{<n})_{\bq_p}
\end{equation}
functorial in the quasi-compact, quasi-separated derived formal scheme $\XX$, where the left map is the canonical map and the right map is the composition
\begin{equation}\label{cristodRmap}
R\Gamma_{\syn}(\XX/p,\bq_p(n))\rightarrow R\Gamma_{\syn}(\XX,\bq_p(n))^{\mathrm{rel}}[1]\underset{\sim}{\xrightarrow{(\mathrm{log}_{\XX})_{\bq_p}[1]}}
(\widehat{\mathrm{dR}}_{\XX}^{<n})_{\bq_p}.
\end{equation}
By construction, (\ref{cristodRmap}) is induced by the isomorphism $\beta:\Prism_{\XX/p}\stackrel{\sim}{\rightarrow}\widehat{\mathrm{dR}}_{\XX}$ of \cite[Thm. 5.4.2]{bhatt-lurie-22}, i.e. by the 
crystalline to de Rham isomorphism, as follows:
$$R\Gamma_{\syn}(\XX/p,\bq_p(n))\stackrel{\sim}{\rightarrow} (\Prism_{\XX/p})^{\varphi=p^n}_{\bq_p}\rightarrow (\Prism_{\XX/p})_{\bq_p}\simeq (\widehat{\mathrm{dR}}_{\XX})_{\bq_p}\rightarrow (\widehat{\mathrm{dR}}_{\XX}^{<n})_{\bq_p}$$
Let ${\bc_p}$ be the completion of an algebraic closure of $\bq_p$. For $\XX=\mathrm{Spf}(\mathcal{O}_{\bc_p})$, the fiber sequence (\ref{fundamental-fib-seq}) recovers the fundamental exact sequence 
$$0\rightarrow \bq_p(n)\rightarrow B_{cris}^+(\mathcal{O}_{\bc_p})^{\varphi=p^n}\rightarrow B_{dR}^+(\mathcal{O}_{\bc_p})/F^{\geq n}\rightarrow 0$$
of $p$-adic Hodge theory. One may write the left term as $\bq_p(n)=\bq_p\cdot t^n$. Dividing out by $t^n$ and taking the colimit over $n\geq 0$, we obtain 
\begin{equation}\label{norm-fundamental-fib-seq}
0\rightarrow \bq_p\rightarrow B_{cris}^{\varphi=1}\rightarrow B_{dR}/F^{\geq 0}\rightarrow 0
\end{equation}
where $B_{cris}=B_{cris}^+(\mathcal{O}_{\bc_p})[1/t]$ and $B_{dR}=B_{dR}^+(\mathcal{O}_{\bc_p})[1/t]$. For any finite extension $K/\bq_p$ and any  $G_K$-representation on a finite dimensional $\bq_p$-vector space $V$, we set 
$$D_{cris}(V)=(B_{cris}\otimes_{\bq_p} V)^{G_K}\hspace{0cm}\mathrm{,}\hspace{0.5cm}D_{dR}(V)=(B_{dR}\otimes_{\bq_p} V)^{G_K}$$
and we define
$$R\Gamma_{f}(K,V):=\mathrm{Fibre}\left(D_{cris}(V)\stackrel{(\varphi-1,\iota)}{\longrightarrow}D_{cris}(V)\oplus D_{dR}(V)/F^{\geq 0}\right)$$
where $\iota$ is the composite map $D_{cris}(V)\rightarrow D_{dR}(V)\rightarrow D_{dR}(V)/F^{\geq 0}$. The exact sequence (\ref{norm-fundamental-fib-seq}) induces 
an equivalence
$$R\Gamma(K,V)\simeq\mathrm{Fibre}\left(R\Gamma(K,B_{cris}\otimes_{\bq_p} V)^{\varphi=1}\to R\Gamma(K,B_{dR}/F^{\geq 0}\otimes_{\bq_p} V)\right)$$
and hence a map
$$\widetilde{\mathrm{exp}}_V:R\Gamma_{f}(K,V)\longrightarrow R\Gamma(K,V).$$
The Bloch-Kato exponential map $\exp_V$ is defined as $H^0$ of the composite
\begin{equation}\label{Bloch-Kato-map}
D_{dR}(V)/F^{\geq 0}\longrightarrow R\Gamma_{f}(K,V)[1]\xrightarrow{\widetilde{\mathrm{exp}}_V[1]} R\Gamma(K,V)[1].
\end{equation}
The boundary map associated to (\ref{fundamental-fib-seq}) followed by the étale comparison map \cite[Thm. 8.3.1]{bhatt-lurie-22}
$$(\gamma_{\syn}^{\et}\{n\})_{\bq_p}:R\Gamma_{\syn}(\XX,\bq_p(n))\longrightarrow R\Gamma_{\et}(\XX_K,\bq_p(n))$$
gives 
\begin{equation}\label{exp}
(\widehat{\mathrm{dR}}_{\XX}^{<n})_{\bq_p}\rightarrow R\Gamma_{\syn}(\XX,\bq_p(n))[1]\rightarrow R\Gamma_{\et}(\XX_K,\bq_p(n))[1].
\end{equation}
The goal of this section is to compare (\ref{Bloch-Kato-map}) and (\ref{exp}), at least for $\XX/\mathcal{O}_K$ smooth proper. In this case we may choose a $G_K$-equivariant decomposition
\begin{equation}\label{Deligne-decomposition}
R\Gamma_{\et}(\XX_{{\bc_p}},\bq_p)\simeq \bigoplus_{i\geq 0}V^{i}[-i]
\end{equation}
where $V^i$ is the $G_K$-representation $V^i:=H^i_{\et}(\mathscr{X}_{\bc_p},\bq_p)$.
By proper base change and Galois descent, we obtain
\begin{equation}\label{Deligne-decomposition-induces}
R\Gamma_{\et}(\XX_{K},\bq_p(n))\simeq R\Gamma(K,R\Gamma_{\et}(\XX_{{\bc_p}},\bq_p(n)))\simeq \bigoplus_{i\geq 0}R\Gamma(K,V^{i}(n))[-i].
\end{equation}
Theorem \ref{main} d) is a consequence of the following statement.
\begin{theorem}\label{exptheorem} Let $K/\bq_p$ be a finite extension and $\XX/\mathcal{O}_K$ a smooth proper formal scheme. The $G_K$-equivariant decomposition (\ref{Deligne-decomposition}) induces an isomorphism of fiber sequences
\[ \xymatrix@C-10pt{
R\Gamma_{\syn}(\XX,\bq_p(n))\ar[r]^{}\ar[d]^{\alpha_f}&R\Gamma_{\syn}(\XX/p,\bq_p(n))\ar[d]^{\alpha_{cris}}\ar[r]^{}&(\widehat{\mathrm{dR}}_{\XX/\mathcal{O}_K}^{<n})_{\bq_p}\ar[d]^{\alpha_{dR}} \\
\bigoplus_{i}R\Gamma_f(K,V^{i}(n))[-i]\ar[r]&\bigoplus_{i}D_{cris}(V^{i}(n))^{\varphi=1}[-i]\ar[r]&\bigoplus_{i}D_{dR}(V^{i}(n))/F^{\geq 0}[-i] 
}
\]
where the isomorphism $\alpha_{cris}$ (resp. $\alpha_{dR}$) is induced by the crystalline (resp. de Rham) comparison theorem (see e.g. \cite{bms18}), and such that the square
\[ \xymatrix@C60pt{
R\Gamma_{\syn}(\XX,\bq_p(n))\ar[r]^{(\gamma_{\syn}^{\et}\{n\})_{\bq_p}}\ar[d]^{\alpha_f}&R\Gamma_{\et}(\XX_{K},\bq_p(n))\ar[d]^{(\ref{Deligne-decomposition-induces})} \\
\bigoplus_{i}R\Gamma_f(K,V^{i}(n))[-i]\ar[r]^{\oplus_i\widetilde{\mathrm{exp}}_{V^{i}(n)}[-i]}&\bigoplus_{i}R\Gamma(K,V^{i}(n))[-i] 
}
\] 
commutes.  
\end{theorem}

\begin{proof}
Let $\varpi\in \mathcal{O}_K$ be a uniformizer and set $k:=\mathcal{O}_K/\varpi$. The residue field $\overline{k}$ of $\mathcal{O}:=\mathcal{O}_{\bc_p}$ is an algebraic closure of $k$. We choose an embedding $W(\overline{k})\rightarrow\mathcal{O}$, which gives a section
$\overline{k}\rightarrow \mathcal{O}/p$. We set $\tilde{\XX}:=\XX\otimes_{\mathcal{O}_K}\mathcal{O}$. Finally, we denote by $\widehat{\otimes}$ the $p$-completed tensor product. 

We consider the pull-back squares
\[ \xymatrix{
R\Gamma_{\syn}(\tilde{\XX},\bq_p(n))\ar[d]\ar[r]^{}&R\Gamma_{\syn}(\tilde{\XX}/p,\bq_p(n))\ar[d]&  \\
(\widehat{\mathrm{dR}}^{\geq n}_{\tilde{\XX}})_{\bq_p}\ar[r]\ar[d]&(\widehat{\mathrm{dR}}_{\tilde{\XX}})_{\bq_p}\ar[d]&\\
(\widehat{\mathrm{dR}}^{\mathrm{hc},\geq n}_{\tilde{\XX}/\mathcal{O}_K})_{\bq_p}\ar[r]&(\widehat{\mathrm{dR}}^{\mathrm{hc}}_{\tilde{\XX}/\mathcal{O}_K})_{\bq_p}&
}
\] 
where the lower row is given by Hodge completed de Rham cohomology relative to $\mathcal{O}_K$. We have isomorphisms \cite[Section 1.5]{beil12}
$$B_{dR}^+\stackrel{\sim}{\rightarrow}(\widehat{\mathrm{dR}}^{\mathrm{hc}}_{\mathcal{O}})_{\mathbb{Q}_p}\stackrel{\sim}{\rightarrow} (\widehat{\mathrm{dR}}^{\mathrm{hc}}_{\mathcal{O}/\mathcal{O}_K})_{\mathbb{Q}_p} $$
and an isomorphism of filtered $\mathbb{E}_{\infty}$-algebras over $\mathcal{O}_K$ 
\begin{equation}\label{identdR}
(\widehat{\mathrm{dR}}^{\mathrm{hc}}_{\tilde{\XX}/\mathcal{O}_K})_{\bq_p}\simeq (\widehat{\mathrm{dR}}^{\mathrm{hc}}_{\XX/\mathcal{O}_K}\widehat{\otimes}_{\mathcal{O}_K} \widehat{\mathrm{dR}}^{\mathrm{hc}}_{\mathcal{O}/\mathcal{O}_K})_{\bq_p}\simeq \widehat{\mathrm{dR}}_{\XX/\mathcal{O}_K}\otimes_{\mathcal{O}_K} B^+_{dR}
\end{equation}
where we use the fact that $\widehat{\mathrm{dR}}_{\XX/\mathcal{O}_K}$ is Hodge complete and perfect as a filtered complex of $\mathcal{O}_K$-modules. Since $R\Gamma_{\add}(\tilde{\XX}/p,\bz_p(n))$ is torsion, we have 
$$R\Gamma_{\syn}(\tilde{\XX}/p,\bq_p(n))\simeq (\Prism_{\tilde{\XX}/p}\{n\})^{\varphi\{n\}=1}_{\bq_p}\simeq (\Prism_{\tilde{\XX}/p})^{\varphi=p^n}_{\bq_p}.$$
Moreover, we have isomorphisms 
\begin{eqnarray*}
(\Prism_{\tilde{\XX}/p})_{\bq_p}&\simeq&  (\Prism^{(1)}_{(\tilde{\XX}\otimes_{\mathcal{O}}\mathcal{O}/p)/A_{\mathrm{inf}}})_{\bq_p}\\
&\simeq&(\Prism^{(1)}_{\tilde{\XX}/A_{\mathrm{inf}}}\widehat{\otimes}_{A_{\mathrm{inf}}} \Prism^{(1)}_{(\mathcal{O}/p)/A_{\mathrm{inf}}})_{\bq_p}\\
&\simeq& (R\Gamma_{A_{\mathrm{inf}}}(\tilde{\XX})\widehat{\otimes}_{A_{\mathrm{inf}}} A_{cris})_{\bq_p}\\
&\simeq& R\Gamma_{cris}((\tilde{\XX}/p)/A_{cris})_{\bq_p}\\
&\simeq& R\Gamma_{cris}(\XX_k/W(k))\otimes_{W(k)}B^+_{cris}
\end{eqnarray*}
given respectively by \cite[Rem. 4.1.8]{bhatt-lurie-22}, \cite[Prop. 4.4.12]{bhatt-lurie-22}, \cite[Thm. 17.2]{bhatt-scholze}, \cite[Thm. 1.8(iii)]{bms18} and \cite[Prop. 13.21]{bms18}. Note that $R\Gamma_{cris}(\XX_k/W(k))$ is a perfect complex of $W(k)$-modules \cite[Thm. 1.8]{bhatt-scholze}.
Note also that the Frobenius twist appearing above merely twists the $A_{\mathrm{inf}}$-module structure. We obtain 
\begin{equation}\label{identcris}
R\Gamma_{\syn}(\tilde{\XX}/p,\bq_p(n))\simeq  (R\Gamma_{cris}(\XX_k/W(k))\otimes_{W(k)}B^+_{cris})^{\varphi=p^n}.
\end{equation}
In view of (\ref{identdR}) and (\ref{identcris}), the total pull-back square above can be written as follows:
\[ \xymatrix{
R\Gamma_{\syn}(\tilde{\XX},\bq_p(n))\ar[d]\ar[r]^{}&(R\Gamma_{cris}(\XX_k/W(k))\otimes_{W(k)}B^+_{cris})^{\varphi=p^n}\ar[d] \\
F^{\geq n}(\widehat{\mathrm{dR}}_{\XX/\mathcal{O}_K}\otimes_{\mathcal{O}_K} B^+_{dR})\ar[r]&\widehat{\mathrm{dR}}_{\XX/\mathcal{O}_K}\otimes_{\mathcal{O}_K} B^+_{dR}
}
\] 
We denote this pull-back square by $\mathrm{PB}(n)$. Here the right vertical map is induced by the canonical map 
\begin{equation}\label{cris-dR}
R\Gamma_{cris}(\XX_k/W(k))\otimes_{W(k)}B^+_{cris}\longrightarrow \widehat{\mathrm{dR}}_{\XX/\mathcal{O}_K}\otimes_{\mathcal{O}_K}B^+_{dR}
\end{equation}
which is in turn defined as the lower horizontal map of the commutative diagram
\[ \xymatrix{
(\Prism_{\tilde{\XX}/p})_{\bq_p}\ar[r]^{(\beta_{\tilde{\XX}})_{\bq_p}}&(\widehat{\mathrm{dR}}_{\tilde{\XX}})_{\bq_p} \\
(\Prism_{\tilde{\XX}}\widehat{\otimes}_{\Prism_{\mathcal{O}}} \Prism_{\mathcal{O}/p})_{\bq_p}\ar[d]^{\simeq}\ar[r]^{(\gamma_{\Prism,\tilde{\XX}}^{dR}\widehat{\otimes}_{\gamma^{dR}_{\Prism,\mathcal{O}}} \beta_{\mathcal{O}})_{\bq_p}}\ar[u]_{\simeq}&(\widehat{\mathrm{dR}}_{\tilde{\XX}}\widehat{\otimes}_{\widehat{\mathrm{dR}}_{\mathcal{O}}}\widehat{\mathrm{dR}}_{\mathcal{O}})_{\bq_p}\ar[d]\ar[u]_{\simeq}\\
(\Prism^{(1)}_{\tilde{\XX}/A_{\mathrm{inf}}}\widehat{\otimes}_{A_{\mathrm{inf}}} A_{cris})_{\bq_p}\ar[d]^{\simeq}\ar[r]&(\widehat{\mathrm{dR}}_{\XX/\mathcal{O}_K}\otimes_{\mathcal{O}_K}\widehat{\mathrm{dR}}^{\mathrm{hc}}_{\mathcal{O}/\mathcal{O}_K})_{\bq_p}\ar[d]^{\simeq}\\
R\Gamma_{cris}(\XX_k/W(k))\otimes_{W(k)} B^+_{cris}\ar[r]^{\hspace{0.5cm}(\ref{cris-dR})}&\widehat{\mathrm{dR}}_{\XX/\mathcal{O}_K}\otimes_{\mathcal{O}_K}B^+_{dR}
}
\] 
where the top square commutes by functoriality of $\beta$ and $\gamma_{\Prism}^{dR}$, and by definition of $\gamma_{\Prism}^{dR}$.

We denote by $$\epsilon=(1,\zeta_p,\zeta_{p^2},\cdots)\in \bz_p(1)=T_p(\mathcal{O}^{\times}_{\bc_p})\stackrel{\sim}{\rightarrow}H_{\syn}^0(\mathrm{Spf}(\mathcal{O}_{\bc_p}),\bz_p(1))$$ a generator. Let $t\in B^{+}_{\mathrm{dR}}$ be the image of $\epsilon$ under the maps
$$\bz_p(1)\rightarrow B^{+}_{cris}\subset B^{+}_{\mathrm{dR}}$$
which is a uniformizer of  $B^{+}_{\mathrm{dR}}$. Multiplication by $\epsilon$ gives a morphism from the pull-back square $\mathrm{PB}(n)$ to $\mathrm{PB}(n+1)$, which can be written as  
\[ \xymatrix{
R\Gamma_{\syn}(\tilde{\XX},\bq_p(n+1))\ar[d]\ar[r]^{}&(R\Gamma_{cris}(\XX_k/W(k))\otimes_{W(k)}t^{-1}B^+_{cris})^{\varphi=p^n}\ar[d] \\
F^{\geq n}(\mathrm{dR}_{\XX/\mathcal{O}_K}\otimes_{\mathcal{O}_K} t^{-1}B^+_{dR})\ar[r]&\mathrm{dR}_{\XX/\mathcal{O}_K}\otimes_{\mathcal{O}_K}t^{-1} B^+_{dR}
}
\] 
The colimit $\mathrm{colim}_m\,\mathrm{PB}(n+m)$, where the transition maps are given by multiplication by $\epsilon$, is the pull-back square
\[ \xymatrix{
\mathrm{colim}\,R\Gamma_{\syn}(\tilde{\XX},\bq_p(n+m))\ar[d]\ar[r]^{}&(R\Gamma_{cris}(\XX_k/W(k))\otimes_{W(k)}B_{cris})^{\varphi=p^n}\ar[d] \\
F^{\geq n}(\mathrm{dR}_{\XX/\mathcal{O}_K}\otimes_{\mathcal{O}_K} B_{dR})\ar[r]&\mathrm{dR}_{\XX/\mathcal{O}_K}\otimes_{\mathcal{O}_K} B_{dR}
}
\] 
where $m$ runs over the non-negative natural numbers $\mathbb{N}$, and the right vertical map is induced by (\ref{cris-dR}). Using the comparison theorems, this  pull-back square is isomorphic to
\[ \xymatrix{
\mathrm{colim}\,R\Gamma_{\syn}(\tilde{\XX},\bq_p(n+m))\ar[d]\ar[r]^{}&R\Gamma_{\et}(\XX_{\bc_p},\bq_p(n))\otimes_{\bq_p}B_{cris}^{\varphi=1}\ar[d] \\
R\Gamma_{\et}(\XX_{\bc_p},\bq_p(n))\otimes_{\bq_p} F^{\geq 0}B_{dR}\ar[r]&R\Gamma_{\et}(\XX_{\bc_p},\bq_p(n))\otimes_{\bq_p} B_{dR}
}
\] 
where the right vertical map (resp. the lower horizontal map) is induced by the inclusion $B_{cris}^{\varphi=1}\rightarrow B_{dR}$ (resp. $F^{\geq 0}B_{dR}\rightarrow B_{dR}$). This last pullback square 
together with the fundamental exact sequence
$$0\rightarrow \bq_p\rightarrow B_{cris}^{\varphi=1}\rightarrow B_{dR}/F^{\geq 0}\rightarrow 0$$
yield an isomorphism
\begin{equation*}\label{BLmap}
\alpha(n):\mathrm{colim}\,R\Gamma_{\syn}(\tilde{\XX},\bq_p(n+m))\stackrel{\sim}{\longrightarrow}R\Gamma_{\et}(\XX_{\bc_p},\bq_p(n)).
\end{equation*}

\begin{prop}\label{compatibility-comp-thms}
The map $\alpha(n)$ is the composition
\begin{eqnarray*}
\mathrm{colim}\,R\Gamma_{\syn}(\tilde{\XX},\bq_p(n+m))&\simeq&(\mathrm{colim}\,R\Gamma_{\syn}(\tilde{\XX},\bz_p(n+m)))_{\bq_p}\\
\label{BL-map}&\rightarrow& ((\mathrm{colim}\,R\Gamma_{\syn}(\tilde{\XX},\bz_p(n+m)))^{\wedge}_p)_{\bq_p}\\
&\stackrel{\sim}{\rightarrow}&R\Gamma_{\et}(\XX_{\bc_p},\bq_p(n)).
\end{eqnarray*}
where the last isomorphism is given by \cite[Thm. 8.5.1]{bhatt-lurie-22}. 
\end{prop}

\begin{proof}
The map $\alpha(n+1)$ is obtained from $\alpha(n)$ using multiplication by $\epsilon$, and a similar observation applies to the composite map of the proposition. Therefore, it suffices to treat the case $n=0$. We have a commutative square
\[ \xymatrix{
R\Gamma_{\syn}(\tilde{\XX},\bz_p(m))\ar[d]^{\gamma_{\syn}^{\et}\{m\}}\ar[r]^{\cup\epsilon\hspace{0.5cm}}& R\Gamma_{\syn}(\tilde{\XX},\bz_p(m+1))\ar[d]^{\gamma_{\syn}^{\et}\{m+1\}}\\
R\Gamma_{\et}(\XX_{{\bc_p}},\bz_p(m))\ar[r]^{\cup\epsilon \hspace{0.5cm}} \ar[r]& R\Gamma_{\et}(\XX_{{\bc_p}},\bz_p(m+1))
}
\] 
where the lower horizontal map is an isomorphism. By \cite{bhatt-mathew}[Prop. 4.12]  the formal scheme $\tilde{\XX}$ is $F$-smooth in the sense of \cite{bhatt-mathew}[Def. 1.7]. It follows from \cite{bhatt-mathew}[Thm. 1.8] that the vertical maps are isomorphisms for $m$ large enough, hence so is the upper horizontal map. In particular, $\mathrm{colim}_m\,R\Gamma_{\syn}(\tilde{\XX},\bz_p(m))$ is $p$-complete. 

Let $(A_0,I)$ be the perfection of the $q$-de Rham prism, so that $A_0$ is the $(p,q-1)$-completion of $\bz[q^{1/p^{\infty}}]$ and $I$ is the ideal generated by $[p]_q=(\frac{q^p-1}{q-1})$, and let $\bz_p^{\mathrm{cyc}}:=A_0/I$. Consider the map of perfect prisms 
$f:(A_0,I)\rightarrow (A_{\mathrm{inf}},\mathrm{Ker}(\theta))$ induced by the map of perfectoid rings $\bz_p^{\mathrm{cyc}}\rightarrow\mathcal{O}_{\bc_p}$ given by $(q^p,q,q^{1/p},\cdots)\mapsto\epsilon$, see \cite[Thm. 3.10]{bhatt-scholze}. Then $q^p\in A_0$ maps to $[\epsilon]\in A_{\mathrm{inf}}$ and $q$ maps to $[\epsilon^{1/p}]$. We denote by $\mu\in A_{\mathrm{inf}}$ the image of $q^p-1$, by $\tilde{\mu}=\varphi^{-1}(\mu)$ the image of $q-1$ and by $\xi$ the image of $[p]_q$ via the map $f:A_0\rightarrow A_{\mathrm{inf}}$. Following  \cite[Notation 2.6.3]{bhatt-lurie-22}, we trivialize the Breuil-Kisin twist $A_{\inf}\{1\}=A_{\inf}\cdot e_{A_{\inf}}$, where  
$$e_{A_{\inf}}=\frac{\mathrm{log}_{\Prism}([\epsilon])}{\tilde{\mu}}=\frac{t}{\tilde{\mu}}\in A_{\inf}\{1\}$$ 
is the image of $\frac{\mathrm{log}_{\Prism}(q^p)}{q-1}\in A_0\{1\}$ via the map $f\{1\}:A_0\{1\}\rightarrow A_{\mathrm{inf}}\{1\}$. Under this trivialization, multiplication by $t$ induces
$$\fil_{\mathcal N}^{\geq 0}\Prism_{\tilde{\XX}}\stackrel{\cdot\tilde{\mu}}{\longrightarrow}\fil_{\mathcal N}^{\geq 1}\Prism_{\tilde{\XX}}\stackrel{\cdot\tilde{\mu}}{\longrightarrow}\fil_{\mathcal N}^{\geq 2}\Prism_{\tilde{\XX}}\stackrel{\cdot\tilde{\mu}}{\longrightarrow}\cdots$$
We have an isomorphism
\begin{eqnarray*}
&&\mathrm{colim}_m\,R\Gamma_{\syn}(\tilde{\XX},\bz_p(m))\\
& \simeq& \mathrm{Fibre}\left(\mathrm{colim}_m\,\fil_{\mathcal N}^{\geq m}\Prism_{\tilde{\XX}}\stackrel{1-\varphi}{\longrightarrow} \Prism_{\tilde{\XX}}\otimes_{A_{\mathrm{inf}}} A_{\mathrm{inf}}[1/\tilde{\mu}]\right)\\
& \stackrel{\sim}{\rightarrow}& \mathrm{Fibre}\left(\mathrm{colim}_m\,\fil_{\mathcal N}^{\geq m}\Prism_{\tilde{\XX}}\stackrel{1-\varphi}{\longrightarrow} \Prism_{\tilde{\XX}}\otimes_{A_{\mathrm{inf}}} A_{\mathrm{inf}}[1/\tilde{\mu}]\right)^{\wedge}_p\\
& \stackrel{\sim}{\rightarrow}& ((\Prism_{\tilde{\XX}}\otimes_{A_{\mathrm{inf}}} A_{\mathrm{inf}}[1/\tilde{\mu}])^{\wedge}_p)^{\varphi=1}
\end{eqnarray*}
since $\mathrm{colim}_m\,R\Gamma_{\syn}(\tilde{\XX},\bz_p(m))$ is $p$-complete and since 
$$\mathrm{colim}_m\,\fil_{\mathcal N}^{\geq m}\Prism_{\tilde{\XX}}\rightarrow \mathrm{colim}_m\,\Prism_{\tilde{\XX}}=  \Prism_{\tilde{\XX}}[1/\tilde{\mu}]$$
is an isomorphism after $p$-completion by \cite[Prop. 8.5.3]{bhatt-lurie-22}. The canonical map $$\mathrm{colim}_m\,R\Gamma_{\syn}(\tilde{\XX},\bz_p(m))\rightarrow \mathrm{colim}_m\,\fil_{\mathcal N}^{\geq m}\Prism_{\tilde{\XX}}\rightarrow\Prism_{\tilde{\XX}}\otimes_{A_{\mathrm{inf}}} A_{\mathrm{inf}}[1/\tilde{\mu}] $$
and the tautological map
$$R\Gamma_{\et}(\XX_{{\bc_p}},\bz_p)\rightarrow R\Gamma_{\et}(\XX_{{\bc_p}},\bz_p)\otimes_{\mathbb{Z}_p} A_{\mathrm{inf}}[1/\tilde{\mu}]$$
sit in the following commutative diagram
\[ \xymatrix{
\mathrm{colim}_m\,R\Gamma_{\syn}(\tilde{\XX},\bz_p(m))\ar[d]^{\simeq}\ar[r]^{\simeq}&R\Gamma_{\et}(\XX_{{\bc_p}},\bz_p)\ar[d]^{\simeq}\\
((\Prism_{\tilde{\XX}}\otimes_{A_{\mathrm{inf}}} A_{\mathrm{inf}}[1/\tilde{\mu}])^{\wedge}_p)^{\varphi=1}\ar[r]^{\simeq \hspace{0.8cm}}\ar[d]& ((R\Gamma_{\et}(\XX_{{\bc_p}},\bz_p)\otimes_{\mathbb{Z}_p} A_{\mathrm{inf}}[1/\tilde{\mu}])^{\wedge}_p)^{\varphi=1}\ar[d] \\
\Prism_{\tilde{\XX}}\otimes_{A_{\mathrm{inf}}} A_{\mathrm{inf}}[1/\tilde{\mu}]\ar[r]^{\simeq\hspace{0.8cm}}& R\Gamma_{\et}(\XX_{{\bc_p}},\bz_p)\otimes_{\mathbb{Z}_p} A_{\mathrm{inf}}[1/\tilde{\mu}]
}
\] 
where the horizontal isomorphisms are the étale comparison maps. Here the top commutative square defines the upper horizontal isomorphism, which is given by the proof of \cite[Thm. 8.5.1]{bhatt-lurie-22}, and  the lower horizontal isomorphism is induced by the isomorphisms
\begin{eqnarray*}
\varphi_{A_{\mathrm{inf}}}^*(\Prism_{\tilde{\XX}}[1/\tilde{\mu}])&\simeq& \varphi_{A_{\mathrm{inf}}}^*(\Prism_{\tilde{\XX}/A_{\mathrm{inf}}}[1/\tilde{\mu}])\\
&\simeq& R\Gamma_{A_{\mathrm{inf}}}(\tilde{\XX})[1/\mu]\\
&\simeq& R\Gamma_{\et}(\XX_{{\bc_p}},\bz_p)\otimes_{\mathbb{Z}_p} A_{\mathrm{inf}}[1/\mu]
\end{eqnarray*}
given by \cite[Prop. 4.4.12]{bhatt-lurie-22}, \cite[Thm. 17.2]{bhatt-scholze} and \cite[Thm. 1.8(iv)]{bms18} respectively. 

The pull-back square
\[ \xymatrix{
\mathrm{colim}_m\,R\Gamma_{\syn}(\tilde{\XX},\bq_p(m))\ar[d]\ar[r]^{}&R\Gamma_{\et}(\XX_{\bc_p},\bq_p)\otimes_{\bq_p}B_{\mathrm{crys}}^{\varphi=1}\ar[d] \\
R\Gamma_{\et}(\XX_{\bc_p},\bq_p)\otimes_{\bq_p}B^+_{\mathrm{dR}}\ar[r]&R\Gamma_{\et}(\XX_{\bc_p},\bq_p)\otimes_{\bq_p} B_{\mathrm{dR}}
}
\] 
obtained before the statement of Proposition \ref{compatibility-comp-thms} is induced by the total square of the diagram below
{\small{
\[ \xymatrix{
\mathrm{colim}_m\,R\Gamma_{\syn}(\tilde{\XX},\bz_p(m))\ar[d]\ar[r]&\mathrm{colim}_m\,R\Gamma_{\syn}(\tilde{\XX}/p,\bz_p(m)) \ar[r]& R\Gamma_{\et}(\XX_{{\bc_p}},\bz_p)\otimes_{\bz_p} B_{\mathrm{cris}}^{\varphi=1}\ar[dd]\\
\mathrm{colim}_m\,\fil_{\mathcal N}^{\geq m}\Prism_{\tilde{\XX}}\ar[r]\ar[d]& \Prism_{\tilde{\XX}}\otimes_{A_{\mathrm{inf}}} A_{\mathrm{inf}}[1/\tilde{\mu}]\ar[d]^{\simeq} &\\
F^{\geq 0}(\mathrm{dR}_{\XX/\mathcal{O}_K}\otimes_{\mathcal{O}_K} B_{\mathrm{dR}})\ar[d]^{\simeq}& R\Gamma_{\et}(\XX_{{\bc_p}},\bz_p)\otimes_{\bz_p} A_{\mathrm{inf}}[1/\tilde{\mu}]\ar[r]& R\Gamma_{\et}(\XX_{{\bc_p}},\bz_p)\otimes_{\bz_p} B_{\mathrm{cris}}\ar[d]\\
R\Gamma_{\et}(\XX_{{\bc_p}},\bz_p)\otimes_{\bz_p} B^+_{\mathrm{dR}}\ar[rr]& & R\Gamma_{\et}(\XX_{{\bc_p}},\bz_p)\otimes_{\bz_p} B_{\mathrm{dR}}
}
\] 
}}
The total square gives a  map 
\begin{eqnarray*}
\mathrm{colim}\,R\Gamma_{\syn}(\tilde{\XX},\bz_p(m))
&\rightarrow&R\Gamma_{\et}(\XX_{\bc_p},\bz_p)\otimes_{\bz_p} (B_{\mathrm{dR}}^+\times_{B_{\mathrm{dR}}}B_{\mathrm{cris}}^{\varphi=1})\\
&=&R\Gamma_{\et}(\XX_{\bc_p},\bq_p)
\end{eqnarray*}
which factors through an isomorphism
\begin{eqnarray*}
\mathrm{colim}\,R\Gamma_{\syn}(\tilde{\XX},\bz_p(m))&\rightarrow&((\Prism_{\tilde{\XX}}\otimes_{A_{\mathrm{inf}}} A_{\mathrm{inf}}[1/\tilde{\mu}])^{\wedge}_p)^{\phi=1}\\
&\simeq& ((R\Gamma_{\et}(\XX_{{\bc_p}},\bz_p)\otimes_{\bz_p} A_{\mathrm{inf}}[1/\tilde{\mu}])^{\wedge}_p)^{\phi=1}\\
&\simeq& R\Gamma_{\et}(\XX_{{\bc_p}},\bz_p)\otimes_{\bz_p} ((A_{\mathrm{inf}}[1/\tilde{\mu}])^{\wedge}_p)^{\phi=1}\\
&\simeq& R\Gamma_{\et}(\XX_{{\bc_p}},\bz_p).
\end{eqnarray*}
which is the map given by \cite[Thm. 8.5.1]{bhatt-lurie-22}.

\end{proof}

Now we conclude the proof of Theorem \ref{exptheorem}. We obtain $G_K$-equivariant morphisms
of fiber sequences
{\small{
\[ \xymatrix@C-10pt{
R\Gamma_{\syn}(\tilde{\XX},\bq_p(n))\ar[r]\ar[d]&R\Gamma_{\syn}(\tilde{\XX}/p,\bq_p(n))\ar[r]\ar[d]&(\mathrm{dR}_{\tilde{\XX}/\mathcal{O}_K}^{<n})_{\bq_p}\ar[d]\\
R\Gamma_{\et}(\XX_{{\bc_p}},\bq_p(n))\ar[r]\ar[d]&(B_{cris}\otimes_{\bq_p}R\Gamma_{\et}(\XX_{\bc_p},\bq_p(n)))^{\varphi=1}\ar[r]\ar[d]&(B_{dR}\otimes_{\bq_p}R\Gamma_{\et}(\XX_{\bc_p},\bq_p(n)))/F^{\geq 0}\ar[d]\\
\bigoplus_{i}V^{i}(n)[-i]\ar[r]&\bigoplus_{i}(B_{cris}\otimes_{\bq_p}V^{i}(n))^{\varphi=1}[-i]\ar[r]&\bigoplus_{i}(B_{dR}\otimes_{\bq_p}V^{i}(n))/F^{\geq 0}[-i]
}
\] }}
given by the discussion above and by the  $G_K$-equivariant decomposition (\ref{Deligne-decomposition}).
Hence we have  morphisms of fiber sequences
{\small{
\[ \xymatrix{
R\Gamma_{\syn}(\XX,\bq_p(n))\ar[r]^{}\ar[d]&R\Gamma_{\syn}(\XX/p,\bq_p(n))\ar[d]\ar[r]^{}&(\mathrm{dR}_{\XX/\mathcal{O}_K}^{<n})_{\bq_p}\ar[d] \\
R\Gamma_{\syn}(\tilde{\XX},\bq_p(n))\ar[r]^{}\ar[d]&R\Gamma_{\syn}(\tilde{\XX}/p,\bq_p(n))\ar[d]\ar[r]^{}&(\mathrm{dR}_{\tilde{\XX}/\mathcal{O}_K}^{<n})_{\bq_p}\ar[d] \\
\bigoplus_{i}V^{i}(n)[-i]\ar[r]&\bigoplus_{i}(B_{cris}\otimes_{\bq_p}V^{i}(n))^{\varphi=1}[-i]\ar[r]&\bigoplus_{i}(B_{dR}\otimes_{\bq_p}V^{i}(n))/F^{\geq 0}[-i]
}
\] 
}}
where the lower left vertical map is the étale comparison map \cite[Thm. 8.3.1]{bhatt-lurie-22}
$$(\gamma_{\syn}^{\et}\{n\})_{\bq_p}:R\Gamma_{\syn}(\tilde{\XX},\bq_p(n))\rightarrow R\Gamma_{\et}(\XX_{{\bc_p}},\bq_p(n))$$
followed by the decomposition (\ref{Deligne-decomposition}). The morphism from the top fiber sequence to the bottom fiber sequence factors through an equivalence
\[ \xymatrix{
R\Gamma_{\syn}(\XX,\bq_p(n))\ar[r]^{}\ar[d]\ar[d]_{\alpha_{f}}^{\simeq}&R\Gamma_{\syn}(\XX/p,\bq_p(n))\ar[d]\ar[d]_{\alpha_{cris}}^{\simeq}\ar[r]^{}&(\mathrm{dR}_{\XX/\mathcal{O}_K}^{<n})_{\bq_p}\ar[d]_{\alpha_{dR}}^{\simeq} \\
\bigoplus_{i}R\Gamma_f(K,V^{i}(n))[-i]\ar[r]&\bigoplus_{i}D_{cris}(V^{i}(n))^{\varphi=1}[-i]\ar[r]&\bigoplus_{i}D_{dR}(V^{i}(n))/F^{\geq 0}[-i] 
}
\] 
Here the left vertical map $\alpha_f$ is the map induced on the fibers of the right horizontal maps. It follows that we have a commutative square
\[ \xymatrix{
R\Gamma_{\syn}(\XX,\bq_p(n))\ar[r]^{}\ar[d]^{\simeq}_{\alpha_f}&R\Gamma_{\et}(\XX_{{\bc_p}},\bq_p(n))\ar[d]_{(\ref{Deligne-decomposition})}^{\simeq} \\
\bigoplus_{i}R\Gamma_f(K,V^{i}(n))[-i]\ar[r]&\bigoplus_{i}V^{i}(n)[-i] 
}
\] 
where the upper horizontal map is the composite map
$$R\Gamma_{\syn}(\XX,\bq_p(n))\rightarrow R\Gamma_{\syn}(\tilde{\XX},\bq_p(n))\xrightarrow{(\gamma_{\syn}^{\et}\{n\})_{\bq_p}} R\Gamma_{\et}(\XX_{{\bc_p}},\bq_p(n))$$
and the lower horizontal map is induced by (\ref{norm-fundamental-fib-seq}).  This commutative square is $G_K$-equivariant, where $G_K$ acts naturally on the right columns and trivially on the left columns. Functoriality of $\gamma_{\syn}^{\et}\{n\}$ then gives a commutative square
\[ \xymatrix@C60pt{
R\Gamma_{\syn}(\XX,\bq_p(n))\ar[r]^{(\gamma_{\syn}^{\et}\{n\})_{\bq_p}}\ar[d]^{\simeq}_{\alpha_f}&R\Gamma_{\et}(\XX_{K},\bq_p(n))\ar[d]_{R\Gamma(K,(\ref{Deligne-decomposition}))}^{\simeq} \\
\bigoplus_{i}R\Gamma_f(K,V^{i}(n))[-i]\ar[r]^{\oplus_i\widetilde{\mathrm{exp}}_{V^{i}(n)}[-i]}&\bigoplus_{i}R\Gamma(K,V^{i}(n))[-i] 
}
\] 
which concludes the proof of Theorem \ref{exptheorem}.

\end{proof}

\section{Special values of Zeta functions}\label{zetasection}

We summarize very briefly the simplifications to \cite{Flach-Morin-16} which are possible in light of the results of the present paper. Let $F$ be a number field and $\X/\co_F$ a smooth projective scheme. The Zeta function of $\X$
$$\zeta(\mathcal{X},s)=\prod_{\underset{ \text{$x$ closed}}{x\in\mathcal{X}}}\dfrac{1}{1-N(x)^{-s}}$$
is defined by a convergent product for $\mathrm{Re}(s)>\mathrm{dim}(\mathcal{X})$ and is expected to have a meromorphic continuation to all $s\in\bc$. The vanishing order and the leading Taylor coefficient of $\zeta(\mathcal{X},s)$ at integer arguments $s\in\bz$ has been a subject of considerable interest, beginning with the analytic class number formula in the 19th century (where $\X=\Spec(\co_F)$ and $s=1$). The prevailing paradigm for attempts to generalize the analytic class number formula has been that of Tamagawa numbers of algebraic groups, beginning with the results of Ono for algebraic Tori, continuing with the conjecture of Birch and Swinnerton-Dyer  and culminating in the Tamagawa number conjecture of Bloch and Kato for motives over $F$ \cite{bk88} (for $\X=\Spec(\co_F)$ this applies to all $s=n\geq 1$). Fontaine and Perrin-Riou found an equivalent formulation of the conjecture of Bloch and Kato which applies to all $s\in\bz$ and is stated in terms of the "fundamental line", a certain product of determinants of rational vector spaces associated to a motive over $F$ \cite{fpr91}. Lichtenbaum has been advocating the point view that special value conjectures should be simpler for Zeta functions than for motivic L-functions and this intuition is in some sense borne out by the results we describe below. If we assume, as we do in this section, that $\X/\co_F$ is smooth and projective (rather than $\X$ regular, proper over $\Spec(\bz)$ as in \cite{Flach-Morin-16, Flach-Morin-20}) there is a tight connection between the two types of functions since we have
\[ \zeta(\X,s)=\prod_{i\in\bz}L(h^i(X),s)^{(-1)^i},\]
i.e. the Zeta function is the motivic L-function associated to the total motive
\begin{equation}\notag
h(X)\simeq \bigoplus_{0\leq i\leq 2(d-1)} h^i(X)[-i]
\end{equation}
of the generic fibre $X:=\X_F$ of $\X$. Here $d$ is the Krull dimension of $\X$ and our notion of motive is a naive one: a motive over $F$ is a collection of rational vector spaces with comparison maps ("motivic structure") as described in \cite{fpr91}, \cite{flach03}. Inspired by the ideas of Lichtenbaum we gave in \cite{Flach-Morin-16}[Thm. 5.27] a further reformulation of the Tamagawa number conjecture of Bloch, Kato, Fontaine and Perrin-Riou for the Zeta function $\zeta(\X,s)$ at any $s\in\bz$ in terms of an integral fundamental line. In this section we recall this conjecture and describe how it can be simplified further.

\subsection{Globalisation of syntomic, additive syntomic and derived de Rham cohomology} The integral fundamental line is an invertible $\bz$-module given by the tensor product of determinants of perfect complexes of abelian groups closely related to the theories in the title of this section.

The definition of derived de Rham cohomology for arbitrary rings and schemes is well known and simply given by omitting $p$-completion in Def.  \ref{padicderiveddeRham}.

\begin{definition} (Hodge filtered derived de Rham cohomology)  Define $\fil^{\geq \star}_{Hod}\mathrm{dR}_{R/A}$ as the left Kan extension of the functor
$$\mathrm{Poly}_{\mathrm{Pairs}} \to \CAlg(\Fun(\bn^{op},\Mod_?));\quad\quad (R,A)\mapsto \Omega^{\geq \star}_{R/A}$$
along $\mathrm{Poly}_{\mathrm{Pairs}}\to\mathrm{Pairs}^{\ani}$ where $\Omega^\bullet_{R/A}$ is the algebraic de Rham complex and the target is the category of pairs $(A,F^{\geq\star}M)$ with $A$ an animated ring and $F^{\geq\star}M\in \CAlg(\Fun(\bn^{op},\Mod_A))$. For a map of derived schemes $X\to\Spec^{der}(A)$ define $\fil^{\geq \star}_{Hod}\mathrm{dR}_{X/A}$ by Zariski descent.
\label{deriveddeRham}\end{definition}

Additive syntomic cohomology is a new integral structure on rational Hodge truncated derived de Rham cohomology which we define here only for the absolute base $A=\bz$.

\begin{definition} (Additive syntomic cohomology) For a derived scheme $X$ define $R\Gamma_{\add}(X/\bz,\bz(n))$ by the pullback square
$$\xymatrix@C+50pt{R\Gamma_{\add}(X/\bz,\bz(n))\ar[r]\ar[d] & \mathrm{dR}_{X/\bz}^{<n}[-1]\ar[d]\\ \prod\limits_pR\Gamma_{\add}(X^\wedge_p/\bz_p,\bz_p(n))
\ar[r]^(.55){\prod\limits_p\gamma^{dR,<n}_{\prism, X^\wedge_p/\bz_p}\{n\}[-1]} &
\prod\limits_p \mathrm{dR}_{(X^\wedge_p)/\bz_p}^{<n}[-1]}$$
\end{definition}

\begin{remark} For $X$ such that $L_{X/\bz}$ has tor-amplitude in $[-1,0]$ $R\Gamma_{\add}(X/\bz,\bz(n))$ was first defined in \cite{Morin20} as the associated graded of a motivic filtration on $TC^+(X):=THH(X)_{S^1}$ (up to shift). There the alternative notation $\mathrm{dR}_{X/\bs}^{<n}[-1]$ is used for $R\Gamma_{\add}(X/\bz,\bz(n))$.
\end{remark} 

\begin{lemma} Let $X$ be a quasicompact, quasiseparated derived scheme.

a) One has an isomorphism $R\Gamma_{\add}(X/\bz,\bz(n))_\bq\simeq \mathrm{dR}_{X/\bz}^{<n}[-1]_\bq$.

b) If $X/\bz$ is proper of finite tor-amplitude and $L_{X/\bz}\in\QCoh(X)$ is perfect then  $\mathrm{dR}_{X/\bz}^{<n}[-1]$ and $R\Gamma_{\add}(X/\bz,\bz(n))$ are perfect complexes of abelian groups.
\end{lemma}

\begin{proof} Part a) is immediate from Lemma \ref{gammaiso}. For part b) note that $\mathrm{dR}_{X/\bz}^{<n}$ has a finite filtration with graded pieces
$R\Gamma(X,L\bigwedge{}^i L_{X/\bz})[-i]$ which are perfect over $\bz$ by our assumptions on $X$. It follows from the filtrations (\ref{addfilt}) and (\ref{nygaardfilt}) that $R\Gamma_{\add}(X^\wedge_p/\bz_p,\bz_p(n))$ is $\bz_p$-perfect for any prime $p$. Together with a) this implies perfectness of $R\Gamma_{\add}(X/\bz,\bz(n))$ over $\bz$.
\end{proof}

Syntomic cohomology globalizes to \'etale motivic cohomology. With any reasonable definition of the motivic complex $\bz(n)$ one has an isomorphism of \'etale sheaves $\bz(n)/p^\nu\simeq \mu_{p^\nu}^{\otimes n}$ on schemes over $\bz[\frac{1}{p}]$. This motivates the following definition of global syntomic cohomology given in \cite{bhatt-lurie-22}[Construction 8.4.1].

\begin{definition} (Global syntomic cohomology a.k.a. $p$-adic \'etale motivic cohomology) For a derived scheme $X$, prime $p$ and $n\in\bz$ define $R\Gamma(X,\bz_p(n))$ by the fibre product 
$$\xymatrix@C+50pt{R\Gamma(X,\bz_p(n))\ar[r]\ar[d] & R\Gamma(X[\frac{1}{p}]_{\et},\bz_p(n))\ar[d]\\ R\Gamma_{\syn}(X^\wedge_p/\bz_p,\bz_p(n))
\ar[r]^{\gamma^{\et}_{\syn}\{n\}} &
R\Gamma(X^\wedge_p[\frac{1}{p}]_{\et},\bz_p(n))}$$
where $\gamma^{\et}_{\syn}\{n\}$ is the \'etale comparison map of \cite{bhatt-lurie-22}[Thm. 8.3.1].
\end{definition}

In order to globalize further and remove the $p$-completion from coefficients we resort to the most classical definition of motivic cohomology which we also use in \cite{Flach-Morin-16}.

\begin{definition} (Motivic cohomology via higher Chow complexes) For a regular scheme $X$ and $n\geq 0$ we consider Bloch's cycle complex 
$$\mathbb{Z}(n):=z^n(-,2n-*)$$ as a complex of sheaves
on the small \'etale topos $X_{\et}$ of the scheme $X$ (see \cite{bloch86}, \cite{Levine01}, \cite{Levine99}, \cite{Geisser04a}). For $n<0$ we define the complex $\mathbb{Z}(n)$ on $X_{\et}$ by
$$\mathbb{Z}(n):=\bigoplus_{p}j_{p,!}(\mu_{p^{\infty}}^{\otimes n})[-1]$$
where $j_p$ is the open immersion $j_p:X[1/p]\rightarrow X$.
\end{definition}

The following result can be proven in slightly greater generality but is sufficient for the discussion in this section.

\begin{prop} Let $X$ be a scheme smooth over a Dedekind ring or a field. For any prime $p$ and $n\in\bz$ there is an isomorphism 
$$ R\Gamma(X_{\et},\bz(n))^\wedge_p\simeq R\Gamma(X,\bz_p(n)).$$
\label{pcomp}\end{prop}

\begin{proof} First assume $n\geq 0$. Then the case of schemes over a field is well known \cite{Geisser04a}[Sec. 5] \cite{bms19}[Sec. 8]. For schemes over Dedekind rings one follows the proof of \cite{Geisser04a}[Thm. 1.3] with \cite{bhatt-mathew}[Thm. 1.8] as the key new ingredient. The case $n<0$ follows from \cite{bhatt-lurie-22}[Example 8.4.5]. \end{proof}

While Prop. \ref{pcomp} provides a complex of abelian groups whose $p$-completion identifies with global syntomic cohomology for all primes $p$ this complex is not in general $\bz$-perfect for arithmetic schemes. The following optimistic expectation generalizes classical finiteness results in number theory such as finite generation of units and class groups ($d=1$, $n=1$, $i=1,2$) or the Mordell-Weil theorem ($d=2$, $n=1$, $i=2$).

\begin{conjecture} ${\bf L}(\mathcal{X}_{\et},n)$: 
Let $F$ be a number field and $\X/\co_F$ a smooth projective scheme of Krull dimension $d$. Then the group $H^i(\mathcal{X}_{\et},\mathbb{Z}(n))$ is finitely generated for $n\in\bz$ and $i\leq 2n+1$.
\label{lxn}\end{conjecture}

If in the situation of ${\bf L}(\mathcal{X}_{\et},n)$ one also assumes conjecture ${\bf L}(\mathcal{X}_{\et},d-n)$ one can show that for $i\geq 2n+1$ there is an isomorphism
$$H^i(\mathcal{X}_{\et},\mathbb{Z}(n))\simeq (\bq/\bz)^{r_{i,n}}\oplus F_{i,n}$$ 
for integers $r_{i,n}$ and finite groups $F_{i,n}$. So if $R\Gamma(\mathcal{X}_{\et},\mathbb{Z}(n))$ is cohomologically bounded there {\em could} be a perfect complex of abelian groups whose $p$-completion is identical to that of $R\Gamma(\mathcal{X}_{\et},\mathbb{Z}(n))$ for all primes $p$. Cohomological boundedness fails for the $2$-primary part if $\X$ has real points but this can be rectified by working with motivic complexes on the Artin-Verdier \'etale topos $\overline{\X}_{\et}$ as we show in \cite{Flach-Morin-16}[App. A]. We then prove the following

\begin{prop} Let $F$ be a number field and $\X/\co_F$ a smooth projective scheme of Krull dimension $d$ so that ${\bf L}(\mathcal{X}_{\et},n)$ and ${\bf L}(\mathcal{X}_{\et},d-n)$ hold true for some $n\in\bz$. Then there exists a perfect complex of abelian groups $R\Gamma_W(\overline{\X},\bz(n))$ and a natural isomorphism
\[ R\Gamma_W(\overline{\X},\bz(n))^\wedge_p\simeq R\Gamma(\X,\bz_p(n)) \]
for odd primes $p$ and an exact triangle
\[ R\Gamma_{\X_\infty}(\overline{\X},\bz_2(n))\to R\Gamma_W(\overline{\X},\bz(n))^\wedge_2\to R\Gamma(\X,\bz_2(n)) \]
where $R\Gamma_{\X_\infty}(\overline{\X},\bz_2(n))$ has finite 2-torsion cohomology groups and is bounded below.
Moreover there is an exact triangle of perfect complexes of abelian groups
\begin{equation}\label{triangle-cpctsupp-fgcoh}
R\Gamma_{W,c}(\mathcal{X},\mathbb{Z}(n))\longrightarrow R\Gamma_{W}(\overline{\mathcal{X}},\mathbb{Z}(n))
\stackrel{i_{\infty}^*}{\longrightarrow}R\Gamma_W(\mathcal{X}_{\infty},\mathbb{Z}(n))
\notag\end{equation}
where $R\Gamma_W(\mathcal{X}_{\infty},\mathbb{Z}(n))$ is defined in \cite{Flach-Morin-16}[Def. 3.23] and satisfies
$$R\Gamma_W(\mathcal{X}_{\infty},\mathbb{Z}(n))_\bq\simeq R\Gamma(\X(\bc),(2\pi i)^n\bq)^{\Gal(\bc/\br)}.$$ 

\end{prop}

\begin{proof} This is \cite{Flach-Morin-16}[Def. 3.6, Prop. 3.8, Cor. 6.8] and \cite{Flach-Morin-16}[Def. 3.26]. 
\end{proof}

\subsection{Special value conjectures} The following version of the fundamental fibre square was proven in \cite{Flach-Morin-16}[Prop. 7.21]. It uses work in syntomic cohomology \cite{colniz15}, \cite{ertlniziol16} which predates the introduction of prismatic cohomology.

\begin{prop} Let $F$ be a number field and $\X/\co_F$ a smooth projective scheme. For any prime $p$ there is a fibre sequence in $\Mod_{\bq_p}$
\begin{equation}\mathrm{dR}_{\mathcal{X}_{\mathbb{Q}_p}/\mathbb{Q}_p}^{<n}[-1]\rightarrow R\Gamma(\mathcal{X}_{\mathbb{Z}_p,\et},\bz(n))^\wedge_{p,\bq_p}\rightarrow R\Gamma(\mathcal{X}^{\mathrm{red}}_{\mathbb{F}_p,\et},\bz(n))^\wedge_{p,\bq_p}\label{oldbeil}\end{equation}
where we denote by $\mathcal{Y}^{\mathrm{red}}$ the reduction of a scheme $\mathcal{Y}$.
\label{beilseq}\end{prop}

In \cite{Flach-Morin-16}[Def. 5.6] we define
$$c_{p}(\mathcal{X},n):=p^{\chi(\mathcal{X}_{\mathbb{F}_p},\mathcal{O},n)}\cdot d_{p}(\mathcal{X},n)\in p^\bz$$
where
$$\chi(\mathcal{X}_{\mathbb{F}_p},\mathcal{O},n):= \sum_{i\leq n, j}(-1)^{i+j} \cdot (n-i)\cdot \mathrm{dim}_{\mathbb{F}_p}H^j_{Zar}(\mathcal{X}^{\mathrm{red}}_{\mathbb{F}_p},\Omega^i)$$
and
\begin{align*}&\mydet_{\bz_p}R\Gamma(\mathcal{X}_{\mathbb{Z}_p,\et},\bz(n))^\wedge_{p}\\=d_{p}(\mathcal{X},n)\,\cdot\,&\mydet_{\bz_p}^{-1}\mathrm{dR}_{\mathcal{X}_{\bz_p}/\mathbb{Z}_p}^{<n}\,\cdot\,\mydet_{\bz_p}R\Gamma(\mathcal{X}^{\mathrm{red}}_{\mathbb{F}_p,\et},\bz(n))^\wedge_{p}\end{align*}
under the rational isomorphism induced by (\ref{oldbeil}). It was shown in \cite{Flach-Morin-16}[Prop. 5.10] that $c_{p}(\mathcal{X},n)=1$ for almost all primes $p$. Hence we can define
\begin{equation}C(\mathcal{X},n):=\prod_{p<\infty} c_p(\mathcal{X},n) \in\bq^\times.\label{cdef}\end{equation}

\begin{definition}\label{deltadef}(The fundamental line) Let $F$ be a number field, $\X/\co_F$ a smooth projective scheme of Krull dimension $d$ and $n\in\bz$ so that ${\bf L}(\mathcal{X}_{\et},n)$ and ${\bf L}(\mathcal{X}_{\et},d-n)$  hold true. Define the fundamental line 
$$\Delta(\mathcal{X},n):=\mathrm{det}_{\mathbb{Z}}R\Gamma_{W,c}(\mathcal{X},\mathbb{Z}(n))
\otimes_{\mathbb{Z}}\mathrm{det}_{\mathbb{Z}}\,\mathrm{dR}_{\mathcal{X}/\mathbb{Z}}^{<n}.$$
\end{definition}

It was shown in \cite{Flach-Morin-16}[Prop. 5.2] that if $\mathcal{X}$ satisfies Conjecture ${\bf B}(\mathcal{X},n)$ as stated in \cite{Flach-Morin-16}[Conj. 2.5] (a version of Beilinson's conjecture relating motivic and Deligne cohomology) then there is a canonical trivialization 
$$\lambda_{\infty}(\mathcal{X},n):\mathbb{R} \stackrel{\sim}{\longrightarrow} \Delta(\mathcal{X},n)\otimes_{\mathbb{Z}}\mathbb{R}.$$

\begin{theorem} Let $F$ be a number field, $\X/\co_F$ a smooth projective scheme of Krull dimension $d$ and $n\in\bz$ so that ${\bf L}(\mathcal{X}_{\et},n)$, ${\bf L}(\mathcal{X}_{\et},d-n)$ and ${\bf B}(\mathcal{X},n)$ hold true. Assume $\zeta(\mathcal{X},s)$ has a meromorphic continuation to $s=n$ and denote  by $\zeta^*(\mathcal{X},n)\in\br$ its leading Taylor coefficient. Then the Tamagawa number conjecture \cite{fpr91} for the motivic structure $h(\X_F)(n)$ (for all primes $p$) is equivalent to the identity
$$\lambda_{\infty}(\zeta^*(\mathcal{X},n)^{-1}\cdot C(\mathcal{X},n)\cdot\mathbb{Z})= \Delta(\mathcal{X},n).
$$
\label{tncthm}\end{theorem}

\begin{proof} This is \cite{Flach-Morin-16}[Thm. 5.27]. \end{proof}

\begin{remark} The reader who is uncomfortable making the rather serious assumptions of Thm. \ref{tncthm} (which are of course also inherent to \cite{bk88}, \cite{fpr91}) may restrict to the example $\X=\Spec(\co_F)$ and any $n\in\bz$ where all assumptions are known \cite{Flach-Morin-16}[Sec. 5.8.3]. \end{remark}

The usefulness of Thm. \ref{tncthm} is limited by the inexplicit nature of the rational factor $C(\mathcal{X},n)$. In \cite{Flach-Morin-16} we were only able to show $C(\mathcal{X},n)=1$ for $n\leq 0$ and $C(\Spec(\co_F),n)=(n-1)!^{-[F:\bq]}$ for $n\geq 1$ if $F$ is a number field all of whose completions $F_v$ are absolutely abelian \cite{Flach-Morin-16}[Prop. 5.3.4]. With the results of the present paper we can show the following result.

\begin{theorem}  Let $F$ be a number field and $\X/\co_F$ a smooth projective scheme. Define
\begin{equation}\notag
C_\infty(\X,n):=\prod_{ i\leq n-1;\, j}(n-1-i)!^{(-1)^{i+j}\mathrm{dim}_{\bq}H^j(\X_{\bq},\Omega^i)}.
\end{equation}
Then for $n\geq 1$
\begin{equation}\notag
C(\X,n)=C_\infty(\X,n)^{-1}. 
\end{equation}
\label{correctionfactor}\end{theorem}

\begin{proof} We fix a prime $p$ and verify the $p$-primary part of Thm. \ref{correctionfactor}. In the following all identities are understood up to elements of $\bz_p^\times$. 
\begin{lemma} There are isomorphisms 
\begin{align*}\widehat{\mathrm{dR}}_{\mathcal{X}^\wedge_p/\mathbb{Z}_p}^{<n}\simeq &\mathrm{dR}_{\mathcal{X}_{\bz_p}/\mathbb{Z}_p}^{<n}\\
R\Gamma_{\syn}(\mathcal{X}^\wedge_p,\bz_p(n))\simeq &R\Gamma(\mathcal{X}_{\mathbb{Z}_p,\et},\bz(n))^\wedge_{p}\\
R\Gamma_{\syn}((\mathcal{X}^\wedge_p/p)^{\mathrm{red}},\bz_p(n))\simeq &R\Gamma(\mathcal{X}^{\mathrm{red}}_{\mathbb{F}_p,\et},\bz(n))^\wedge_{p}\end{align*}
such that (\ref{oldbeil}) is isomorphic to the fibre sequence (\ref{fundamental-fib-seq}) arising from the de Rham logarithm. 
\end{lemma}

\begin{proof} Since $\mathcal{X}_{\bz_p}/\mathbb{Z}_p$ is proper the complex $\mathrm{dR}_{\mathcal{X}_{\bz_p}/\mathbb{Z}_p}^{<n}$ is $\bz_p$-perfect and hence coincides with its $p$-completion. Therefore
$$\mathrm{dR}_{\mathcal{X}_{\bz_p}/\mathbb{Z}_p}^{<n}\simeq \widehat{\mathrm{dR}}_{\mathcal{X}_{\bz_p}/\mathbb{Z}_p}^{<n}\simeq \widehat{\mathrm{dR}}_{\mathcal{X}^\wedge_p/\mathbb{Z}_p}^{<n}$$
where the second isomorphism follows from $p$-completeness. 
The second isomorphism in the Lemma is a consequence of Prop. \ref{pcomp} since $\mathcal{X}_{\mathbb{Z}_p}$ is smooth over the product of Dedekind rings $(\co_F)_{\bz_p}$. One also needs to use the isomorphism of \'etale cohomology with $\mu_{p^\nu}^{\otimes n}$ coefficients of the scheme $\X_{\bq_p}$ on the one hand and the rigid analytic variety $\X^\wedge_p[\frac{1}{p}]$ on the other. The third isomorphism is also a consequence of Prop. \ref{pcomp} since $\mathcal{X}^{\mathrm{red}}_{\mathbb{F}_p}$ is smooth over $\bF_p$. The proof of Lemma \ref{mulemma} below shows that the natural map induces an isomorphism
$$ R\Gamma_{\syn}((\mathcal{X}^\wedge_p/p)^{\mathrm{red}},\bz_p(n))_{\bq_p}\simeq R\Gamma_{\syn}(\mathcal{X}^\wedge_p/p,\bz_p(n))_{\bq_p}$$
and that therefore the fibre sequence (\ref{fundamental-fib-seq}) induces a fibre sequence with the same terms as in (\ref{oldbeil}). Like the de Rham logarithm, the logarithm map implicit in (\ref{oldbeil}) is induced by the cristalline-to-de Rham comparison map, as is clear from (128) in \cite{Flach-Morin-16}[Prop. 7.21]. Therefore we can identify the two fibre sequences.
\end{proof}

Using Thm. \ref{thm:synlogproper} we have 
\begin{align*}&\mydet_{\bz_p}R\Gamma_{\syn}(\mathcal{X}^\wedge_p,\bz_p(n))\\
\simeq &\mydet_{\bz_p}R\Gamma_{\syn}(\mathcal{X}^\wedge_p,\bz_p(n))^{\rel}\otimes\mydet_{\bz_p}R\Gamma_{\syn}(\mathcal{X}^\wedge_p/p,\bz_p(n))\\
\simeq &\mydet_{\bz_p}^{-1}\widehat{\mathrm{dR}}_{\mathcal{X}^\wedge_p/\mathbb{Z}_p}^{<n,\rel}\cdot C_\infty(\X,n)^{-1} \otimes\mydet_{\bz_p}R\Gamma_{\syn}(\mathcal{X}^\wedge_p/p,\bz_p(n))\\
\simeq &\mydet_{\bz_p}^{-1}\widehat{\mathrm{dR}}_{\mathcal{X}^\wedge_p/\mathbb{Z}_p}^{<n}\cdot \mydet_{\bz_p}\widehat{\mathrm{dR}}_{(\mathcal{X}^\wedge_p/p)/\mathbb{Z}_p}^{<n}\cdot C_\infty(\X,n)^{-1} \otimes\mydet_{\bz_p}R\Gamma_{\syn}(\mathcal{X}^\wedge_p/p,\bz_p(n))
\end{align*}
Define $\mu_{\syn},\mu_{\add}\in p^\bz$ by
\begin{align*}\mydet_{\bz_p}R\Gamma_{\syn}(\mathcal{X}^\wedge_p/p,\bz_p(n))=&\mu_{\syn}\cdot\mydet_{\bz_p}R\Gamma_{\syn}((\mathcal{X}^\wedge_p/p)^{\mathrm{red}},\bz_p(n))\\
\mydet_{\bz_p}\widehat{\mathrm{dR}}_{(\mathcal{X}^\wedge_p/p)/\mathbb{Z}_p}^{<n}=&\mu^{-1}_{\add}\cdot\mydet_{\bz_p}\widehat{\mathrm{dR}}_{(\mathcal{X}^\wedge_p/p)^{\mathrm{red}}/\mathbb{Z}_p}^{<n}.
\end{align*}
Then we have
\begin{align*}d_{p}(\mathcal{X},n)=&\mydet_{\bz_p}\widehat{\mathrm{dR}}_{(\mathcal{X}^\wedge_p/p)/\mathbb{Z}_p}^{<n}\cdot C_\infty(\X,n)^{-1}\cdot\mu_{\syn}\\
=&\mu^{-1}_{\add}\cdot\mydet_{\bz_p}\widehat{\mathrm{dR}}_{(\mathcal{X}^\wedge_p/p)^{\mathrm{red}}/\mathbb{Z}_p}^{<n}\cdot C_\infty(\X,n)^{-1}\cdot\mu_{\syn}
\notag\end{align*}
and since $(\mathcal{X}^\wedge_p/p)^{\mathrm{red}}=\X_{\bF_p}^{\mathrm{red}}$ is smooth projective over $\bF_p$ the main theorem of \cite{Morin15} gives
$$ \mydet_{\bz_p}\widehat{\mathrm{dR}}_{(\mathcal{X}^\wedge_p/p)^{\mathrm{red}}/\mathbb{Z}_p}^{<n}=\chi^\times\left(\widehat{\mathrm{dR}}_{(\mathcal{X}^\wedge_p/p)^{\mathrm{red}}/\mathbb{Z}_p}^{<n}\right)^{-1}=p^{-\chi(\mathcal{X}_{\mathbb{F}_p},\mathcal{O},n)}.$$
Hence we find
$$ c_{p}(\mathcal{X},n)=p^{\chi(\mathcal{X}_{\mathbb{F}_p},\mathcal{O},n)}\cdot d_{p}(\mathcal{X},n)=\mu^{-1}_{\add}\cdot C_\infty(\X,n)^{-1}\cdot\mu_{\syn}.$$
Theorem \ref{correctionfactor} then follows from the next Lemma.
\end{proof}

\begin{lemma}\label{mulemma} With notation introduced above one has $\mu_{\add}=\mu_{\syn}$. 
\end{lemma}
\begin{proof} We have $\co_F/p=\prod_{v\mid p}\co_{F_v}/p$ and
$$\mathcal{X}^\wedge_p/p\simeq \coprod_{v\mid p}\mathcal{X}_{\co_{F_v}/p},\quad (\mathcal{X}^\wedge_p/p)^{\mathrm{red}}\simeq \coprod_{v\mid p}\mathcal{X}_{\co_{F_v}/\varpi_v}$$ 
where $\varpi_v$ is a uniformizer of $\co_{F_v}$. With obvious notation there are factorizations $\mu_{\syn}=\prod_{v\mid p}\mu_{\syn,v}$ and $\mu_{\add}=\prod_{v\mid p}\mu_{\add,v}$ and we shall show $\mu_{\syn,v}= \mu_{\add,v}$ for all $v\mid p$.

Consider $\X_v:=\mathcal{X}_{\co_{F_v}/p}$ as a (formal) scheme with filtered structure sheaf $\Phi^{\geq\star}\co_{\mathcal{X}_v}$ given by the $\varpi_v$-adic filtration of finite length $e_v:=e(F_v/\bq_p)$. By Cor. \ref{globalfiltrations}
$$\mathrm{gr}^0_\Phi R\Gamma_{?}(\X_v/\bz_p,\bz_p(n))\simeq R\Gamma_{?}(\X_{\co_{F_v}/\varpi_v}/\bz_p,\bz_p(n))$$
for $?=\syn,\add$ and therefore $\mu_{?,v}=\chi^\times\Phi^{\geq 1} R\Gamma_{?}(\X_v/\bz_p,\bz_p(n))$
where in the case $?=\add$ we also use $\chi^\times(\fibre(\gamma^{dR,<n}_{\prism,\X_{\co_{F_v}/\varpi_v}} ))=1$ and $\chi^\times(\fibre(\gamma^{dR,<n}_{\prism,\X_v} ))=1$ as in the proof of Thm. \ref{thm:synlogproper} c).
By Prop. \ref{griso} we have for $i\geq 1$
\begin{equation}\mathrm{gr}^i_\Phi R\Gamma_{\syn}(\X_v/\bz_p,\bz_p(n))\simeq \mathrm{gr}^i_\Phi R\Gamma_{\add}(\X_v/\bz_p,\bz_p(n))\notag\end{equation}
and these complexes are torsion and $\bz_p$-perfect and vanish for all but finitely many $i$ by Prop. \ref{completeandbounded} b).  Therefore
\begin{align*} \mu_{\syn,v}=&\chi^\times\Phi^{\geq 1} R\Gamma_{\syn}(\X_v/\bz_p,\bz_p(n))=\prod_{i\geq 1}\chi^\times \mathrm{gr}^i_\Phi R\Gamma_{\syn}(\X_v/\bz_p,\bz_p(n)) \\
=&\prod_{i\geq 1}\chi^\times \mathrm{gr}^i_\Phi R\Gamma_{\add}(\X_v/\bz_p,\bz_p(n))=\chi^\times\Phi^{\geq 1} R\Gamma_{\add}(\X_v/\bz_p,\bz_p(n))=\mu_{\add,v} 
\end{align*}
concluding the proof.
\end{proof}

Replacing Hodge truncated derived de Rham cohomology by additive syntomic cohomology we can eliminate the correction factor $C(\X,n)$ entirely.

\begin{corollary} Let $F$ be a number field, $\X/\co_F$ a smooth projective scheme of Krull dimension $d$ and $n\in\bz$ so that ${\bf L}(\mathcal{X}_{\et},n)$, ${\bf L}(\mathcal{X}_{\et},d-n)$ and ${\bf B}(\mathcal{X},n)$ hold true. Assume $\zeta(\mathcal{X},s)$ has a meromorphic continuation to $s=n$ and denote  by $\zeta^*(\mathcal{X},n)\in\br$ its leading Taylor coefficient. Then the Tamagawa number conjecture \cite{fpr91} for the motivic structure $h(\X_F)(n)$ (for all primes $p$) is equivalent to the identity
$$\lambda_{\infty}(\zeta^*(\mathcal{X},n)^{-1}\cdot\mathbb{Z})= \Delta(\mathcal{X},n)^{\mathrm{new}}
$$
where
$$\Delta(\mathcal{X},n)^{\mathrm{new}}:=\mathrm{det}_{\mathbb{Z}}R\Gamma_{W,c}(\mathcal{X},\mathbb{Z}(n))
\otimes_{\mathbb{Z}}\mathrm{det}^{-1}_{\mathbb{Z}}\,R\Gamma_{\add}(\X/\bz,\bz(n)).$$
\label{tnccor}\end{corollary}

\begin{proof} It suffices to show
$$ \mathrm{det}^{-1}_{\mathbb{Z}}\,R\Gamma_{\add}(\X/\bz,\bz(n))=C_\infty(\X,n)\cdot \mathrm{det}_{\mathbb{Z}}\,\mathrm{dR}_{\mathcal{X}/\mathbb{Z}}^{<n}$$
which can be checked after $p$-completion for all primes $p$. The $p$-complete statement follows from the filtrations (\ref{addfilt}) and (\ref{nygaardfilt}) as in the proof of Thm. \ref{thm:synlogproper} c).
\end{proof}

\begin{corollary} Let $F$ be a number field, $\X/\co_F$ a smooth projective scheme of Krull dimension $d$ and $n\in\bz$ so that ${\bf L}(\mathcal{X}_{\et},n)$, ${\bf L}(\mathcal{X}_{\et},d-n)$ and ${\bf B}(\mathcal{X},n)$ hold true. Assume $\zeta(\mathcal{X},s)$ has a meromorphic continuation to $s=n$ and $s=d-n$ and satisfies the expected functional equation \cite{Flach-Morin-20}[Conj. 1.3]. Then the Tamagawa number conjecture \cite{fpr91} (at any prime $p$) holds true for the motivic structure $h(\X_F)(n)$ if and only if it holds true for the motivic structure $h(\X_F)(d-n)$. 
\label{funeqcor}\end{corollary}

\begin{proof} This follows immediately from Cor. \ref{tnccor} and \cite{Flach-Morin-20}[Thm. 1.4]. Note also that ${\bf B}(\mathcal{X},n)$ implies ${\bf B}(\mathcal{X},d-n)$ \cite{Flach-Morin-16}[Rem. 2.6]. \end{proof}

\begin{remark} \label{errata} (Errata to \cite{Flach-Morin-16}) We take this opportunity to correct a few typos in \cite{Flach-Morin-16}.
\begin{itemize}
\item[1)] In the definition of $C(\X,n)$ after Conj. 5.11 omit $|\ |_p$. The correct definition is (\ref{cdef}).
\item[2)] Replace $C(\X,n)$ by $C(\X,n)^{-1}$ in the displayed formula following the displayed formula defining $\gamma_p$ in the proof of Thm. 5.27.
\item[3)] In the definition of $\eta_V(\omega)$ in the proof of Prop. 5.34 replace $|\ |^{-1}_p$ by $|\ |_p$ (see \cite{pr95}[App. C.2.8]) so that $\eta_V(\omega)\sim D_K^{1-n}$. In 
\cite{Flach-Morin-16}[eq. (89)] replace $D_K^{n-1}$ by  $D_K^{1-n}$. 
\item[4)] In the last displayed formula of the proof of Prop. 5.34 omit $|\ |_p$ and replace $[F_v:\bq_p]$ by $-[F_v:\bq_p]$.
\end{itemize}
\end{remark}

\begin{theorem} Let $K/\bq_p$ be a finite extension with discriminant $D_K$ and residue field of cardinality $q$. Then for $n\geq 2$
$$\mydet_{\bz_p}\left(\exp_{\bq_p(n)}\co_K\right)\cdot (1-q^{-n})^{-1}\cdot (n-1)!^{[K:\bq_p]}\cdot D_K^{n-1}=\mydet_{\bz_p}^{-1}R\Gamma(K,\bz_p(n)) $$
inside $\mydet_{\bq_p}^{-1}R\Gamma(K,\bq_p(n))=\mydet_{\bq_p}H^1(K,\bq_p(n))$.
\label{bktheorem2} \end{theorem}

\begin{proof} We continue the argument of the introduction. As remarked there, for $n\geq 2$ all maps in (\ref{logandexp}) are isomorphisms and  \cite{bhatt-mathew}[Thm. 1.8] gives an isomorphism
$$ c_{\et}: R\Gamma_{\syn}(\Spf(\co_K),\bz_p(n))\simeq R\Gamma(K,\bz_p(n)).$$
So Thm. \ref{bktheorem2} amounts to the identity 
\begin{equation}\alpha=(1-q^{-n})^{-1}\cdot (n-1)!^{[K:\bq_p]}\cdot D_K^{n-1}\sim q^n\cdot (n-1)!^{[K:\bq_p]}\cdot D_K^{n-1}\label{localvolume}\end{equation}  
where $\alpha\in \bq_p^\times$ is such that 
$$\mydet_{\bz_p}\left(\co_K\right)\cdot\alpha=\mydet_{\bq_p}^{-1}(\log_{\co_K,\bq})\left(\mydet_{\bz_p}^{-1}R\Gamma_{\syn}(\Spf(\co_K),\bz_p(n))\right)$$
inside $\mydet_{\bq_p}K$. By Thm. \ref{main} c) the right hand side is equal to
\begin{align*}&\mydet_{\bq_p}^{-1}\log_{\co_K,\bq}\left(\mydet_{\bz_p}^{-1}R\Gamma_{\syn}(\Spf(\co_K),\bz_p(n))^{\rel}\right)\cdot \mydet_{\bz_p}^{-1}R\Gamma_{\syn}(\Spf(\co_K/p),\bz_p(n))\\
= &\mydet_{\bz_p}\widehat{\mathrm{dR}}^{<n,\rel}_{\co_K}\cdot  (n-1)!^{[K:\bq_p]}\cdot \mydet_{\bz_p}^{-1}R\Gamma_{\syn}(\Spf(\co_K/p),\bz_p(n))\\
= &\mydet_{\bz_p}\widehat{\mathrm{dR}}^{<n}_{\co_K}\cdot  (n-1)!^{[K:\bq_p]}\cdot \mydet_{\bz_p}^{-1}R\Gamma_{\syn}(\Spf(\co_K/p),\bz_p(n))\cdot\mydet^{-1}_{\bz_p}\widehat{\mathrm{dR}}^{<n}_{\co_K/p}
\end{align*}
Denote by $\kappa\simeq(\co_K/p)^\mathrm{red}$ the residue field of $K$. Lemma \ref{mulemma} and the fact that finite fields have trivial higher $p$-adic K-theory show that this last determinant equals
$$\mydet_{\bz_p}\widehat{\mathrm{dR}}^{<n}_{\co_K}\cdot  (n-1)!^{[K:\bq_p]}\cdot\mydet^{-1}_{\bz_p}\widehat{\mathrm{dR}}^{<n}_{\kappa}.$$
By  \cite{Flach-Morin-16}[Prop. 5.36] one has $\mydet_{\bz_p}\widehat{\mathrm{dR}}^{<n}_{\co_K}=\mydet_{\bz_p}\left(\co_K\right)\cdot D_K^{n-1}$ and the main theorem of \cite{Morin15} shows $\mydet^{-1}_{\bz_p}\widehat{\mathrm{dR}}^{<n}_{\kappa}=q^n\cdot\bz_p$. This finishes the proof of (\ref{localvolume}).
\end{proof}

\begin{remark} By Cor. \ref{fneqcomp} Thm. \ref{bktheorem2} proves Conjecture $C_{EP}(\bq_p(n))$ of \cite{pr95}[App. C.2.9]. Since the de Rham and the syntomic logarithm are functorial we expect that the equivariant version of Conjecture $C_{EP}(\bq_p(n))$, the local epsilon conjecture of Fukaya and Kato \cite{fk} for Tate motives \cite{flach-daigle-16}, also readily follows from our methods. Addressing Conjecture $C_{EP}(V)$ for more general $p$-adic representations $V$ would require to first extend the construction of the de Rham and the syntomic logarithm to filtered prismatic F-gauges, i.e. (certain) objects of   $\gsheaf_{\widehat{\ba}^1/\widehat{\bg}_m}$. We hope to come back to this question in a subsequent article.
\end{remark}
\section{Appendix A}

In this appendix we discuss generalities on $A$-modules in presentable stable $\infty$-categories and give the precise definition of $\gsheaf_\mathfrak{A}$.

\begin{definition} Let $\mathscr{M}$ be a presentable stable $\infty$-category tensored over $\Mod_\bz$. For any $A\in\CAlg_\bz$ we define the presentable stable category of $A$-modules in $\mathscr{M}$
$$ \mathscr{M}_A:= \mathscr{M}\otimes_{\Mod_\bz}\Mod_A$$
where the relative tensor product is discussed for example in \cite{lurieSAG}[App. D.2.1].
\end{definition}

Clearly, the definition of $\mathscr{M}_A$ can be further generalized to $\mathscr{M}\otimes_{\Mod_\bz}\mathscr{N}$ for any presentable stable $\infty$-category $\mathscr{N}$  tensored over $\Mod_\bz$. We will only be interested in the example
\begin{equation}\mathscr{N}= \cd(\mathfrak{A})\simeq \varprojlim_{\Delta}\cd(A^\bullet)\label{stackexample}\end{equation}
for a formal $\delta$-stack $\mathfrak{A}$ in the sense of Def. \ref{deltastackdef}. Here the second isomorphism follows from faithfully flat descent \cite{lurieSAG}[Cor. D.6.3.3]. In this example we use the notation 
$$\mathscr{M}_\mathfrak{A}:=\mathscr{M}\otimes_{\Mod_\bz} \cd(\mathfrak{A}).$$

Recall from \cite{akn23}[Def. A.9] that $\mathscr{M}$ is called $p$-complete if 
$$\mathscr{M}\to\mathscr{M}^\wedge_p:=\mathscr{M}\otimes_{\Mod_\bz}\cd(\bz)$$ 
is an equivalence where $\cd(\bz)=\cd(\Spf(\bz))\simeq\varprojlim_n\Mod_{\bz/p^n}$ is the category of $p$-complete objects of $\Mod_\bz$.

\begin{lemma} Let $\mathscr{M}$ be a presentable stable $\infty$-category tensored over $\Mod_\bz$.

a) For any homomorphism $A\to A'$ in $\CAlg_\bz$ there is a pair of adjoint functors 
$$ -\otimes_AA': \mathscr{M}_A\rightleftarrows\mathscr{M}_{A'}:\mathrm{Res}^{A'}_A.$$

b) Assume $X\in\mathscr{M}$ is an $A$-module in the sense that there is a morphism of $\be_1$-$\bz$-algebras 
$$ a:A\to\End_{\mathscr{M}}(X).$$
Then $X$ gives rise to an object of $\mathscr{M}_A$ and conversely, any object of $\mathscr{M}_A$ restricts to an $A$-module in $\mathscr{M}$.

c) If $\mathscr{M}$ is $p$-complete then $\mathscr{M}_A\simeq \mathscr{M}\otimes_{\Mod_\bz}\cd(A)$ where $\cd(A)=(\Mod_A)^\wedge_p$ is the category of $p$-complete objects of $\Mod_A$.

d) If $\mathscr{M}$ is compactly assembled (in particular if $\mathscr{M}$ is compactly generated) and $\mathfrak{A}$ is as in (\ref{stackexample}) we have
$$  \mathscr{M}_{\mathfrak{A}} \simeq  \varprojlim_{\Delta} \mathscr{M}_{A^\bullet}.$$

e) If $\mathscr{M}\in\CAlg(\Mod_{\Mod_\bz}(\mathcal{P}r^L))$ then $\mathscr{M}_A\in\CAlg(\Mod_{\Mod_\bz}(\mathcal{P}r^L))$.
\label{DAgeneralities}\end{lemma}

\begin{proof} a) The left adjoint $-\otimes_AA'=\id_\mathscr{M}\otimes (-\otimes_AA')$ is induced by the corresponding functor $-\otimes_AA':\Mod_A\to \Mod_{A'}$ and functoriality of the tensor product in $\Mod_{\Mod_\bz}(\mathcal{P}r^L)$. The right adjoint exists for any functor in $\Mod_{\Mod_\bz}(\mathcal{P}r^L)$.

b) We may view $a$ as a $\Mod_\bz$ enriched functor $Ba:BA\to\mathscr{M}$ where $BA$ is the category with a single object which has endomorphism algebra $A$. Since $\Mod_A$ is freely generated under colimits by $A$ the functor $Ba$ extends uniquely to a $\Mod_\bz$-linear functor $F:\Mod_A\to\mathscr{M}$. We have 
$$\Fun_{\Mod_\bz}(\Mod_A,\mathscr{M})\simeq \Mod_A\otimes_{\Mod_\bz} \mathscr{M}=\mathscr{M}_A$$ 
since $\Mod_A$ is dualizable with dual $\Mod_{A^{op}}=\Mod_A$. Conversely, any object of $\mathscr{M}_A$ gives rise to such a functor $F$ which we can evaluate on $A\in\Mod_A$ to obtain $a$.

c) Since $\cd(\bz)=(\Mod_\bz)^\wedge_p$ is an idempotent algebra in $\Mod_{\Mod_\bz}(\mathcal{P}r^L)$ \cite{akn23}[Rem. A.8] we have
$$ \mathscr{M}_A= \mathscr{M}\otimes_{\Mod_\bz}\Mod_A\simeq  \mathscr{M}\otimes_{\Mod_\bz}\cd(\bz)\otimes_{\Mod_\bz}\Mod_A\simeq \mathscr{M}\otimes_{\Mod_\bz}\cd(A).$$

d) If $\mathscr{M}$ is compactly assembled it is dualizable in  $\Mod_{\Mod_\bz}(\mathcal{P}r^L)$ \cite{lurieSAG}[Thm. 7.0.7] and therefore
\begin{align*}\mathscr{M}_\mathfrak{A}=&\mathscr{M}\otimes_{\Mod_\bz} \cd(\mathfrak{A}) \simeq \Fun_{\Mod_\bz}(\mathscr{M}^\vee,\varprojlim_{\Delta}\cd(A^\bullet))\\\simeq &\varprojlim_{\Delta}  \Fun_{\Mod_\bz}(\mathscr{M}^\vee,\cd(A^\bullet))
\simeq \varprojlim_{\Delta}\mathscr{M}\otimes_{\Mod_\bz} \cd(A^\bullet)\simeq \varprojlim_{\Delta}\mathscr{M}_{A^\bullet}.
\end{align*}

e) This follows since $\Mod_A\in\CAlg(\Mod_{\Mod_\bz}(\mathcal{P}r^L))$.
\end{proof}

If $A$ is a $\delta$-ring and $\mathfrak{A}=\Spf(A)$ then the following definition makes precise Def. \ref{fgauges}. Denote by $\cdf_\bn\subseteq\cdf$ the full subcategory of $\bn^{op}$-indexed filtrations.

\begin{definition} (Lax prismatic F-gauges with coefficients in a formal $\delta$-stack) Let $\mathfrak{A}$ be a formal $\delta$-stack in the sense of Def. \ref{deltastackdef}. Define
\begin{align*}  \tau_1: &\cdf_\bn(\fil^{\geq\star}_{\mathcal N}F^*\co_{\mathrm{WCart}}) _\mathfrak{A}\xrightarrow{-\otimes_{ \fil^{\geq\star}_{\mathcal N}F^*\co_{\mathrm{WCart}},\Phi}\mathscr{I}^\star\co_{\mathrm{WCart}} } \cdf(\mathscr{I}^\star\co_{\mathrm{WCart}}) _\mathfrak{A} \\
\tau_2: &\cdf(\mathscr{I}^\star\co_{\mathrm{WCart}}) _\mathfrak{A} \xrightarrow{\mathscr{I}^\star\mathscr{E}\mapsto \mathscr{I}^0\mathscr{E}=:\mathscr{E}}
\cd(\mathrm{WCart}) _\mathfrak{A}\\
\varphi^*_{\mathfrak{A}}: &\cd(\mathrm{WCart}) _\mathfrak{A}\to \cd(\mathrm{WCart}) _\mathfrak{A}\\
(F^*,\rho_{dR}^*): &\cd(\mathrm{WCart}) _\mathfrak{A}\to \cd(\mathrm{WCart}) _\mathfrak{A} \times_{\iota^*,\mathcal{D}(\mathrm{WCart}^{\mathrm{HT}}) _\mathfrak{A},\pi^*} \mathcal{D}(\mathfrak{A})
\end{align*}
where the last functor is induced by the commutative square (\ref{wcart}). Define
\begin{align*}\fil\iota^*: &\cdf_\bn(\fil^{\geq\star}_{\mathcal N}F^*\co_{\mathrm{WCart}}) _\mathfrak{A}\to \cdf_\bn(\mathrm{WCart}^{\mathrm{HT}}) _\mathfrak{A},\ \fil^{\geq\star}_{\mathcal N}\mathscr{E}'\mapsto \fil^{\geq\star}_{\mathcal N}\mathscr{E}'/\mathscr{I}\fil^{\geq\star-1}_{\mathcal N}\mathscr{E}'\\
\fil\pi^*: &\cdf_\bn(\mathfrak{A})\to \cdf_\bn(\mathrm{WCart}^{\mathrm{HT}}) _\mathfrak{A},\ \fil^{\geq\star}_{H}D\mapsto \fil^{\geq\star}_{H}D\otimes\co_{\mathrm{WCart}^{\mathrm{HT}}} 
\end{align*}
and
$$\mathscr{G}^{\mathcal N} _\mathfrak{A}:=\cdf_\bn(\fil^{\geq\star}_{\mathcal N}F^*\co_{\mathrm{WCart}}) _\mathfrak{A} \times_{\fil\iota^*,\cdf_\bn(\mathrm{WCart}^{\mathrm{HT}}) _\mathfrak{A},\fil\pi^*} \cdf_\bn(\mathfrak{A})$$ 
$$\fil^{\geq 0}: \mathscr{G}^{\mathcal N} _\mathfrak{A}\to \cd(\mathrm{WCart}) _\mathfrak{A} \times_{\iota^*,\mathcal{D}(\mathrm{WCart}^{\mathrm{HT}}) _\mathfrak{A},\pi^*} \mathcal{D}(\mathfrak{A}).$$
Let $\gsheaf _\mathfrak{A}$ be the lax equalizer of the ordered pair of functors
$$(F^*,\rho_{dR}^*)\,\varphi^*_{\mathfrak{A}}\,\tau_2\,\tau_1\,\mathrm{pr}_1,\  \fil^{\geq 0}:\mathscr{G}^{\mathcal N} _\mathfrak{A}\to \cd(\mathrm{WCart}) _\mathfrak{A} \times_{\iota^*,\mathcal{D}(\mathrm{WCart}^{\mathrm{HT}}) _\mathfrak{A},\pi^*} \mathcal{D}(\mathfrak{A}).$$
Here the lax equalizer $\mathrm{Leq}(F,G)$ of an ordered pair of functors $F,G:\mathscr{C}\to\mathscr{D}$ is the fibre product
$$\begin{CD}\mathrm{Leq}(F,G) @>>> \mathscr{C}\\
@VVV @VV (F,G) V\\
\Fun(\Delta^1,\mathscr{D}) @>(s,t)>> \mathscr{D}\times\mathscr{D}.
\end{CD}
$$
\label{rigorousfgauges}\end{definition} 

\begin{lemma} a)  Let $A^\bullet$ be a cosimplicial $\delta$-ring with flat differentials such that $\Spf(A^\bullet)$ is a groupoid object in the $\infty$-topos of formal stacks and put
\begin{equation} \mathfrak{A} \simeq \varinjlim_{\Delta^{op}}\Spf(A^\bullet).
\notag\end{equation}
Then the natural functor
$$\gsheaf _\mathfrak{A}\to \varprojlim_\Delta\gsheaf_{A^\bullet}=\Tot \gsheaf_{A^\bullet}$$
is an equivalence. In particular, $\gsheaf_A$ satisfies descent for fpqc-covers $A\to A^0$ of $\delta$-rings.

b) In the situation of a) the functor $\gsheaf _\mathfrak{A}\to\gsheaf _{A^0}$ is conservative.

\label{omnibus}\end{lemma}

\begin{proof} a) The limit $\varprojlim_\Delta$ commutes with the fibre products in the definition of $\gsheaf _\mathfrak{A}$ and with the functor category $\Fun(\Delta^1,\mathscr{C})$ since the latter is a limit over a constant diagram with value $\mathscr{C}$ indexed by $\Delta^1$ \cite{lurieHTT}[Cor. 3.3.3.2]. We are therefore reduced to show the equivalence
$$  \mathscr{M}_{\mathfrak{A}} \simeq  \varprojlim_{\Delta} \mathscr{M}_{A^\bullet}$$
for $\mathscr{M}$ any of the following categories:
\begin{align}&\cd(\mathrm{WCart}), \mathcal{D}(\mathrm{WCart}^{\mathrm{HT}}), \mathcal{D}(\bz_p),\label{list}\\ 
&\cdf_\bn(\fil^{\geq\star}_{\mathcal N}F^*\co_{\mathrm{WCart}}),\cdf_\bn(\mathrm{WCart}^{\mathrm{HT}}), \cdf_\bn(\bz_p)
\notag\end{align}
This follows from  Lemma \ref{cg} and Lemma \ref{DAgeneralities} d), noting that $\cdf_\bn(\mathfrak{X})\simeq\Fun(\bn^{op},\cd(\mathfrak{X}))$.

b) This is a general property of totalizations.
\end{proof}

\begin{lemma}\label{cg} a) The categories
$$\mathcal{D}(\mathrm{WCart}^{\mathrm{HT}}),\ \cd(\mathrm{WCart}),\ \mathcal{D}(\bz_p)$$
are compactly generated.

b) If $\mathscr{C}$ is compactly generated and $I$ is a small category then $\Fun(I,\mathscr{C})$ is compactly generated.

c) If $\mathscr{C}$ is compactly generated and symmetric monoidal with tensor product preserving colimits in both variables, and $A\in\CAlg(\mathscr{C})$ then $\Mod_A(\mathscr{C})$ is compactly generated.
\end{lemma}

\begin{proof} a) Put $\overline{\co}:=\co_{\mathrm{WCart}^{\mathrm{HT}}}$. By \cite{bhatt-lurie-22}[Cor. 3.5.13] the global section functor 
$$R\Gamma(\mathrm{WCart}^{\mathrm{HT}},-)\simeq R\Hom_{\mathrm{WCart}^{\mathrm{HT}}}(\overline{\co},-)$$
commutes with colimits when viewed as taking values in $\cd(\bz_p)$. It follows that the functor
$$R\Hom_{\mathrm{WCart}^{\mathrm{HT}}}(\overline{\co}\{n\}/p,-)\simeq R\Hom_{\mathrm{WCart}^{\mathrm{HT}}}(\overline{\co},-\otimes\overline{\co}\{-n\})/p[-1]$$ commutes with colimits when viewed as taking values in $\Mod_\bz$ and with filtered colimits when viewed as taking values in Anima. Hence $\overline{\co}\{n\}/p$ is a compact object of $\mathcal{D}(\mathrm{WCart}^{\mathrm{HT}})$. If
$$R\Hom_{\mathrm{WCart}^{\mathrm{HT}}}(\overline{\co}\{n\}[k]/p,\mathcal{F})\simeq R\Hom_{\mathrm{WCart}^{\mathrm{HT}}}(\overline{\co}\{n\}[k],\mathcal{F})/p[-1]=0$$
then $R\Hom_{\mathrm{WCart}^{\mathrm{HT}}}(\overline{\co}\{n\}[k],\mathcal{F})=0$ by $p$-completeness and therefore $\mathcal{F}=0$ by
\cite{bhatt-lurie-22}[Prop. 3.5.15] (stating that $\mathcal{D}(\mathrm{WCart}^{\mathrm{HT}})$ is generated under shifts and colimits by $\overline{\co}\{n\}$ with $n\in\bz$). It follows that the objects $\overline{\co}\{n\}[k]/p$ are compact generators of $\mathcal{D}(\mathrm{WCart}^{\mathrm{HT}})$. 

The functor $\iota_*:\mathcal{D}(\mathrm{WCart}^{\mathrm{HT}})\to \cd(\mathrm{WCart})$ has right adjoint
$$\iota^!(\mathcal{E})\simeq R\Hom_{\co}(\overline{\co},\mathcal{E})\simeq R\Hom_{\co}((\mathcal{I}\to\co),\mathcal{E})\simeq (\mathcal{E}\to\mathcal{I}^{-1}\mathcal{E})[-1]\simeq \iota^*\mathcal{I}^{-1}\mathcal{E}[-1]$$
which commutes with all (filtered) colimits. Hence the objects $\iota_*\overline{\co}\{n\}[k]/p$ for $k,n\in\bz$ are compact objects of $\cd(\mathrm{WCart})$. They are also generators since if $\iota^!\mathcal{E}=0$ then $\iota^*\mathcal{E}=0$ and hence $\mathcal{E}=0$ by $\mathcal{I}$-completeness. It follows that $\cd(\mathrm{WCart})$ is compactly generated.

Finally, reasoning as above one shows that $\cd(\bz_p)$ is compactly generated by $\bF_p[k]$ for $k\in\bz$ (more generally, $\cd(R)$ is compactly generated by $R/p[k]$ for any animated ring $R$).

b) We have $\Fun(I,\mathscr{C})\simeq \Fun^L(\mathscr{P}(I),\mathscr{C})$ by the universal property of presheaf categories \cite{lurieHTT}[Thm. 5.1.5.6]. By \cite{lurieHTT}[Thm. 5.3.5.12] $\mathscr{P}(I)$ is compactly generated. Moreover $\mathscr{P}(I)$ is dualizable in $\mathcal{P}r^L$ with dual $\mathscr{P}(I^{op})$. Hence
$\Fun^L(\mathscr{P}(I),\mathscr{C})\simeq \mathscr{P}(I^{op})\otimes\mathscr{C}$ is  compactly generated by \cite{lurieHA}[Lemma 5.3.2.11].

c) This follows as in \cite{lurieHA}[Lemma 5.3.2.12 (3)].
\end{proof}

\begin{prop} There exist symmetric monoidal structures on $\mathscr{G}^{\mathcal N} _\mathfrak{A}$ and $\gsheaf_{\mathfrak{A}}$ such that the natural functors
$$\gsheaf_{\mathfrak{A}}\to\mathscr{G}^{\mathcal N} _\mathfrak{A}\to\cdf_\bn(\fil^{\geq\star}_{\mathcal N}F^*\co_{\mathrm{WCart}}) _\mathfrak{A} \times \cdf_\bn(\mathfrak{A})$$ 
are symmetric monoidal.
\end{prop}

\begin{proof} We note that all categories $\mathscr{M}$ in (\ref{list}) are symmetric monoidal, as is $\cd(\mathfrak{A})$, hence we obtain symmetric monoidal structures on all $\mathscr{M}_{\mathfrak{A}}$. The functor $\fil\iota^*$ is symmetric monoidal, being equivalent to scalar extension along the map of filtered, quasi-coherent algebras
$\fil^{\geq\star}_{\mathcal N}F^*\co_{\mathrm{WCart}}\to\iota_*\co_{\mathrm{WCart}^{\mathrm{HT}}}$ where the target is equipped with the trivial filtration. Similarly, $\fil\pi^*$ is symmetric monoidal, hence so is $\mathscr{G}^{\mathcal N} _\mathfrak{A}$.

The functors $F^*$, $\rho_{dR}^*$ and $\varphi^*_{\mathfrak{A}}$ are symmetric monoidal, being pullback functors associated to morphisms of stacks, $\tau_2$ is a symmetric monoidal equivalence and $\tau_1$ is again scalar extension along a map of filtered algebras, hence symmetric monoidal. Finally $\fil^{\geq 0}$ is symmetric monoidal since we restrict to $\bn^{op}$-indexed filtrations. It follows that $\gsheaf _\mathfrak{A}$ is symmetric monoidal. More concretely, one replaces all categories $\mathscr{M}$ in Def. \ref{rigorousfgauges} by their corresponding $\infty$-operads $\mathscr{M}^{\otimes}$ in order to define the symmetric monoidal structure.
\end{proof}

\end{document}